\documentclass[11pt,a4paper,reqno]{amsart}
\usepackage{amsfonts}
\usepackage{amsmath,amssymb,amsthm,amsxtra}
\usepackage{float}
\usepackage[colorlinks, linkcolor=blue,anchorcolor=Periwinkle,
    citecolor=red,urlcolor=Emerald]{hyperref}
\usepackage[usenames,dvipsnames]{xcolor}
\usepackage{geometry,array} 
\usepackage{graphicx,subfigure,bookmark,tikz,pgfplots,url,transparent,pdflscape,transparent}
 \usepackage{epstopdf}
\usepackage[normalem]{ulem}
\title[Geometric model for module categories of Dynkin quivers]
{Geometric model for module categories of Dynkin quivers via hearts of total stability conditions}
\author{Wen Chang}
\address{Wen Chang:
School of Mathematics and Statistics,
Shaanxi Normal University,
Xi'an 710062, China.}
\email{changwen161@163.com}

\author{Yu Qiu}
\address{Yu Qiu:
	Yau Mathematical Sciences Center and Department of Mathematical Sciences,
	Tsinghua University,
    100084 Beijing,
    China.
    \&
    Beijing Institute of Mathematical Sciences and Applications, Yanqi Lake, Beijing, China}
\email{yu.qiu@bath.edu}

\author{Xiaoting Zhang}
\address{Xiaoting Zhang:
    Beijing Advanced Innovation Center for Imaging Theory and Technology, Academy for Multidisciplinary Studies, Capital Normal University, Beijing 100048, China}
\email{xiaoting.zhang09@hotmail.com}

\dedicatory{Dedicated to Alastair King on the occasion of his sixtieth birthday}
\renewcommand{\floatpagefraction}{0.60}
\usetikzlibrary{matrix,positioning,decorations.markings,arrows,decorations.pathmorphing,
    backgrounds,fit,positioning,shapes.symbols,chains,shadings,fadings,calc}

\tikzset{->-/.style={decoration={  markings,  mark=at position #1 with
    {\arrow{>}}},postaction={decorate}}}
\tikzset{-<-/.style={decoration={  markings,  mark=at position #1 with
    {\arrow{<}}},postaction={decorate}}}

\usepackage[all]{xy}

\def\Serre{\mathbb{S}}
\def\xx{\mathbb{X}}
\def\QStap{\QStab^\circ}
\def\dexc{cyan!21}
\def\back{green!25}
\newcommand{\Note}[1]{\textcolor{red}{\texttt{#1}}}

\theoremstyle{plain}
\newtheorem{theorem}{Theorem}[section]

\newtheorem{lemma}[theorem]{Lemma}

\newtheorem{proposition}[theorem]{Proposition}
\newtheorem{conjecture}[theorem]{Conjecture}
\theoremstyle{definition}
\newtheorem{definition}[theorem]{Definition}

\newtheorem{example}[theorem]{Example}

\newtheorem{remark}[theorem]{Remark}

\numberwithin{equation}{section}

\def\hh{\mathcal}
\def\ha{\hh{A}}
\def\kong{\mathbb}
\def\<{\langle}
\def\>{\rangle}
\def\NN{\mathbb{N}}
\def\ZZ{\mathbb{Z}}
\def\QQ{\mathbb{Q}}
\def\RR{\mathbb{R}}
\def\bP{\mathbb{P}}

\def\Obj{\operatorname{Obj}}
\def\CC{\mathbb{C}}
\def\Add{\operatorname{Add}}
\def\Aut{\operatorname{Aut}}
\def\Ind{\operatorname{Ind}}
\def\Sim{\operatorname{Sim}}
\def\Hom{\operatorname{Hom}}
\def\hom{{\hh{H}}om}
\def\End{\operatorname{End}}
\def\Ext{\operatorname{Ext}}
\def\Irr{\operatorname{Irr}}
\def\tri{\varsigma}
\def\Stab{\operatorname{Stab}}
\def\stab{\operatorname{St}}
\def\Stap{\operatorname{Stab}^\circ}
\def\diff{\operatorname{d}}
\def\Br{\operatorname{Br}}
\def\Rep{\operatorname{Rep}}
\def\thick{\operatorname{thick}}
\def\arg{\operatorname{arg}}
\def\rank{\operatorname{rank}}
\def\PSL{\operatorname{PSL}_2(\ZZ)}
\def\deg{\operatorname{deg}}
\newcommand{\h}{\hh{H}}            
\newcommand{\nh}{\widehat{\h}}
\newcommand{\nS}{\widehat{S}}
\newcommand{\ns}{\widehat{\sigma}}
\newcommand{\nz}{\widehat{Z}}
\newcommand{\np}{\widehat{\hh{P}}}
\renewcommand{\k}{\mathbf{k}}
\renewcommand{\mod}{\operatorname{mod}}
\newcommand{\Ho}[1]{\operatorname{\bf H}_{#1}}
\newcommand{\tilt}[3]{{#1}^{#2}_{#3}}
\newcommand{\Cone}{\operatorname{Cone}}
\newcommand{\Ae}[1]{\mathcal{#1}^e}
\renewcommand{\Re}{\operatorname{Re}}

\newcommand{\bahao}{\fontsize{6.25pt}{\baselineskip}\selectfont}

\renewcommand{\labelenumi}{\theenumi{$^\circ$}.} 
\newcommand{\id}{\operatorname{id}}
\newcommand{\HH}{\operatorname{HH}}

\newcommand{\im}{\operatorname{im}}
\newcommand{\Proj}{\operatorname{Proj}}
\newcommand{\Inj}{\operatorname{Inj}}
\newcommand{\EG}{\operatorname{EG}}       
\newcommand{\EGp}{\operatorname{EG}^\circ}       
\newcommand{\shift}[1]{\operatorname{\Sigma}_{#1}}
\newcommand{\CEG}[2]{\operatorname{CEG}_{#1}(#2)}             
\newcommand{\D}{\operatorname{\hh{D}}}

\newcommand{\per}{\operatorname{per}}
\newcommand\Sph{\operatorname{Sph}}
\def\zero{\hh{H}_\Gamma}
\newcommand{\Tri}{\bigtriangleup}
\def\edge{NavyBlue}
\def\vertex{red}
\def\arrow{red}
\def\surf{\mathbf{S}}                       
\def\SS{\mathbf{S}}                       
\def\sss{\hh{S}}
\def\ssp{\underline{\hh{S}}}
\def\SC{\mathrm{S}}                       
\def\SSo{\mathbf{S}_\Tri}                       
\def\Surf{\mathbf{S}^\diamond}              
\def\xm{\Sigma}
\def\type{\Sigma_\surf^\Tri}                            
\def\xmt{\Sigma_{\T}}
\def\ts{\Sigma_{\TT}}
\def\tsx{\Sigma_{\TT}^\bullet}
\def\xms{\Sigma_{\TT}}
\def\Xm{\Sigma^\diamond_{\surf}}            
\def\Xmt{\Sigma^\diamond_{\T}}            
\newcommand\XM[1]{\Sigma^{#1}_{\surf}}      
\newcommand{\DT}{\operatorname{DT}}        
\newcommand{\ST}{\operatorname{ST}}        
\newcommand{\BT}{\operatorname{BT}}        
\def\kg{g_\xm}                              
\def\bj{b_\xm}                              
\newcommand{\AEG}{\operatorname{AEG}}       
\newcommand{\MCG}{\operatorname{MCG}}
\newcommand{\mcg}{\operatorname{MCG}_\circ}
\newcommand{\corank}{\operatorname{corank}}
\newcommand{\Int}{\operatorname{Int}}
\newcommand{\ain}{\operatorname{AI}}         
\newcommand{\gin}{\operatorname{GI}}         
\newcommand{\EGT}{\operatorname{EG}^{\times}}
\newcommand{\PP}{\Gamma_{\TT}}             
\newcommand{\HD}{\operatorname{\Psi}}   
\newcommand{\hd}{\operatorname{\Psi}}  
\newcommand\Dehn[1]{\operatorname{D}_{#1}}
\newcommand\dehn[1]{\operatorname{D}_{#1}^{-1}}
\def\sign{\operatorname{sign}}
\def\wt{\operatorname{wt}}
\def\st{\T_{\times}}
\def\Homeo{\operatorname{Homeo}}
\newcommand\coho[1]{\operatorname{H}^{#1}}
\newcommand\ho[1]{\operatorname{H}_{#1}}
\newcommand\coch[1]{\operatorname{Z}^{#1}}

\def\Esurf{\widetilde{\surf}}
\def\Est{\widetilde{\st}}
\def\Exms{\widetilde{\xms}}
\def\EDfd{\D(\EPP)}
\def\EPP{\widetilde{\PP}}
\def\coh{\operatorname{coh}}
\def\Symp{\operatorname{Symp}}
\def\Diff{\operatorname{Diff}}
\def\Fuk{\operatorname{Fuk}}
\def\Dcoh{\D^b(\coh X)}
\def\DFuk{\D^b\Fuk(M, \omega)}
\def\Ezero{\widetilde{\hh{H}_\Gamma}}
\def\TT{\mathbf{T}}
\def\T{\kong{T}}
\def\lk{\operatorname{lk}}
\def\M{\mathbf{M}}
\def\P{\mathbf{P}}
\def\cconx{\ccon^\bullet}
\def\clc{\mathbf{C}_{\TT}}
\def\Clc{\clc^\diamond}
\def\Ccon{\ccon^\diamond}
\newcommand{\CG}{\operatorname{CG}}             
\newcommand{\CCG}{\operatorname{CCG}}             
\newcommand{\rep}{\operatorname{\xi}}

\def\dian{\small{$\bullet$}}
\def\jiantou{edge[->=stealth]}
\def\jt{edge[->=stealth]}
\newcommand{\gm}{\bigtriangleup}
\newcommand{\nil}{\operatorname{nil}}
\newcommand{\AT}{\operatorname{AT}}        
\def\RHom{\operatorname{RHom}}
\def\ee{\operatorname{\mathrm{E}}}
\def\EE{\operatorname{\kong{E}}}
\def\OPR{\operatorname{\kong{R}}}
\def\OPL{\operatorname{\kong{L}}}

\def\surfo{{\mathbf{S}}_\Tri}
\def\surfx{{\mathbf{S}}^\diamond}
\def\CA{\operatorname{CA}}
\def\cA{\underline{\CA}}
\def\oA{\operatorname{OA}}
\def\HA{\operatorname{HA}}
\def\oAp{\operatorname{OA}^\circ}
\def\Z{\mathbf{Z}}
\newcommand{\Q}[1]{\mathcal{Q}(#1)}
\newcommand\Bt[1]{\operatorname{B}_{#1}}
\newcommand\bt[1]{\operatorname{B}_{#1}^{-1}}
\newcommand{\EGC}{\operatorname{EG}^{\operatorname{c}}}
\def\ff{\operatorname{\mathrel{\Big|}}}
\def\add{\operatorname{add}}
\newcommand{\MMCG}{\operatorname{MCG}_\bullet}
\newcommand{\twi}{\Psi} 

\def\filt{\operatorname{filt}}
\def\Grot{\operatorname{\mathrm{K}}}
\def\gldim{\operatorname{gldim}}
\def\XStab{\QStab}

\def\sli{\mathcal{P}}
\newcommand{\norm}[1]{\lVert #1 \rVert}
\def\hua{\mathcal}
\def\HMF{\operatorname{HMF}}
\def\tshift{\overline{\tau}}
\def\AR{\operatorname{AR}}
\def\Dwq{\D_\infty}
\def\RC{\mathcal{R}}
\def\DQ{\D_\infty(Q)}

\newcommand{\qq}[1]{\operatorname{\Gamma}_{#1}Q}
\def\cd{\operatorname{gd}_Q}
\def\Lag{\hua{L}}

\def\nn{node{$\bullet$}}
\def\ww{node[white]{$\bullet$}node[red]{$\circ$}}

\def\sun{Emerald}

\def\fblue{blue!20}
\def\fgreen{green!30}
\def\forange{red!30!orange!30!white}
\def\fcyan{Emerald!49!cyan!20!white}

\def\fire{orange!40}
\def\fires{yellow!20}
\def\ice{blue!15!cyan!50!white}
\def\ices{blue!30!cyan!10!white}

\def\parity{\varrho}

\def\roots{\Lambda}
\def\xqq{circle(.05)}

\def\gms{\surf^\lambda}
\usepackage{marvosym,wasysym}
\def\vot{\text{\Biohazard}}
\def\Vot{\text{\Cancer}}

\newcommand{\za}{\alpha}
\newcommand{\zb}{\beta}
\newcommand{\zg}{\gamma}
\newcommand{\zo}{\omega}
\newcommand{\zD}{\vartriangle}
\newcommand{\sjx}{\bigtriangleup}
\newcommand{\sbx}{\square}
\newcommand{\wbx}{\pentagon}

\newcommand{\wen}{\color{red}}
\newcommand{\gray}{\color{gray!50}}
\newcommand{\qy}[1]{\textcolor{TealBlue}{#1}}
\renewcommand{\Im}{\operatorname{Im}}
\newcommand{\imz}{{\Im}Z}
\def\UD{\mathrm{UD}}
\def\AD{\mathrm{Ad}}

\newcommand{\UP}[1]{\upharpoonleft\hskip -.05in{#1}\hskip -.05in\upharpoonright}
\begin{document}

\def\Sth{\operatorname{Stgon}}
\def\hgon{\mathbf{V}}
\def\agon{\mathbf{J}}
\def\ohgon{\overrightarrow{\hgon}}

\def\ihgon{\hgon_{\mathrm{ice}}}
\def\fhgon{\hgon_{\mathrm{fire}}}
\def\icore{\ihgon^{\text{\tiny$\copyright$}}}
\def\fcore{\fhgon^{\text{\tiny$\copyright$}}}
\def\uQ{\underline{Q}}
\def\XX{\mathfrak{M}}
\def\ToSt{\operatorname{ToSt}}
\begin{abstract}
We derive a geometric model for the module category $\mod \k Q$ of a Dynkin quiver $Q$
via the heart of a total stability condition on the bounded derived category of $\mod \k Q$.
As an application, we prove Reineke's conjecture that
there is a stability function on $\mod \k Q$ making any indecomposables stable.

    \vskip .3cm
    {\parindent =0pt
    \it Key words:
    total stability condition, Dynkin diagram, module category, geometric model, Reineke's conjecture}

\end{abstract}
\maketitle
\tableofcontents\addtocontents{toc}{\setcounter{tocdepth}{1}}


\section{Introduction}
The stability structures, e.g. King's $\theta$-stability on abelian categories and Bridgeland's stability conditions on triangulated categories, play an important role in many areas, such as, geometric invariant theory, Donaldson-Thomas theory, cluster theory, etc.

In the prequel \cite{QZ22}, Qiu-Zhang give a geometric model for the root category of the bounded derived category $\Dwq(Q)$
associated to a Dynkin diagram and describe the space $\ToSt(Q)$ of all total stability conditions on $\Dwq(Q)$.
In this paper, we describe the heart $\h_{\sigma}$ for $\sigma\in\ToSt(Q)$ on the geometric model and hence obtain a model
for the module category of certain Dynkin quiver whose underlying diagram is the same as $Q$.
In fact, we show that such a heart can realize the module category of a Dynkin quiver with any given orientation.
As an application, we prove Reineke's conjecture that, for any Dynkin quiver $Q$,
there is a stability function on $\mod \k Q$ such that all indecomposable objects are stable.

\subsection{Stability structure on abelian and triangulated categories}
An algebraic version of the notion of stability in geometric invariant theory is introduced by King \cite{K} in the representation theory of finite-dimensional algebras.
The King's $\theta$-stability is generalized by Rudakov \cite{R97} to any abelian category.
A stability function on the module category $\h(Q)\colon=\mod \k Q$ of a quiver $Q$
is a group homomorphism $Z\colon K \h(Q)\to \CC$.
The associated slope function is
\[
    \mu_Z(\alpha)=\frac{\Im Z(\alpha)}{\Re Z(\alpha)}.
\]

An object $M$ is $Z$-stable (resp. $Z$-semistable) if and only if
\[
    \mu_Z(L)<\mu_Z(M),\quad \text{(resp. $\mu_Z(L)\le\mu_Z(M)$)}
\]
for any subobject $0\neq L\subsetneq M$.

Later, motivated by $\Pi$-stability from string theory,
Bridgeland \cite{B07} introduce the notion of stability conditions on a triangulated category.
To be more precise, a stability condition $\sigma=(Z,\sli)$ on a triangulated category $\D$ consists of
a group homomorphism $Z\colon K\D\to\CC$, called the central charge,
and an $\RR$-collection of abelian subcategories $\sli(\phi)$ of $\D$,
known as the slicing, which satisfy certain conditions.
The objects in each $\sli(\phi)$ are called semistable
and the simple objects there are called stable.
Such a notion is a kind of triangulated generalization of the stability for an abelian category mentioned above.
Indeed, a stability condition $\sigma$ on $\D$ induces
a stability function $Z$ on the abelian category, i.e. its heart $\h_\sigma$.
Conversely, a stability function $Z$ on a heart $\h$ of $\D$,
satisfying the so-called Harder-Narasimhan property, is equivalent to a stability condition with heart $\h$ on $\D$.

\subsection{Reineke's conjecture and total stability}
Denote by $\Sim\h(Q)$ the set of simples in $\h(Q)$ and
$\hh{A}_Q$ be the quantum affine space
\begin{gather*}
    \kong{Q}(q^{1/2})\< y^{S} \mid S\in\Sim\h(Q) \>  \big{/}
    (y^{S_i} y^{S_j} - q^{\lambda_Q(i,j)} y^{S_j} y^{S_i} ),
\end{gather*}
where $q^{1/2}$ is an indeterminate and $\lambda_Q(i,j)=\<S_j,S_i\>-\<S_i,S_j\>$ with
$\<-,-\>$ given by the Euler form associated to $Q$.
One can obtain the completion $\widehat{\hh{A}}_Q$ of $\hh{A}_Q$
with respect to the ideal generated by $y^S, \forall S\in\Sim\h(Q)$.
Suppose that a stability function $Z$ on $\h(Q)$ is discrete,
i.e. $\mu_Z(M)\neq \mu_Z(L)$ for any indecomposable $M\not\cong L$ in $\h(Q)$.
Then an invariant $\DT(\h(Q))$ (cf. e.g. \cite{K11}) can be calculated as
\begin{gather}\label{eq:DT}
    \DT(\h)\colon=\prod^{\longleftarrow}_{ \text{$M$ stable} } \kong{E}(y^{\dim M})
    \quad \in \hh{A}_Q
\end{gather}
where the product is taking in phase-decreasing order and $\kong{E}$ is the quantum dilogarithm defined as the formal series
\[
    \kong{E}(X)=\sum_{j=0}^{\infty} \frac{q^{j^2/2}X^j}{
    \prod_{k=0}^{j-1}  (q^j-q^k)  }.
\]
The formula~\ref{eq:DT} for calculating quantum dilogarithm identities is proved by Reineke \cite{R03},
using Hall algebra and the integration map,
for $\h(Q)$ of a Dynkin quiver $Q$.
An alternative proof can be found in \cite{Q15} which uses the fundamental group of exchange graph of hearts.
In the setting of Calabi-Yau-$3$ categories,
this invariant produces the Donladson-Thomas invariants,
e.g. studied by Kontsevich-Sobliman \cite{KS08}.
In \cite{R03},
Reineke conjectures the following.

\begin{conjecture}\label{conj}
For any Dynkin quiver $Q$, there is a stability function $Z$ such that
the number of factors in the product \eqref{eq:DT} reaches maximal, i.e. equals $n\cdot h_Q/2$,
where $h_Q$ is the Coxeter number associated to $Q$.
In other words, any indecomposable object in $\h(Q)$ is $Z$-stable.
\end{conjecture}

Such a stability function $Z$ is called a total stability function.
This conjecture has been confirmed for any quiver $Q$ of type $A_n$ (with arbitrary orientation) in several papers via different approaches, such as \cite{BGMS19,AI20,HH20,K20}, and also for any quiver $Q$ of type $D_n$ in \cite{A22}.

In this paper, we prove Reineke's conjecture from the point of view of triangulated categories via geometric models.
In the triangulated setting, one also has the notion of total stability for a Bridgeland's stability condition $\sigma$, (cf. \cite{Q15, Q18}),
i.e. requiring all indecomposable objects in the triangulated category $\sigma$-stable.
The existence of such total stability conditions is rare, if and only if under the circumstance that $\D=\Dwq(Q)$ for a Dynkin diagram $Q$ (cf. \cite{KOT19,Q20}).
Clearly, a total stability condition $\sigma$ on $\D$ will induce a total stability function on its heart $\h_\sigma$.
The subtle point is to describe the heart $\h_\sigma$ so that one can obtain the module category of any Dynkin quiver (with arbitrary orientation).
The geometric description (cf. \cite{QZ22}) of the space $\ToSt(Q)$ is the key here,
that is we have a geometric model for any module category of Dynkin type.

In type $A_n$, our approach is essentially the same as the one in \cite{BGMS19}.

\begin{remark}
Note that in Reineke's original paper \cite{R03},
he takes $\Re Z(M)=\dim M$ to be the dimension function of $M$ for any $M\in\h(Q)$, which is just for simplicity.
Very recently, it is proved in \cite{M22} that, when restricted to $\Re Z(M)=\dim M$,
Reineke's conjecture fails for the linear quiver $Q$ of type $E_7$.
However, such a constrain is unnecessary from the modern point of view.
Our approach proves the correct version of Reineke's conjecture.

Another remark is that Hille-Juteau also have a proof of Reineke's conjecture (unpublished) according to \cite[\S~1]{K11}.
\end{remark}

\subsection{Geometric model for module category of a Dynkin diagram}
There are many works on geometric models for various categories in representation theory of algebras.
Such models are in some sense, certain type of Fukaya categories, cf. \cite{HKK17, OPS18, HZZ20, BS21}.
The geometric model allows one to use surface combinatorics/ topology to produce many applications, i.e.
realizing arcs as indecomposables, intersections as morphisms, (total) rotations of arcs as Auslander-Reiten translations, etc.
In particular, Qiu-Zhang \cite{QZ22} construct a geometric model for the root category $\Dwq(Q)/[2]$ of each Dynkin diagram $Q$. In details, for a Dynkin diagram $Q$,
they construct stable $h_Q$-gons $\hgon$, which provide geometric models for $\Dwq(Q)/[2]$, in the sense that each indecomposable object in $\Dwq(Q)/[2]$ is interpreted as an oriented admissible diagonal in $\hgon$, and naturally induce a central charge of $\Dwq(Q)/[2]$.
Note that by upgrading orientation ($\ZZ_2$-grading) to $\ZZ$-grading for admissible diagonals,
the model also works for the bounded derived category $\Dwq(Q)$.
Furthermore, by the construction, contractible triangles in the geometric model imply triangles in the associated triangulated categories.

We will derive a geometric model for any heart of a total stability condition in $\ToSt(Q)$
from the geometric model of the root category. In Proposition \ref{prop:obj-in-heart}, we realize each admissible upward diagonal of a stable $h_Q$-gon associated to a total stability condition $\sigma$ as an indecomposable object in the heart $\h_\sigma$. Moreover, we locate all the simples in the stable $h_Q$-gon and obtain the following.

\begin{theorem}[Theorems \ref{thm:An}, \ref{thm:Dn} and \ref{thm:En}]\label{thm:intro}
Let $\sigma \in \ToSt(Q)$ for a quiver $Q$ of Dynkin type and $\hgon_\sigma$ be the stable $h_Q$-gon of type $Q$ associated to $\sigma$.
Then there is a bijection
$$\XX\colon\UD(\hgon_\sigma) \longrightarrow \Ind \h_\sigma$$
sending an upward diagonal $\ell$ in $\hgon_\sigma$ to an indecomposable object $\XX(\ell)$ in $\h_\sigma$.

Moreover, it restricts to a bijection
\[
    \Sim\hgon_\sigma\to\Sim\h_\sigma,
\]
between the set of $\Im Z$-Ind diagonals of $\hgon_\sigma$ and the set of simples of $\h_\sigma$.
As a consequence,
$\h_\sigma$ is equivalent to the module category of the (ungraded) intersection quiver of the $\Im Z$-Ind diagonals of $\hgon_\sigma$.
\end{theorem}

Furthermore, we shall use such a model to prove Reineke's conjecture.
\begin{theorem}[Theorem~\ref{thm:R's}]
Conjecture~\ref{conj} holds.
\end{theorem}

In type $A_n$ and $D_n$, the proof bases on building stable $h_Q$-gon directly.
In type $E_{n \in\{6,7,8\}}$,
we show that essentially we only need to give $2^{n-5}$ concrete examples (of stable $h_Q$-gons) to verify the conjecture.

\subsection*{Acknowledgments}
Qy would like to thank Alastair King for many inspiring discussions,
Ryan Kinser for the invitation of giving a talk in FD Seminar that pushed for this paper
and Jan Schr\"oer for pointing out the reference \cite{A22}.
This work is supported by
National Key R\&D Program of China (No.2020YFA0713000),
Beijing Natural Science Foundation (Grant No.Z180003),
National Natural Science Foundation of China (Grant No.12031007) and
National Natural Science Foundation of China (Grant No.12101422).


\section{Preliminaries}

In this section, we will recall some conventions and constructions from \cite{QZ22}.
\subsubsection*{Conventions}
Consider the complex plane $\CC=\RR^2$.
For any two points $W_1=(x_1,y_1)$ and $W_2=(x_2,y_2)$ in $\CC$, we define
\begin{gather}\label{eq:order}
W_1 < W_2 \Longleftrightarrow
\begin{cases}
y_1<y_2 \text{\quad or}\\
y_1=y_2,\, x_1<x_2.
\end{cases}
\end{gather}
We write a vector $\overrightarrow{VW}$ by $VW$ for simplicity and
denote by $\arg(VW)$ the angle of $VW$ which takes value in $[0,2\pi)$.

Let $\hgon$ be a (labelled) $h$-gon in $\CC$ with vertices $V_0,V_1,\cdots,V_h=V_0$
and (oriented) edges $0\neq z_j=V_{j-1}V_j$, where $j\in\ZZ_h$.
Then $\hgon$ is \emph{positively convex} if all other $V_i$'s are on the left hand side of
the edge $z_j=V_{j-1}V_j$.
Note that the positively convexity implies that the vertices $V_j$'s are in anti-clockwise order.
For any integer $0\le s\le h/2$, the \emph{length-$s$ diagonals} of $\hgon$ are those $V_jV_{j+s}$.
For example, the length-$1$ diagonals are just the edges $z_j$.
For a positively convex $h$-gon $\hgon$,
the \emph{level-$s$ diagonal-gon} is the convex polygon bounded by its length-$s$ diagonals (i.e. on the left hand side of).

\begin{definition}\label{def:h-gon}
An \emph{$h$-gon $\hgon$ of type $Q$} (a Dynkin quiver) is defined respectively as:
\begin{description}
\item[$A_n\, (n\geq 1)$] an $(n+1)$-gon.
\item[$D_n\, (n\geq 4)$] a (centrally) symmetric doubly punctured $2(n-1)$-gon with punctures $B_\pm$.
\item[$E_6$] a $12$-gon satisfying:
\begin{gather}\label{eq:E6-rel}
    \begin{cases}
     z_j+z_{j+4}+z_{j+8}=0,\\
     z_{j}-z_{j-3}+z_{j-6}-z_{j-9}=0,
    \end{cases}\quad \forall j\in\ZZ_{12}.
\end{gather}

\item[$E_7$] a symmetric $18$-gon satisfying:
\begin{gather}\label{eq:E7-rel}
     z_{j}+z_{j+1}+z_{j+6}+z_{j+7}+z_{j+12}+z_{j+13}=0,\qquad\forall j\in\ZZ_{18}.
\end{gather}

\item[$E_8$] a symmetric $30$-gon satisfying:
\begin{gather}\label{eq:E8-rel}
    \begin{cases}
     z_{j}+z_{j+10}+z_{j+20}=0,\\
     z_{j}+z_{j+6}+z_{j+12}+z_{j+18}+z_{j+24}=0,
    \end{cases}\quad \forall j\in\ZZ_{30}.
\end{gather}
%
%
\end{description}
To summarize, an $h$-gon of type $Q$ is an $h_Q$-gon (satisfying above extra conditions),
where $h_Q$ is the Coxeter number associated to the underlying Dynkin diagram of $Q$.
In this case, we have $h=h_Q$. For simplicity, we prefer to write $h$ instead of $h_Q$ if there is no confusion.

\end{definition}

Let $\hgon$ be an $h$-gon of the exceptional type $E_n$ $(n=6,7,8)$. We follow the notion in \cite{QZ22} respectively:
\begin{description}
\item[$E_6$] for each $j\in\ZZ_{12}$ we can draw an triangle $\TT_j\colon=V_{j-1}V_jW_{j}$ with edges
\[
    V_{j-1}V_{j}=z_j,\quad V_jW_{j}=z_{j+4},\quad
    W_{j}V_{j-1}=z_{j+8},
\]
cf. the triangle relations in \eqref{eq:E6-rel}.

\item[$E_7$] for each $j\in\ZZ_{18}$ we can draw a hexagon
$\mathbf{L}_{j}\colon=V_{j-1}V_{j}V_{j+1} W_{j+1} U_{j} W_{j-1}$ with edges
\begin{gather*}
    V_{j-1}V_{j}=z_{j},\quad V_{j}V_{j+1}=z_{j+1},\quad
    V_{j+1} W_{j+1}=z_{j+6},\\ W_{j+1} U_{j}=z_{j+7},\quad
    U_{j} W_{j-1}=z_{j+12},\quad W_{j-1}V_{j-1}=z_{j+13},
\end{gather*}
cf. the hexagon relations in \eqref{eq:E7-rel}.
\item[$E_8$]
for each $j\in\ZZ_{30}$ we can draw
a triangle $\TT_{j}\colon=V_{j-1}V_{j}W_{j-1}$
with edges
\[V_{j-1}V_{j}=z_{j},\; V_{j}W_{j-1}=z_{j+10},\;
    W_{j-1}V_{j-1}=z_{j+20},\]
cf. the triangle relations in \eqref{eq:E8-rel}.
\end{description}
The key feature in the exceptional cases is that there are two $h/2$-gons inside $\hgon$,
i.e. the ice/fire core with vertices $W_j$ of even/odd indices.

\begin{definition}\label{def:adm-diag}
Let $\hgon$ be an $h$-gon of type $Q$.
The \emph{admissible diagonals} of $\hgon$ are defined respectively as:
\begin{description}
\item[$A_n$] any usual (oriented) diagonals of $\hgon$.
\item[$D_n$] either the length-$s$ diagonals $\pm V_jV_{j+s}=\pm V_{j+s+h/2}V_{j+h/2}$ for $1\leq s< h/2$ and $j\in \ZZ_{h}$
    or vectors $\pm V_jB_{\pm}=\pm B_{\mp}V_{j+h/2}$ for $j\in \ZZ_{h}.$
\item[$E_n$] either the length-$s$ diagonals $\pm V_{j}V_{j+s}$ for $1 \leq s \leq n-3$ and $j\in \ZZ_h$
    or vectors
$$\begin{cases}
    \pm W_{j}W_{j+2},\\
    \pm V_{j-1}W_{j+2}=\pm W_{j+1}V_{j+n-3},\\
    \pm V_{j-1}W_{j+1}=\pm W_{j+2}V_{j+n-3},
    \end{cases}\quad \forall j\in \ZZ_h.
$$
\end{description}
Here, the equalities above, together with the ones in the construction of the exceptional case, i.e.
\[
    \pm V_{j-1}V_j=\pm W_{j+4}V_{j+n-3}=\pm V_{j-2+n}W_{j-2+n},  \quad \forall j\in \ZZ_h,
\]
and
\[\begin{cases}
    \pm W_{j}W_{j+2}=\pm W_{j+8}V_{j+7}=\pm V_{j-5}W_{j-6},& \text{for $E_7$}\\
    \pm W_{j}W_{j+2}=\pm W_{j+14}V_{j+13}=\pm V_{j-10}W_{j-12},& \text{for $E_8$}
    \end{cases}\quad \forall j\in \ZZ_h.
\]
as well as (central) symmetry in type $E_7/E_8$, are considered as equivalent relations.
We denote by $\AD(\hgon)$ the set of (equivalent classes of) admissible diagonals in $\hgon$.
\end{definition}

Let $\Dwq(Q)$ be the bounded derived category of the module category of $\k Q$ and $\RC(Q)$ the corresponding root category $\Dwq(Q)/[2]$.
Denote by $\Stab(Q)$ the space of stability conditions on $\Dwq(Q)$, cf \cite{B07} for the detailed definition, and
$\ToSt(Q)$ the space of total stability conditions on $\Dwq(Q)$.
Now we recall the main result from \cite{QZ22} that the admissible diagonals give a central charge, which is a combination of Proposition~3.9, Theorems~4.5 and 6.7 there.
In fact, we slightly generalize it to obtain a bijection between the set of (equivalent classes of) admissible diagonals and the set of indecomposable objects in $\RC(Q)$.

\begin{theorem}\label{thm:geo-mod-root-cat}\cite{QZ22}
A stable $h$-gon $\hgon$ of type $Q$ gives a geometric model for $\RC(Q)$
in the sense that we have a bijection
\begin{gather}\label{eq:XAD}\begin{array}{rcl}
\XX:\AD(\hgon) &\longrightarrow& \Ind \RC(Q),\\
\ell &\mapsto& \XX(\ell)\end{array}
\end{gather}
which naturally induces a central charge $Z_\hgon: K \Dwq(Q)\to\CC$ with $Z_\hgon(\XX(\ell))=\ell$.
\end{theorem}

There are two natural actions on $\Stab(Q)$, i.e. the $\CC$-action and the $\Aut\Dwq(Q)$-action,
which also can be induced on $\ToSt(Q)$. In particular, the action of $[1]$ on $\ToSt(Q)$ coincide with the action of $1\in\CC$ on $\ToSt(Q)$.
\begin{definition}\label{def:sth-gon}
An $h$-gon $\hgon$ of type $Q$ is \emph{stable} if it is positively convex and moreover:
\begin{description}
\item[$D_n$] the two punctures $B_{\pm}$ are inside the level-$(n-2)$ diagonal-gon.
\item[$E_n$] the vertices $W_j$'s are inside the level-$(n-3)$ diagonal-gon, for $n\in\{6,7,8\}$.
\end{description}
Note that a stable $h$-gon, without stating the types explicitly, just means a positively convex $h$-gon (with above extra conditions).

Denote by $\Sth(Q)$ the moduli space of stable $h$-gons of type $Q$, where two stable $h$-gons are equivalent if and only if they are parallel, i.e. related by a translation of $\mathbb{C}$.
\end{definition}

\begin{theorem}\label{thm:iso-mfd}\cite[Thm.~1]{QZ22}
There is an isomorphism between complex manifolds
$$\ToSt(Q)/[2]\cong\Sth(Q),$$
which sends a total stability condition $\sigma$ to a stable $h$-gon $\hgon_\sigma$ of type $Q$ with central charge $Z_{\hgon_\sigma}=Z_\sigma$.
\end{theorem}

\section{Hearts and triangles}
In this section, we will describe the heart $\h_\sigma$ of $\sigma\in\ToSt(Q)$ for a Dynkin quiver $Q$
and special short exact sequence in $\h_\sigma$.
\subsection{Hearts of total stability conditions}\label{sec:hearts}

Recall that a stability condition $\sigma=(Z,\sli)$ consists of the central charge $Z$ and the slicing $\sli$,
which is a $\RR$-collection of abelian subcategory $\sli(\phi)$ for $\phi\in\RR$.
Its heart $\h_\sigma\colon=\sli[0,1)$ is defined to be the abelian category generated by $\sli(\phi)$, where $\phi\in[0,1)$.
In particular, if $\sigma$ is totally stable, then the indecomposables in $\h_\sigma$ are precisely those objects whose phases are in $[0,1)$.

Define
\begin{gather}\label{eq:order3}
\UP{VW}\colon=
\begin{cases}
VW & \text{if~}~ V<W,\\
WV & \text{if~}~ W<V.
\end{cases}
\end{gather}
Then we obtain $\arg(\UP{VW})\in[0,\pi)$.

\begin{definition}
An admissible diagonal $\ell$ of an $h$-gon $\hgon$ of type $Q$ is called \emph{upward} if $\arg(\ell)\in[0,\pi)$.
Denote by $\UD(\hgon)$ the set of upward (admissible) diagonals in $\hgon$.
\end{definition}

For any admissible diagonal $VW$, the upward version is exactly $\UP{VW}$,
where $\UP{\cdot}$ is defined in \eqref{eq:order3}.

\begin{proposition}\label{prop:obj-in-heart}
Let $\sigma\in \ToSt(Q)$ and $\hgon_\sigma$ be the corresponding stable $h$-gon of type $Q$.
Then $\XX$ in \eqref{eq:XAD} induces a bijection
\begin{equation}\label{eq:XUD}
\XX\colon\UD(\hgon_\sigma) \longrightarrow \Ind \h_\sigma.
\end{equation}
\end{proposition}
\begin{proof}
We choose $\h_\sigma[1]\cup\h_\sigma$ as a fundamental domain for $\RC(Q)$.
For an upward diagonal $\ell$, the phase of $\XX(\ell)$ is $\arg(\ell)/\pi$, which belongs to $[0,1)$.
Hence the map $\XX$ in \eqref{eq:XAD} restricts to the required bijection.
\end{proof}

Recall that the heart $\h$ in $\Dwq(Q)$ is \emph{standard}
if it is equivalent to the module category of some quiver with the same underlying graph of $Q$.
When $\sigma$ is totally stable, we have
$\Ind\Dwq(Q)=\Ind\sli_\sigma\cup\Ind\sli_\sigma^\perp$ for $\sli_\sigma=\sli[0,+\infty)$.
Thus, by \cite[Proposition 2.7]{Q15} we have the following statement.

\begin{lemma}\label{lem3.3}
The heart $\h_\sigma$ is standard if $\sigma\in\ToSt(Q)$.
\end{lemma}

\subsection{Intersections as homomorphisms}\label{subsec:int and hom}
In this section, we let $\hgon$ be an $h$-gon of type $Q$ and
will establish a correspondence between certain intersections of admissible diagonals in $\AD(\hgon)$
and the homomorphisms between the associated objects in the root category $\RC(Q)$.

\begin{definition}
An \emph{oriented intersection} from an upward diagonal $\ell_1$ to another upward diagonal $\ell_2$ is an anticlockwise angle $\za$ from $\ell_1$ to $\ell_2$ locally at an intersection $P\in \ell_1\cap \ell_2$ such that $\za$ is in the interior of $\hgon$.
\end{definition}

The opposite angles are considered equivalent if $P$ is an interior intersection, cf. Figure~\ref{fig:def-or-inter}.

For an oriented intersection $\za$ from $\ell_1$ to $\ell_2$, we set
\begin{gather}\label{eq:int-index}
\rho(\za)\colon=
\begin{cases}
0 & \text{if~} arg(\ell_1)<arg(\ell_2),\\
1 & \text{if~} arg(\ell_2)<arg(\ell_1),
\end{cases}
\end{gather}
called the \emph{intersection index} of $\za$. Note that $arg(\ell_1)\neq arg(\ell_2)$ if $\ell_1$ and $\ell_2$ intersect.
For instance, in Figure~\ref{fig:def-or-inter}, $\za$ is an oriented intersection from $\ell_1$ to $\ell_2$ with index $\rho(\za)=0$,
while $\zb$ is an oriented intersection from $\ell_2$ to $\ell_1$ with index $\rho(\zb)=1$.

\begin{figure}
\begin{tikzpicture}[>=stealth,scale=0.5]
\draw [very thick,blue,<-] (2,2)--(-4.5,-4.5);
\draw [very thick,blue,<-](-3,3)--(3,-3);

\draw [bend right,->,thick] (.5,.5)to(-0.5,.5);
\node at (0,1.4) {$\za$};
\draw [bend left,<-,thick] (.5,-.5)to(-0.5,-.5);
\node at (0,-1.4) {$\za$};

\draw [bend left,<-,thick] (.5,.5)to(0.5,-.5);
\node at (1.4,0) {$\zb$};
\draw [bend right,->,thick] (-.5,.5)to(-0.5,-.5);
\node at (-1.4,0) {$\zb$};
%
\node at (-3,-2) {$\ell_1$};
\node at (3,-2) {$\ell_2$};
\node at (0,0) {\tiny$\bullet$};
\end{tikzpicture}
\begin{center}
\caption{Two oriented intersections $\za$ and $\zb$}\label{fig:def-or-inter}
\end{center}
\end{figure}
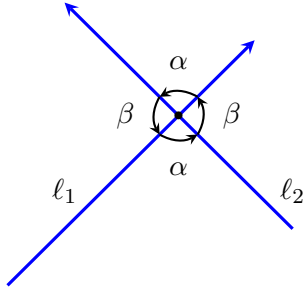

\begin{lemma}\label{lem:tri-vs-tri}
Assume that there is a contractible triangle formed by upward diagonals $\ell_1$, $\ell_2$ and $\ell_3$ in $\hgon$
such that $\ell_3=\ell_1+\ell_2$ and $arg(\ell_1) < arg(\ell_2)$.
Then we have the following triangle in $\RC(Q)$:
\begin{gather}\label{eq:tr}
    \XX(\ell_1)  \xrightarrow{\za} \XX(\ell_3)\xrightarrow{\zb} \XX(\ell_2)  \xrightarrow{\zg} \XX(\ell_1)[1],
\end{gather}
where morphisms correspond to oriented intersections.

Moreover, such a triangle is a short exact sequence when restricted to $\h_\sigma$:
\[ 0  \xrightarrow{} \XX(\ell_1)  \xrightarrow{\za} \XX(\ell_3)\xrightarrow{\zb} \XX(\ell_2)  \xrightarrow{} 0.\]
\end{lemma}

\begin{figure}[H]
\begin{tikzpicture}[>=stealth,yscale=0.8,scale=.8]
\draw [blue, very thick,->] (3,-3)--(0,0);
\draw [blue, very thick,->](0,0)--(1,2);
\draw [blue, very thick,->](3,-3)--(1,2);

\draw [thick, bend left,<-] (.2,.35)to(0.3,-.3);
\node at (.6,0) {$\zg$};

\draw [thick, bend right,->] (.8,1.6)to(1.2,1.6);
\node at (1,1.2) {$\za$};

\draw [thick, bend right,->] (2.6,-2)to(2.3,-2.3);
\node at (2.2,-1.7) {$\zb$};

\node at (.2,1.2) {$\ell_1$};
\node at (1,-1.5) {$\ell_2$};
\node at (2.2,0) {$\ell_3$};
\node at (-.5,0) {$P$};
\node at (0,0) {$\bullet$};
\node at (1,2) {$\bullet$};
\node at (3,-3) {$\bullet$};
\node at (6,0) {};
\end{tikzpicture}
\begin{tikzpicture}[>=stealth,yscale=0.8,scale=.8]
\draw [blue, very thick,->] (-3,-3)--(0,0);
\draw [blue, very thick,->](0,0)--(-1,2);
\draw [blue, very thick,->](-3,-3)--(-1,2);

\draw [thick,bend right,->] (-.2,.35)to(-0.3,-.3);
\node at (-.6,0) {$\zg$};

\draw [thick,bend left,<-] (-.8,1.6)to(-1.2,1.6);
\node at (-1,1.1) {$\zb$};

\draw [thick,bend left,<-] (-2.6,-2)to(-2.3,-2.3);
\node at (-2.2,-1.7) {$\za$};

\node at (-.3,1.3) {$\ell_2$};
\node at (-1,-1.5) {$\ell_1$};
\node at (-2.2,0) {$\ell_3$};
\node at (.5,0) {$P$};
\node at (0,0) {$\bullet$};
\node at (-1,2) {$\bullet$};
\node at (-3,-3) {$\bullet$};
\end{tikzpicture}
\begin{center}
\caption{Contractible triangle of upward diagonals}\label{fig:lem-tri-heart}
\end{center}
\end{figure}
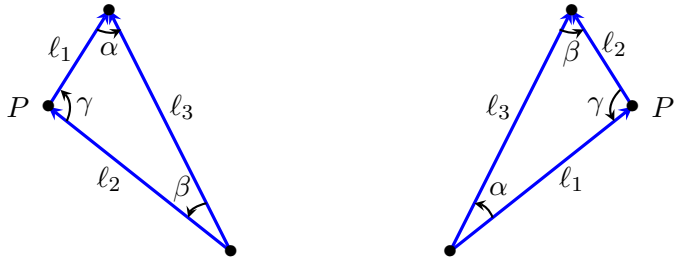
\begin{proof}
There are two cases of the contractible triangle depending on the positions of $\ell_1$ and $\ell_2$, see Figure~\ref{fig:lem-tri-heart},
but the triangle \eqref{eq:tr} will always be as above (where the order of objects depends on $\arg$ of $\ell_i$'s).

The statement follows from \cite{HKK17,OPS18} and \cite{QZZ} for type $A$ and type $D$, respectively.
As for the exceptional cases, it can be checked straightforwardly,
cf. \cite[\S~5]{QZ22} for couple forms of the triangles.
\end{proof}

\begin{remark}
In general, oriented intersections between admissible diagonals (arcs) induce (a basis of) morphisms between the corresponding objects
and the index of an intersection corresponds to the degree of the morphism.
More precisely, one has (cf. \cite{HKK17,OPS18,QZZ} or \cite[(6.7)]{IQZ20})
\begin{gather}\label{eq:int=dim}
    \Int^d(\ell_1,\ell_2)=\dim\Hom^d(\XX(\ell_1),\XX(\ell_2))
\end{gather}
in the setting of (graded skew) gentle algebras.

Moreover, the mapping cone of such a morphism usually (e.g. when the intersection is at the boundary) corresponds
to the diagonals obtained from the union of the original two diagonals by smoothing out the oriented intersection.
\end{remark}

\begin{definition}
For two upward diagonals $l_1$ and $l_2$,
there is an \emph{essential oriented intersection} from $l_2$ to $l_1$ if
there is a contractible triangle with edges (equivalent to) $l_1$ and $l_2$,
as shown in one of the pictures in Figure~\ref{fig:lem-tri-heart}.
\end{definition}

Such a notion will be used to define the intersection quiver of an $h$-gon of type $Q$ in Section~\ref{sec:hearts}.
Note that there is at most one essential oriented intersection between two upward diagonals.
In type $A_n$, any oriented intersection at the boundary between two upward diagonals is essential.
However, this is not true in other types.
For instance, in type $D_n$, the oriented intersection of upward diagonals $V_jB_-$ and $V_jB_+$ is not essential since $B_-B_+$ is not an admissible diagonal.

\section{Geometric model of hearts}\label{sec:geo}
In this section, we give algorithms to find admissible diagonals corresponding to simples in the hearts of total stability conditions
and hence obtain all the information (that we need) of the hearts in the geometric model.
\subsection{Type $A_n$}\label{subsec:type An}
Let $\hgon$ be a stable $(n+1)$-gon in $\CC$. We reorder the vertices $V_j$ of $\hgon$ as
\[
    Y_1 < Y_2 < \cdots <Y_{n+1}
\]
with respect to the order \eqref{eq:order}, where $Y_i=(x_i,y_i)$.
Note that there are at most two $y_i$'s coincide.
Then an upward diagonal in $\hgon$ is of the form $Y_iY_j$ with $1 \leq i<j\leq n+1$.

\begin{definition}\label{def:InZ-ind:A}
The diagonals $s_i:=Y_{i}Y_{i+1}$ are called \emph{$\imz$-ind diagonals}, for $1 \leq i \leq n$.
We denote by
$$\Sim \hgon \colon=\{ s_i \mid 1\leq i\leq n\}.$$
\end{definition}

Note that any upward diagonal $Y_iY_j$ ($1 \leq i<j\leq n+1$) uniquely decomposes into the sum:
\begin{gather}\label{eq:decom}
    Y_iY_j=Y_iY_{i+1}+\cdots+Y_{j-1}Y_j=s_i+s_{i+1}\cdots+s_{j-1},
\end{gather}
with respect to $y$-coordinate (i.e. the imaginary part of the central charge).

\begin{lemma}\label{lem:filtration-An}
Let $\sigma \in \ToSt(A_n)$ and $Y_iY_j\in\UD(\hgon_\sigma)$ with $\hgon_\sigma$ being the associated stable $(n+1)$-gon.
The decomposition \eqref{eq:decom} induces a filtration of short exact sequences for $\XX(Y_iY_j)$ in the heart $\h_\sigma$ with factors
$\XX(s_k)$ for $i\leq k\leq j-1$.
\end{lemma}
\begin{proof}
Note that $Y_iY_j$, $Y_iY_{j-1}$ and $s_{j-1}=Y_{j-1}Y_j$ form a contractible triangle.
By Lemma \ref{lem:tri-vs-tri}, there exists a short exact sequence (in $\h_\sigma$) either of the form
\[ 0  \xrightarrow{} \XX(Y_iY_{j-1})  \xrightarrow{} \XX(Y_iY_j)\xrightarrow{} \XX(s_{j-1})  \xrightarrow{} 0,\]
if $arg(Y_iY_{j-1})<arg(s_{j-1})$, or
\[ 0  \xrightarrow{} \XX(s_{j-1})  \xrightarrow{} \XX(Y_iY_j)\xrightarrow{} \XX(Y_iY_{j-1})  \xrightarrow{} 0,\]
if $arg(s_{j-1})<arg(Y_iY_{j-1})$.
Thus the required filtration can be constructed inductively.
\end{proof}

\begin{definition}\label{def:An-quiver}
The \emph{intersection quiver} $\hh{Q}(\hgon)$ of $\hgon$ is the graded quiver whose
\begin{itemize}
  \item vertices are $\imz$-ind diagonals and
  \item arrows are essential oriented intersections with gradings given by intersection indices.
\end{itemize}
\end{definition}

For example, one can draw $\hh{Q}(\hgon)$ for a stable $6$-gon $\hgon$ of type $A_5$ as in Figure~\ref{fig:tt}.
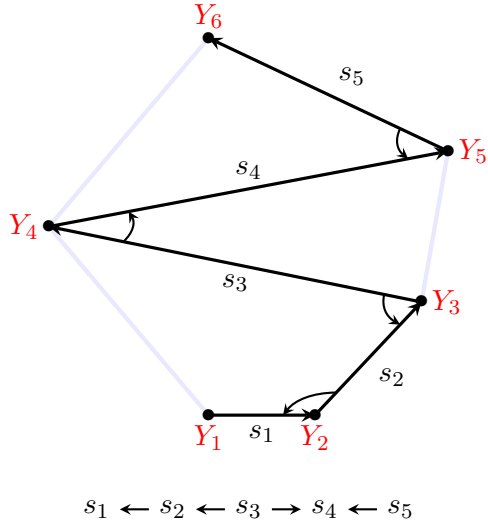
\begin{figure}[ht]\centering
\begin{tikzpicture}[xscale=.7,yscale=.5]
\path[red] (-1,0) coordinate (t1) node[below]{$Y_1$};
\path[red] (1,0) coordinate (t2) node[below]{$Y_2$};
\path[red] (3,3) coordinate (t3) node[right]{$Y_3$};
\path[red] (3.5,7) coordinate (t5)node[right]{$Y_5$};
\path[red] (-1,10) coordinate (t6)node[above]{$Y_6$};
\path[red] (-4,5) coordinate (t4)node[left]{$Y_4$};
\draw[blue!30,ultra thick,opacity=.3](t1)to(t4);
\draw[blue!30,ultra thick,opacity=.3](t3)to(t5);
\draw[blue!30,ultra thick,opacity=.3](t4)to(t6);
\foreach \j in {1,2,3,4,5,6}
{\draw (t\j) node {$\bullet$};}
\foreach \j/\i in {1/2,2/3,3/4,4/5,5/6}
{\draw[->,>=stealth,very thick,black] (t\j) to (t\i);}
\draw[black]
    ($(t1)!.5!(t2)$) node[below] {$s_1$}
    ($(t3)!.5!(t2)$) node[below right] {$s_2$}
    ($(t3)!.5!(t4)$) node[below] {$s_3$}
    ($(t5)!.5!(t4)$) node[above] {$s_4$}
    ($(t5)!.5!(t6)$) node[above right] {$s_5$};

\draw     [bend right,->,>=stealth, thick]($(t2)!.2!(t3)$) to ($(t2)!.3!(t1)$);
\draw    [bend right,->,>=stealth, thick]($(t3)!.1!(t4)$) to ($(t3)!.2!(t2)$);
\draw     [bend right,->,>=stealth, thick]($(t4)!.2!(t3)$) to ($(t4)!.2!(t5)$);
\draw     [bend right,->,>=stealth, thick]($(t5)!.2!(t6)$) to ($(t5)!.1!(t4)$);
\end{tikzpicture}

\begin{tikzpicture}[scale=.5,xscale=1,ar/.style={->,thick,>=stealth}]
\draw(0,0)node(v1){$s_1$}(2,0)node(v2){$s_2$}(4,0)node(v3){$s_3$}
(6,0)node(v4){$s_4$}(8,0)node(v5){$s_5$};

\draw[ar](v2)to(v1);
\draw[ar](v3)to(v2);
\draw[ar](v3)to(v4);
\draw[ar](v5)to(v4);
\draw(0,1)node{};
\end{tikzpicture}
\caption{The intersection quiver of a $6$-gon of type $A_5$}\label{fig:tt}
\end{figure}

Note that any oriented intersection at the boundary is essential in a stable $(n+1)$-gon.
Denote by $\hh{Q}^{1-?}(\hgon)$ the quiver obtained from $\hh{Q}(\hgon)$ by grading change $?\mapsto(1-?)$.
We will show that the quiver $\hh{Q}^{1-?}(\hgon)$ will have trivial grading for all Dynkin case.

Recall from \cite{KQ} that the \emph{Ext-quiver} $\hh{Q}(\h)$ of a (finite length) heart $\h$ is the graded quiver
whose vertices are the simples of $\h$
and whose graded arrows from $S$ to $T$ are given by a basis of $\Ext^\bullet(S,T)$.

Let $\Sim\hua{A}$ denote the set of simples of an abelian category $\hua{A}$.
As a consequence, we have the following.

\begin{theorem}\label{thm:An}
Let $\sigma \in \ToSt(Q)$ with heart $\h_\sigma$ and the corresponding stable $h$-gon $\hgon_\sigma$ of type $Q$.
If $Q$ is of type $A_n$,
then the bijection \eqref{eq:XUD} restricts to a bijection
$$\XX\colon\Sim\hgon_\sigma\to\Sim\h_\sigma.$$
Moreover, the Ext-quiver $\hh{Q}(\h_\sigma)$ coincides with the intersection quiver $\hh{Q}(\hgon_\sigma)$.
Thus, we have
\begin{gather}\label{eq:Heart}
    \h_\sigma\cong\h(\hh{Q}^{1-?}(\hgon_\sigma)).
\end{gather}
Recall that $\h(\hh{Q}^{1-?}(\hgon_\sigma))=\mod \k \hh{Q}^{1-?}(\hgon_\sigma)$.
\end{theorem}
\begin{proof}
Any object in $\h_\sigma$ admits a filtration with factors in $\{\XX(s) \mid s\in\Sim\hgon_\sigma\}$.
Hence we have
$$\Sim\h_\sigma\subset\{\XX(s) \mid s\in\Sim\hgon_\sigma\}.$$
Noticing that both sets have $n$ elements, they must coincide.
The intersection quiver matches the $\Ext$-quiver follows from
the general formula \eqref{eq:int=dim}.

The final statement, i.e. \eqref{eq:Heart}, follows from the standard simple-projective duality: $\Ext^1(S_i,S_j)\cong\Irr(P_j,P_i)^*$.
In particular, since the heart $\h_\sigma$ is hereditary by Lemma \ref{lem3.3}, the grading of each arrow in $\hh{Q}(\h_\sigma)$ is one. Thus $\hh{Q}^{1-?}(\hgon_\sigma)$ has trivial grading.
\end{proof}
\begin{remark}
By Theorem~\ref{thm:An}, the filtration in Lemma \ref{lem:filtration-An} is exactly the Jordan-H\"{o}lder filtration of $\XX(Y_iY_j)$ in $\h_\sigma$.

If we forget the central charges in the theorem above and treat $\hgon$ as
a topological polygon, then it becomes the result of \cite{BGMS19}.
\end{remark}

\subsection{Type $D_n$}
Let $\hgon$ be a stable $2(n-1)$-gon of type $D_n$ with center $O=(0,0)$ in $\CC$ and
$B_{\pm}=(\pm x_B,\pm y_B)$ the two punctures.
Without loss of generality,
we assume that $B_-\le B_+$, i.e. either $B_-< B_+$ or $B_-=O=B_+$.

Reorder the points in $\{V_j, B_\pm\mid j\in\ZZ_{2(n-1)}\}$ as
\begin{gather}\label{eq:orderD}
    Y_{-n} < \cdots Y_{-1} \le {\gray{O}} \le Y_1 < \cdots Y_n
\end{gather}
with respect to \eqref{eq:order}, where $Y_i=(x_i,y_i)\in\CC$.
The symmetry of $\hgon$ implies that $Y_{-i}=-Y_i$.
Also note that the two equalities in~\eqref{eq:orderD} can only hold simultaneously and, in which case,
we have $Y_{\pm1}=O=B_{\pm}$.
The stability of $\hgon$ implies that
there are at least $(n-2)$ $V_j$'s, which are less then $B_-$.
Hence $B_\pm$ must be either $Y_{\pm1}$ or $Y_{\pm2}$.
In such notation, an upward diagonal in $\hgon$ is of the form $Y_iY_j$ such that $-n \leq i < j \leq n$ and $i+j\neq 0$.

Denote by
\begin{itemize}
\item $\hgon_n^{-}$ the convex hull of vertices in $\{Y_{-n},\cdots, Y_{-2}, Y_{-1}\}$ and
\item $\hgon_n^{+}$ the convex hull of vertices in $(\{Y_{-n},\cdots, Y_{-2}, Y_{-1}\}\setminus \{B_-\})\cup \{B_+\}$.
\end{itemize}

\begin{lemma}
Both $\hgon_n^{-}$ and $\hgon_n^{+}$ are stable $n$-gons of type $A_{n-1}$.
\end{lemma}
\begin{proof}
Consider the polygon with $(n-1)$ consecutive vertices $V_{j+1},\ldots V_{j+(n-1)}$,
which inherit the positive convexity of $\hgon$.
By the stability of $\hgon$, $B_\pm$ is bounded in the parallelogram with vertices $V_j, V_{j+1}, V_{j+(n-1)}, V_{j+n}$.
Thus, the polygon with vertices $V_{j+1},\ldots V_{j+(n-1)}$ and $B_\epsilon$ is positively convex, for $\epsilon\in\{+,-\}$.

In particular, $\hgon_n^{\pm}$ are both positively convex
since the vertices of $\hgon^\pm_n$ except $B_\pm$ are $(n-1)$ consecutive vertices of $\hgon$.
\end{proof}

Set $s_i\colon=Y_{i-n-1}Y_{i-n}$ for $1 \leq i \leq n-1$.
Then, by Definition~\ref{def:InZ-ind:A}, we have
\begin{gather}\label{eq:SimDn}
\Sim\hgon_n^-=\{ s_i \mid 1\leq i\leq n-1\},\quad \text{and}
\\ \label{eq:SimDn2}
    \Sim\hgon_n^+=\begin{cases}
    \{s_i \mid 1 \leq i \leq n-2\}\cup\{ Y_{-2}B_+\}  & \text{if $B_-=Y_{-1}$},\\
    \{s_i \mid 1 \leq i \leq n-3\}\cup\{ Y_{-3}Y_{-1}, Y_{-1}B_+\}  & \text{if $B_-=Y_{-2}$}.
    \end{cases}
\end{gather}

\begin{definition}\label{def:Dn-sim}
For type $D_n$,
the set $\Sim \hgon$ of \emph{$\imz$-ind diagonals} is defined to be $\{s_i\mid 1\le i\le n-1\}$ together
with a diagonal
\begin{gather}\label{eq:extra-diag}
s_{n}=
\begin{cases}
Y_{-2}B_+=B_-Y_2 & \text{if~}~ B_\pm=Y_{\pm1},\\
Y_{-1}B_+=B_-Y_1 & \text{if~}~ B_\pm=Y_{\pm2}.
\end{cases}
\end{gather}
\end{definition}

\begin{lemma}\label{lem:filtration}
Let $\sigma \in \ToSt(D_n)$ and $Y_iY_j\in\UD(\hgon_\sigma)$ with $\hgon:=\hgon_\sigma$ being the associated stable $2(n-1)$-gon.
Then $Y_iY_j$ decomposes into a sum of $\imz$-ind diagonals in $\Sim\hgon$,
which induces a filtration of short exact sequences for $\XX(Y_iY_j)$ in the heart $\h_\sigma$ with factors $\XX(s), s\in\Sim\hgon$.
\end{lemma}
\begin{proof}
We only prove the case that $B_\pm=Y_{\pm1}$ as the argument for $B_\pm=Y_{\pm2}$ is similar.

Note that there are the following four possibilities of the indices $i,j$:
\begin{itemize}
\item[$($a$)$] $i<j\le -1$;
\item[$($b$)$] $i<-1<1\le j$;
\item[$($a$')$] $j>i\geq 1$;
\item[$($b$')$] $j>1>-1\geq i$,
\end{itemize}
where the latter two cases are the dual of the former two since $Y_iY_j=Y_{-j}Y_{-i}$. It suffices to discuss the former two cases.
To be more precise,
\begin{itemize}
\item[$($a$)$] if $i<j\leq-1$, then $Y_iY_j$ is an upward diagonal in $\hgon_n^-$ and thus
a sum of elements in $\Sim\hgon_n^-(\subset \Sim\hgon)$ by Lemma~\ref{lem:filtration-An};
\item[$($b1$)$] if $i<-1$ and $j=1$, then $Y_iY_j=Y_iB_+$ is an upward diagonal in $\hgon_n^+$ and thus a sum of elements in $\Sim\hgon_n^+(\subset \Sim\hgon)$ by Lemma~\ref{lem:filtration-An};
\item[$($b2$)$] if $i<-1<1< j$, then we have
\[Y_iY_j = Y_iB_-+B_-Y_j= Y_iB_-+Y_{-j}B_+.\]
Note that $Y_iY_{-1}\in\UD(\hgon_n^-)$ and $Y_{-j}B_+\in\UD(\hgon_n^+)$.
Hence, combining the decompositions in the two situations above, we obtain a required decomposition.
\end{itemize}

Furthermore, similar to the proof of Lemma \ref{lem:filtration-An},
we can iteratively use Lemma \ref{lem:tri-vs-tri} to show that the sum gives rise to a filtration of short exact sequences whose factors are $\imz$-ind diagonals.
\end{proof}

Following the Definition \ref{def:An-quiver}, we can define the \emph{intersection quiver} $\hh{Q}(\hgon)$ for a stable $2(n-1)$-gon $\hgon$ of type $D_n$.
Note that the arrows in $\hh{Q}(\hgon)$ arise from essential oriented intersections.
See Figure~\ref{fig:int-quiver-D5} for an example,
that there is no arrow from $Y_{-2}Y_1$ to $Y_{-2}Y_{-1}$ as the oriented intersection at $Y_{-2}$ is not essential.

\begin{figure}[ht]\centering
\begin{tikzpicture}[scale=2, rotate=-28, xscale=1,arrow/.style={->,>=stealth,thick}]
\clip[rotate=30] (-2,-2) rectangle (2.2,2);
\path (0,0) coordinate (O)
      (40:1.12) coordinate (v1)
      (78:1.3) coordinate (v2)
      (123:1.2) coordinate (v3)
      (172:1.7) coordinate (v4);
\draw[cyan, ultra thick]
      ($(O)-.5*(v1)-.5*(v2)-.5*(v3)-.5*(v4)$) coordinate (V0) coordinate (V8)  coordinate (W5)
      ($(V0)+(v1)$) coordinate (V1) coordinate (W6)
      ($(V1)+(v2)$) coordinate (V2) coordinate (W7)
      ($(V2)+(v3)$) coordinate (V3) coordinate (W8) coordinate (W0)
      ($(V3)+(v4)$) coordinate (V4) coordinate (W1);
\path ($(O)-(V1)$) coordinate (V5) coordinate (W2)
      ($(O)-(V2)$) coordinate (V6) coordinate (W3)
      ($(O)-(V3)$) coordinate (V7) coordinate (W4);
\draw[cyan, ultra thick, fill=cyan!10](V0)
    \foreach \j in {1,...,8} {to(V\j)};
\foreach \j in {0,...,7}
    {\draw[white,fill=white](V\j)to($3*(V\j)$)to($3*(W\j)$)to(W\j);}
\foreach \j in {0,...,7} {\draw[Green!20, very thick](V\j)to(W\j);
\draw[thin,gray,dotted](V\j)to($(O)-(V\j)$)
     ($1.1*(V0)$)node[red]{\footnotesize{$Y_{-5}$}}
     ($1.1*(V1)$)node[red]{\footnotesize{$Y_{-4}$}}
     ($1.13*(V2)$)node[red]{\footnotesize{$Y_{-2}$}}
     ($1.1*(V3)$)node[red]{\footnotesize{$Y_{3}$}}
     ($1.1*(V4)$)node[red]{\footnotesize{$Y_{5}$}}
     ($1.1*(V5)$)node[red]{\footnotesize{$Y_{4}$}}
     ($1.1*(V6)$)node[red]{\footnotesize{$Y_{2}$}}
     ($1.1*(V7)$)node[red]{\footnotesize{$Y_{-3}$}};}
\draw[cyan!50, ultra thick](V0)
    \foreach \j in {1,...,8} {to(V\j)\nn };
\draw[dotted](125:.3) coordinate
    (p) \nn ($1.5*(p)$) node[red]{\footnotesize{$Y_1$}}to[thin,gray,dotted]
    ($(O)-(p)$)\nn ($(O)-1.5*(p)$) node[red]{\footnotesize{$Y_{-1}$}};

    \foreach \j/\i in {0/1,1/7,7/2}
{\draw[->,>=stealth,very thick,black] (V\j) to (V\i);}
\draw[->,>=stealth,very thick,black] (V2) to (p);
\draw[->,>=stealth,very thick,black] (V2) to ($(O)-(p)$);

\draw     [bend right,->,>=stealth, thick]($(V1)!.15!(V7)$) to ($(V1)!.3!(V0)$);
\draw     [bend right,->,>=stealth, thick]($(V7)!.2!(V1)$) to ($(V7)!.15!(V2)$);
\draw     [bend right,->,>=stealth, thick]($(V2)!.25!(p)$) to ($(V2)!.15!(V7)$);
\draw     [bend right,->,>=stealth, thick]($(V2)!.5!($(O)-(p)$)$) to ($(V2)!.25!(V7)$);

\draw[black]
    ($(V0)!.5!(V1)$) node[below] {$s_1$}
    ($(V1)!.5!(V7)$) node[below] {$s_2$}
    ($(V2)!.5!(V7)$) node[below] {$s_3$}
    ($(p)!.5!(V2)$) node[above] {$s_5$}
    ($($(O)-(p)$)!.3!(V2)$) node[above] {$s_4$};

\end{tikzpicture}

\begin{tikzpicture}[scale=.5,xscale=1,ar/.style={->,thick,>=stealth}]
\draw(0,0)node(v1){$s_1$}(2,0)node(v2){$s_2$}(4,0)node(v3){$s_3$}
(6,-1)node(v4){$s_4$}(6,1)node(v5){$s_5$};

\draw[ar](v2)to(v1);
\draw[ar](v2)to(v3);
\draw[ar](v4)to(v3);
\draw[ar](v5)to(v3);
\end{tikzpicture}
\caption{The intersection quiver of an $8$-gon of type $D_5$.}
\label{fig:int-quiver-D5}
\end{figure}
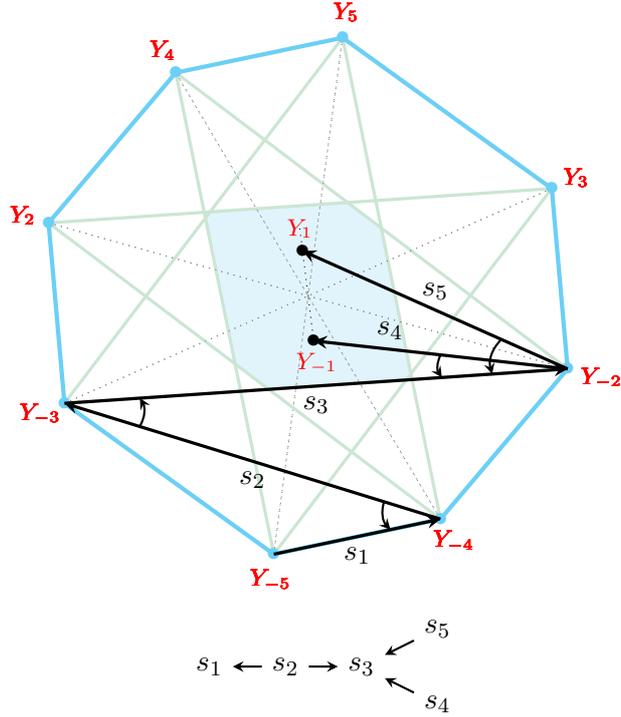

As a consequence, we have the following theorem.

\begin{theorem}\label{thm:Dn}
Theorem~\ref{thm:An} holds for any quiver $Q$ of type $D_n$.
In particular, the underlying diagram of the Ext-quiver/intersection quiver is of the form:
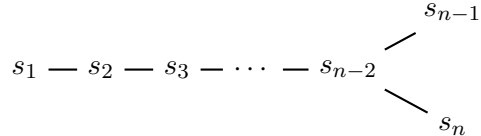
\begin{center}
\begin{tikzpicture}[scale=.5,xscale=1,ar/.style={-,thick}]
\draw(0,0)node(v1){$s_1$}(2,0)node(v2){$s_2$}(4,0)node(v3){$s_3$}
(6,0)node(v4){$\cdots$}(8.5,0)node(vn2){$s_{n-2}$}
(11.25,1.5)node(vn1){$s_{n-1}$}
(11.25,-1.5)node(vn){$s_n$};
\draw[ar](v2)edge(v1)edge(v3);
\draw[ar](v4)edge(v3)edge(vn2);
\draw[ar](vn2)edge(vn1)edge(vn);
\end{tikzpicture}
\end{center}
\end{theorem}
\begin{proof}
Mutatis mutandis the proof of Theorem~\ref{thm:An}.
\end{proof}

\subsection{Exceptional type $E_6$}\label{s4.3}

Let $\hgon$ be a stable $12$-gon of type $E_6$ in $\CC$.
Reorder the points in $\{V_j, W_j\mid j \in \ZZ_{12}\}$ with respect to \eqref{eq:order} as $\{Y_i \mid 1\le i\le 24\}$ such that $Y_i<Y_{i+1}$ for all $i$.
The stability of $\hgon$ implies that there are at least three $V_j$'s that are less than all $W_j$'s. In other words,
$Y_1,Y_2,Y_3$ must all be $V_j$'s and thus are consecutive vertices of $\hgon$.
Without loss of generality, we can assume that $\{Y_1,Y_2,Y_3\}=\{V_1,V_2,V_3\}$.
Denote by $\sjx_-$ the triangle with vertices $V_1$, $V_2$ and $V_3$.

\begin{remark}\label{rem:e6only}
Similarly, one can deduce that $Y_{22},Y_{23},Y_{24}$ are consecutive vertices of $\hgon$ and they are vertices in
$\{V_7,V_8,V_9\}$. We leave it as an exercise to the reader.
Denote by $\sjx_0$ the triangle with vertices $V_7$, $V_8$ and $V_9$.
\end{remark}

Consider the two triangles (cf. Figure~\ref{fig:E6sjx}) respectively:
\begin{itemize}
\item $\sjx_L$ with vertices $W_3$, $W_1$ and $V_0$,
\item $\sjx_R$ with vertices $V_4$, $W_4$ and $W_2$.
\end{itemize}

\begin{figure}[th]\centering
\definecolor{ffqqqq}{rgb}{1.,0.,0.}
\definecolor{qqffff}{rgb}{0.,1.,1.}
\definecolor{qqffqq}{rgb}{0.,1.,0.}
\definecolor{ffcctt}{rgb}{1.,0.8,0.2}
\definecolor{qqccqq}{rgb}{0.,0.8,0.}
\definecolor{xdxdff}{rgb}{0.49019607843137253,0.49019607843137253,1.}
\definecolor{qqzzff}{rgb}{0.,0.6,1.}
\definecolor{ffwwqq}{rgb}{1.,0.4,0.}
\definecolor{rvwvcq}{rgb}{0.08235294117647059,0.396078431372549,0.7529411764705882}
\begin{tikzpicture}[scale=.4,arrow/.style={->,>=stealth,>=stealth,thick}]
\clip(5.13890508660171,-2.3781855621584524) rectangle (31.252915609287367,23.351281298211735);
\fill[fill=ffcctt,fill opacity=0.5] (16.671906096116764,20.73388171029249) -- (10.289687300774663,17.99411050563459) -- (21.557116438041643,21.7494300041815) -- cycle;

\fill[line width=1.2pt,color=qqffqq,fill=qqffqq,fill opacity=0.5] (15.,5.) -- (8.617781204657899,2.2602287953421007) -- (19.88521034192488,6.015548293889011) -- cycle;

\fill[line width=1.2pt,color=qqffqq,fill=qqffqq,fill opacity=0.5] (24.736432434219985,6.196642251719948) -- (18.354213638877884,3.456871047062049) -- (29.621642776144864,7.212190545608959) -- cycle;
\fill[line width=1.2pt,color=qqffff,fill=qqffff,fill opacity=0.4]  (26.263250812189575,0.8399102176574627) -- (14.941845314902121,-0.058154685097879966) -- (20.098854191814286,-0.84657565018107) -- cycle;
\draw [line width=1.2pt,color=ffwwqq] (15.,5.)-- (14.941845314902121,-0.058154685097879966);
\draw [line width=1.2pt,color=ffwwqq] (14.941845314902121,-0.058154685097879966)-- (8.617781204657899,2.2602287953421007);

\draw [line width=1.2pt,color=qqzzff] (14.941845314902121,-0.058154685097879966)-- (18.354213638877884,3.456871047062049);
\draw [line width=1.2pt,color=qqzzff] (18.354213638877884,3.456871047062049)-- (20.098854191814286,-0.84657565018107);
\draw [line width=1.2pt,color=qqzzff] (20.098854191814286,-0.84657565018107)-- (14.941845314902121,-0.058154685097879966);

\draw [line width=1.2pt,color=xdxdff] (26.263250812189575,0.8399102176574627)-- (29.621642776144864,7.212190545608959);
\draw [line width=1.2pt,color=xdxdff] (29.621642776144864,7.212190545608959)-- (24.736432434219985,6.196642251719948);
\draw [line width=1.2pt,color=xdxdff] (24.736432434219985,6.196642251719948)-- (26.263250812189575,0.8399102176574627);
\draw [line width=1.2pt,color=ffwwqq] (23.29757866590064,9.530574026048939)-- (29.679797461242742,12.270345230706837);
\draw [line width=1.2pt,color=ffwwqq] (29.679797461242742,12.270345230706837)-- (29.621642776144864,7.212190545608959);
\draw [line width=1.2pt,color=ffwwqq] (29.621642776144864,7.212190545608959)-- (23.29757866590064,9.530574026048939);
\draw [line width=1.2pt,color=qqzzff] (29.679797461242742,12.270345230706837)-- (24.522788584330577,13.058766195790026);
\draw [line width=1.2pt,color=qqzzff] (24.522788584330577,13.058766195790026)-- (27.93515690830634,16.573791927949955);
\draw [line width=1.2pt,color=qqzzff] (27.93515690830634,16.573791927949955)-- (29.679797461242742,12.270345230706837);
\draw [line width=1.2pt,color=qqccqq] (21.557116438041643,21.7494300041815)-- (21.77076028793105,14.887306060111422);
\draw [line width=1.2pt,color=qqccqq] (21.77076028793105,14.887306060111422)-- (27.93515690830634,16.573791927949955);
\draw [line width=1.2pt,color=qqccqq] (27.93515690830634,16.573791927949955)-- (21.557116438041643,21.7494300041815);
\draw [line width=1.2pt,color=xdxdff] (18.198724474086355,15.377149676230006)-- (21.557116438041643,21.7494300041815);
\draw [line width=1.2pt,color=xdxdff] (21.557116438041643,21.7494300041815)-- (16.671906096116764,20.73388171029249);
\draw [line width=1.2pt,color=xdxdff] (16.671906096116764,20.73388171029249)-- (18.198724474086355,15.377149676230006);
\draw [line width=1.2pt,color=ffwwqq] (10.289687300774663,17.99411050563459)-- (16.671906096116764,20.73388171029249);
\draw [line width=1.2pt,color=ffwwqq] (16.671906096116764,20.73388171029249)-- (16.613751411018885,15.675727025194611);
\draw [line width=1.2pt,color=ffwwqq] (16.613751411018885,15.675727025194611)-- (10.289687300774663,17.99411050563459);
\draw [line width=1.2pt,color=qqzzff] (12.034327853711066,13.690663808391474)-- (6.8773189767989,14.479084773474662);
\draw [line width=1.2pt,color=qqzzff] (6.8773189767989,14.479084773474662)-- (10.289687300774663,17.99411050563459);
\draw [line width=1.2pt,color=qqzzff] (10.289687300774663,17.99411050563459)-- (12.034327853711066,13.690663808391474);
\draw [line width=1.2pt,color=qqccqq] (6.8773189767989,14.479084773474662)-- (7.090962826688308,7.616960829404583);
\draw [line width=1.2pt,color=qqccqq] (7.090962826688308,7.616960829404583)-- (13.255359447063597,9.303446697243116);
\draw [line width=1.2pt,color=qqccqq] (13.255359447063597,9.303446697243116)-- (6.8773189767989,14.479084773474662);
\draw [line width=1.2pt,color=xdxdff] (8.617781204657899,2.2602287953420994)-- (11.976173168613187,8.632509123293595);
\draw [line width=1.2pt,color=xdxdff] (11.976173168613187,8.632509123293595)-- (7.090962826688308,7.616960829404583);
\draw [line width=1.2pt,color=xdxdff] (7.090962826688308,7.616960829404583)-- (8.617781204657899,2.2602287953420994);
\draw [line width=1.2pt,color=qqzzff] (19.88521034192488,6.015548293889011)-- (23.29757866590064,9.530574026048939);
\draw [line width=1.2pt,color=xdxdff] (23.29757866590064,9.530574026048939)-- (21.77076028793105,14.887306060111422);
\draw [line width=1.2pt,color=qqzzff] (21.77076028793105,14.887306060111422)-- (16.613751411018885,15.675727025194611);
\draw [line width=1.2pt,color=xdxdff] (16.613751411018885,15.675727025194611)-- (13.255359447063597,9.303446697243116);
\draw [line width=1.2pt,color=qqzzff] (13.255359447063597,9.303446697243116)-- (15.,5.);
\draw [line width=1.2pt,color=qqccqq] (24.736432434219985,6.196642251719948)-- (24.522788584330577,13.058766195790026);
\draw [line width=1.2pt,color=ffwwqq] (24.522788584330577,13.058766195790026)-- (18.198724474086355,15.377149676230006);
\draw [line width=1.2pt,color=qqccqq] (18.198724474086355,15.377149676230006)-- (12.034327853711066,13.690663808391474);
\draw [line width=1.2pt,color=ffwwqq] (12.034327853711066,13.690663808391474)-- (11.976173168613187,8.632509123293595);
\draw [line width=1.2pt,color=qqccqq] (11.976173168613187,8.632509123293595)-- (18.354213638877884,3.456871047062049);

\draw [line width=1.2pt,color=gray!50] (10.289687300774663,17.99411050563459)-- (21.557116438041643,21.7494300041815);
\draw [line width=1.2pt,color=ffwwqq] (16.671906096116764,20.73388171029249)-- (10.289687300774663,17.99411050563459);
\draw [line width=1.2pt,color=xdxdff] (21.557116438041643,21.7494300041815)-- (16.671906096116764,20.73388171029249);

\draw [line width=1.2pt,color=gray!50] (18.354213638877884,3.456871047062049)-- (29.621642776144864,7.212190545608959);
\draw [line width=1.2pt,color=gray!50] (8.617781204657899,2.2602287953421007)-- (19.88521034192488,6.015548293889011);
\draw [line width=1.2pt,color=ffwwqq] (24.736432434219985,6.196642251719948)-- (18.354213638877884,3.456871047062049);

\draw [line width=1.2pt,color=ffwwqq] (15.,5.) -- (8.617781204657899,2.2602287953421007);
\draw [line width=1.2pt,color=xdxdff] (19.88521034192488,6.015548293889011) -- (15.,5.);

\draw [line width=1.2pt,color=qqffqq] (20.098854191814286,-0.84657565018107)-- (26.263250812189575,0.8399102176574627);
\draw [line width=1.2pt,color=gray!50] (14.941845314902121,-0.058154685097879966)-- (26.263250812189575,0.8399102176574627);

\draw [->,>=stealth,line width=2pt,color=ffqqqq] (8.617781204657899,2.2602287953421007) -- (18.354213638877884,3.456871047062049);
\draw [->,>=stealth,line width=2pt,color=ffqqqq] (19.88521034192488,6.015548293889011) -- (29.621642776144864,7.212190545608959);

\draw [line width=1.2pt,color=qqccqq] (19.88521034192488,6.015548293889011)-- (26.263250812189575,0.8399102176574627);
\draw [line width=1.2pt,color=qqccqq] (19.88521034192488,6.015548293889011)-- (20.098854191814286,-0.84657565018107);
\begin{scriptsize}
\draw [fill=rvwvcq] (8.617781204657899,2.2602287953421007) circle (2.5pt)node[color=rvwvcq,left] {$V_0$};
\draw [fill=rvwvcq] (15.,5.15) circle (2.5pt)node[color=rvwvcq,left] {$W_1$};
\draw [fill=rvwvcq] (14.941845314902121,-0.058154685097879966) circle (2.5pt)node[color=rvwvcq,below] {$V_1$};
\draw [fill=rvwvcq] (20.098854191814286,-0.84657565018107) circle (2.5pt)node[color=rvwvcq,below] {$V_2$};
\draw [fill=rvwvcq] (18.354213638877884,3.456871047062049) circle (2.5pt)node[color=rvwvcq,above] {$W_2$};
\draw [fill=rvwvcq] (19.88521034192488,6.015548293889011) circle (2.5pt);
\draw [color=rvwvcq] (19.3,6.5)node {$W_3$};
\draw [fill=rvwvcq] (26.263250812189575,0.8399102176574627) circle (2.5pt)node[color=rvwvcq,below] {$V_3$};
\draw [fill=rvwvcq] (29.621642776144864,7.212190545608959) circle (2.5pt)node[color=rvwvcq,right] {$V_4$};
\draw [fill=rvwvcq] (14.96,5.02) circle (2.5pt);
\draw [fill=rvwvcq] (26.263250812189575,0.8399102176574627) circle (2.5pt);
\draw [fill=rvwvcq] (24.736432434219985,6.196642251719948) circle (2.5pt);
\draw [color=rvwvcq] (25.25,7.15) node {$W_4$};
\draw [fill=rvwvcq] (23.29757866590064,9.530574026048939) circle (2.5pt)node[color=rvwvcq,left] {$W_5$};
\draw [fill=rvwvcq] (29.679797461242742,12.270345230706837) circle (2.5pt)node[color=rvwvcq,right] {$V_5$};
\draw [fill=rvwvcq] (29.621642776144864,7.212190545608959) circle (2.5pt);
\draw [fill=rvwvcq] (29.679797461242742,12.270345230706837) circle (2.5pt);
\draw [fill=rvwvcq] (24.522788584330577,13.058766195790026) circle (2.5pt);
\draw[color=rvwvcq] (24.2,13.8)node {$W_6$};
\draw [fill=rvwvcq] (27.93515690830634,16.573791927949955) circle (2.5pt)node[color=rvwvcq,right] {$V_6$};
\draw [fill=rvwvcq] (21.557116438041643,21.7494300041815) circle (2.5pt)node[color=rvwvcq,above] {$V_7$};
\draw [fill=rvwvcq] (21.77076028793105,14.887306060111422) circle (2.5pt);
\draw[color=rvwvcq] (22.3,15.5) node {$W_7$};
\draw [fill=rvwvcq] (27.93515690830634,16.573791927949955) circle (2.5pt);
\draw [fill=rvwvcq] (18.198724474086355,15.377149676230006) circle (2.5pt)node[color=rvwvcq,below] {$W_8$};
\draw [fill=rvwvcq] (21.557116438041643,21.7494300041815) circle (2.5pt);
\draw [fill=rvwvcq] (16.671906096116764,20.73388171029249) circle (2.5pt)node[color=rvwvcq,above] {$V_8$};
\draw [fill=rvwvcq] (10.289687300774663,17.99411050563459) circle (2.5pt)node[color=rvwvcq,above] {$V_9$};
\draw [fill=rvwvcq] (16.671906096116764,20.73388171029249) circle (2.5pt);
\draw [fill=rvwvcq] (16.613751411018885,15.675727025194611) circle (2.5pt);
\draw[color=rvwvcq] (15.95,16.6) node {$W_9$};
\draw [fill=rvwvcq] (12.034327853711066,13.690663808391474) circle (2.5pt);
\draw[color=rvwvcq] (12.9,13.388104592769473) node {$W_{10}$};
\draw [fill=rvwvcq] (6.8773189767989,14.479084773474662) circle (2.5pt)node[color=rvwvcq,left] {$V_{10}$};
\draw [fill=rvwvcq] (10.289687300774663,17.99411050563459) circle (2.5pt);
\draw [fill=rvwvcq] (6.8773189767989,14.479084773474662) circle (2.5pt);
\draw [fill=rvwvcq] (7.090962826688308,7.616960829404583) circle (2.5pt)node[color=rvwvcq,left] {$V_{11}$};
\draw [fill=rvwvcq] (13.255359447063597,9.303446697243116) circle (2.5pt)node[color=rvwvcq,right] {$W_{11}$};
\draw [fill=rvwvcq] (8.617781204657899,2.2602287953420994) circle (2.5pt);
\draw [fill=rvwvcq] (11.976173168613187,8.632509123293595) circle (2.5pt);
\draw[color=rvwvcq] (12.1,7.55003262852787) node {$W_{0}$};
\draw [fill=rvwvcq] (7.090962826688308,7.616960829404583) circle (2.5pt);
\draw[color=ffcctt] (13,20) node {$\sjx_{0}$};
\draw[color=qqffqq] (16.063440947832323,4.228973726713782) node {$\sjx_L$};
\draw[color=qqffqq] (24.313650430233682,4.648475903785036) node {$\sjx_R$};
\draw[color=cyan] (21.586886279270523,1.1875829429471982) node {$\sjx_{-}$};
\draw[color=ffqqqq] (13.511469370648852,2.2) node {$s_{3,L}$};
\draw[color=ffqqqq] (23.33481201706742,7.1) node {$s_{3}$};
\end{scriptsize}
\end{tikzpicture}\vskip -.5cm
\caption{The triangles $\sjx_{?}$ in type $E_6$ with $?=\{0,-,L,R\}$.}
\label{fig:E6sjx}

\centering
\begin{tikzpicture}[line cap=round,line join=round,arrow/.style={->,>=stealth,thick},scale=.5]
\draw (10,8) node {};
\clip(6,-2) rectangle (30,7);
\fill[line width=1.2pt,color=qqffff,fill=qqffff,fill opacity=0.4]
(8.617781204657899,2.2602287953421007) -- (15.,5.) -- (19.88521,6.01555) -- (26.263250812189575,0.8399102176574627) -- (20.098854191814286,-0.84657565018107) -- (14.941845314902121,-0.058154685097879966) -- cycle;
\draw [line width=1.2pt,color=ffwwqq] (15.,5.)-- (14.941845314902121,-0.058154685097879966);
\draw [line width=1.2pt,color=ffwwqq] (14.941845314902121,-0.058154685097879966)-- (8.617781204657899,2.2602287953421007);
\draw [line width=1.2pt,color=qqzzff] (14.941845314902121,-0.058154685097879966)-- (18.354213638877884,3.456871047062049);
\draw [line width=1.2pt,color=qqzzff] (18.354213638877884,3.456871047062049)-- (20.098854191814286,-0.84657565018107);
\draw [line width=1.2pt,color=qqccqq] (19.88521,6.01555)-- (20.098854191814286,-0.84657565018107);
\draw [line width=1.2pt,color=qqccqq] (20.098854191814286,-0.84657565018107)-- (26.263250812189575,0.8399102176574627);
\draw [line width=1.2pt,color=qqccqq] (26.263250812189575,0.8399102176574627)-- (19.88521,6.01555);
\draw [line width=1.2pt,color=xdxdff] (26.263250812189575,0.8399102176574627)-- (29.621642434219986,7.212192251719948);
\draw [line width=1.2pt,color=xdxdff] (29.621642434219986,7.212192251719948)-- (24.736432434219985,6.196642251719948);
\draw [line width=1.2pt,color=xdxdff] (24.736432434219985,6.196642251719948)-- (26.263250812189575,0.8399102176574627);
\draw [line width=1.2pt,color=ffwwqq] (23.297578323975763,9.530575732159928)-- (29.679797119317865,12.270346936817827);
\draw [line width=1.2pt,color=ffwwqq] (29.679797119317865,12.270346936817827)-- (29.621642434219986,7.212192251719948);
\draw [line width=1.2pt,color=ffwwqq] (29.621642434219986,7.212192251719948)-- (23.297578323975763,9.530575732159928);
\draw [line width=1.2pt,color=qqzzff] (29.679797119317865,12.270346936817827)-- (24.5227882424057,13.058767901901016);
\draw [line width=1.2pt,color=qqzzff] (24.5227882424057,13.058767901901016)-- (27.935156566381462,16.573793634060944);
\draw [line width=1.2pt,color=qqzzff] (27.935156566381462,16.573793634060944)-- (29.679797119317865,12.270346936817827);
\draw [line width=1.2pt,color=qqccqq] (21.557115754191887,21.74943341640348)-- (21.770759946006173,14.887307766222412);
\draw [line width=1.2pt,color=qqccqq] (21.770759946006173,14.887307766222412)-- (27.935156566381462,16.573793634060944);
\draw [line width=1.2pt,color=qqccqq] (27.935156566381462,16.573793634060944)-- (21.557115754191887,21.74943341640348);
\draw [line width=1.2pt,color=xdxdff] (18.198724132161477,15.377151382340996)-- (21.557115754191887,21.74943341640348);
\draw [line width=1.2pt,color=xdxdff] (21.557115754191887,21.74943341640348)-- (16.671905754191886,20.73388341640348);
\draw [line width=1.2pt,color=xdxdff] (16.671905754191886,20.73388341640348)-- (18.198724132161477,15.377151382340996);
\draw [line width=1.2pt,color=ffwwqq] (10.289686958849785,17.99411221174558)-- (16.671905754191886,20.73388341640348);
\draw [line width=1.2pt,color=ffwwqq] (16.671905754191886,20.73388341640348)-- (16.613751069094008,15.6757287313056);
\draw [line width=1.2pt,color=ffwwqq] (16.613751069094008,15.6757287313056)-- (10.289686958849785,17.99411221174558);
\draw [line width=1.2pt,color=qqzzff] (12.034327511786188,13.690665514502463)-- (6.8773186348740225,14.479086479585652);
\draw [line width=1.2pt,color=qqzzff] (6.8773186348740225,14.479086479585652)-- (10.289686958849785,17.99411221174558);
\draw [line width=1.2pt,color=qqzzff] (10.289686958849785,17.99411221174558)-- (12.034327511786188,13.690665514502463);
\draw [line width=1.2pt,color=qqccqq] (6.8773186348740225,14.479086479585652)-- (7.090962826688308,7.616960829404583);
\draw [line width=1.2pt,color=qqccqq] (7.090962826688308,7.616960829404583)-- (13.255359447063597,9.303446697243116);
\draw [line width=1.2pt,color=qqccqq] (13.255359447063597,9.303446697243116)-- (6.8773186348740225,14.479086479585652);
\draw [line width=1.2pt,color=xdxdff] (8.617781204657899,2.2602287953420994)-- (11.976172826688309,8.632510829404584);
\draw [line width=1.2pt,color=xdxdff] (11.976172826688309,8.632510829404584)-- (7.090962826688308,7.616960829404583);
\draw [line width=1.2pt,color=xdxdff] (7.090962826688308,7.616960829404583)-- (8.617781204657899,2.2602287953420994);
\draw [line width=1.2pt,color=qqzzff] (19.88521,6.01555)-- (23.297578323975763,9.530575732159928);
\draw [line width=1.2pt,color=xdxdff] (23.297578323975763,9.530575732159928)-- (21.770759946006173,14.887307766222412);
\draw [line width=1.2pt,color=qqzzff] (21.770759946006173,14.887307766222412)-- (16.613751069094008,15.6757287313056);
\draw [line width=1.2pt,color=xdxdff] (16.613751069094008,15.6757287313056)-- (13.255359447063597,9.303446697243116);
\draw [line width=1.2pt,color=qqzzff] (13.255359447063597,9.303446697243116)-- (15.,5.);
\draw [line width=1.2pt,color=ffwwqq] (18.354213638877884,3.456871047062049)-- (24.736432434219985,6.196642251719948);
\draw [line width=1.2pt,color=qqccqq] (24.736432434219985,6.196642251719948)-- (24.5227882424057,13.058767901901016);
\draw [line width=1.2pt,color=ffwwqq] (24.5227882424057,13.058767901901016)-- (18.198724132161477,15.377151382340996);
\draw [line width=1.2pt,color=qqccqq] (18.198724132161477,15.377151382340996)-- (12.034327511786188,13.690665514502463);
\draw [line width=1.2pt,color=ffwwqq] (12.034327511786188,13.690665514502463)-- (11.976172826688309,8.632510829404584);
\draw [line width=1.2pt,color=qqccqq] (11.976172826688309,8.632510829404584)-- (18.354213638877884,3.456871047062049);
\draw [->,>=stealth,line width=1.5pt] (20.098854191814286,-0.84657565018107) -- (14.941845314902121,-0.05815468509788002);
\draw [->,>=stealth,line width=1.5pt] (14.941845314902121,-0.058154685097879966) -- (26.263250812189575,0.8399102176574627);
\draw [->,>=stealth,line width=1.5pt] (26.263250812189575,0.8399102176574627) -- (8.617781204657899,2.2602287953421007);
\draw [->,>=stealth,line width=1.5pt] (8.617781204657899,2.2602287953421007) -- (15.,5.);
\draw [->,>=stealth,line width=1.5pt] (8.617781204657899,2.2602287953421007) -- (18.354213638877884,3.456871047062049);
\draw [->,>=stealth,line width=1.5pt] (15.,5.) -- (19.88521,6.01555);
\begin{scriptsize}
\draw [fill=rvwvcq] (8.617781204657899,2.2602287953421007) circle (2.5pt)node[color=rvwvcq,left] {$V_0$};
\draw [fill=rvwvcq] (15.,5.) circle (2.5pt);
\draw [color=rvwvcq] (14,5.5) node {$W_1$};
\draw [fill=rvwvcq] (14.941845314902121,-0.058154685097879966) circle (2.5pt)node[color=rvwvcq,below] {$V_1$};
\draw [fill=rvwvcq] (20.098854191814286,-0.84657565018107) circle (2.5pt)node[color=rvwvcq,below] {$V_2$};
\draw [fill=rvwvcq] (18.354213638877884,3.456871047062049) circle (2.5pt)node[color=rvwvcq,above] {$W_2$};
\draw [fill=rvwvcq] (19.88521034192488,6.015548293889011) circle (2.5pt)node[color=rvwvcq,right] {$W_3$};
\draw [fill=rvwvcq] (26.263250812189575,0.8399102176574627) circle (2.5pt)node[color=rvwvcq,below] {$V_3$};
\draw [fill=rvwvcq] (29.621642776144864,7.212190545608959) circle (2.5pt)node[color=rvwvcq,right] {$V_4$};
\draw [fill=rvwvcq] (14.96,5.02) circle (2.5pt);
\draw [fill=rvwvcq] (26.263250812189575,0.8399102176574627) circle (2.5pt);
\draw [fill=rvwvcq] (24.736432434219985,6.196642251719948) circle (2.5pt);
\draw [color=rvwvcq] (24.1,6.45) node {$W_4$};
\draw [fill=rvwvcq] (23.29757866590064,9.530574026048939) circle (2.5pt)node[color=rvwvcq,left] {$W_5$};
\draw [fill=rvwvcq] (29.679797461242742,12.270345230706837) circle (2.5pt)node[color=rvwvcq,right] {$V_5$};
\draw [fill=rvwvcq] (29.621642776144864,7.212190545608959) circle (2.5pt);
\draw [fill=rvwvcq] (29.679797461242742,12.270345230706837) circle (2.5pt);
\draw [fill=rvwvcq] (24.522788584330577,13.058766195790026) circle (2.5pt);
\draw[color=rvwvcq] (24.2,13.8)node {$W_6$};
\draw [fill=rvwvcq] (27.93515690830634,16.573791927949955) circle (2.5pt)node[color=rvwvcq,right] {$V_6$};
\draw [fill=rvwvcq] (21.557116438041643,21.7494300041815) circle (2.5pt)node[color=rvwvcq,above] {$V_7$};
\draw [fill=rvwvcq] (21.77076028793105,14.887306060111422) circle (2.5pt);
\draw[color=rvwvcq] (22.3,15.5) node {$W_7$};
\draw [fill=rvwvcq] (27.93515690830634,16.573791927949955) circle (2.5pt);
\draw [fill=rvwvcq] (18.198724474086355,15.377149676230006) circle (2.5pt)node[color=rvwvcq,below] {$W_8$};
\draw [fill=rvwvcq] (21.557116438041643,21.7494300041815) circle (2.5pt);
\draw [fill=rvwvcq] (16.671906096116764,20.73388171029249) circle (2.5pt)node[color=rvwvcq,above] {$V_8$};
\draw [fill=rvwvcq] (10.289687300774663,17.99411050563459) circle (2.5pt)node[color=rvwvcq,above] {$V_9$};
\draw [fill=rvwvcq] (16.671906096116764,20.73388171029249) circle (2.5pt);
\draw [fill=rvwvcq] (16.613751411018885,15.675727025194611) circle (2.5pt);
\draw[color=rvwvcq] (15.95,16.6) node {$W_9$};
\draw [fill=rvwvcq] (12.034327853711066,13.690663808391474) circle (2.5pt);
\draw[color=rvwvcq] (12.9,13.388104592769473) node {$W_{10}$};
\draw [fill=rvwvcq] (6.8773189767989,14.479084773474662) circle (2.5pt)node[color=rvwvcq,left] {$V_{10}$};
\draw [fill=rvwvcq] (10.289687300774663,17.99411050563459) circle (2.5pt);
\draw [fill=rvwvcq] (6.8773189767989,14.479084773474662) circle (2.5pt);
\draw [fill=rvwvcq] (7.090962826688308,7.616960829404583) circle (2.5pt)node[color=rvwvcq,left] {$V_{11}$};
\draw [fill=rvwvcq] (13.255359447063597,9.303446697243116) circle (2.5pt)node[color=rvwvcq,right] {$W_{11}$};
\draw [fill=rvwvcq] (8.617781204657899,2.2602287953420994) circle (2.5pt);
\draw [fill=rvwvcq] (11.976173168613187,8.632509123293595) circle (2.5pt);
\draw[color=rvwvcq] (12.1,7.55003262852787) node {$W_{0}$};
\draw [fill=rvwvcq] (7.090962826688308,7.616960829404583) circle (2.5pt);
\draw[color=black] (17.467361977093386,-1) node {$s_6$};
\draw[color=black] (17.776992839882436,.6) node {$s_5$};
\draw[color=black] (14,1.2) node {$s_4$};
\draw[color=black] (11.73919101549599,4.2) node {$s_2$};
\draw[color=black] (13.94531091286796,3.3) node {$s_3$};
\draw[color=black] (17.080323398607074,6) node {$s_1$};

\draw [->,>=stealth,line width=.8pt,bend right] (18.8,-0.6) to (18.9,.2);
\draw [->,>=stealth,line width=.8pt,bend right] (20.8,1.3) to (20.9,.5);
\draw [->,>=stealth,line width=.8pt,bend right] (14,1.9) to (14,2.85);
\draw [->,>=stealth,line width=.8pt,bend right] (12,2) to (12,3.6);
\draw [->,>=stealth,line width=.8pt,bend right] (13.5,4.3) to (16,5.2);
\end{scriptsize}
\end{tikzpicture}
\begin{tikzpicture}[scale=.5,xscale=1,ar/.style={->,thick,>=stealth}]
\draw(0,0)node(v1){$s_1$}(2,0)node(v2){$s_2$}(4,0)node(v4){$s_4$}(6,0)node(v5){$s_5$}(8,0)node(v6)
{$s_6$}(4,2)node(v3){$s_3$};

\draw[ar](v2)to(v1);
\draw[ar](v4)to(v2);
\draw[ar](v4)to(v3);
\draw[ar](v4)to(v5);
\draw[ar](v6)to(v5);
\end{tikzpicture}
\caption{The intersection quiver of a $12$-gon of type $E_6$}
\label{fig:int-quiver-E6}
\end{figure}


\begin{lemma}\label{lem:sjx-LR}
The triangles $\sjx_L$ and $\sjx_R$ (as well as $\sjx_0$) are parallel (i.e. related by translations).
\end{lemma}
\begin{proof}
The statement follows from the equalities of diagonals/vectors:
\[
    V_7V_8=V_4W_4=W_3W_1 \quad\text{and}\quad V_8V_9=W_4W_2=W_1V_0,
\]
cf. \cite[Construction~6.3 or \S~6.3]{QZ22}.
\end{proof}

Comparing the vertices in $\sjx_L$ and $\sjx_R$ with respect to \eqref{eq:order},
we set $\sjx_+$ to be the smaller one among them in the sense that
\begin{gather}\label{def:up-tri}
\sjx_+\colon=
\begin{cases}
\sjx_L & \text{if~}~ W_1<W_4,\\
\sjx_R & \text{if~}~ W_4<W_1.
\end{cases}
\end{gather}

\begin{lemma}\label{lem:agon6}
Let $\agon_6$ be the convex hull of $\sjx_-\cup\sjx_+$.
Then it is a positively convex/stable hexagon (of type $A_5$).
\end{lemma}
\begin{proof}
By \cite[\S~6.3]{QZ22}, we know that $\agon_6$ is in fact a narrow hexagon either $\mathbf{H}_1$ or $\mathbf{H}_2$, where
$$\mathbf{H}_j=V_{j-1}V_jV_{j+1}V_{j+2}W_{j+2}W_{j}.$$
Note that those $\mathbf{H}_j$'s inherit the positive convexity from $\hgon$.
\end{proof}

Denote the $\imz$-ind diagonals of $\agon_6$ from top to bottom by $s_1, s_2, s_4, s_5,$ and $s_6$ respectively.
Set $s_3\colon=\UP{V_0W_2}=\UP{W_3V_4}$, i.e.
\begin{gather}\label{def:s3-1}
s_3=
\begin{cases}
V_0W_2=W_3V_4 & \text{if~}~ V_0<W_2 (\Longleftrightarrow W_3<V_4),\\
W_2V_0=V_4W_3 & \text{if~}~ W_2<V_0 (\Longleftrightarrow V_4<W_3)
\end{cases}
\end{gather}
and define $\Sim\hgon\colon=\{s_i\mid 1\leq i \leq 6\}$.
Then, following Definition \ref{def:An-quiver}, we also have the \emph{intersection quiver} $\hh{Q}(\hgon)$.
See Figure~\ref{fig:int-quiver-E6} for an example, which corresponds to Figure~\ref{fig:E6sjx}.
\begin{figure}[ht]\centering
\definecolor{ffcctt}{rgb}{1.,0.8,0.2}
\begin{tikzpicture}[xscale=.85,yscale=.85,rotate=-45, font=\tiny,
arrow/.style={->,>=stealth,thick}]
\clip[rotate=-45](-.8,2)rectangle(6.3,19);
\foreach \j in {-3,...,8}{
\foreach \k in {0,...,8}{
    \draw[teal,->,>=stealth,very thick]
        (2*\k,2*\j+.5)to++(0,1);
    \draw[teal,->,>=stealth,very thick]
        (2*\k+.5,2*\j)to++(1,0);
}}
\draw[rotate=-45,white,fill=white](0,2)rectangle(-2,19)(8,2)rectangle(5.5,19);
\foreach \j in {-3,...,8}{
\foreach \k in {0,...,8}{
    \draw[white](2*\k,2*\j) node[circle,draw=teal] (x\k\j) {$000$};
}}

\foreach \j in {0,...,5}{
    \draw[teal,->,>=stealth,very thick] (2*\j+4+1+.005+.35,2*\j+1-.005+.35)to (2*\j+4+2-.35,2*\j+2-.35);
    \draw[teal,->,>=stealth,very thick] (2*\j+4+.35,2*\j+.35) to (2*\j+4+1+.005-.35,2*\j+1-.005-.35);
}
\draw[green!15](6,4) node[circle,draw=teal,fill=green!15] {$000$};
\draw[ffcctt!15](2*2+2+1+.005,2*2-1-.005) node[circle,draw=teal,fill=ffcctt!15] {$000$};
\draw[blue!15](8,2) node[cyan!15,circle,draw=teal,fill=cyan!15] {$000$};
\foreach \j in {1,...,6}{
    \draw[yellow!10](2*\j+2+1+.005,2*\j-1-.005) node[circle,draw=teal](m\j) {$000$};
    \draw(m\j)node{$M_{\j}$};
}
\foreach \j/\k in {0/3,1/4,2/5,3/6,4/7,5/8}{
    \draw[](2*\j+6,2*\j+2) node {$V_\j V_\k$};
}
\draw[](-2+6,-2+2) node {$V_{11} V_2$};
\foreach \j/\k in {0/2,1/3,2/4,3/5,4/6,5/7}{
    \draw[](2*\j+6,2*\j) node {$V_\j V_\k$};
}
\foreach \j/\k/\l in {8/11/2,9/0/3,10/1/4,11/2/5,12/3/6,13/4/7}{
    \draw[](2*\j+6-18,2*\j-14) node {$V_{\k} W_\l$};
}
\foreach \j/\k in {0/1,1/2,2/3,3/4,4/5,5/6}{
    \draw[](2*\j+6,2*\j-2) node[white,circle,draw=teal] {$000$} node {$V_\j V_\k$};}
\foreach \j/\k in {1/2,2/3}{
    \draw[](2*\j+6,2*\j-2) node[cyan!15,circle,draw=teal,fill=cyan!15] {$000$} node {$V_\j V_\k$};}
\foreach \j/\k/\l in {8/0/2,9/1/3,10/2/4,11/3/5,12/4/6,13/5/7,14/6/8}{
    \draw[](2*\j-2-12,2*\j-2-12) node[green!15,circle,draw=teal] {$000$} node {$W_\k W_\l$};}
\foreach \j/\k/\l in {8/1/3,9/2/4}{
    \draw[](2*\j-12,2*\j-12) node[green!15,circle,draw=teal,fill=green!15] {$000$} node {$W_\k W_\l$};}
\end{tikzpicture}
\caption{Part of the AR quiver of $\Dwq(E_6)$}
\label{fig:AR-E6}
\end{figure}
\begin{proposition}\label{pp:thmE6}
Theorem~\ref{thm:An} holds for any quiver $Q$ of type $E_6$.
In particular, the underlying diagram of the Ext-quiver/intersection quiver is of the form (of homogenous degree one):
\begin{gather}\label{eq:E6}
\begin{tikzpicture}[scale=.5,xscale=1,ar/.style={-,thick}]
\draw(0,0)node(v1){$s_1$}(2,0)node(v2){$s_2$}(4,0)node(v4){$s_4$}(6,0)node(v5){$s_5$}(8,0)node(v6){$s_6$}
    (4,2)node(v3){$s_3$};
\draw[ar](v2)edge(v1)edge(v4);
\draw[ar](v3)edge(v4);
\draw[ar](v5)edge(v6)edge(v4);
\end{tikzpicture}
\end{gather}
\end{proposition}

\begin{proof}
Recall the assumption that $\sigma\in\ToSt(Q)$ and $\hgon:=\hgon_\sigma$ be the corresponding stable $12$-gon of type $E_6$.
Following the construction above, we have the set of upward diagonals
\[\Sim\hgon=\Sim\agon_6\cup\{s_3\}=\{s_i\mid 1\le i\le 6\},\]
where $\agon_6$ is the narrow hexagon (either $\mathbf{H}_1$ or $\mathbf{H}_2$) in Lemma~\ref{lem:agon6} and $\Sim\agon_6=\{s_1,s_2,s_4,s_5,s_6\}$.

Denote by $\thick\mathbb{L}$ the thick subcategory of $\Dwq(Q)$ generated by objects in $\{\XX(\ell)\mid\ell\in\mathbb{L}\}$ for some set $\mathbb{L}$ consisting of upward diagonals.

We claim that $\Sim_0\colon=\{\XX(s_i)\mid 1\le i\le 6\}\subset\h_\sigma$ is a simple minded collection of $\Dwq(Q)$, in the sense that
\begin{itemize}
  \item $\Hom^{\le0}(S,T)=0$ for any $S,T\in\Sim_0$ and
  \item $\thick\Sim\hgon=\Dwq(Q)$.
\end{itemize}
If so, $\Sim_0$ generates a heart that is contained in $\h_\sigma$, which must coincide with $\h_\sigma$.
In other words, $\Sim_0=\Sim\h_\sigma$ and then the proposition follows easily.

For the first condition, we actually have
\begin{gather}\label{eq:d=1}
    \Hom^\bullet(S,T)=\Hom^{1}(S,T).
\end{gather}
To verify this, we first list several full subcategories of $\Dwq(Q)$:
\begin{itemize}
  \item[$($a$)$] $\thick\Sim\agon_6\cong\Dwq(A_5)$ since $\agon_6$ is a stable hexagon (of type $A_5$).
  \item[$($b$)$] $\thick\{s_3,s_4,s_5,s_6\}\cong\Dwq(A_4)$ with the corresponding stable pentagon $\agon_5$ (of type $A_4$)
 formed by the vertices of $s_{3,?}$ and those of $\sjx_-$, where
  \[s_{3,?}\colon=\begin{cases}
    s_{3,L}=V_0W_2& \text{if $\sjx_+=\sjx_L$,}\\
    s_{3,R}=V_4W_3& \text{if $\sjx_+=\sjx_R$.}
  \end{cases}\]
  Indeed, we have $\Sim\agon_5=\{s_3,s_4,s_5,s_6\}$.
  \item[$($c$)$] $\thick\Sim\sjx_\pm\cong\Dwq(A_2)$ with the corresponding stable triangle $\sjx_\pm$ of type $A_2$.
  Its indecomposable objects (up to shift) are the ones in green (for $\sjx_+$) and in blue (for $\sjx_-$) in Figure~\ref{fig:AR-E6}.
\end{itemize}
By Theorem~\ref{thm:An} and $($a$)$, any objects $S,T$ in $\{\XX(s)\mid s\in\Sim\agon_6\}$ satisfies \eqref{eq:d=1}.
Note that the object with labeling $XY\in\AD(\hgon)$ in Figure~\ref{fig:AR-E6} is precisely $\XX(XY)\in\RC(Q)$
and the one in yellow corresponds to some shift of $\XX(s_3)$.
Using $($c$)$, it is straightforward to see that $\Hom^\bullet$ vanishes between (i.e. from and to)
$\XX(s_3)$ and $\XX(s_i)\in\thick\Sim\sjx_\pm$ for $i\in\{1,2,5,6\}$.
By $($b$)$, there exists exactly one $\Hom^1$ between $\XX(s_3)$ and $\XX(s_4)$ as they are consecutive $\Im Z$-Ind diagonals of $\agon_5$.
In all, the equality \eqref{eq:d=1} holds for any $S,T\in\Sim_0$, which implies the first condition.
In fact, from the discussion above,
we have actually shown that the Ext-quiver of $\Sim_0$ is of the form \eqref{eq:E6}.

For any $j\in\ZZ_{12}$, define
\[\UD_{j}\colon=\UD(\mathbf{H}_{j})\cup\{\UP{V_{j-1}W_{j+1}},\UP{W_{j+1}V_{j+2}}\}.\]
The second condition can be proved as follows:
\begin{itemize}
\item[$($i$)$] By Lemma~\ref{lem:tri-vs-tri}, there is a triangle with objects $\UP{\ell}$, where $\ell$ runs through the edges of a contractible triangle.
\item[$($ii$)$] By Lemma~\ref{lem:filtration-An},
we know that $\thick\Sim\agon_6=\thick\UD(\agon_6)$. Note that
\[\Sim\hgon=\Sim\agon_6\cup\{\UP{V_0W_2}=\UP{W_3V_4}\},\]
where $\agon_6=\mathbf{H}_j$ and $\UP{V_0W_2}=\UP{W_3V_4}$ is one element in $\UD_j\setminus\UD(\mathbf{H}_j)$ for some $j\in\{1,2\}$. It immediately follows from $($i$)$ that $\thick\Sim\hgon=\thick\UD_1$ or $\thick\UD_2$.
\item[$($iii$)$] For any $\ell\in\UD_{j+1}$, one checks that $\ell$ is an edge of a contractible triangle whose other two edges are in $\UD_{j}$. By $($i$)$, we have $\thick\UD_{j+1}\subset\thick\UD_{j}$. Hence all $\UD_{j}$'s coincide for $j\in\ZZ_{12}$.

 \item [$($iv$)$] By definition, we have
 \[\UD(\hgon)=\bigcup_{j\in\ZZ_{12}}\UD_j.\]
 Then $($iii$)$ implies that $\thick\UD_j=\thick\UD(\hgon)=\Dwq(Q).$ Together with $($ii$)$, we obtain the required claim.\qedhere
\end{itemize}
\end{proof}
%
%
%
%
%

\subsection{Exceptional types $E_{n}$}

\begin{figure}[ht]\centering
\makebox[\textwidth]{
\definecolor{ffqqqq}{rgb}{1.,0.,0.}
\definecolor{qqffqq}{rgb}{0.,1.,0.}
\definecolor{wrwrwr}{rgb}{0.3803921568627451,0.3803921568627451,0.3803921568627451}
\definecolor{qqffff}{rgb}{0.,1.,1.}
\definecolor{ffwwqq}{rgb}{1.,0.4,0.}
\definecolor{qqccqq}{rgb}{0.,0.8,0.}
\definecolor{rvwvcq}{rgb}{0.08235294117647059,0.396078431372549,0.7529411764705882}
\centering
\begin{tikzpicture}[line cap=round,line join=round,arrow/.style={->,>=stealth,thick},scale=.8]
\clip(-7.6637476144654455,-4.711430363598874) rectangle (4.957924266923282,9.070450910786567);

\fill[fill=qqffqq,fill opacity=0.5] (-6.64,-1.14) -- (-3.92,-0.82) -- (-1.04,-1.3) -- cycle;
\fill[fill=qqffqq,fill opacity=0.5] (-2.8,-2.06) -- (-0.08,-1.74) -- (2.8,-2.22) -- cycle;
\fill[fill=qqffff,fill opacity=0.4] (-5.04,-2.78) -- (-3.36,-3.48) -- (-1.34,-3.96) -- (0.7,-3.62) -- cycle;

\draw [line width=1.2pt,color=qqccqq] (0.7,-3.62)-- (2.8,-2.22);
\draw [line width=1.2pt,color=qqccqq] (2.8,-2.22)-- (3.92,-0.26);
\draw [line width=1.2pt,color=ffwwqq] (3.92,-0.26)-- (4.48,1.16);
\draw [line width=1.2pt,color=ffwwqq] (4.48,1.16)-- (4.78,3.82);
\draw [line width=1.2pt,color=rvwvcq] (4.78,3.82)-- (4.,5.7);
\draw [line width=1.2pt,color=rvwvcq] (4.,5.7)-- (2.4,7.34);
\draw [line width=1.2pt,color=qqccqq] (2.4,7.34)-- (0.72,8.04);
\draw [line width=1.2pt,color=qqccqq] (0.72,8.04)-- (-1.3,8.52);
\draw [line width=1.2pt,color=ffwwqq] (-1.3,8.52)-- (-3.34,8.18);
\draw [line width=1.2pt,color=ffwwqq] (-3.34,8.18)-- (-5.44,6.78);
\draw [line width=1.2pt,color=rvwvcq] (-5.44,6.78)-- (-6.56,4.82);
\draw [line width=1.2pt,color=rvwvcq] (-6.56,4.82)-- (-7.12,3.4);
\draw [line width=1.2pt,color=qqccqq] (-7.12,3.4)-- (-7.42,0.74);
\draw [line width=1.2pt,color=qqccqq] (-7.42,0.74)-- (-6.64,-1.14);
\draw [line width=1.2pt,color=ffwwqq] (-6.64,-1.14)-- (-5.04,-2.78);
\draw [line width=1.7pt,color=qqffff] (-0.08,-1.74)-- (2.24,0.44);
\draw [line width=1.7pt,color=qqffff] (2.24,0.44)-- (2.74,3.48);
\draw [line width=1.7pt,color=qqffff] (2.74,3.48)-- (1.28,5.38);
\draw [line width=1.7pt,color=qqffff] (1.28,5.38)-- (-1.6,5.86);
\draw [line width=1.7pt,color=qqffff] (-1.6,5.86)-- (-3.84,5.14);
\draw [line width=1.7pt,color=qqffff] (-3.84,5.14)-- (-5.1,2.92);
\draw [line width=1.7pt,color=qqffff] (-5.1,2.92)-- (-4.54,0.26);
\draw [line width=1.7pt,color=qqffff] (-4.54,0.26)-- (-2.8,-2.06);
\draw [line width=1.7pt,color=ffwwqq] (0.16,6.62)-- (-2.56,6.3);
\draw [line width=1.7pt,color=ffwwqq] (-2.56,6.3)-- (-4.88,4.12);
\draw [line width=1.7pt,color=ffwwqq] (-4.88,4.12)-- (-5.38,1.08);
\draw [line width=1.7pt,color=ffwwqq] (-5.38,1.08)-- (-3.92,-0.82);
\draw [line width=1.7pt,color=ffwwqq] (-1.04,-1.3)-- (1.2,-0.58);
\draw [line width=1.7pt,color=ffwwqq] (1.2,-0.58)-- (2.46,1.64);
\draw [line width=1.7pt,color=ffwwqq] (2.46,1.64)-- (1.9,4.3);
\draw [line width=1.7pt,color=ffwwqq] (1.9,4.3)-- (0.16,6.62);
\draw [line width=1.2pt,color=qqffqq] (-6.64,-1.14)-- (-3.92,-0.82);
\draw [line width=2.pt,color=qqffqq] (-3.92,-0.82)-- (-1.04,-1.3);
\draw [line width=1.2pt,color=qqffqq] (-1.04,-1.3)-- (-6.64,-1.14);
\draw [line width=1.2pt,color=qqffqq] (-2.8,-2.06)-- (-0.08,-1.74);
\draw [line width=1.2pt,color=qqffqq] (-0.08,-1.74)-- (2.8,-2.22);
\draw [line width=1.2pt,color=qqffqq] (2.8,-2.22)-- (-2.8,-2.06);
\draw [line width=1.2pt,color=qqffff] (-5.04,-2.78)-- (-3.36,-3.48);
\draw [line width=1.2pt,color=qqffff] (-3.36,-3.48)-- (-1.34,-3.96);
\draw [line width=1.2pt,color=qqffff] (-1.34,-3.96)-- (0.7,-3.62);
\draw [line width=1.2pt,color=qqffff] (0.7,-3.62)-- (-5.04,-2.78);

\draw [line width=.8pt,dashed,gray] (-3.36,-3.48)-- (-2.8,-2.06);
\draw [line width=.8pt,dashed,gray] (-1.34,-3.96)-- (-1.04,-1.3);
\draw [line width=.8pt,dashed,gray] (0.7,-3.62)-- (-0.08,-1.74);
\draw [line width=.8pt,dashed,gray] (2.8,-2.22)-- (1.2,-0.58);
\draw [line width=.8pt,dashed,gray] (3.92,-0.26)-- (2.24,0.44);
\draw [line width=.8pt,dashed,gray] (4.48,1.16)-- (2.46,1.64);
\draw [line width=.8pt,dashed,gray] (4.78,3.82)-- (2.74,3.48);
\draw [line width=.8pt,dashed,gray] (4.,5.7)-- (1.9,4.3);
\draw [line width=.8pt,dashed,gray] (1.28,5.38)-- (2.4,7.34);
\draw [line width=.8pt,dashed,gray] (0.72,8.04)-- (0.16,6.62);
\draw [line width=.8pt,dashed,gray] (-1.6,5.86)-- (-1.3,8.52);
\draw [line width=.8pt,dashed,gray] (-2.56,6.3)-- (-3.34,8.18);
\draw [line width=.8pt,dashed,gray] (-3.84,5.14)-- (-5.44,6.78);
\draw [line width=.8pt,dashed,gray] (-4.88,4.12)-- (-6.56,4.82);
\draw [line width=.8pt,dashed,gray] (-7.12,3.4)-- (-5.1,2.92);
\draw [line width=.8pt,dashed,gray] (-5.38,1.08)-- (-7.42,0.74);
\draw [line width=.8pt,dashed,gray] (-6.64,-1.14)-- (-4.54,0.26);
\draw [line width=.8pt,dashed,gray] (-3.92,-0.82)-- (-5.04,-2.78);
\draw [line width=0.5pt,dotted,color=wrwrwr] (-3.36,-3.48)-- (-3.92,-0.82);
\draw [line width=0.5pt,dotted,color=wrwrwr] (-2.8,-2.06)-- (-5.04,-2.78);
\draw [line width=0.5pt,dotted,color=wrwrwr] (-1.04,-1.3)-- (-3.36,-3.48);
\draw [line width=0.5pt,dotted,color=wrwrwr] (-2.8,-2.06)-- (-1.34,-3.96);
\draw [line width=0.5pt,dotted,color=wrwrwr] (-0.08,-1.74)-- (-1.34,-3.96);
\draw [line width=0.5pt,dotted,color=wrwrwr] (-1.04,-1.3)-- (0.7,-3.62);
\draw [line width=0.5pt,dotted,color=wrwrwr] (0.7,-3.62)-- (1.2,-0.58);
\draw [line width=0.5pt,dotted,color=wrwrwr] (2.8,-2.22)-- (2.24,0.44);
\draw [line width=0.5pt,dotted,color=wrwrwr] (1.2,-0.58)-- (3.92,-0.26);
\draw [line width=0.5pt,dotted,color=wrwrwr] (2.46,1.64)-- (3.92,-0.26);
\draw [line width=0.5pt,dotted,color=wrwrwr] (4.48,1.16)-- (2.24,0.44);
\draw [line width=0.5pt,dotted,color=wrwrwr] (2.46,1.64)-- (4.78,3.82);
\draw [line width=0.5pt,dotted,color=wrwrwr] (4.48,1.16)-- (2.74,3.48);
\draw [line width=0.5pt,dotted,color=wrwrwr] (4.,5.7)-- (2.74,3.48);
\draw [line width=0.5pt,dotted,color=wrwrwr] (1.9,4.3)-- (4.78,3.82);
\draw [line width=0.5pt,dotted,color=wrwrwr] (1.9,4.3)-- (2.4,7.34);
\draw [line width=0.5pt,dotted,color=wrwrwr] (4.,5.7)-- (1.28,5.38);
\draw [line width=0.5pt,dotted,color=wrwrwr] (1.28,5.38)-- (0.72,8.04);
\draw [line width=0.5pt,dotted,color=wrwrwr] (0.16,6.62)-- (2.4,7.34);
\draw [line width=0.5pt,dotted,color=wrwrwr] (0.16,6.62)-- (-1.3,8.52);
\draw [line width=0.5pt,dotted,color=wrwrwr] (-1.6,5.86)-- (0.72,8.04);
\draw [line width=0.5pt,dotted,color=wrwrwr] (-1.3,8.52)-- (-2.56,6.3);
\draw [line width=0.5pt,dotted,color=wrwrwr] (-3.34,8.18)-- (-1.6,5.86);
\draw [line width=0.5pt,dotted,color=wrwrwr] (-3.34,8.18)-- (-3.84,5.14);
\draw [line width=0.5pt,dotted,color=wrwrwr] (-2.56,6.3)-- (-5.44,6.78);
\draw [line width=0.5pt,dotted,color=wrwrwr] (-5.44,6.78)-- (-4.88,4.12);
\draw [line width=0.5pt,dotted,color=wrwrwr] (-4.88,4.12)-- (-7.12,3.4);
\draw [line width=0.5pt,dotted,color=wrwrwr] (-7.12,3.4)-- (-5.38,1.08);
\draw [line width=0.5pt,dotted,color=wrwrwr] (-5.38,1.08)-- (-6.64,-1.14);
\draw [line width=0.5pt,dotted,color=wrwrwr] (-5.04,-2.78)-- (-4.54,0.26);
\draw [line width=0.5pt,dotted,color=wrwrwr] (-4.54,0.26)-- (-7.42,0.74);
\draw [line width=0.5pt,dotted,color=wrwrwr] (-7.42,0.74)-- (-5.1,2.92);
\draw [line width=0.5pt,dotted,color=wrwrwr] (-5.1,2.92)-- (-6.56,4.82);
\draw [line width=0.5pt,dotted,color=wrwrwr] (-6.56,4.82)-- (-3.84,5.14);
\draw [line width=0.5pt,dashed,color=qqffff] (-2.8,-2.06)-- (-2.5,0.6);
\draw [line width=0.5pt,dashed,color=qqffff] (-2.5,0.6)-- (-4.54,0.26);
\draw [line width=0.5pt,dashed,color=qqffff] (-0.08,-1.74)-- (-1.68,-0.1);
\draw [line width=0.5pt,dashed,color=qqffff] (-1.68,-0.1)-- (-2.8,-2.06);
\draw [line width=0.5pt,dashed,color=qqffff] (2.24,0.44)-- (0.22,0.92);
\draw [line width=0.5pt,dashed,color=qqffff] (0.22,0.92)-- (-0.08,-1.74);
\draw [line width=0.5pt,dashed,color=qqffff] (-4.54,0.26)-- (-3.42,2.22);
\draw [line width=0.5pt,dashed,color=qqffff] (-3.42,2.22)-- (-5.1,2.92);
\draw [line width=0.5pt,dashed,color=qqffff] (2.74,3.48)-- (0.64,2.08);
\draw [line width=0.5pt,dashed,color=qqffff] (0.64,2.08)-- (2.24,0.44);
\draw [line width=0.5pt,dashed,color=qqffff] (1.28,5.38)-- (0.72,3.96);
\draw [line width=0.5pt,dashed,color=qqffff] (2.74,3.48)-- (0.72,3.96);
\draw [line width=0.5pt,dashed,color=qqffff] (-1.6,5.86)-- (-0.82,3.98);
\draw [line width=0.5pt,dashed,color=qqffff] (-0.82,3.98)-- (1.28,5.38);
\draw [line width=0.5pt,dashed,color=qqffff] (-3.84,5.14)-- (-2.16,4.44);
\draw [line width=0.5pt,dashed,color=qqffff] (-2.16,4.44)-- (-1.6,5.86);
\draw [line width=0.5pt,dashed,color=qqffff] (-5.1,2.92)-- (-3.06,3.26);
\draw [line width=0.5pt,dashed,color=qqffff] (-3.06,3.26)-- (-3.84,5.14);
\draw [line width=.8pt,dashed,color=ffwwqq] (-0.14,3.96)-- (1.9,4.3);
\draw [line width=.8pt,dashed,color=ffwwqq] (0.16,6.62)-- (-0.14,3.96);
\draw [line width=.8pt,dashed,color=ffwwqq] (-0.96,4.66)-- (0.16,6.62);
\draw [line width=.8pt,dashed,color=ffwwqq] (-2.56,6.3)-- (-0.96,4.66);
\draw [line width=.8pt,dashed,color=ffwwqq] (-2.86,3.64)-- (-2.56,6.3);
\draw [line width=.8pt,dashed,color=ffwwqq] (-4.88,4.12)-- (-2.86,3.64);
\draw [line width=.8pt,dashed,color=ffwwqq] (-3.28,2.48)-- (-4.88,4.12);
\draw [line width=.8pt,dashed,color=ffwwqq] (-5.38,1.08)-- (-3.28,2.48);
\draw [line width=.8pt,dashed,color=ffwwqq] (-5.38,1.08)-- (-3.36,0.6);
\draw [line width=.8pt,dashed,color=ffwwqq] (-3.92,-0.82)-- (-3.36,0.6);
\draw [line width=.8pt,dashed,color=ffwwqq] (-1.82,0.58)-- (-3.92,-0.82);
\draw [line width=.8pt,dashed,color=ffwwqq] (-1.04,-1.3)-- (-1.82,0.58);
\draw [line width=.8pt,dashed,color=ffwwqq] (-0.48,0.12)-- (-1.04,-1.3);
\draw [line width=.8pt,dashed,color=ffwwqq] (1.2,-0.58)-- (-0.48,0.12);
\draw [line width=.8pt,dashed,color=ffwwqq] (0.42,1.3)-- (1.2,-0.58);
\draw [line width=.8pt,dashed,color=ffwwqq] (2.46,1.64)-- (0.42,1.3);
\draw [line width=.8pt,dashed,color=ffwwqq] (0.78,2.34)-- (2.46,1.64);
\draw [line width=.8pt,dashed,color=ffwwqq] (1.9,4.3)-- (0.78,2.34);

\draw [line width=0.5pt,dotted,color=wrwrwr] (-6.64,-1.14) -- (-3.92,-0.82);
\draw [line width=2pt,color=ffwwqq] (-3.92,-0.82) -- (-1.04,-1.3);
\draw [line width=1.2pt,color=gray!50] (-6.64,-1.14) -- (-1.04,-1.3);

\draw [line width=2pt,color=qqffff] (-2.8,-2.06) -- (-0.08,-1.74);
\draw [line width=0.5pt,dotted,color=wrwrwr] (-0.08,-1.74) -- (2.8,-2.22);
\draw [line width=1.2pt,color=gray!50] (2.8,-2.22)-- (-2.8,-2.06);

\draw [line width=1.2pt,color=ffwwqq] (-5.04,-2.78) -- (-3.36,-3.48);
\draw [line width=1.2pt,color=rvwvcq] (-3.36,-3.48) -- (-1.34,-3.96);
\draw [line width=1.2pt,color=rvwvcq] (-1.34,-3.96) -- (0.7,-3.62);
\draw [line width=1.2pt,color=gray!50] (0.7,-3.62) -- (-5.04,-2.78);

\draw [->,>=stealth,line width=2pt,color=ffqqqq] (2.8,-2.22) -- (-1.04,-1.3);
\draw [->,>=stealth,line width=2pt,color=ffqqqq] (-2.8,-2.06) -- (-6.64,-1.14);

\begin{scriptsize}
\draw [fill=rvwvcq] (-3.36,-3.48) circle (1.5pt)node [color=rvwvcq,below]{$V_2$};
\draw [fill=rvwvcq] (-1.34,-3.96) circle (1.5pt)node [color=rvwvcq,below]{$V_3$};
\draw [fill=rvwvcq] (0.7,-3.62) circle (1.5pt)node [color=rvwvcq,below]{$V_4$};
\draw [fill=rvwvcq] (-0.08,-1.74) circle (1.5pt)node [color=rvwvcq,left]{$W_4$};
\draw [fill=rvwvcq] (-2.8,-2.06) circle (1.5pt)node [color=rvwvcq,above]{$W_2$};
\draw [fill=rvwvcq] (2.8,-2.22) circle (1.5pt)node [color=rvwvcq,right]{$V_5$};
\draw [fill=rvwvcq] (3.92,-0.26) circle (1.5pt);
\draw [fill=rvwvcq] (2.24,0.44) circle (1.5pt);
\draw [fill=rvwvcq] (4.48,1.16) circle (1.5pt);
\draw [fill=rvwvcq] (4.78,3.82) circle (1.5pt);
\draw [fill=rvwvcq] (2.74,3.48) circle (1.5pt);
\draw [fill=rvwvcq] (2.74,3.48) circle (1.5pt);
\draw [fill=rvwvcq] (4.78,3.82) circle (1.5pt);
\draw [fill=rvwvcq] (4.,5.7) circle (1.5pt);
\draw [fill=rvwvcq] (2.4,7.34) circle (1.5pt);
\draw [fill=rvwvcq] (1.28,5.38) circle (1.5pt);
\draw [fill=rvwvcq] (1.28,5.38) circle (1.5pt);
\draw [fill=rvwvcq] (2.4,7.34) circle (1.5pt);
\draw [fill=rvwvcq] (0.72,8.04) circle (1.5pt);
\draw [fill=rvwvcq] (-1.3,8.52) circle (1.5pt);
\draw [fill=rvwvcq] (-1.6,5.86) circle (1.5pt);
\draw [fill=rvwvcq] (-5.44,6.78) circle (1.5pt);
\draw [fill=rvwvcq] (-3.84,5.14) circle (1.5pt);
\draw [fill=rvwvcq] (-1.6,5.86) circle (1.5pt);
\draw [fill=rvwvcq] (-1.3,8.52) circle (1.5pt);
\draw [fill=rvwvcq] (-3.34,8.18) circle (1.5pt);
\draw [fill=rvwvcq] (-7.12,3.4) circle (1.5pt);
\draw [fill=rvwvcq] (-5.1,2.92) circle (1.5pt);
\draw [fill=rvwvcq] (-3.84,5.14) circle (1.5pt);
\draw [fill=rvwvcq] (-5.44,6.78) circle (1.5pt);
\draw [fill=rvwvcq] (-6.56,4.82) circle (1.5pt);
\draw [fill=rvwvcq] (-6.64,-1.14) circle (1.5pt)node [color=rvwvcq,left]{$V_0$};
\draw [fill=rvwvcq] (-4.54,0.26) circle (1.5pt);
\draw [fill=rvwvcq] (-5.1,2.92) circle (1.5pt);
\draw [fill=rvwvcq] (-7.12,3.4) circle (1.5pt);
\draw [fill=rvwvcq] (-7.42,0.74) circle (1.5pt);
\draw [fill=rvwvcq] (-6.64,-1.14) circle (1.5pt);
\draw [fill=rvwvcq] (-5.04,-2.78) circle (1.5pt)node [color=rvwvcq,below]{$V_1$};
\draw [fill=rvwvcq] (-3.36,-3.48) circle (1.5pt);
\draw [fill=rvwvcq] (-2.8,-2.06) circle (1.5pt);
\draw [fill=rvwvcq] (-4.54,0.26) circle (1.5pt);
\draw [fill=wrwrwr] (-1.32,2.28) circle (2.0pt);
\draw [fill=rvwvcq] (0.16,6.62) circle (1.5pt);
\draw [fill=rvwvcq] (-2.56,6.3) circle (1.5pt);
\draw [fill=rvwvcq] (-4.88,4.12) circle (1.5pt);
\draw [fill=rvwvcq] (-5.38,1.08) circle (1.5pt);
\draw [rvwvcq] (-4.4,-0.65) node {$W_1$};

\draw [rvwvcq] (-1.5,-1) node {$W_3$};
\draw [fill=rvwvcq] (1.2,-0.58) circle (1.5pt);
\draw [fill=rvwvcq] (2.46,1.64) circle (1.5pt);
\draw [fill=rvwvcq] (1.9,4.3) circle (1.5pt);
\draw[color=qqffqq] (-4,-1.5) node {$\sjx_L$};
\draw[color=qqffqq] (0.5368233223978434,-2.4) node {$\sjx_R$};
\draw [fill=wrwrwr] (-2.8,-2.06) circle (2.0pt);
\draw [fill=wrwrwr] (-2.8,-2.06) circle (2.0pt);
\draw[color=cyan] (-2.486752671472437,-3.95) node {$\sbx_{4}$};
\draw[color=ffqqqq] (1.4,-1.6) node {$s_{3,R}$};
\draw[color=ffqqqq] (-4.1,-2.1) node {$s_{3}$};

\draw [fill=rvwvcq] (-2.8,-2.06) circle (1.5pt);
\draw [fill=rvwvcq] (-4.54,0.26) circle (1.5pt);
\draw [fill=rvwvcq] (-0.08,-1.74) circle (1.5pt);
\draw [fill=rvwvcq] (-2.8,-2.06) circle (1.5pt);
\draw [fill=rvwvcq] (2.24,0.44) circle (1.5pt);
\draw [fill=rvwvcq] (-0.08,-1.74) circle (1.5pt);
\draw [fill=rvwvcq] (-4.54,0.26) circle (1.5pt);
\draw [fill=rvwvcq] (-5.1,2.92) circle (1.5pt);
\draw [fill=rvwvcq] (2.74,3.48) circle (1.5pt);
\draw [fill=rvwvcq] (2.24,0.44) circle (1.5pt);
\draw [fill=rvwvcq] (1.28,5.38) circle (1.5pt);
\draw [fill=rvwvcq] (2.74,3.48) circle (1.5pt);
\draw [fill=rvwvcq] (-1.6,5.86) circle (1.5pt);
\draw [fill=rvwvcq] (1.28,5.38) circle (1.5pt);
\draw [fill=rvwvcq] (-3.84,5.14) circle (1.5pt);
\draw [fill=rvwvcq] (-1.6,5.86) circle (1.5pt);
\draw [fill=rvwvcq] (-5.1,2.92) circle (1.5pt);
\draw [fill=rvwvcq] (-3.84,5.14) circle (1.5pt);
\draw [fill=rvwvcq] (1.9,4.3) circle (1.5pt);
\draw [fill=rvwvcq] (0.16,6.62) circle (1.5pt);
\draw [fill=rvwvcq] (0.16,6.62) circle (1.5pt);
\draw [fill=rvwvcq] (-2.56,6.3) circle (1.5pt);
\draw [fill=rvwvcq] (-2.56,6.3) circle (1.5pt);
\draw [fill=rvwvcq] (-4.88,4.12) circle (1.5pt);
\draw [fill=rvwvcq] (-4.88,4.12) circle (1.5pt);
\draw [fill=rvwvcq] (-5.38,1.08) circle (1.5pt);
\draw [fill=rvwvcq] (-5.38,1.08) circle (1.5pt);
\draw [fill=rvwvcq] (-3.92,-0.82) circle (1.5pt);
\draw [fill=rvwvcq] (-3.92,-0.82) circle (1.5pt);
\draw [fill=rvwvcq] (-1.04,-1.3) circle (1.5pt);
\draw [fill=rvwvcq] (1.2,-0.58) circle (1.5pt);
\draw [fill=rvwvcq] (1.2,-0.58) circle (1.5pt);
\draw [fill=rvwvcq] (2.46,1.64) circle (1.5pt);
\draw [fill=rvwvcq] (2.46,1.64) circle (1.5pt);
\draw [fill=rvwvcq] (1.9,4.3) circle (1.5pt);
\end{scriptsize}
\end{tikzpicture}
}\vskip -.5cm
\caption{The triangles $\sjx_L,\sjx_R$ and the square $\sbx_-$ in type $E_7$}
\label{fig:polygon-E7}
\vskip .6cm
\definecolor{qqffqq}{rgb}{0.,1.,0.}
\makebox[\textwidth]{
\definecolor{wrwrwr}{rgb}{0.3803921568627451,0.3803921568627451,0.3803921568627451}
\definecolor{qqffff}{rgb}{0.,1.,1.}
\definecolor{ffwwqq}{rgb}{1.,0.4,0.}
\definecolor{qqccqq}{rgb}{0.,0.8,0.}
\definecolor{rvwvcq}{rgb}{0.08235294117647059,0.396078431372549,0.7529411764705882}
\centering
\begin{tikzpicture}[line cap=round,line join=round,arrow/.style={->,>=stealth,thick},scale=1.1]
\clip(-5.5,-4.616555547102639) rectangle (3.5,-.9);
\fill[line width=1.7pt,color=qqffff,fill=qqffff,fill opacity=0.4] (-5.04,-2.78) -- (-2.8,-2.06) -- (-0.08,-1.74) -- (2.8,-2.22) -- (0.7,-3.62) -- (-1.34,-3.96) -- (-3.36,-3.48) -- cycle;
\draw [line width=1.2pt,color=qqccqq] (0.7,-3.62)-- (2.8,-2.22);
\draw [line width=1.2pt,color=qqccqq] (2.8,-2.22)-- (3.92,-0.26);
\draw [line width=1.2pt,color=ffwwqq] (3.92,-0.26)-- (4.48,1.16);
\draw [line width=1.2pt,color=ffwwqq] (4.48,1.16)-- (4.78,3.82);
\draw [line width=1.2pt,color=rvwvcq] (4.78,3.82)-- (4.,5.7);
\draw [line width=1.2pt,color=rvwvcq] (4.,5.7)-- (2.4,7.34);
\draw [line width=1.2pt,color=qqccqq] (2.4,7.34)-- (0.72,8.04);
\draw [line width=1.2pt,color=qqccqq] (0.72,8.04)-- (-1.3,8.52);
\draw [line width=1.2pt,color=ffwwqq] (-1.3,8.52)-- (-3.34,8.18);
\draw [line width=1.2pt,color=ffwwqq] (-3.34,8.18)-- (-5.44,6.78);
\draw [line width=1.2pt,color=rvwvcq] (-5.44,6.78)-- (-6.56,4.82);
\draw [line width=1.2pt,color=rvwvcq] (-6.56,4.82)-- (-7.12,3.4);
\draw [line width=1.2pt,color=qqccqq] (-7.12,3.4)-- (-7.42,0.74);
\draw [line width=1.2pt,color=qqccqq] (-7.42,0.74)-- (-6.64,-1.14);
\draw [line width=1.2pt,color=ffwwqq] (-6.64,-1.14)-- (-5.04,-2.78);
\draw [line width=1.7pt,color=qqffff] (-0.08,-1.74)-- (2.24,0.44);
\draw [line width=1.7pt,color=qqffff] (2.24,0.44)-- (2.74,3.48);
\draw [line width=1.7pt,color=qqffff] (2.74,3.48)-- (1.28,5.38);
\draw [line width=1.7pt,color=qqffff] (1.28,5.38)-- (-1.6,5.86);
\draw [line width=1.7pt,color=qqffff] (-1.6,5.86)-- (-3.84,5.14);
\draw [line width=1.7pt,color=qqffff] (-3.84,5.14)-- (-5.1,2.92);
\draw [line width=1.7pt,color=qqffff] (-5.1,2.92)-- (-4.54,0.26);
\draw [line width=1.7pt,color=qqffff] (-4.54,0.26)-- (-2.8,-2.06);
\draw [line width=1.7pt,color=ffwwqq] (0.16,6.62)-- (-2.56,6.3);

\draw [line width=1.7pt,color=ffwwqq] (-2.56,6.3)-- (-4.88,4.12);
\draw [line width=1.7pt,color=ffwwqq] (-4.88,4.12)-- (-5.38,1.08);
\draw [line width=1.7pt,color=ffwwqq] (-5.38,1.08)-- (-3.92,-0.82);
\draw [line width=1.7pt,color=ffwwqq] (-1.04,-1.3)-- (1.2,-0.58);
\draw [line width=1.7pt,color=ffwwqq] (1.2,-0.58)-- (2.46,1.64);
\draw [line width=1.7pt,color=ffwwqq] (2.46,1.64)-- (1.9,4.3);
\draw [line width=1.7pt,color=ffwwqq] (1.9,4.3)-- (0.16,6.62);
\draw [line width=2pt,color=ffwwqq] (-3.92,-0.82) -- (-1.04,-1.3);

\draw [line width=1.2pt,color=qqffff] (-3.36,-3.48)-- (-1.34,-3.96);
\draw [line width=1.2pt,color=qqffff] (-1.34,-3.96)-- (0.7,-3.62);

\draw [line width=.8pt,dashed,gray] (-3.36,-3.48)-- (-2.8,-2.06);
\draw [line width=.8pt,dashed,gray] (-1.34,-3.96)-- (-1.04,-1.3);
\draw [line width=.8pt,dashed,gray] (0.7,-3.62)-- (-0.08,-1.74);
\draw [line width=.8pt,dashed,gray] (2.8,-2.22)-- (1.2,-0.58);
\draw [line width=.8pt,dashed,gray] (3.92,-0.26)-- (2.24,0.44);
\draw [line width=.8pt,dashed,gray] (4.48,1.16)-- (2.46,1.64);
\draw [line width=.8pt,dashed,gray] (4.78,3.82)-- (2.74,3.48);
\draw [line width=.8pt,dashed,gray] (4.,5.7)-- (1.9,4.3);
\draw [line width=.8pt,dashed,gray] (1.28,5.38)-- (2.4,7.34);
\draw [line width=.8pt,dashed,gray] (0.72,8.04)-- (0.16,6.62);
\draw [line width=.8pt,dashed,gray] (-1.6,5.86)-- (-1.3,8.52);
\draw [line width=.8pt,dashed,gray] (-2.56,6.3)-- (-3.34,8.18);
\draw [line width=.8pt,dashed,gray] (-3.84,5.14)-- (-5.44,6.78);
\draw [line width=.8pt,dashed,gray] (-4.88,4.12)-- (-6.56,4.82);
\draw [line width=.8pt,dashed,gray] (-7.12,3.4)-- (-5.1,2.92);
\draw [line width=.8pt,dashed,gray] (-5.38,1.08)-- (-7.42,0.74);
\draw [line width=.8pt,dashed,gray] (-6.64,-1.14)-- (-4.54,0.26);
\draw [line width=.8pt,dashed,gray] (-3.92,-0.82)-- (-5.04,-2.78);
\draw [line width=0.5pt,dotted,color=wrwrwr] (-6.64,-1.14)-- (-3.92,-0.82);
\draw [line width=0.5pt,dotted,color=wrwrwr] (-3.36,-3.48)-- (-3.92,-0.82);
\draw [line width=0.5pt,dotted,color=wrwrwr] (-2.8,-2.06)-- (-5.04,-2.78);
\draw [line width=0.5pt,dotted,color=wrwrwr] (-1.04,-1.3)-- (-3.36,-3.48);
\draw [line width=0.5pt,dotted,color=wrwrwr] (-2.8,-2.06)-- (-1.34,-3.96);
\draw [line width=0.5pt,dotted,color=wrwrwr] (-0.08,-1.74)-- (-1.34,-3.96);
\draw [line width=0.5pt,dotted,color=wrwrwr] (-1.04,-1.3)-- (0.7,-3.62);
\draw [line width=0.5pt,dotted,color=wrwrwr] (0.7,-3.62)-- (1.2,-0.58);
\draw [line width=0.5pt,dotted,color=wrwrwr] (2.8,-2.22)-- (2.24,0.44);
\draw [line width=0.5pt,dotted,color=wrwrwr] (1.2,-0.58)-- (3.92,-0.26);
\draw [line width=0.5pt,dotted,color=wrwrwr] (2.46,1.64)-- (3.92,-0.26);
\draw [line width=0.5pt,dotted,color=wrwrwr] (4.48,1.16)-- (2.24,0.44);
\draw [line width=0.5pt,dotted,color=wrwrwr] (2.46,1.64)-- (4.78,3.82);
\draw [line width=0.5pt,dotted,color=wrwrwr] (4.48,1.16)-- (2.74,3.48);
\draw [line width=0.5pt,dotted,color=wrwrwr] (4.,5.7)-- (2.74,3.48);
\draw [line width=0.5pt,dotted,color=wrwrwr] (1.9,4.3)-- (4.78,3.82);
\draw [line width=0.5pt,dotted,color=wrwrwr] (1.9,4.3)-- (2.4,7.34);
\draw [line width=0.5pt,dotted,color=wrwrwr] (4.,5.7)-- (1.28,5.38);
\draw [line width=0.5pt,dotted,color=wrwrwr] (1.28,5.38)-- (0.72,8.04);
\draw [line width=0.5pt,dotted,color=wrwrwr] (0.16,6.62)-- (2.4,7.34);
\draw [line width=0.5pt,dotted,color=wrwrwr] (0.16,6.62)-- (-1.3,8.52);
\draw [line width=0.5pt,dotted,color=wrwrwr] (-1.6,5.86)-- (0.72,8.04);
\draw [line width=0.5pt,dotted,color=wrwrwr] (-1.3,8.52)-- (-2.56,6.3);
\draw [line width=0.5pt,dotted,color=wrwrwr] (-3.34,8.18)-- (-1.6,5.86);
\draw [line width=0.5pt,dotted,color=wrwrwr] (-3.34,8.18)-- (-3.84,5.14);
\draw [line width=0.5pt,dotted,color=wrwrwr] (-2.56,6.3)-- (-5.44,6.78);
\draw [line width=0.5pt,dotted,color=wrwrwr] (-5.44,6.78)-- (-4.88,4.12);
\draw [line width=0.5pt,dotted,color=wrwrwr] (-4.88,4.12)-- (-7.12,3.4);
\draw [line width=0.5pt,dotted,color=wrwrwr] (-7.12,3.4)-- (-5.38,1.08);
\draw [line width=0.5pt,dotted,color=wrwrwr] (-5.38,1.08)-- (-6.64,-1.14);
\draw [line width=0.5pt,dotted,color=wrwrwr] (-5.04,-2.78)-- (-4.54,0.26);
\draw [line width=0.5pt,dotted,color=wrwrwr] (-4.54,0.26)-- (-7.42,0.74);
\draw [line width=0.5pt,dotted,color=wrwrwr] (-7.42,0.74)-- (-5.1,2.92);
\draw [line width=0.5pt,dotted,color=wrwrwr] (-5.1,2.92)-- (-6.56,4.82);
\draw [line width=0.5pt,dotted,color=wrwrwr] (-6.56,4.82)-- (-3.84,5.14);
\draw [line width=0.5pt,dashed,color=qqffff] (-2.8,-2.06)-- (-2.5,0.6);
\draw [line width=0.5pt,dashed,color=qqffff] (-2.5,0.6)-- (-4.54,0.26);
\draw [line width=0.5pt,dashed,color=qqffff] (-0.08,-1.74)-- (-1.68,-0.1);
\draw [line width=0.5pt,dashed,color=qqffff] (-1.68,-0.1)-- (-2.8,-2.06);
\draw [line width=0.5pt,dashed,color=qqffff] (2.24,0.44)-- (0.22,0.92);
\draw [line width=0.5pt,dashed,color=qqffff] (0.22,0.92)-- (-0.08,-1.74);
\draw [line width=0.5pt,dashed,color=qqffff] (-4.54,0.26)-- (-3.42,2.22);
\draw [line width=0.5pt,dashed,color=qqffff] (-3.42,2.22)-- (-5.1,2.92);
\draw [line width=0.5pt,dashed,color=qqffff] (2.74,3.48)-- (0.64,2.08);
\draw [line width=0.5pt,dashed,color=qqffff] (0.64,2.08)-- (2.24,0.44);
\draw [line width=0.5pt,dashed,color=qqffff] (1.28,5.38)-- (0.72,3.96);
\draw [line width=0.5pt,dashed,color=qqffff] (2.74,3.48)-- (0.72,3.96);
\draw [line width=0.5pt,dashed,color=qqffff] (-1.6,5.86)-- (-0.82,3.98);
\draw [line width=0.5pt,dashed,color=qqffff] (-0.82,3.98)-- (1.28,5.38);
\draw [line width=0.5pt,dashed,color=qqffff] (-3.84,5.14)-- (-2.16,4.44);
\draw [line width=0.5pt,dashed,color=qqffff] (-2.16,4.44)-- (-1.6,5.86);
\draw [line width=0.5pt,dashed,color=qqffff] (-5.1,2.92)-- (-3.06,3.26);
\draw [line width=0.5pt,dashed,color=qqffff] (-3.06,3.26)-- (-3.84,5.14);
\draw [line width=.8pt,dashed,color=ffwwqq] (-0.14,3.96)-- (1.9,4.3);
\draw [line width=.8pt,dashed,color=ffwwqq] (0.16,6.62)-- (-0.14,3.96);
\draw [line width=.8pt,dashed,color=ffwwqq] (-0.96,4.66)-- (0.16,6.62);
\draw [line width=.8pt,dashed,color=ffwwqq] (-2.56,6.3)-- (-0.96,4.66);
\draw [line width=.8pt,dashed,color=ffwwqq] (-2.86,3.64)-- (-2.56,6.3);
\draw [line width=.8pt,dashed,color=ffwwqq] (-4.88,4.12)-- (-2.86,3.64);
\draw [line width=.8pt,dashed,color=ffwwqq] (-3.28,2.48)-- (-4.88,4.12);
\draw [line width=.8pt,dashed,color=ffwwqq] (-5.38,1.08)-- (-3.28,2.48);
\draw [line width=.8pt,dashed,color=ffwwqq] (-5.38,1.08)-- (-3.36,0.6);
\draw [line width=.8pt,dashed,color=ffwwqq] (-3.92,-0.82)-- (-3.36,0.6);
\draw [line width=.8pt,dashed,color=ffwwqq] (-1.82,0.58)-- (-3.92,-0.82);
\draw [line width=.8pt,dashed,color=ffwwqq] (-1.04,-1.3)-- (-1.82,0.58);
\draw [line width=.8pt,dashed,color=ffwwqq] (-0.48,0.12)-- (-1.04,-1.3);
\draw [line width=.8pt,dashed,color=ffwwqq] (1.2,-0.58)-- (-0.48,0.12);
\draw [line width=.8pt,dashed,color=ffwwqq] (0.42,1.3)-- (1.2,-0.58);
\draw [line width=.8pt,dashed,color=ffwwqq] (2.46,1.64)-- (0.42,1.3);
\draw [line width=.8pt,dashed,color=ffwwqq] (0.78,2.34)-- (2.46,1.64);
\draw [line width=.8pt,dashed,color=ffwwqq] (1.9,4.3)-- (0.78,2.34);

\begin{scriptsize}
\draw [fill=rvwvcq] (-3.36,-3.48) circle (1.5pt)node [color=rvwvcq,below]{$V_2$};
\draw [fill=rvwvcq] (-1.34,-3.96) circle (1.5pt)node [color=rvwvcq,below]{$V_3$};
\draw [fill=rvwvcq] (0.7,-3.62) circle (1.5pt)node [color=rvwvcq,below]{$V_4$};
\draw[color=rvwvcq] (0.2958395052839391,-1.95) node {$W_4$};
\draw [fill=rvwvcq] (-2.8,-2.06) circle (1.5pt)node [color=rvwvcq,above]{$W_2$};
\draw [fill=rvwvcq] (2.8,-2.22) circle (1.5pt)node [color=rvwvcq,right]{$V_5$};
\draw [fill=rvwvcq] (3.92,-0.26) circle (1.5pt);
\draw [fill=rvwvcq] (2.24,0.44) circle (1.5pt);
\draw [fill=rvwvcq] (4.48,1.16) circle (1.5pt);
\draw [fill=rvwvcq] (4.78,3.82) circle (1.5pt);
\draw [fill=rvwvcq] (2.74,3.48) circle (1.5pt);
\draw [fill=rvwvcq] (2.74,3.48) circle (1.5pt);
\draw [fill=rvwvcq] (4.78,3.82) circle (1.5pt);
\draw [fill=rvwvcq] (4.,5.7) circle (1.5pt);
\draw [fill=rvwvcq] (2.4,7.34) circle (1.5pt);
\draw [fill=rvwvcq] (1.28,5.38) circle (1.5pt);
\draw [fill=rvwvcq] (1.28,5.38) circle (1.5pt);
\draw [fill=rvwvcq] (2.4,7.34) circle (1.5pt);
\draw [fill=rvwvcq] (0.72,8.04) circle (1.5pt);
\draw [fill=rvwvcq] (-1.3,8.52) circle (1.5pt);
\draw [fill=rvwvcq] (-1.6,5.86) circle (1.5pt);
\draw [fill=rvwvcq] (-5.44,6.78) circle (1.5pt);
\draw [fill=rvwvcq] (-3.84,5.14) circle (1.5pt);
\draw [fill=rvwvcq] (-1.6,5.86) circle (1.5pt);
\draw [fill=rvwvcq] (-1.3,8.52) circle (1.5pt);
\draw [fill=rvwvcq] (-3.34,8.18) circle (1.5pt);
\draw [fill=rvwvcq] (-7.12,3.4) circle (1.5pt);
\draw [fill=rvwvcq] (-5.1,2.92) circle (1.5pt);
\draw [fill=rvwvcq] (-3.84,5.14) circle (1.5pt);
\draw [fill=rvwvcq] (-5.44,6.78) circle (1.5pt);
\draw [fill=rvwvcq] (-6.56,4.82) circle (1.5pt);
\draw [fill=rvwvcq] (-6.64,-1.14) circle (1.5pt)node [color=rvwvcq,left]{$V_0$};
\draw [fill=rvwvcq] (-4.54,0.26) circle (1.5pt);
\draw [fill=rvwvcq] (-5.1,2.92) circle (1.5pt);
\draw [fill=rvwvcq] (-7.12,3.4) circle (1.5pt);
\draw [fill=rvwvcq] (-7.42,0.74) circle (1.5pt);
\draw [fill=rvwvcq] (-6.64,-1.14) circle (1.5pt);
\draw [fill=rvwvcq] (-5.04,-2.78) circle (1.5pt)node [color=rvwvcq,below]{$V_1$};
\draw [fill=rvwvcq] (-3.36,-3.48) circle (1.5pt);
\draw [fill=rvwvcq] (-2.8,-2.06) circle (1.5pt);
\draw [fill=rvwvcq] (-4.54,0.26) circle (1.5pt);
\draw [fill=wrwrwr] (-1.32,2.28) circle (2.0pt);
\draw [fill=rvwvcq] (0.16,6.62) circle (1.5pt);
\draw [fill=rvwvcq] (-2.56,6.3) circle (1.5pt);
\draw [fill=rvwvcq] (-4.88,4.12) circle (1.5pt);
\draw [fill=rvwvcq] (-5.38,1.08) circle (1.5pt);
\draw [fill=rvwvcq] (-3.92,-0.82) circle (1.5pt)node [color=rvwvcq,left]{$W_1$};
\draw [fill=rvwvcq] (-1.04,-1.3) circle (1.5pt)node [color=rvwvcq,left]{$W_3$};
\draw [fill=rvwvcq] (1.2,-0.58) circle (1.5pt);
\draw [fill=rvwvcq] (2.46,1.64) circle (1.5pt);
\draw [fill=rvwvcq] (1.9,4.3) circle (1.5pt);

\draw[color=black] (-0.3142178053551662,-4) node {$s_{7}$};
\draw[color=black] (-1.5,-3.35) node {$s_{6}$};
\draw[color=black] (-4.313482397322634,-3.3) node {$s_5$};
\draw[color=black] (-0.9073290795876299,-2.7) node {$s_4$};
\draw[color=black] (-1.8563071183595714,-2.3) node {$s_2$};
\draw[color=black] (-1.8563071183595714,-1.7) node {$s_1$};
\draw[color=black] (0.7194904154499845,-1.5) node {$s_3$};
\draw [fill=rvwvcq] (3.92,-0.26) circle (1.5pt);
\draw [fill=rvwvcq] (2.24,0.44) circle (1.5pt);
\draw [fill=rvwvcq] (4.78,3.82) circle (1.5pt);
\draw [fill=rvwvcq] (2.74,3.48) circle (1.5pt);
\draw [fill=rvwvcq] (4.,5.7) circle (1.5pt);
\draw [fill=rvwvcq] (2.4,7.34) circle (1.5pt);
\draw [fill=rvwvcq] (1.28,5.38) circle (1.5pt);
\draw [fill=rvwvcq] (0.72,8.04) circle (1.5pt);
\draw [fill=rvwvcq] (-1.6,5.86) circle (1.5pt);
\draw [fill=rvwvcq] (-5.44,6.78) circle (1.5pt);
\draw [fill=rvwvcq] (-3.84,5.14) circle (1.5pt);
\draw [fill=rvwvcq] (-3.34,8.18) circle (1.5pt);
\draw [fill=rvwvcq] (-7.12,3.4) circle (1.5pt);
\draw [fill=rvwvcq] (-5.1,2.92) circle (1.5pt);
\draw [fill=rvwvcq] (-6.56,4.82) circle (1.5pt);
\draw [fill=rvwvcq] (-4.54,0.26) circle (1.5pt);
\draw [fill=rvwvcq] (-7.42,0.74) circle (1.5pt);
\draw [fill=rvwvcq] (-2.8,-2.06) circle (1.5pt);
\draw [fill=rvwvcq] (-4.54,0.26) circle (1.5pt);
\draw [fill=rvwvcq] (-0.08,-1.74) circle (1.5pt);
\draw [fill=rvwvcq] (-2.8,-2.06) circle (1.5pt);
\draw [fill=rvwvcq] (2.24,0.44) circle (1.5pt);
\draw [fill=rvwvcq] (-0.08,-1.74) circle (1.5pt);
\draw [fill=rvwvcq] (2.74,3.48) circle (1.5pt);
\draw [fill=rvwvcq] (1.28,5.38) circle (1.5pt);
\draw [fill=rvwvcq] (-1.6,5.86) circle (1.5pt);
\draw [fill=rvwvcq] (-3.84,5.14) circle (1.5pt);
\draw [fill=rvwvcq] (-1.6,5.86) circle (1.5pt);
\draw [fill=rvwvcq] (-5.1,2.92) circle (1.5pt);
\draw [fill=rvwvcq] (-3.84,5.14) circle (1.5pt);
\draw [fill=rvwvcq] (-5.1,2.92) circle (1.5pt);
\draw [fill=rvwvcq] (0.16,6.62) circle (1.5pt);
\draw [fill=rvwvcq] (-2.56,6.3) circle (1.5pt);
\draw [fill=rvwvcq] (-2.56,6.3) circle (1.5pt);
\draw [fill=rvwvcq] (-4.88,4.12) circle (1.5pt);
\draw [fill=rvwvcq] (-5.38,1.08) circle (1.5pt);
\draw [fill=rvwvcq] (-3.92,-0.82) circle (1.5pt);
\draw [fill=rvwvcq] (-1.04,-1.3) circle (1.5pt);
\draw [fill=rvwvcq] (-2.8,-2.06) circle (1.5pt);
\draw [fill=rvwvcq] (-0.08,-1.74) circle (1.5pt);
\draw [fill=rvwvcq] (-1.04,-1.3) circle (1.5pt);
\draw [fill=rvwvcq] (1.2,-0.58) circle (1.5pt);
\draw [fill=rvwvcq] (-0.08,-1.74) circle (1.5pt);
\draw [fill=rvwvcq] (-0.08,-1.74) circle (1.5pt);
\draw [fill=rvwvcq] (1.2,-0.58) circle (1.5pt);
\draw [fill=rvwvcq] (2.24,0.44) circle (1.5pt);
\draw [fill=rvwvcq] (1.2,-0.58) circle (1.5pt);
\draw [fill=rvwvcq] (2.46,1.64) circle (1.5pt);
\draw [fill=rvwvcq] (2.74,3.48) circle (1.5pt);
\draw [fill=rvwvcq] (2.46,1.64) circle (1.5pt);
\draw [fill=rvwvcq] (2.46,1.64) circle (1.5pt);
\draw [fill=rvwvcq] (1.9,4.3) circle (1.5pt);
\draw [fill=rvwvcq] (1.28,5.38) circle (1.5pt);
\draw [fill=rvwvcq] (0.16,6.62) circle (1.5pt);
\end{scriptsize}

\draw [line width=0.5pt,dotted,color=wrwrwr] (-0.08,-1.74) -- (2.8,-2.22);
\draw [line width=1.2pt,color=rvwvcq] (-3.36,-3.48) -- (-1.34,-3.96);

\draw [->,>=stealth,line width=1.5pt] (-1.34,-3.96) -- (0.7,-3.62);
\draw [->,>=stealth,line width=1.5pt] (0.7,-3.62) -- (-3.36,-3.48);
\draw [->,>=stealth,line width=1.5pt] (-3.36,-3.48) -- (-5.04,-2.78);
\draw [->,>=stealth,line width=1.5pt] (-5.04,-2.78) -- (2.8,-2.22);
\draw [->,>=stealth,line width=1.5pt] (2.8,-2.22) -- (-2.8,-2.06);
\draw [->,>=stealth,line width=1.5pt] (-2.8,-2.06) -- (-0.08,-1.74);
\draw [->,>=stealth,line width=1.5pt] (2.8,-2.22) -- (-1.04,-1.3);

\draw [->,>=stealth,line width=.8pt,bend right] (.8,-1.75) to (.8,-2.35);
\draw [->,>=stealth,line width=.8pt,bend right=20] (-.9,-3.6) to (-.9,-3.85);
\draw [<-,>=stealth,line width=.8pt,bend right=-20] (-.9,-1.85) to (-.9,-2.1);
\draw [->,>=stealth,line width=.8pt,bend right=20] (-.6,-2.15) to (-.6,-2.45);
\draw [<-,>=stealth,line width=.8pt,bend left] (-4,-2.73) to (-4,-3.2);
\draw [->,>=stealth,line width=.8pt,bend right] (-2.8,-3.5) to (-3.8,-3.25);
\end{tikzpicture}
}
\vskip -.3cm
\begin{tikzpicture}[scale=.5,xscale=1,ar/.style={->,thick,>=stealth}]
\draw(0,0)node(v1){$s_1$}(2,0)node(v2){$s_2$}(4,0)node(v4){$s_4$}(6,0)node(v5){$s_5$}(8,0)node(v6)
{$s_6$}(10,0)node(vn){$s_7$}
    (4,2)node(v3){$s_3$};

\draw[ar](v6)to(vn);
\draw[ar](v2)to(v1);
\draw[ar](v2)to(v4);
\draw[ar](v3)to(v4);
\draw[ar](v5)to(v4);
\draw[ar](v6)to(v5);
\end{tikzpicture}
\caption{The intersection quiver of an $18$-gon of type $E_7$}
\label{fig:int-quiver-E7}
\end{figure}

Let $\sigma\in\ToSt(E_n)$, where $n\in\{6,7,8\}$, and $\hgon\colon=\hgon_\sigma$ the associated stable $h$-gon of type $E_n$.
Reorder the points in $\{V_j, W_j\mid j \in \ZZ_{h_Q}\}$ increasingly with respect to \eqref{eq:order} as $\{Y_i \mid 1\le i\le 2h_Q\}$.
As in \S~\ref{s4.3} for the case $n=6$, we have the following in all exceptional cases:
\begin{itemize}
\item the points in $\{Y_i\mid 1\le i\le n-3\}$ are consecutive vertices of $\hgon$.
 Thus we can assume that they are $\{V_j\mid 1\le j\le n-3\}$.
 Those vertices $V_j$'s form a positively convex $(n-3)$-gon $\mathbf{P}_{n-3}$. More precisely, $\mathbf{P}_3$ is a triangle $\sjx_-$, $\mathbf{P}_4$ is a square $\sbx_-$ and $\mathbf{P}_5$ is a pentagon $\wbx_-$.
\item there are two distinct parallel triangles $\sjx_L$ and $\sjx_R$ (since $W_1V_0=W_4W_2$ and $W_3W_1=V_{n-2}W_4$), where
  \begin{itemize}
      \item  $\sjx_L$ is the triangle with vertices $V_0,W_1$ and $ W_3$, and
      \item  $\sjx_R$ is the triangle with vertices $W_2,W_4$ and $V_{n-2}$.
  \end{itemize}
\item
define $\sjx_+$ to be the smaller one among the above two triangles as in \eqref{def:up-tri} and $\agon_n$ the convex hull of $\mathbf{P}_{n-3}\cup\sjx_+$, which is a positively convex $n$-gon due to the stability of
$\hgon$.
\item denote the $\imz$-ind diagonals of $\agon_n$ from top to bottom by $s_1, s_2, s_4,\ldots, s_n$ respectively and set
$s_3\colon=\UP{V_0W_2}=\UP{W_3V_{n-2}}$.
\end{itemize}
\begin{definition}
For type $E_n$, the set $\Sim \hgon$ of \emph{$\imz$-ind diagonals} of $\hgon$ is defined to be $\{s_i\mid 1\le i\le n\}$.
Thus we have the \emph{intersection quiver} $\hh{Q}(\hgon)$ as in Definition \ref{def:An-quiver}.
See Figures~\ref{fig:polygon-E7}-\ref{fig:int-quiver-E7} and Figures~\ref{fig:polygon-E8}-\ref{fig:int-quiver-E8}
for examples of the associated intersection quivers in types $E_7$ and $E_8$ respectively.
\end{definition}

Now, we can upgrade Proposition~\ref{pp:thmE6} to all exceptional cases.

\begin{theorem}\label{thm:En}
Theorem~\ref{thm:An} holds for any quiver $Q$ of type $E_n$, where $n\in\{6,7,8\}$.
In particular, the underlying diagram of the Ext-quiver/intersection quiver is of the form
\begin{gather}\label{eq:En}
\begin{tikzpicture}[scale=.5,xscale=1,ar/.style={-,thick}]
\draw(0,0)node(v1){$s_1$}(2,0)node(v2){$s_2$}(4,0)node(v4)
{$s_4$}(6,0)node(v5){$s_5$}(8,0)node(v6){$\cdots$}(10,0)node(vn){$s_n.$}
    (4,2)node(v3){$s_3$};
\draw[ar](v2)edge(v1)edge(v4);
\draw[ar](v3)edge(v4);
\draw[ar](v5)edge(v6)edge(v4);
\draw[ar](v6)edge(vn);
\end{tikzpicture}
\end{gather}
\end{theorem}

\section{Reineke's conjecture}
In this section, we will use the geometric model constructed in \S~\ref{sec:geo} to prove
Reineke's conjecture about the existence of total stability functions on the module category of any Dynkin quiver.
The conjecture for type $A$ case is proved in \cite{BGMS19}, where the construction is more or less the same.
However, we include the details as well for self-containedness.

\subsection{Type $A$ and $D$}
\begin{proposition}\label{prop:An-quiver to polygon}
For any quiver $Q$ of type $A_n$, there is a stable $(n+1)$-gon $\hgon$ such that $\hh{Q}^{1-?}(\hgon)\cong Q$.
\end{proposition}

\begin{figure}[ht]\centering
\begin{tikzpicture}[scale=.5, rotate=0, xscale=1,arrow/.style={->,>=stealth,thick}]
\draw [cyan!50,very thick] (0,0) circle (6cm);

\draw[thin,gray,dashed](-6,0)to(6,0);
\draw[thin,gray,dashed](-5.7,-1.5)to(5.7,-1.5);
\draw[thin,gray,dashed](-5,-3)to(5,-3);
\draw[thin,gray,dashed](-4,-4.4)to(4,-4.4);
\draw[thin,gray,dashed](-5.7,1.5)to(5.7,1.5);
\draw[thin,gray,dashed](-5,3)to(5,3);
\draw[thin,gray,dashed](-4,4.4)to(4,4.4);
\draw [fill] (0,-6) circle (.1cm);
    \draw[red] (0,-6.8) node {$Y_1$};
\draw [fill] (0,6) circle (.1cm);
    \draw[red] (0,6.8) node {$Y_{n+1}$};

\draw [fill] (5.2,-3) circle (.1cm);
    \draw[red] (6,-3) node {$Y_i$};
    \draw (8,-3) node {$\epsilon_i=-$};

\draw [fill] (-5.2,3) circle (.1cm);
    \draw[red] (-6,3) node {$Y_j$};
    \draw (-8,3) node {$\epsilon_j=+$};

\draw  (4.7,-2) coordinate (vi) node {$s_i$};
\draw  (3.4,-4) coordinate (vi1) node {$s_{i-1}$};
\draw [thick, bend right,arrow] ($(vi)!.2!(vi1)$) to ($(vi)!.8!(vi1)$);

\draw  (-4.7,2) coordinate (vj) node {$s_{j-1}$};
\draw  (-3.4,4) coordinate (vj1) node {$s_j$};
\draw [thick, bend right,arrow] ($(vj)!.2!(vj1)$) to ($(vj)!.8!(vj1)$);
\end{tikzpicture}
\caption{Stable $(n+1)$-gon of type $A_n$}
\label{fig:exA}
\end{figure}
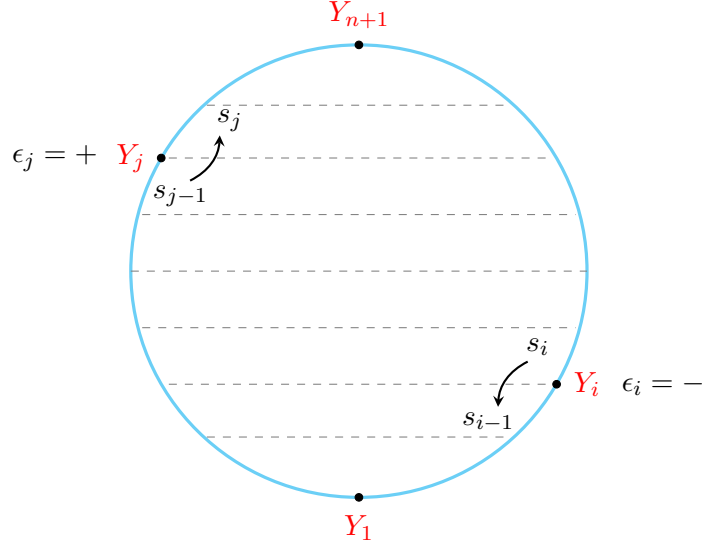

\begin{proof}
Assume that the underlying diagram of $Q$ is as follows:
\begin{center}
\begin{tikzpicture}[scale=.5,xscale=1,ar/.style={-,thick}]
\draw(0,0)node(v1){$1$}(2,0)node(v2){$2$}(4,0)node(v3){$3$}
(6.5,0)node(v4){$\cdots$}(9,0)node(vn){$n.$};
\draw[ar](v2)edge(v1)edge(v3);
\draw[ar](v4)edge(v3)edge(vn);
\end{tikzpicture}
\end{center}
We define a sign function $\epsilon$ on $Q_1=\{a_i\mid 2\le i\le n\}$, where $a_i$ is the arrow between $i-1$ and $i$, as:
\begin{gather}\label{eq:eps}
\epsilon_i\colon= \begin{cases}
						+ & \textrm{if $i-1 \xrightarrow{a_i} i$}, \\
						- & \textrm{if $i-1 \xleftarrow{a_i} i$}.
					\end{cases}
\end{gather}
Set $Y_1=(0,-1)$ and $Y_{n+1}=(1,0)$ in $\CC$.
For any $2\le i\le n$, let $Y_i=(x_i,y_i)$ be the point on $S^1=\{x^2+y^2=1\}$ with
$$y_i=-1+\frac{2(i-1)}{n},$$
and $\sign(x_i)=-\epsilon_i$. Clearly, we have
\[Y_1<Y_2<\cdots<Y_{n+1}.\]
Denote by $\hgon$ the convex $(n+1)$-gon with vertices $\{Y_i\mid 1 \leq i\leq n+1\}$.
By Definition \ref{def:An-quiver}, the intersection quiver of $\hgon$ is isomorphic to $Q$,
cf. Figure~\ref{fig:exA}.
\end{proof}

\begin{proposition}\label{prop:Dn-quiver to polygon}
For any quiver $Q$ of type $D_n$, there is a stable $2(n-1)$-gon $\hgon$ of type $D_n$ such that $\hh{Q}^{1-?}(\hgon)\cong Q$.
\end{proposition}
\begin{proof}
Assume that the underlying diagram of $Q$ is as follows:
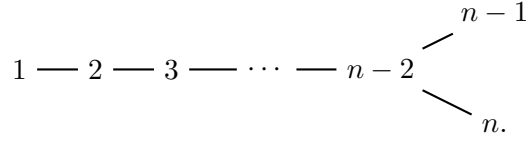
\begin{center}
\begin{tikzpicture}[scale=.5,xscale=1,ar/.style={-,thick}]
\draw(0,0)node(v1){$1$}(2,0)node(v2){$2$}(4,0)node(v3){$3$}
(6.5,0)node(v4){$\cdots$}(9.5,0)node(vn2){$n-2$}
(12.5,1.5)node(vn1){$n-1$}
(12.5,-1.5)node(vn){$n.$};
\draw[ar](v2)edge(v1)edge(v3);
\draw[ar](v4)edge(v3)edge(vn2);
\draw[ar](vn2)edge(vn1)edge(vn);
\end{tikzpicture}
\end{center}
Denote by $Q_1=\{a_i\mid 2\le i\le n\}$, where $a_i$ is the arrow between $i-1$ and $i$ for $2\le i\le n-1$ and $a_n$ is
the arrow between $n-2$ and $n$.
We define a sign function $\epsilon$ on $Q_1$ with $\epsilon_i$ given by \eqref{eq:eps} for $2\le i\le n-1$ and
\begin{gather}
\epsilon_n\colon= \begin{cases}
						+ & \textrm{if $n-2 \xrightarrow{a_n} n$}, \\
						- & \textrm{if $n-2 \xleftarrow{a_n} n$}.
					\end{cases}
\end{gather}

Set $Y_{\pm n}=(0,\pm1)$ in $\CC$.
For any $3\le i\le n-1$, let $Y_i=(x_i,y_i)$ be the point on $S^1=\{x^2+y^2=1\}$ with
$$y_{\pm i}=\frac{\pm i}{n},$$
and $\sign(x_{\pm i})=\pm\epsilon_{n+1-|i|}$.
To choose vertices $Y_{\pm1}$ and $Y_{\pm2}$, there are three cases.

\begin{itemize}
\item[$($a$)$]\label{case1} If $\epsilon_{n-1}=\epsilon_n$,
then we set $Y_{\pm2}=(x_{\pm2},\pm 2/n)$ be the points on $S^1$ with $\sign(x_{\pm2})=\pm \epsilon_{n-1}$
and $Y_{\pm1}=(0,0)$.
Let $\hgon$ be the convex $2(n-1)$-gon with vertices $\{Y_{\pm i}\mid 2 \leq i\leq n\}$ and two punctures $B_\pm=Y_{\pm1}$,
which is a stable $2(n-1)$-gon of type $D_n$ by construction.
Moreover, the ordering \eqref{eq:orderD} holds.
By Definitions~\ref{def:An-quiver} and~\ref{def:Dn-sim},
the intersection quiver of $\hgon$ is isomorphic to $Q$.
The pictures in Figure~\ref{fig:Dn-1} give two examples for the case $n=6$.
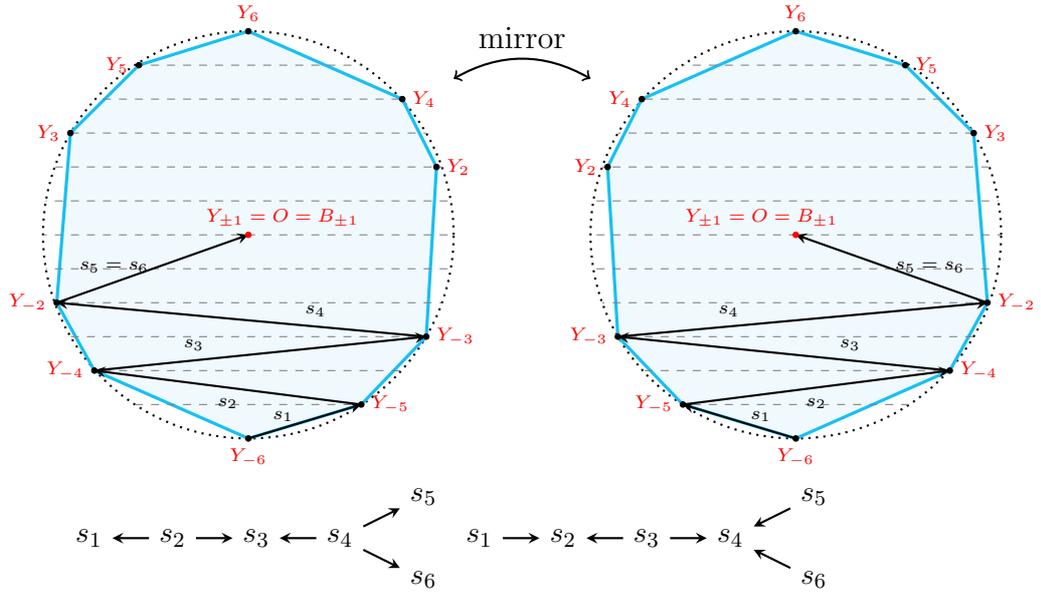
\begin{figure}[ht]\centering
\begin{tikzpicture}[scale=.45, rotate=0, xscale=1,arrow/.style={->,>=stealth,thick},font=\tiny]
    \draw (-18,0) node {};
\draw [thick,dotted] (-8,0) circle (6cm);
\draw [thick,dotted] (8,0) circle (6cm);
\path (-8,-6) coordinate (Y-6) (-4.7,-5) coordinate (Y-5) (-12.5,-4) coordinate (Y-4) (-2.8,-3) coordinate (Y-3)(-13.6,-2) coordinate (Y-2) (-8,6) coordinate (Y6) (-11.2,5) coordinate (Y5) (-3.5,4) coordinate (Y4) (-13.2,3) coordinate (Y3) (-2.5,2) coordinate (Y2) (-8,0) coordinate (O);
\begin{scope}
\draw[cyan!70,very thick,fill=cyan!5] (Y-6) to (Y-5) to (Y-3) to (Y2) to (Y4) to (Y6) to (Y5) to (Y3) to (Y-2) to (Y-4) to (Y-6);
\end{scope}

\begin{scope}
\clip (-8,0) circle (6cm);
\foreach \j in {-5,...,5}{
\draw[thin,gray,dashed](-16,\j)to(-2,\j);}
\end{scope}

\draw [fill] (Y-6) circle (.08cm);
    \draw (Y-6) node[red,below] {$Y_{-6}$};

\draw [fill] (Y-5) circle (.08cm);
    \draw (Y-5) node[red,right] {$Y_{-5}$};

\draw [fill] (Y-4)  circle (.08cm);
    \draw (Y-4) node[red,left] {$Y_{-4}$};

\draw [fill] (Y-3)  circle (.08cm);
    \draw (Y-3) node[red,right] {$Y_{-3}$};

\draw [fill] (Y-2)  circle (.08cm);
    \draw (Y-2) node[red,left] {$Y_{-2}$};

\draw [fill] (Y6)  circle (.08cm);
    \draw[red] (Y6) node[above] {$Y_{6}$};

\draw [fill] (Y5)  circle (.08cm);
    \draw (Y5) node[red,left] {$Y_{5}$};

\draw [fill] (Y4)  circle (.08cm);
    \draw (Y4) node[red,right] {$Y_{4}$};

\draw [fill] (Y3) circle (.08cm);
    \draw (Y3) node[red,left] {$Y_{3}$};

\draw [fill] (Y2)  circle (.08cm);
    \draw (Y2) node[red,right] {$Y_{2}$};

\draw[->,>=stealth,thick,black] (Y-6) to (Y-5);
\draw[->,>=stealth,thick,black] (Y-5) to (Y-4);
\draw[->,>=stealth,thick,black] (Y-4) to (Y-3);
\draw[->,>=stealth,thick,black] (Y-3) to (Y-2);
\draw[->,>=stealth,thick,black] (Y-2) to (O);

\draw[black]
    ($(Y-6)!.4!(Y-5)+(.3,.25)$) node[left] {$s_1$}
    ($(Y-5)!.5!(Y-4)$) node[left,below] {$s_2$}
    ($(Y-4)!.3!(Y-3)$) node[left,above] {$s_3$}
    ($(Y-3)!.3!(Y-2)$) node[right,above] {$s_4$}
    ($(Y-2)!.3!(O)$) node[above] {$s_5=s_6$};


\path (8,-6) coordinate (Y-26) (4.7,-5) coordinate (Y-25) (12.5,-4) coordinate (Y-24) (2.8,-3) coordinate (Y-23)(13.6,-2) coordinate (Y-22) (8,6) coordinate (Y26) (11.2,5) coordinate (Y25) (3.5,4) coordinate (Y24) (13.2,3) coordinate (Y23) (2.5,2) coordinate (Y22) (8,0) coordinate (2O);

\begin{scope}
\draw[cyan!70,very thick,fill=cyan!5] (Y-26) to (Y-25) to (Y-23) to (Y22) to (Y24) to (Y26) to (Y25) to (Y23) to (Y-22) to (Y-24) to (Y-26);
\end{scope}

\begin{scope}
\clip (8,0) circle (6cm);
\foreach \j in {-5,...,5}{
\draw[thin,gray,dashed](16,\j)to(2,\j);}
\end{scope}

\draw [fill] (Y-26) circle (.08cm);
    \draw (Y-26) node[red,below] {$Y_{-6}$};

\draw [fill] (Y-25) circle (.08cm);
    \draw (Y-25) node[red,left] {$Y_{-5}$};

\draw [fill] (Y-24)  circle (.08cm);
    \draw (Y-24) node[red,right] {$Y_{-4}$};

\draw [fill] (Y-23)  circle (.08cm);
    \draw (Y-23) node[red,left] {$Y_{-3}$};

\draw [fill] (Y-22)  circle (.08cm);
    \draw (Y-22) node[red,right] {$Y_{-2}$};

\draw [fill] (Y26)  circle (.08cm);
    \draw[red] (Y26) node[above] {$Y_{6}$};

\draw [fill] (Y25)  circle (.08cm);
    \draw (Y25) node[red,right] {$Y_{5}$};

\draw [fill] (Y24)  circle (.08cm);
    \draw (Y24) node[red,left] {$Y_{4}$};

\draw [fill] (Y23) circle (.08cm);
    \draw (Y23) node[red,right] {$Y_{3}$};

\draw [fill] (Y22)  circle (.08cm);
    \draw (Y22) node[red,left] {$Y_{2}$};

\draw[->,>=stealth,thick,black] (Y-26) to (Y-25);
\draw[->,>=stealth,thick,black] (Y-25) to (Y-24);
\draw[->,>=stealth,thick,black] (Y-24) to (Y-23);
\draw[->,>=stealth,thick,black] (Y-23) to (Y-22);
\draw[->,>=stealth,thick,black] (Y-22) to (2O);

\draw[black]
    ($(Y-26)!.4!(Y-25)+(-.3,.25)$) node[right] {$s_1$}
    ($(Y-25)!.5!(Y-24)$) node[right,below] {$s_2$}
    ($(Y-24)!.3!(Y-23)$) node[right,above] {$s_3$}
    ($(Y-23)!.3!(Y-22)$) node[left,above] {$s_4$}
    ($(Y-22)!.3!(2O)$) node[above] {$s_5=s_6$};


\draw [red,fill=red] (O)  circle (.08cm);
    \draw (-7,.5) node[red] {$Y_{\pm1}=O=B_{\pm1}$};

\draw [red,fill=red] (2O)  circle (.08cm);
    \draw (7,.5) node[red] {$Y_{\pm1}=O=B_{\pm1}$};

\draw [thick,<->,bend left] (-2,4.6) to (2,4.6);
    \draw[font=\large] (0,5.8) node {mirror};
\end{tikzpicture}

\begin{tikzpicture}[scale=.55,xscale=1,ar/.style={->,thick,>=stealth}]
\draw(-2,0)node(v1){$s_1$}(0,0)node(v2){$s_2$}(2,0)node(v3){$s_3$}(4,0)node(v4){$s_4$}
(6,-1)node(v6){$s_6$}(6,1)node(v5){$s_5$};

\draw[ar](v2)to(v1);
\draw[ar](v2)to(v3);
\draw[ar](v4)to(v3);
\draw[ar](v4)to(v5);
\draw[ar](v4)to(v6);

\end{tikzpicture}
\begin{tikzpicture}[scale=.55,xscale=1,ar/.style={->,thick,>=stealth}]
\draw(-2,0)node(v1){$s_1$}(0,0)node(v2){$s_2$}(2,0)node(v3){$s_3$}(4,0)node(v4){$s_4$}
(6,-1)node(v6){$s_6$}(6,1)node(v5){$s_5$};

\draw[ar](v1)to(v2);
\draw[ar](v3)to(v2);
\draw[ar](v3)to(v4);
\draw[ar](v5)to(v4);
\draw[ar](v6)to(v4);
\end{tikzpicture}
\caption{Two mirror stable $10$-gons of type $D_6$ in Case~$($a$)$}
\label{fig:Dn-1}
\end{figure}

\item[$($b$)$]\label{case2} If $\epsilon_{n-1}=-\epsilon_n$ and $(\epsilon_{n-1},\epsilon_n)=(-,+)$, then we
set $Y_{\pm1}=(\pm1,0)$. Let $\hgon$ be the convex $2(n-1)$-gon with vertices $\{Y_{\pm i}\mid 1 \leq i\leq n, i\ne2\}$.
As $\hgon$ is inscribed in $S^1$, the level-$(n-2)$ diagonal-gon is the intersection of $(n-1)$ symmetric rectangles containing the origin $O$.
Hence, such level-$(n-2)$ diagonal-gon contains a disk with center at $O$ whose diameter $d$ is the shortest edge of those rectangles.
Take $B_\pm=Y_{\pm2}=(0,\pm d/3)$ inside such a disk.
Then the convex $2(n-1)$-gon $\hgon$ with punctures $B_\pm$ is a stable $2(n-1)$-gon of type $D_n$.
Moreover, the ordering \eqref{eq:orderD} holds.
By Definitions~\ref{def:An-quiver} and~\ref{def:Dn-sim},
the intersection quiver of $\hgon$ is isomorphic to $Q$.
The left picture in Figure~\ref{fig:Dn-2} gives an example for the case $n=6$.

\item[$($b$')$]\label{case3}
If If $\epsilon_{n-1}=-\epsilon_n$ and $(\epsilon_{n-1},\epsilon_n)=(+,-)$, then we
take the stable $2(n-1)$-gon $\hgon'$ of type $D_n$ obtained from $\hgon$ in Case~$($b$)$ by clockwise rotating
$R_{-\theta}=e^{-\mathbf{i} \pi \theta}$ of a small angle $0<\theta\ll1$.
Such a rotation does not change the order between almost all $Y_i$'s, except the one between $Y_{\pm1}$.
Set $Y'_i=R_{-\theta}(Y_i)$ for $i\neq\pm1$ and $Y'_{\pm1}=R_{-\theta}(Y_{\mp1})$.
Then the ordering \eqref{eq:orderD} holds for all $Y'_i$'s.
By Definitions~\ref{def:An-quiver} and~\ref{def:Dn-sim},
the intersection quiver of $\hgon'$ is isomorphic to $Q$.
The right picture in Figure~\ref{fig:Dn-2} gives an example for the case $n=6$.\qedhere
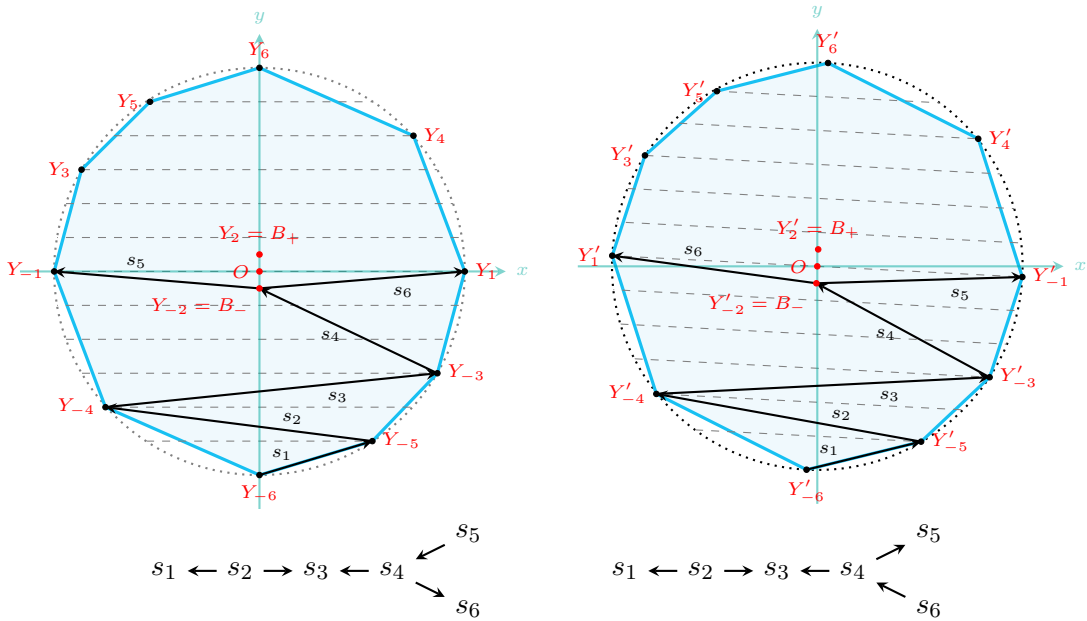
\begin{figure}[ht]\centering
\begin{tikzpicture}[scale=.45, rotate=0, xscale=1,arrow/.style={->,>=stealth,thick},font=\tiny]
\draw [thick,dotted,gray] (0,0) circle (6cm);
\path (0,-6) coordinate (Y-6) (3.3,-5) coordinate (Y-5) (-4.5,-4) coordinate (Y-4) (5.2,-3) coordinate (Y-3) (0,-.5) coordinate (Y-2)(-6,0) coordinate (Y-1)(6,0) coordinate (Y1)(0,6) coordinate (Y6)(-3.2,5)  coordinate (Y5)(4.5,4) coordinate (Y4)(-5.2,3) coordinate (Y3) (0,.5) coordinate (Y2);
\begin{scope}
\clip[draw] (Y-6) to (Y-5) to (Y-3) to (Y1) to (Y4) to (Y6) to (Y5) to (Y3) to (Y-1) to (Y-4) to (Y-6);
\draw[fill=cyan!10](0,0) circle (10cm);
\end{scope}
\begin{scope}
\draw[cyan!70,very thick,fill=cyan!5] (Y-6) to (Y-5) to (Y-3) to (Y1) to (Y4) to (Y6) to (Y5) to (Y3) to (Y-1) to (Y-4) to (Y-6);
\end{scope}

\begin{scope}
\clip (0,0) circle (6cm);
\foreach \j in {-5,...,5}{
\draw[thin,gray,dashed](-8,\j)to(8,\j);}
\end{scope}

\draw[thick,arrow,Emerald,opacity=.5](-7,0)to(7.2,0)node[right]{$x$};
\draw[thick,arrow,Emerald,opacity=.5](0,-7)to(0,7)node[above]{$y$};

\draw [fill] (Y-6) circle (.08cm) node[red,below] {$Y_{-6}$};
\draw [fill] (Y-5) circle (.08cm)node[red,right] {$Y_{-5}$};
\draw [fill] (Y-4) circle (.08cm)node[red,left] {$Y_{-4}$};
\draw [fill] (Y-3) circle (.08cm)node[red,right] {$Y_{-3}$};
\draw [fill] (Y1) circle (.08cm)node[red,right] {$Y_{1}$};
\draw [fill] (Y6) circle (.08cm)node[red,above] {$Y_{6}$};
\draw [fill] (Y5) circle (.08cm)node[red,left] {$Y_{5}$};
\draw [fill] (Y4) circle (.08cm)node[red,right] {$Y_{4}$};
\draw [fill] (Y3) circle (.08cm)node[red,left] {$Y_{3}$};
\draw [fill] (Y-1) circle (.08cm)node[red,left] {$Y_{-1}$};

\draw[->,>=stealth,thick,black] (Y-6) to (Y-5);
\draw[->,>=stealth,thick,black] (Y-5) to (Y-4);
\draw[->,>=stealth,thick,black] (Y-4) to (Y-3);
\draw[->,>=stealth,thick,black] (Y-3) to (Y-2);
\draw[->,>=stealth,thick,black] (Y-2) to (Y-1);
\draw[->,>=stealth,thick,black] (Y-2) to (Y1);

\draw[black]
    ($(Y-6)!.2!(Y-5)+(0,-.05)$) node[above] {$s_1$}
    ($(Y-5)!.3!(Y-4)-(0,.1)$) node[above] {$s_2$}
    ($(Y-4)!.7!(Y-3)+(0,.05)$) node[below] {$s_3$}
    ($(Y-3)!.6!(Y-2)+(0,.05)$) node[below] {$s_4$}
    ($(Y-2)!.6!(Y-1)$) node[above] {$s_5$}
    ($(Y-2)!.7!(Y1)$) node[below] {$s_6$};

\draw [red,fill=red] (0,0) circle (.08cm)node[left] {$O$};
\draw [red,fill=red] (0,.5) circle (.08cm)node[red,above] {$Y_{2}=B_+$};
\draw [red,fill=red] (0,-.5) circle (.08cm) node[red,below left] {$Y_{-2}=B_-$};
\end{tikzpicture}\quad
\begin{tikzpicture}[scale=.45, rotate=-3, xscale=1,arrow/.style={->,>=stealth,thick},font=\tiny]
\draw [thick,dotted] (0,0) circle (6cm);

\path (0,-6) coordinate (Y-6) (3.3,-5) coordinate (Y-5) (-4.5,-4) coordinate (Y-4) (5.2,-3) coordinate (Y-3) (0,-.5) coordinate (Y-2)(-6,0) coordinate (Y-1)(6,0) coordinate (Y1)(0,6) coordinate (Y6)(-3.2,5)  coordinate (Y5)(4.5,4) coordinate (Y4)(-5.2,3) coordinate (Y3) (0,.5) coordinate (Y2);

\begin{scope}
\clip[draw] (Y-6) to (Y-5) to (Y-3) to (Y1) to (Y4) to (Y6) to (Y5) to (Y3) to (Y-1) to (Y-4) to (Y-6);
\draw[fill=cyan!10](0,0) circle (10cm);
\end{scope}

\begin{scope}
\draw[cyan!70,very thick,fill=cyan!5] (Y-6) to (Y-5) to (Y-3) to (Y1) to (Y4) to (Y6) to (Y5) to (Y3) to (Y-1) to (Y-4) to (Y-6);
\end{scope}

\begin{scope}
\clip (0,0) circle (6cm);
\foreach \j in {-5,...,5}{
\draw[thin,gray,dashed](-8,\j)to(8,\j);}
\end{scope}

\draw[rotate=3,thick,arrow,Emerald,opacity=.5](-7,0)to(7.2,0)node[right]{$x$};
\draw[rotate=3,thick,arrow,Emerald,opacity=.5](0,-7)to(0,7)node[above]{$y$};

\draw [fill] (Y-6) circle (.08cm) node[red,below] {$Y'_{-6}$};
\draw [fill] (Y-5) circle (.08cm)node[red,right] {$Y'_{-5}$};
\draw [fill] (Y-4) circle (.08cm)node[red,left] {$Y'_{-4}$};
\draw [fill] (Y-3) circle (.08cm)node[red,right] {$Y'_{-3}$};
\draw [fill] (Y1) circle (.08cm)node[red,right] {$Y'_{-1}$};
\draw [fill] (Y6) circle (.08cm)node[red,above] {$Y'_{6}$};
\draw [fill] (Y5) circle (.08cm)node[red,left] {$Y'_{5}$};
\draw [fill] (Y4) circle (.08cm)node[red,right] {$Y'_{4}$};
\draw [fill] (Y3) circle (.08cm)node[red,left] {$Y'_{3}$};
\draw [fill] (Y-1) circle (.08cm)node[red,left] {$Y'_{1}$};

\draw[->,>=stealth,thick,black] (Y-6) to (Y-5);
\draw[->,>=stealth,thick,black] (Y-5) to (Y-4);
\draw[->,>=stealth,thick,black] (Y-4) to (Y-3);
\draw[->,>=stealth,thick,black] (Y-3) to (Y-2);
\draw[->,>=stealth,thick,black] (Y-2) to (Y-1);
\draw[->,>=stealth,thick,black] (Y-2) to (Y1);

\draw[black]
    ($(Y-6)!.2!(Y-5)+(0,-.05)$) node[above] {$s_1$}
    ($(Y-5)!.3!(Y-4)-(0,.1)$) node[above] {$s_2$}
    ($(Y-4)!.7!(Y-3)+(0,.05)$) node[below] {$s_3$}
    ($(Y-3)!.6!(Y-2)+(0,.05)$) node[below] {$s_4$}
    ($(Y-2)!.6!(Y-1)$) node[above] {$s_6$}
    ($(Y-2)!.7!(Y1)$) node[below] {$s_5$};

\draw [red,fill=red] (0,0) circle (.08cm)node[left] {$O$};
\draw [red,fill=red] (0,.5) circle (.08cm)node[red,above] {$Y'_{2}=B_+$};
\draw [red,fill=red] (0,-.5) circle (.08cm) node[red,below left] {$Y'_{-2}=B_-$};
\end{tikzpicture}
\begin{tikzpicture}[scale=.5,xscale=1,ar/.style={->,thick,>=stealth}]
\draw(-2,0)node(v1){$s_1$}(0,0)node(v2){$s_2$}(2,0)node(v3){$s_3$}(4,0)node(v4){$s_4$}
(6,-1)node(v6){$s_6$}(6,1)node(v5){$s_5$};

\draw[ar](v2)to(v1);
\draw[ar](v2)to(v3);
\draw[ar](v4)to(v3);
\draw[ar](v5)to(v4);
\draw[ar](v4)to(v6);

\draw(9,0)node{};
\end{tikzpicture}
\begin{tikzpicture}[scale=.5,xscale=1,ar/.style={->,thick,>=stealth}]
\draw(0,0)node{};
\draw(-2,0)node(v1){$s_1$}(0,0)node(v2){$s_2$}(2,0)node(v3){$s_3$}(4,0)node(v4){$s_4$}
(6,-1)node(v6){$s_6$}(6,1)node(v5){$s_5$};

\draw[ar](v2)to(v1);
\draw[ar](v2)to(v3);
\draw[ar](v4)to(v3);
\draw[ar](v4)to(v5);
\draw[ar](v6)to(v4);
\end{tikzpicture}
\caption{Two stable $10$-gons of type $D_6$ in Cases~$($b$)$ and $($b$')$}
\label{fig:Dn-2}
\end{figure}
\end{itemize}
\end{proof}
Note that any quiver in Case~$($b$')$ is isomorphic to some quiver in Case~$($b$)$ by swapping the vertices $n-1$ and $n$.
Thus, one can omit the proof in Case~$($b$')$.
However we keep it to demonstrate the $\CC$-action effect.

\subsection{Exceptional type}
Firstly, note that if we prove the conjecture for $E_8$,
then $E_6$ and $E_7$ follows directly by restricted to full subcategories.
Secondly, we have the following easy observation, cf. Lemma~\ref{lem:ob} and an illustration in Figure~\ref{fig:Dn-1}.

Consider the mirror operation $\hgon \mapsto \hgon^\vee$ with respect to $y$-axis,
which also mirror on the order \eqref{eq:order} to be
\begin{gather}\label{eq:order2}
W_1 < W_2 \Longleftrightarrow
\begin{cases}
y_1<y_2 \text{\quad or}\\
y_1=y_2,\, x_1>x_2,
\end{cases}
\end{gather}
for any two points $W_1=(x_1,y_1)$ and $W_2=(x_2,y_2)$ in $\CC$.
Let $\sigma^\vee$ be the total stability condition to $\hgon^\vee$.
The order \eqref{eq:order2} effects the definition of upward diagonals, in the sense that
we need to take $\arg$ to be in $(0,\pi]$.
Consequently, we need choose $\h^\vee\colon=\sli^\vee(0,1]$ as the heart of $\sigma^\vee$
so that everything in the paper still hold, e.g. Proposition~\ref{prop:obj-in-heart}.

\begin{lemma}\label{lem:ob}
The intersection quiver $\hh{Q}(\hgon^\vee)$ is the opposite quiver of $\hh{Q}(\hgon)$ and
it is the Ext-quiver of the heart $\h^\vee$.
\end{lemma}
Note that by this lemma,
one can also easily reduce Case~$($b$')$ to Case~$($b$)$ using the mirror operation.

\begin{remark}
As there are only finitely many indecomposable objects in $\Dwq(Q)/[2]$,
then $\sli(\theta)=\emptyset$ for $0<\theta\ll1$.
This implies that
\[
    \sli(0,1] = \sli[\theta,1+\theta),
\]
which equals the canonical heart of $\theta\cdot\sigma$ (the $\CC$-action).
Thus, in the setting of Lemma~\ref{lem:ob},
if one combines the mirror operation with a rotation smaller enough,
the effect on the Ext-quiver of the canonical heart would be precisely taking opposite quiver.
\end{remark}

Next, we have another important observation.

\textbf{Observation:} Suppose that
we have constructed a stable $h$-gon $\hgon$ of type $E_n$, where $n\in\{6,7,8\}$, with the intersection quiver
\begin{gather}\label{eq:En2}
\begin{tikzpicture}[scale=.5,xscale=1,ar/.style={->,thick,>=stealth}]
\draw(0,0)node(v1){$s_1$}(2,0)node(v2){$s_2$}(4,0)node(v4){$s_4$}(6,0)node(v5){$s_5$}(8.25,0)node(v6){$\cdots$}(10.5,0)node(vn){$s_n$}
    (4,2)node(v3){$s_3$};
\draw[-,thick](v2)edge(v1)edge(v4);
\draw[-,thick](v5)edge(v6)edge(v4);
\draw[ar](v6)to(vn);
\draw[ar](v2)to(v1);
\draw[ar](v4)to(v3);
\end{tikzpicture}
\end{gather}
such that
\begin{gather}\label{eq:arg=0}
    \arg Z_{\hgon}(S_1)=\arg Z_{\hgon}(S_3)=\arg Z_{\hgon}(S_n)=0,
\end{gather}
where $S_j=\XX(s_j)$ for $1\le j\le 8$.
Then, by a small deformation, $\arg Z_{\hgon}(S_i)$ can become near $\pi$
for any leaf vertex $i\in\{1,3,n\}$ and thus reverse the arrow at $S_i$.
Note that since the space $\Sth(Q)$ of stable $h$-gons of type $Q$ is open, such small deformation is permitted.

\begin{proposition}\label{prop:E-quiver to polygon}
For any quiver $Q$ of type $E_n$, where $n\in\{6,7,8\}$, there is a stable $h$-gon $\hgon$ of type $E_n$ such that $\hh{Q}^{1-?}(\hgon)\cong Q$.
\end{proposition}
\begin{proof}
By Lemma~\ref{lem:ob}, we can fix the orientation of a chosen arrow (say between vertex 2 and 4).
Further, by the observation above, we only need to
construct $2^3=8$ stable $30$-gons $\hgon$ with the following intersection quivers:
\begin{gather}\label{eq:En3}
\begin{tikzpicture}[scale=.5,xscale=1,ar/.style={->,thick,>=stealth}]
\draw(0,0)node(v1){$s_1$}(2,0)node(v2){$s_2$}(4,0)node(v4){$s_4$}(6,0)node(v5){$s_5$}(8,0)node(v6){$s_6$}
(10,0)node(v7){$s_7$}(12,0)node(v8){$s_8$}
    (4,2)node(v3){$s_3$};
\draw[ar](v2)edge(v1);
\draw[ar](v4)edge(v2);
\draw[ar](v4)edge(v3);
\draw[ar](v7)edge(v8);
\draw[-,thick](v4)to(v5)(v5)to(v6)(v6)to(v7);
\end{tikzpicture}
\end{gather}
satisfying \eqref{eq:arg=0}.
All possibilities of those $8$ stable $30$-gons are shown in Appendix~\ref{app:E}.
\end{proof}
\subsection{Summary}
Combining Propositions~\ref{prop:An-quiver to polygon}~\ref{prop:Dn-quiver to polygon} and~\ref{prop:E-quiver to polygon},
we verify Reineke's conjecture:
\begin{theorem}\label{thm:R's}
For any Dynkin quiver $Q$,
there is a stability function on $\h(Q)$ such that any indecomposable is stable.
\end{theorem}

%

%
%

Final remark is that all the examples in the exceptional cases are produce by
\href{https://www.geogebra.org/}{GeoGebra}
where the precise coordinate information is in the tex/tikz code.

\appendix
\section{List of examples of stable $30$-gons in type $E_8$}\label{app:E}
\begin{example}
Figures~\ref{fig:E8-1} and~\ref{fig:E8-2} give the eight examples of stable 30-gon $\hgon$,
in which we only draw the partial part containing $\Sim\hgon$ such that the associated intersection quiver $\hh{Q}(\hgon)$ can be read.
More precisely, the intersection quiver is of the form
\begin{gather}\label{eq:En4}
\begin{tikzpicture}[scale=.5,xscale=1,ar/.style={->,thick,>=stealth}]
\draw(0,0)node(v1){$s_1$}(2,0)node(v2){$s_2$}(4,0)node(v4){$s_4$}(6,0)node(v5){$s_5$}(8,0)node(v6){$s_6$}
(10,0)node(v7){$s_7$}(12,0)node(v8){$s_8$}
    (4,2)node(v3){$s_3$};
\draw[ar](v7)edge(v8);
\draw[ar](v2)edge(v1);
\draw[ar](v4)edge(v3);
\draw[ar](v4)edge(v2);
\draw[black,font=\tiny]
    ($(v4)!.5!(v5)$) node[above] {$a_5$}
    ($(v5)!.5!(v6)$) node[above] {$a_6$}
    ($(v6)!.5!(v7)$) node[above] {$a_7$};
\draw[black,thick]
    (v4)to(v5)
    (v5)to(v6)
    (v6)to(v7);
\end{tikzpicture}
\end{gather}
with $\epsilon_i$ denoting the sign of the orientation of $a_i$ for $5\le i\le 7$ as in \eqref{eq:eps}.
And we have the sign labeling $\epsilon_5\epsilon_6\epsilon_7$ with the following eight possibilities:
\begin{itemize}
  \item $---,+--,-+-,--+$ for the pictures in Figure~\ref{fig:E8-1}, respectively, and
  \item $++-,+-+,-++,+++$ for the pictures in Figure~\ref{fig:E8-2}, respectively.
\end{itemize}

\begin{figure}[ht]\centering
\makebox[\textwidth][c]{
\definecolor{rvwvcq}{rgb}{0.08235294117647059,0.396078431372549,0.7529411764705882}
\definecolor{ffxfqq}{rgb}{1.,0.4980392156862745,0.}
\definecolor{qqffff}{rgb}{0.,1.,1.}
\definecolor{qqffqq}{rgb}{0.,1.,0.}
\definecolor{ffqqqq}{rgb}{1.,0.,0.}
\definecolor{ttffcc}{rgb}{0.2,1.,0.8}
\definecolor{uuuuuu}{rgb}{0.26666666666666666,0.26666666666666666,0.26666666666666666}
\definecolor{zzttqq}{rgb}{0.6,0.2,0.}
\definecolor{aqaqaq}{rgb}{0.6274509803921569,0.6274509803921569,0.6274509803921569}
\definecolor{xfqqff}{rgb}{0.4980392156862745,0.,1.}
\definecolor{ududff}{rgb}{0.30196078431372547,0.30196078431372547,1.}
\definecolor{ffwwqq}{rgb}{1.,0.4,0.}
\definecolor{ttffqq}{rgb}{0.2,1.,0.}
\definecolor{qqzzff}{rgb}{0.,0.6,1.}
\definecolor{qqqqff}{rgb}{0.08235294117647059,0.396078431372549,0.7529411764705882}

\centering
\begin{tikzpicture}[line cap=round,line join=round,arrow/.style={->,>=stealth,thick},scale=.8,rotate=3]
\clip(-3.1,-0.4718120250644906) rectangle (11.783935758155993,15.839620174270646);
\fill[line width=1.2pt,dotted,color=ffxfqq,fill=ffxfqq!5] (3.5061961158567763,1.4834932259945088) -- (5.475976115856779,2.048223225994509) -- (6.911022593073872,3.692956985899212) -- (7.891055574020006,5.672889935910688) -- (8.414584166485634,7.918859776195598) -- (7.758085503194081,10.395289861398567) -- (5.999346839902525,12.306989946601536) -- (3.960324751165938,13.120949296423401) -- (1.913698273948845,13.120945536518697) -- (-0.08895031851678414,12.519705696233787) -- (-1.6098303185167868,10.874975696233788) -- (-1.8208716552252326,8.39854561103082) -- (-1.2168960437057388,5.939852501304252) -- (-0.1503025474347801,3.9599233111974783) -- (1.3288174525652208,2.3151933111974787) -- cycle;
\fill[line width=1.2pt,color=qqffff,fill=qqffff!5] (0.9603247511659356,13.120949296423401) -- (2.930104751165938,13.6856792964234) -- (5.107483414457493,12.85397921122043) -- (6.586603414457494,11.209249211220431) -- (7.653196910728453,9.229320021113658) -- (8.257172522247947,6.770626911387091) -- (8.046131185539501,4.2941968261841215) -- (6.525251185539498,2.649466826184122) -- (4.522602593073869,2.048226985899212) -- (2.4759761158567763,2.048223225994509) -- (0.4369540271201888,2.862182575816373) -- (-1.3217846361713672,4.773882661019343) -- (-1.9782832994629196,7.250312746222312) -- (-1.4547547069972921,9.49628258650722) -- (-0.4747217260511576,11.476215536518698) -- cycle;

\fill[line width=1.2pt,color=qqffqq,fill=qqffqq,fill opacity=0.5] (0.5061961158567739,1.4834932259945088) -- (2.4759761158567763,2.048223225994509) -- (4.522602593073869,2.048226985899212) -- cycle;
\fill[line width=1.2pt,color=qqffqq,fill=qqffqq,fill opacity=0.5] (3.5061961158567763,1.4834932259945088) -- (5.475976115856779,2.048223225994509) -- (7.522602593073872,2.048226985899212) -- cycle;

\fill[fill=qqffff,fill opacity=0.4] (1.5035475233911475,0.8822533857095992) -- (3.0875561158567764,0.4034932259945089) -- (3.9550961158567763,0.4034932259945089) -- (5.545218204593364,0.6695338761726454) -- (6.699981256365425,1.2165269006962423) -- cycle;

\draw [line width=1.2pt,color=qqzzff] (3.0875561158567764,0.4034932259945089)-- (2.4759761158567763,2.048223225994509);
\draw [line width=1.2pt,color=qqzzff] (2.4759761158567763,2.048223225994509)-- (1.5035475233911475,0.8822533857095992);
\draw [line width=1.2pt,color=ttffqq] (3.9550961158567763,0.4034932259945089)-- (3.5061961158567763,1.4834932259945088);
\draw [line width=1.2pt,color=ttffqq] (3.5061961158567763,1.4834932259945088)-- (3.0875561158567764,0.4034932259945089);
\draw [line width=1.2pt,color=ffwwqq] (5.545218204593364,0.6695338761726454)-- (4.522602593073869,2.048226985899212);
\draw [line width=1.2pt,color=ffwwqq] (4.522602593073869,2.048226985899212)-- (3.9550961158567763,0.4034932259945089);
\draw [line width=1.2pt,color=xfqqff] (6.699981256365425,1.2165269006962423)-- (7.522602593073872,2.048226985899212);
\draw [line width=1.2pt,color=xfqqff] (7.522602593073872,2.048226985899212)-- (6.525251185539498,2.649466826184122);
\draw [line width=1.2pt,color=xfqqff] (6.525251185539498,2.649466826184122)-- (6.699981256365425,1.2165269006962423);
\draw [line width=1.2pt,color=xfqqff] (1.3288174525652208,2.3151933111974787)-- (1.5035475233911475,0.8822533857095992);
\draw [line width=1.2pt,color=xfqqff] (1.5035475233911475,0.8822533857095992)-- (0.5061961158567739,1.4834932259945088);
\draw [line width=1.2pt,color=xfqqff] (0.5061961158567739,1.4834932259945088)-- (1.3288174525652208,2.3151933111974787);
\draw [line width=1.2pt,color=ttffcc] (6.699981256365425,1.2165269006962423)-- (5.475976115856779,2.048223225994509);
\draw [line width=1.2pt,color=ttffcc] (5.475976115856779,2.048223225994509)-- (5.545218204593364,0.6695338761726454);
\draw [line width=1.2pt,color=qqzzff] (6.911022593073872,3.692956985899212)-- (7.522602593073872,2.048226985899212);
\draw [line width=1.2pt,color=qqzzff] (7.522602593073872,2.048226985899212)-- (8.495031185539501,3.214196826184122);
\draw [line width=1.2pt,color=qqzzff] (8.495031185539501,3.214196826184122)-- (6.911022593073872,3.692956985899212);
\draw [line width=1.2pt,color=ttffqq] (8.495031185539501,3.214196826184122)-- (8.046131185539501,4.2941968261841215);
\draw [line width=1.2pt,color=ttffqq] (8.046131185539501,4.2941968261841215)-- (8.913671185539501,4.2941968261841215);
\draw [line width=1.2pt,color=ttffqq] (8.913671185539501,4.2941968261841215)-- (8.495031185539501,3.214196826184122);
\draw [line width=1.2pt,color=ffwwqq] (7.891055574020006,5.672889935910688)-- (9.481177662756593,5.938930586088825);
\draw [line width=1.2pt,color=ffwwqq] (9.481177662756593,5.938930586088825)-- (8.913671185539501,4.2941968261841215);
\draw [line width=1.2pt,color=ffwwqq] (8.913671185539501,4.2941968261841215)-- (7.891055574020006,5.672889935910688);
\draw [line width=1.2pt,color=ttffcc] (8.257172522247947,6.770626911387091)-- (9.411935574020008,7.317619935910688);
\draw [line width=1.2pt,color=ttffcc] (9.411935574020008,7.317619935910688)-- (9.481177662756593,5.938930586088825);
\draw [line width=1.2pt,color=ttffcc] (9.481177662756593,5.938930586088825)-- (8.257172522247947,6.770626911387091);
\draw [line width=1.2pt,color=xfqqff] (9.237205503194081,8.750559861398568)-- (9.411935574020008,7.317619935910688);
\draw [line width=1.2pt,color=xfqqff] (9.411935574020008,7.317619935910688)-- (8.414584166485634,7.918859776195598);
\draw [line width=1.2pt,color=xfqqff] (8.414584166485634,7.918859776195598)-- (9.237205503194081,8.750559861398568);
\draw [line width=1.2pt,color=qqzzff] (7.653196910728453,9.229320021113658)-- (9.237205503194081,8.750559861398568);
\draw [line width=1.2pt,color=qqzzff] (9.237205503194081,8.750559861398568)-- (8.625625503194081,10.395289861398567);
\draw [line width=1.2pt,color=qqzzff] (8.625625503194081,10.395289861398567)-- (7.653196910728453,9.229320021113658);
\draw [line width=1.2pt,color=ttffqq] (7.758085503194081,10.395289861398567)-- (8.625625503194081,10.395289861398567);
\draw [line width=1.2pt,color=ttffqq] (8.625625503194081,10.395289861398567)-- (8.176725503194081,11.475289861398567);
\draw [line width=1.2pt,color=ttffqq] (8.176725503194081,11.475289861398567)-- (7.758085503194081,10.395289861398567);
\draw [line width=1.2pt,color=ffwwqq] (6.586603414457494,11.209249211220431)-- (8.176725503194081,11.475289861398567);
\draw [line width=1.2pt,color=ffwwqq] (8.176725503194081,11.475289861398567)-- (7.154109891674587,12.853982971125134);
\draw [line width=1.2pt,color=ffwwqq] (7.154109891674587,12.853982971125134)-- (6.586603414457494,11.209249211220431);
\draw [line width=1.2pt,color=ttffcc] (5.999346839902525,12.306989946601536)-- (7.154109891674587,12.853982971125134);
\draw [line width=1.2pt,color=ttffcc] (7.154109891674587,12.853982971125134)-- (5.93010475116594,13.6856792964234);
\draw [line width=1.2pt,color=ttffcc] (5.93010475116594,13.6856792964234)-- (5.999346839902525,12.306989946601536);
\draw [line width=1.2pt,color=qqzzff] (4.932753343631567,14.28691913670831)-- (3.348744751165938,14.7656792964234);
\draw [line width=1.2pt,color=qqzzff] (3.348744751165938,14.7656792964234)-- (3.960324751165938,13.120949296423401);
\draw [line width=1.2pt,color=qqzzff] (3.960324751165938,13.120949296423401)-- (4.932753343631567,14.28691913670831);
\draw [line width=1.2pt,color=ttffqq] (3.348744751165938,14.7656792964234)-- (2.481204751165938,14.7656792964234);
\draw [line width=1.2pt,color=ttffqq] (2.481204751165938,14.7656792964234)-- (2.930104751165938,13.6856792964234);
\draw [line width=1.2pt,color=ttffqq] (2.930104751165938,13.6856792964234)-- (3.348744751165938,14.7656792964234);
\draw [line width=1.2pt,color=ffwwqq] (2.481204751165938,14.7656792964234)-- (0.8910826624293504,14.499638646245264);
\draw [line width=1.2pt,color=ffwwqq] (0.8910826624293504,14.499638646245264)-- (1.913698273948845,13.120945536518697);
\draw [line width=1.2pt,color=ffwwqq] (1.913698273948845,13.120945536518697)-- (2.481204751165938,14.7656792964234);
\draw [line width=1.2pt,color=xfqqff] (-0.2636803893427109,13.952645621721667)-- (-1.0863017260511576,13.120945536518697);
\draw [line width=1.2pt,color=xfqqff] (-1.0863017260511576,13.120945536518697)-- (-0.08895031851678414,12.519705696233787);
\draw [line width=1.2pt,color=xfqqff] (-0.08895031851678414,12.519705696233787)-- (-0.2636803893427109,13.952645621721667);
\draw [line width=1.2pt,color=xfqqff] (5.107483414457493,12.85397921122043)-- (4.932753343631567,14.28691913670831);
\draw [line width=1.2pt,color=xfqqff] (4.932753343631567,14.28691913670831)-- (5.93010475116594,13.6856792964234);
\draw [line width=1.2pt,color=xfqqff] (5.93010475116594,13.6856792964234)-- (5.107483414457493,12.85397921122043);
\draw [line width=1.2pt,color=ttffcc] (0.8910826624293504,14.499638646245264)-- (-0.2636803893427109,13.952645621721667);
\draw [line width=1.2pt,color=ttffcc] (-0.2636803893427109,13.952645621721667)-- (0.9603247511659356,13.120949296423401);
\draw [line width=1.2pt,color=ttffcc] (0.9603247511659356,13.120949296423401)-- (0.8910826624293504,14.499638646245264);
\draw [line width=1.2pt,color=qqzzff] (-0.4747217260511576,11.476215536518698)-- (-1.0863017260511576,13.120945536518697);
\draw [line width=1.2pt,color=qqzzff] (-1.0863017260511576,13.120945536518697)-- (-2.058730318516787,11.954975696233788);
\draw [line width=1.2pt,color=qqzzff] (-2.058730318516787,11.954975696233788)-- (-0.4747217260511576,11.476215536518698);
\draw [line width=1.2pt,color=ttffqq] (-2.058730318516787,11.954975696233788)-- (-1.6098303185167868,10.874975696233788);
\draw [line width=1.2pt,color=ttffqq] (-1.6098303185167868,10.874975696233788)-- (-2.477370318516787,10.874975696233788);
\draw [line width=1.2pt,color=ttffqq] (-2.477370318516787,10.874975696233788)-- (-2.058730318516787,11.954975696233788);
\draw [line width=1.2pt,color=ffwwqq] (-1.4547547069972921,9.49628258650722)-- (-3.044876795733879,9.230241936329085);
\draw [line width=1.2pt,color=ffwwqq] (-3.044876795733879,9.230241936329085)-- (-2.477370318516787,10.874975696233788);
\draw [line width=1.2pt,color=ffwwqq] (-2.477370318516787,10.874975696233788)-- (-1.4547547069972921,9.49628258650722);
\draw [line width=1.2pt,color=ttffcc] (-1.8208716552252326,8.39854561103082)-- (-2.975634706997294,7.851552586507221);
\draw [line width=1.2pt,color=ttffcc] (-2.975634706997294,7.851552586507221)-- (-3.044876795733879,9.230241936329085);
\draw [line width=1.2pt,color=ttffcc] (-3.044876795733879,9.230241936329085)-- (-1.8208716552252326,8.39854561103082);
\draw [line width=1.2pt,color=xfqqff] (-2.800904636171367,6.418612661019342)-- (-2.975634706997294,7.851552586507221);
\draw [line width=1.2pt,color=xfqqff] (-2.975634706997294,7.851552586507221)-- (-1.9782832994629196,7.250312746222312);
\draw [line width=1.2pt,color=xfqqff] (-1.9782832994629196,7.250312746222312)-- (-2.800904636171367,6.418612661019342);
\draw [line width=1.2pt,color=qqzzff] (-1.2168960437057388,5.939852501304252)-- (-2.800904636171367,6.418612661019342);
\draw [line width=1.2pt,color=qqzzff] (-2.800904636171367,6.418612661019342)-- (-2.189324636171367,4.773882661019343);
\draw [line width=1.2pt,color=qqzzff] (-2.189324636171367,4.773882661019343)-- (-1.2168960437057388,5.939852501304252);
\draw [line width=1.2pt,color=ttffqq] (-1.3217846361713672,4.773882661019343)-- (-2.189324636171367,4.773882661019343);
\draw [line width=1.2pt,color=ttffqq] (-2.189324636171367,4.773882661019343)-- (-1.7404246361713671,3.6938826610193427);
\draw [line width=1.2pt,color=ttffqq] (-1.7404246361713671,3.6938826610193427)-- (-1.3217846361713672,4.773882661019343);
\draw [line width=1.2pt,color=ffwwqq] (-0.1503025474347801,3.9599233111974783)-- (-1.7404246361713671,3.6938826610193427);
\draw [line width=1.2pt,color=ffwwqq] (-1.7404246361713671,3.6938826610193427)-- (-0.7178090246518725,2.315189551292775);
\draw [line width=1.2pt,color=ffwwqq] (-0.7178090246518725,2.315189551292775)-- (-0.1503025474347801,3.9599233111974783);
\draw [line width=1.2pt,color=ttffcc] (0.4369540271201888,2.862182575816373)-- (-0.7178090246518725,2.315189551292775);
\draw [line width=1.2pt,color=ttffcc] (-0.7178090246518725,2.315189551292775)-- (0.506196115856774,1.4834932259945095);
\draw [line width=1.2pt,color=ttffcc] (0.506196115856774,1.4834932259945095)-- (0.4369540271201888,2.862182575816373);

\draw [line width=1.2pt,color=gray!50] (0.5061961158567739,1.4834932259945088)-- (2.4759761158567763,2.048223225994509);
\draw [line width=1.2pt,color=qqffqq] (2.4759761158567763,2.048223225994509)-- (4.522602593073869,2.048226985899212);
\draw [line width=1.2pt,color=gray!50] (4.522602593073869,2.048226985899212)-- (0.5061961158567739,1.4834932259945088);

\draw [line width=1.2pt,color=gray!50] (7.522602593073872,2.048226985899212)-- (3.5061961158567763,1.4834932259945088);

\draw [line width=1.2pt,color=qqzzff] (1.5035475233911475,0.8822533857095992)-- (3.0875561158567764,0.4034932259945089);
\draw [line width=1.2pt,color=qqffqq] (3.0875561158567764,0.4034932259945089)-- (3.9550961158567763,0.4034932259945089);
\draw [line width=1.2pt,color=ffwwqq] (3.9550961158567763,0.4034932259945089)-- (5.545218204593364,0.6695338761726454);
\draw [line width=1.2pt,color=gray] (5.545218204593364,0.6695338761726454)-- (6.699981256365425,1.2165269006962423);
\draw [line width=1.2pt,color=gray!50] (6.699981256365425,1.2165269006962423)-- (1.5035475233911475,0.8822533857095992);

\draw [line width=1.2pt,color=ffxfqq] (3.5061961158567763,1.4834932259945088)-- (5.475976115856779,2.048223225994509);
\draw [line width=1.2pt,color=ffxfqq] (5.475976115856779,2.048223225994509)-- (6.911022593073872,3.692956985899212);
\draw [line width=1.2pt,color=ffxfqq] (6.911022593073872,3.692956985899212)-- (7.891055574020006,5.672889935910688);
\draw [line width=1.2pt,color=ffxfqq] (7.891055574020006,5.672889935910688)-- (8.414584166485634,7.918859776195598);
\draw [line width=1.2pt,color=ffxfqq] (8.414584166485634,7.918859776195598)-- (7.758085503194081,10.395289861398567);
\draw [line width=1.2pt,color=ffxfqq] (7.758085503194081,10.395289861398567)-- (5.999346839902525,12.306989946601536);
\draw [line width=1.2pt,color=ffxfqq] (5.999346839902525,12.306989946601536)-- (3.960324751165938,13.120949296423401);
\draw [line width=1.2pt,color=ffxfqq] (3.960324751165938,13.120949296423401)-- (1.913698273948845,13.120945536518697);
\draw [line width=1.2pt,color=ffxfqq] (1.913698273948845,13.120945536518697)-- (-0.08895031851678414,12.519705696233787);
\draw [line width=1.2pt,color=ffxfqq] (-0.08895031851678414,12.519705696233787)-- (-1.6098303185167868,10.874975696233788);
\draw [line width=1.2pt,color=ffxfqq] (-1.6098303185167868,10.874975696233788)-- (-1.8208716552252326,8.39854561103082);
\draw [line width=1.2pt,color=ffxfqq] (-1.8208716552252326,8.39854561103082)-- (-1.2168960437057388,5.939852501304252);
\draw [line width=1.2pt,color=ffxfqq] (-1.2168960437057388,5.939852501304252)-- (-0.1503025474347801,3.9599233111974783);
\draw [line width=1.2pt,color=ffxfqq] (-0.1503025474347801,3.9599233111974783)-- (1.3288174525652208,2.3151933111974787);
\draw [line width=1.2pt,color=ffxfqq] (1.3288174525652208,2.3151933111974787)-- (3.5061961158567763,1.4834932259945088);
\draw [line width=1.2pt,color=qqffff] (0.9603247511659356,13.120949296423401)-- (2.930104751165938,13.6856792964234);
\draw [line width=1.2pt,color=qqffff] (2.930104751165938,13.6856792964234)-- (5.107483414457493,12.85397921122043);
\draw [line width=1.2pt,color=qqffff] (5.107483414457493,12.85397921122043)-- (6.586603414457494,11.209249211220431);
\draw [line width=1.2pt,color=qqffff] (6.586603414457494,11.209249211220431)-- (7.653196910728453,9.229320021113658);
\draw [line width=1.2pt,color=qqffff] (7.653196910728453,9.229320021113658)-- (8.257172522247947,6.770626911387091);
\draw [line width=1.2pt,color=qqffff] (8.257172522247947,6.770626911387091)-- (8.046131185539501,4.2941968261841215);
\draw [line width=1.2pt,color=qqffff] (8.046131185539501,4.2941968261841215)-- (6.525251185539498,2.649466826184122);
\draw [line width=1.2pt,color=qqffff] (6.525251185539498,2.649466826184122)-- (4.522602593073869,2.048226985899212);
\draw [line width=1.2pt,color=qqffff] (4.522602593073869,2.048226985899212)-- (2.4759761158567763,2.048223225994509);
\draw [line width=1.2pt,color=qqffff] (2.4759761158567763,2.048223225994509)-- (0.4369540271201888,2.862182575816373);
\draw [line width=1.2pt,color=qqffff] (0.4369540271201888,2.862182575816373)-- (-1.3217846361713672,4.773882661019343);
\draw [line width=1.2pt,color=qqffff] (-1.3217846361713672,4.773882661019343)-- (-1.9782832994629196,7.250312746222312);
\draw [line width=1.2pt,color=qqffff] (-1.9782832994629196,7.250312746222312)-- (-1.4547547069972921,9.49628258650722);
\draw [line width=1.2pt,color=qqffff] (-1.4547547069972921,9.49628258650722)-- (-0.4747217260511576,11.476215536518698);
\draw [line width=1.2pt,color=qqffff] (-0.4747217260511576,11.476215536518698)-- (0.9603247511659356,13.120949296423401);
\draw [line width=0.5pt,dashed,color=qqffff] (2.475976115856776,2.0482232259945086)-- (2.027076115856776,3.1282232259945086);
\draw [line width=0.5pt,dashed,color=qqffff] (2.027076115856776,3.1282232259945086)-- (0.4369540271201884,2.862182575816373);
\draw [line width=0.5pt,dashed,color=ffwwqq] (3.506196115856776,1.4834932259945095)-- (2.483580504337281,2.862186335721077);
\draw [line width=0.5pt,dashed,color=ffwwqq] (2.483580504337281,2.862186335721077)-- (1.3288174525652199,2.315193311197479);
\draw [line width=0.5pt,dashed,color=qqffff] (4.522602593073869,2.0482269858992126)-- (3.2985974525652226,2.8799233111974782);
\draw [line width=0.5pt,dashed,color=qqffff] (3.298597452565223,2.8799233111974782)-- (2.4759761158567763,2.0482232259945086);
\draw [line width=0.5pt,dashed,color=ffwwqq] (4.478624708322405,2.6494630662794183)-- (5.475976115856779,2.0482232259945086);
\draw [line width=0.5pt,dashed,color=ffwwqq] (4.478624708322405,2.6494630662794183)-- (3.506196115856776,1.4834932259945095);
\draw [line width=0.5pt,dashed,color=qqffff] (6.525251185539498,2.6494668261841223)-- (4.94124259307387,3.1282269858992127);
\draw [line width=0.5pt,dashed,color=qqffff] (4.52260259307387,2.0482269858992126)-- (4.94124259307387,3.1282269858992127);
\draw [line width=0.5pt,dashed,color=ffwwqq] (6.911022593073872,3.6929569858992117)-- (6.043482593073872,3.6929569858992117);
\draw [line width=0.5pt,dashed,color=ffwwqq] (5.4759761158567795,2.0482232259945086)-- (6.043482593073872,3.6929569858992117);
\draw [line width=0.5pt,dashed,color=qqffff] (8.046131185539501,4.2941968261841215)-- (6.456009096802913,4.028156176005986);
\draw [line width=0.5pt,dashed,color=qqffff] (6.525251185539499,2.6494668261841223)-- (6.456009096802914,4.028156176005986);
\draw [line width=0.5pt,dashed,color=ffwwqq] (7.891055574020006,5.672889935910687)-- (6.736292522247945,5.125896911387089);
\draw [line width=0.5pt,dashed,color=ffwwqq] (6.911022593073872,3.69295698589921)-- (6.736292522247945,5.125896911387089);
\draw [line width=0.5pt,dashed,color=qqffff] (8.257172522247949,6.770626911387091)-- (7.434551185539501,5.9389268261841215);
\draw [line width=0.5pt,dashed,color=qqffff] (7.434551185539501,5.9389268261841215)-- (8.046131185539501,4.294196826184122);
\draw [line width=0.5pt,dashed,color=ffwwqq] (8.414584166485634,7.918859776195598)-- (7.4421555740200045,6.752889935910689);
\draw [line width=0.5pt,dashed,color=ffwwqq] (7.4421555740200045,6.752889935910689)-- (7.891055574020005,5.672889935910689);
\draw [line width=0.5pt,dashed,color=qqffff] (7.234556910728454,8.149320021113658)-- (7.653196910728454,9.229320021113658);
\draw [line width=0.5pt,dashed,color=qqffff] (7.234556910728454,8.149320021113658)-- (8.257172522247949,6.77062691138709);
\draw [line width=0.5pt,dashed,color=ffwwqq] (7.190579025976989,8.750556101493864)-- (7.758085503194081,10.395289861398567);
\draw [line width=0.5pt,dashed,color=ffwwqq] (7.190579025976989,8.750556101493864)-- (8.414584166485636,7.918859776195598);
\draw [line width=0.5pt,dashed,color=ffwwqq] (4.409224751165938,12.0409492964234)-- (5.999346839902525,12.306989946601536);
\draw [line width=0.5pt,dashed,color=ffwwqq] (3.9603247511659383,13.120949296423401)-- (4.409224751165938,12.0409492964234);
\draw [line width=0.5pt,dashed,color=ffwwqq] (3.137703414457491,12.289249211220431)-- (3.960324751165938,13.120949296423401);
\draw [line width=0.5pt,dashed,color=ffwwqq] (1.913698273948845,13.120945536518697)-- (3.1377034144574916,12.289249211220431);
\draw [line width=0.5pt,dashed,color=ffwwqq] (1.9136982739488442,13.120945536518697)-- (1.4950582739488443,12.040945536518697);
\draw [line width=0.5pt,dashed,color=ffwwqq] (-0.08895031851678414,12.519705696233787)-- (1.4950582739488443,12.040945536518697);
\draw [line width=0.5pt,dashed,color=ffwwqq] (-0.08895031851678503,12.519705696233787)-- (-0.019708229780199815,11.141016346411924);
\draw [line width=0.5pt,dashed,color=ffwwqq] (-1.6098303185167868,10.874975696233788)-- (-0.019708229780198927,11.141016346411924);
\draw [line width=0.5pt,dashed,color=ffwwqq] (-0.9982503185167868,9.230245696233787)-- (-1.6098303185167868,10.874975696233786);
\draw [line width=0.5pt,dashed,color=ffwwqq] (-1.8208716552252344,8.39854561103082)-- (-0.9982503185167868,9.230245696233787);
\draw [line width=0.5pt,dashed,color=ffwwqq] (-0.7982560437057398,7.019852501304252)-- (-1.8208716552252344,8.39854561103082);
\draw [line width=0.5pt,dashed,color=ffwwqq] (-0.7982560437057398,7.019852501304252)-- (-1.2168960437057397,5.939852501304252);
\draw [line width=0.5pt,dashed,color=qqffff] (3.952720362685433,12.306986186696832)-- (5.107483414457494,12.85397921122043);
\draw [line width=0.5pt,dashed,color=qqffff] (2.9301047511659384,13.6856792964234)-- (3.952720362685433,12.306986186696832);
\draw [line width=0.5pt,dashed,color=qqffff] (1.957676158700309,12.519709456138491)-- (2.9301047511659384,13.6856792964234);
\draw [line width=0.5pt,dashed,color=qqffff] (1.957676158700309,12.519709456138491)-- (0.9603247511659356,13.120949296423401);
\draw [line width=0.5pt,dashed,color=qqffff] (0.9603247511659347,13.120949296423401)-- (0.3928182739488424,11.476215536518698);
\draw [line width=0.5pt,dashed,color=qqffff] (-0.4747217260511576,11.476215536518698)-- (0.3928182739488424,11.476215536518698);
\draw [line width=0.5pt,dashed,color=qqffff] (-0.4747217260511576,11.4762155365187)-- (-0.29999165522523086,10.04327561103082);
\draw [line width=0.5pt,dashed,color=qqffff] (-1.4547547069972921,9.496282586507222)-- (-0.29999165522523086,10.04327561103082);
\draw [line width=0.5pt,dashed,color=qqffff] (-1.0058547069972903,8.41628258650722)-- (-1.4547547069972904,9.49628258650722);
\draw [line width=0.5pt,dashed,color=qqffff] (-1.9782832994629196,7.250312746222312)-- (-1.0058547069972903,8.41628258650722);
\draw [line width=0.5pt,dashed,color=qqffff] (-0.7542781589542749,6.418616420924046)-- (-1.9782832994629214,7.250312746222312);
\draw [line width=0.5pt,dashed,color=qqffff] (-0.7542781589542749,6.418616420924046)-- (-1.3217846361713672,4.773882661019343);
\draw [line width=0.5pt,dashed,color=qqffff] (6.6558455031940795,9.830559861398568)-- (6.586603414457494,11.209249211220431);
\draw [line width=0.5pt,dashed,color=qqffff] (7.653196910728453,9.229320021113658)-- (6.6558455031940795,9.830559861398568);
\draw [line width=0.5pt,dashed,color=ffxfqq] (6.174076910728452,10.874050021113657)-- (5.999346839902525,12.306989946601536);
\draw [line width=0.5pt,dashed,color=ffxfqq] (6.174076910728452,10.874050021113657)-- (7.758085503194081,10.395289861398567);
\draw [line width=0.5pt,dashed,color=qqffff] (5.107483414457493,12.85397921122043)-- (5.719063414457493,11.209249211220431);
\draw [line width=0.5pt,dashed,color=qqffff] (5.719063414457493,11.209249211220431)-- (6.586603414457493,11.209249211220431);
\draw [line width=0.5pt,dashed,color=ffxfqq] (-1.2168960437057388,5.939852501304252)-- (-0.21954463617136533,5.338612661019342);
\draw [line width=0.5pt,dashed,color=ffxfqq] (-0.21954463617136533,5.338612661019342)-- (-0.1503025474347801,3.9599233111974783);
\draw [line width=0.5pt,dashed,color=qqffff] (0.2622239562942621,4.295122501304252)-- (-1.3217846361713672,4.773882661019343);
\draw [line width=0.5pt,dashed,color=qqffff] (0.2622239562942621,4.295122501304252)-- (0.4369540271201888,2.862182575816373);
\draw [line width=0.5pt,dashed,color=ffxfqq] (1.3288174525652208,2.315193311197479)-- (0.7172374525652208,3.9599233111974783);
\draw [line width=0.5pt,dashed,color=ffxfqq] (0.7172374525652208,3.9599233111974783)-- (-0.15030254743477922,3.9599233111974783);

\draw [->,>=stealth,line width=2pt,color=ffqqqq] (0.6,1.4834932259945088) -- (3.49,1.4834932259945088);
\draw [->,>=stealth,line width=2pt,color=ffqqqq] (4.6,2.048226985899212) -- (7.5,2.048226985899212);

\begin{scriptsize}
\draw [fill=qqqqff] (1.5035475233911475,0.8822533857095992) circle (1.5pt);
\draw[color=qqqqff] (1.4004630598389611,0.5910339560092084) node {$V_1$};
\draw [fill=qqqqff] (3.0875561158567764,0.4034932259945089) circle (1.5pt);
\draw[color=qqqqff] (2.9621966543508016,0.1138378012414252) node {$V_2$};
\draw [fill=qqqqff] (2.4759761158567763,2.048223225994509) circle (1.5pt)node[color=qqqqff,above] {$W_1$};
\draw [fill=qqqqff] (3.9550961158567763,0.4034932259945089) circle (1.5pt);
\draw[color=qqqqff] (3.9382801509207015,0.09214706693379868) node {$V_3$};
\draw [fill=qqqqff] (3.5061961158567763,1.4834932259945088) circle (1.5pt) node[color=qqqqff,above] {$W_2$};
\draw [fill=qqqqff] (5.545218204593364,0.6695338761726454) circle (1.5pt);
\draw[color=qqqqff] (5.651848956010082,0.37412661293294336) node {$V_4$};
\draw [fill=qqqqff] (4.522602593073869,2.048226985899212) circle (1.5pt)node[color=qqqqff,above] {$W_3$};
\draw [fill=qqqqff] (6.699981256365425,1.2165269006962423) circle (1.5pt);
\draw[color=qqqqff] (6.823149151893961,0.8947042363159796) node {$V_5$};
\draw [fill=qqqqff] (0.5061961158567739,1.4834932259945088) circle (1.5pt);
\draw[color=qqqqff] (0.2725443526915209,1.1766837823151244) node {$V_0$};
\draw [fill=qqqqff] (7.522602593073872,2.048226985899212) circle (1.5pt);
\draw[color=qqqqff] (7.7124696709909815,1.76233360862104) node {$V_6$};
\draw [fill=ududff] (6.525251185539498,2.649466826184122) circle (1.5pt);
\draw [fill=ududff] (0.506196115856774,1.4834932259945088) circle (1.5pt);
\draw [fill=ududff] (1.3288174525652208,2.3151933111974787) circle (1.5pt);
\draw [fill=qqqqff] (5.475976115856779,2.048223225994509) circle (1.5pt)node[color=qqqqff,above] {$W_4$};
\draw [fill=ududff] (6.911022593073872,3.692956985899212) circle (1.5pt);
\draw [fill=ududff] (7.522602593073872,2.048226985899212) circle (1.5pt);
\draw [fill=ududff] (8.495031185539501,3.214196826184122) circle (1.5pt);
\draw [fill=ududff] (8.495031185539501,3.214196826184122) circle (1.5pt);
\draw [fill=ududff] (8.046131185539501,4.2941968261841215) circle (1.5pt);
\draw [fill=ududff] (8.913671185539501,4.2941968261841215) circle (1.5pt);
\draw [fill=ududff] (7.891055574020006,5.672889935910688) circle (1.5pt);
\draw [fill=ududff] (9.481177662756593,5.938930586088825) circle (1.5pt);
\draw [fill=ududff] (8.913671185539501,4.2941968261841215) circle (1.5pt);
\draw [fill=ududff] (8.257172522247947,6.770626911387091) circle (1.5pt);
\draw [fill=ududff] (9.411935574020008,7.317619935910688) circle (1.5pt);
\draw [fill=ududff] (9.481177662756593,5.938930586088825) circle (1.5pt);
\draw [fill=ududff] (9.237205503194081,8.750559861398568) circle (1.5pt);
\draw [fill=ududff] (9.411935574020008,7.317619935910688) circle (1.5pt);
\draw [fill=ududff] (8.414584166485634,7.918859776195598) circle (1.5pt);
\draw [fill=ududff] (7.653196910728453,9.229320021113658) circle (1.5pt);
\draw [fill=ududff] (9.237205503194081,8.750559861398568) circle (1.5pt);
\draw [fill=ududff] (8.625625503194081,10.395289861398567) circle (1.5pt);
\draw [fill=ududff] (7.758085503194081,10.395289861398567) circle (1.5pt);
\draw [fill=ududff] (8.625625503194081,10.395289861398567) circle (1.5pt);
\draw [fill=ududff] (8.176725503194081,11.475289861398567) circle (1.5pt);
\draw [fill=ududff] (6.586603414457494,11.209249211220431) circle (1.5pt);
\draw [fill=ududff] (8.176725503194081,11.475289861398567) circle (1.5pt);
\draw [fill=ududff] (7.154109891674587,12.853982971125134) circle (1.5pt);
\draw [fill=ududff] (5.999346839902525,12.306989946601536) circle (1.5pt);
\draw [fill=ududff] (7.154109891674587,12.853982971125134) circle (1.5pt);
\draw [fill=ududff] (5.93010475116594,13.6856792964234) circle (1.5pt);
\draw [fill=uuuuuu] (3.218150433511357,7.584586261208955) circle (2.0pt);
\draw [fill=ududff] (4.932753343631567,14.28691913670831) circle (1.5pt);
\draw [fill=ududff] (3.348744751165938,14.7656792964234) circle (1.5pt);
\draw [fill=ududff] (3.960324751165938,13.120949296423401) circle (1.5pt);
\draw [fill=ududff] (3.348744751165938,14.7656792964234) circle (1.5pt);
\draw [fill=ududff] (2.481204751165938,14.7656792964234) circle (1.5pt);
\draw [fill=ududff] (2.930104751165938,13.6856792964234) circle (1.5pt);
\draw [fill=ududff] (2.481204751165938,14.7656792964234) circle (1.5pt);
\draw [fill=ududff] (0.8910826624293504,14.499638646245264) circle (1.5pt);
\draw [fill=ududff] (1.913698273948845,13.120945536518697) circle (1.5pt);
\draw [fill=ududff] (-0.2636803893427109,13.952645621721667) circle (1.5pt);
\draw [fill=ududff] (-1.0863017260511576,13.120945536518697) circle (1.5pt);
\draw [fill=ududff] (-0.08895031851678414,12.519705696233787) circle (1.5pt);
\draw [fill=ududff] (5.107483414457493,12.85397921122043) circle (1.5pt);
\draw [fill=ududff] (4.932753343631567,14.28691913670831) circle (1.5pt);
\draw [fill=ududff] (0.8910826624293504,14.499638646245264) circle (1.5pt);
\draw [fill=ududff] (-0.2636803893427109,13.952645621721667) circle (1.5pt);
\draw [fill=ududff] (0.9603247511659356,13.120949296423401) circle (1.5pt);
\draw [fill=ududff] (-0.4747217260511576,11.476215536518698) circle (1.5pt);
\draw [fill=ududff] (-1.0863017260511576,13.120945536518697) circle (1.5pt);
\draw [fill=ududff] (-2.058730318516787,11.954975696233788) circle (1.5pt);
\draw [fill=ududff] (-2.058730318516787,11.954975696233788) circle (1.5pt);
\draw [fill=ududff] (-1.6098303185167868,10.874975696233788) circle (1.5pt);
\draw [fill=ududff] (-2.477370318516787,10.874975696233788) circle (1.5pt);
\draw [fill=ududff] (-1.4547547069972921,9.49628258650722) circle (1.5pt);
\draw [fill=ududff] (-3.044876795733879,9.230241936329085) circle (1.5pt);
\draw [fill=ududff] (-2.477370318516787,10.874975696233788) circle (1.5pt);
\draw [fill=ududff] (-1.8208716552252326,8.39854561103082) circle (1.5pt);
\draw [fill=ududff] (-2.975634706997294,7.851552586507221) circle (1.5pt);
\draw [fill=ududff] (-3.044876795733879,9.230241936329085) circle (1.5pt);
\draw [fill=ududff] (-2.800904636171367,6.418612661019342) circle (1.5pt);
\draw [fill=ududff] (-2.975634706997294,7.851552586507221) circle (1.5pt);
\draw [fill=ududff] (-1.9782832994629196,7.250312746222312) circle (1.5pt);
\draw [fill=ududff] (-1.2168960437057388,5.939852501304252) circle (1.5pt);
\draw [fill=ududff] (-2.800904636171367,6.418612661019342) circle (1.5pt);
\draw [fill=ududff] (-2.189324636171367,4.773882661019343) circle (1.5pt);
\draw [fill=ududff] (-1.3217846361713672,4.773882661019343) circle (1.5pt);
\draw [fill=ududff] (-2.189324636171367,4.773882661019343) circle (1.5pt);
\draw [fill=ududff] (-1.7404246361713671,3.6938826610193427) circle (1.5pt);
\draw [fill=ududff] (-0.1503025474347801,3.9599233111974783) circle (1.5pt);
\draw [fill=ududff] (-1.7404246361713671,3.6938826610193427) circle (1.5pt);
\draw [fill=ududff] (-0.7178090246518725,2.315189551292775) circle (1.5pt);
\draw [fill=ududff] (0.4369540271201888,2.862182575816373) circle (1.5pt);
\draw [fill=ududff] (-0.7178090246518725,2.315189551292775) circle (1.5pt);
\draw [fill=ududff] (0.506196115856774,1.4834932259945095) circle (1.5pt);
\draw [fill=ududff] (5.93010475116594,13.6856792964234) circle (1.5pt);
\draw [fill=ududff] (5.107483414457493,12.85397921122043) circle (1.5pt);
\draw[color=ffqqqq] (2.268092834567761,1.2) node {$s_{3,L}$};
\draw[color=qqffqq] (1.7,2.06) node {$\sjx_L$};
\draw[color=qqffqq] (5.8,1.5671169998524015) node {$\sjx_R$};
\draw[color=cyan] (4.784219181281282,0.3) node {$\wbx_{-}$};
\draw[color=ffqqqq] (6.16,2.28) node {$s_{3}$};
\draw [fill=ududff] (2.475976115856776,2.0482232259945086) circle (1.5pt);
\draw [fill=rvwvcq] (-6.490538400610113,14.57319898793836) circle (1.5pt);
\draw[color=rvwvcq] (-3.18712937403957,16.023991415885472) node {$B$};
\draw [fill=ududff] (0.4369540271201884,2.862182575816373) circle (1.5pt);
\draw [fill=ududff] (3.506196115856776,1.4834932259945095) circle (1.5pt);
\draw [fill=ududff] (1.3288174525652199,2.315193311197479) circle (1.5pt);
\draw [fill=ududff] (4.522602593073869,2.0482269858992126) circle (1.5pt);
\draw [fill=ududff] (2.4759761158567763,2.0482232259945086) circle (1.5pt);
\draw [fill=ududff] (3.506196115856776,1.4834932259945095) circle (1.5pt);
\draw [fill=ududff] (6.525251185539498,2.6494668261841223) circle (1.5pt);
\draw [fill=ududff] (4.52260259307387,2.0482269858992126) circle (1.5pt);
\draw [fill=ududff] (6.911022593073872,3.6929569858992117) circle (1.5pt);
\draw [fill=ududff] (5.4759761158567795,2.0482232259945086) circle (1.5pt);
\draw [fill=ududff] (8.046131185539501,4.2941968261841215) circle (1.5pt);
\draw [fill=ududff] (6.525251185539499,2.6494668261841223) circle (1.5pt);
\draw [fill=ududff] (7.891055574020006,5.672889935910687) circle (1.5pt);
\draw [fill=ududff] (6.911022593073872,3.69295698589921) circle (1.5pt);
\draw [fill=ududff] (8.257172522247949,6.770626911387091) circle (1.5pt);
\draw [fill=ududff] (8.046131185539501,4.294196826184122) circle (1.5pt);
\draw [fill=ududff] (8.414584166485634,7.918859776195598) circle (1.5pt);
\draw [fill=ududff] (7.891055574020005,5.672889935910689) circle (1.5pt);
\draw [fill=ududff] (7.653196910728454,9.229320021113658) circle (1.5pt);
\draw [fill=ududff] (8.257172522247949,6.77062691138709) circle (1.5pt);
\draw [fill=ududff] (7.758085503194081,10.395289861398567) circle (1.5pt);
\draw [fill=ududff] (8.414584166485636,7.918859776195598) circle (1.5pt);
\draw [fill=ududff] (5.999346839902525,12.306989946601536) circle (1.5pt);
\draw [fill=ududff] (3.9603247511659383,13.120949296423401) circle (1.5pt);
\draw [fill=ududff] (3.960324751165938,13.120949296423401) circle (1.5pt);
\draw [fill=ududff] (1.913698273948845,13.120945536518697) circle (1.5pt);
\draw [fill=ududff] (1.9136982739488442,13.120945536518697) circle (1.5pt);
\draw [fill=ududff] (-0.08895031851678414,12.519705696233787) circle (1.5pt);
\draw [fill=ududff] (-0.08895031851678503,12.519705696233787) circle (1.5pt);
\draw [fill=ududff] (-1.6098303185167868,10.874975696233788) circle (1.5pt);
\draw [fill=ududff] (-1.6098303185167868,10.874975696233786) circle (1.5pt);
\draw [fill=ududff] (-1.8208716552252344,8.39854561103082) circle (1.5pt);
\draw [fill=ududff] (-1.8208716552252344,8.39854561103082) circle (1.5pt);
\draw [fill=ududff] (-1.2168960437057397,5.939852501304252) circle (1.5pt);
\draw [fill=ududff] (5.107483414457494,12.85397921122043) circle (1.5pt);
\draw [fill=ududff] (2.9301047511659384,13.6856792964234) circle (1.5pt);
\draw [fill=ududff] (2.9301047511659384,13.6856792964234) circle (1.5pt);
\draw [fill=ududff] (0.9603247511659347,13.120949296423401) circle (1.5pt);
\draw [fill=ududff] (-0.4747217260511576,11.476215536518698) circle (1.5pt);
\draw [fill=ududff] (-0.4747217260511576,11.4762155365187) circle (1.5pt);
\draw [fill=ududff] (-1.4547547069972921,9.496282586507222) circle (1.5pt);
\draw [fill=ududff] (-1.4547547069972904,9.49628258650722) circle (1.5pt);
\draw [fill=ududff] (-1.9782832994629196,7.250312746222312) circle (1.5pt);
\draw [fill=ududff] (-1.9782832994629214,7.250312746222312) circle (1.5pt);
\draw [fill=ududff] (-1.3217846361713672,4.773882661019343) circle (1.5pt);
\draw [fill=ududff] (6.586603414457494,11.209249211220431) circle (1.5pt);
\draw [fill=qqqqff] (7.653196910728453,9.229320021113658) circle (1.5pt);
\draw [fill=ududff] (5.999346839902525,12.306989946601536) circle (1.5pt);
\draw [fill=ududff] (5.107483414457493,12.85397921122043) circle (1.5pt);
\draw [fill=ududff] (6.586603414457493,11.209249211220431) circle (1.5pt);
\draw [fill=qqqqff] (-1.2168960437057388,5.939852501304252) circle (1.5pt);
\draw [fill=ududff] (-0.1503025474347801,3.9599233111974783) circle (1.5pt);
\draw [fill=ududff] (0.4369540271201888,2.862182575816373) circle (1.5pt);
\draw [fill=ududff] (1.3288174525652208,2.315193311197479) circle (1.5pt);
\draw [fill=ududff] (-0.15030254743477922,3.9599233111974783) circle (1.5pt);
\end{scriptsize}
\end{tikzpicture}
}
\vskip -.7cm
\caption{The triangles $\sjx_L,\sjx_R$ and the pentagon $\wbx_-$ in type $E_8$}
\label{fig:polygon-E8}
\vskip .6cm

{
\definecolor{qqffff}{rgb}{0.,1.,1.}
\definecolor{ffxfqq}{rgb}{1.,0.4980392156862745,0.}
\definecolor{ttffcc}{rgb}{0.2,1.,0.8}
\definecolor{uuuuuu}{rgb}{0.26666666666666666,0.26666666666666666,0.26666666666666666}
\definecolor{aqaqaq}{rgb}{0.6274509803921569,0.6274509803921569,0.6274509803921569}
\definecolor{xfqqff}{rgb}{0.4980392156862745,0.,1.}
\definecolor{ududff}{rgb}{0.30196078431372547,0.30196078431372547,1.}
\definecolor{ffwwqq}{rgb}{1.,0.4,0.}
\definecolor{ttffqq}{rgb}{0.2,1.,0.}
\definecolor{qqzzff}{rgb}{0.,0.6,1.}
\definecolor{ffqqqq}{rgb}{1.,0.,0.}
\definecolor{rvwvcq}{rgb}{0.08235294117647059,0.396078431372549,0.7529411764705882}
\begin{tikzpicture}[line cap=round,line join=round,arrow/.style={->,>=stealth,thick},scale=1.8,rotate=3]
\clip[rotate=-3](0.1,0.015242192519770936) rectangle (7,2.6);
\fill[fill=qqffff,fill opacity=0.4] (0.5121851973855259,1.4834932259945088) -- (2.4759761158567763,2.048223225994509) -- (4.539942022366593,2.0640379358303678) -- (6.699981256365425,1.2165269006962423) -- (5.545218204593364,0.6695338761726454) -- (3.9550961158567763,0.4034932259945089) -- (3.0875561158567764,0.4034932259945089) -- (1.5095366049198995,0.8822533857095992) -- cycle;

\draw [line width=1.2pt,color=ffxfqq] (4.539942022366593,2.0640379358303678)-- (5.545218204593364,0.6695338761726454);
\draw [line width=1.2pt,color=ffxfqq] (4.539942022366593,2.0640379358303678)-- (3.9550961158567763,0.4034932259945089);
\draw [line width=1.2pt,color=qqzzff] (1.5095366049198995,0.8822533857095992)-- (3.0875561158567764,0.4034932259945089);
\draw [line width=1.2pt,color=qqzzff] (3.0875561158567764,0.4034932259945089)-- (2.4759761158567763,2.048223225994509);
\draw [line width=1.2pt,color=qqzzff] (2.4759761158567763,2.048223225994509)-- (1.5095366049198995,0.8822533857095992);
\draw [line width=1.2pt,color=ttffqq] (3.9550961158567763,0.4034932259945089)-- (3.5061961158567763,1.4834932259945088);
\draw [line width=1.2pt,color=ttffqq] (3.5061961158567763,1.4834932259945088)-- (3.0875561158567764,0.4034932259945089);
\draw [line width=1.2pt,color=xfqqff] (6.699981256365425,1.2165269006962423)-- (7.533952940837843,2.0640379358303678);
\draw [line width=1.2pt,color=xfqqff] (7.533952940837843,2.0640379358303678)-- (6.53660153330347,2.6652777761152775);
\draw [line width=1.2pt,color=xfqqff] (6.53660153330347,2.6652777761152775)-- (6.699981256365425,1.2165269006962423);
\draw [line width=1.2pt,color=xfqqff] (1.3461568818579446,2.3310042611286343)-- (1.5095366049198995,0.8822533857095992);
\draw [line width=1.2pt,color=xfqqff] (1.5095366049198995,0.8822533857095992)-- (0.5121851973855259,1.4834932259945088);
\draw [line width=1.2pt,color=xfqqff] (0.5121851973855259,1.4834932259945088)-- (1.3461568818579446,2.3310042611286343);
\draw [line width=1.2pt,color=ttffcc] (5.545218204593364,0.6695338761726454)-- (6.699981256365425,1.2165269006962423);
\draw [line width=1.2pt,color=ttffcc] (6.699981256365425,1.2165269006962423)-- (5.469987034328026,2.048223225994509);
\draw [line width=1.2pt,color=ttffcc] (5.469987034328026,2.048223225994509)-- (5.545218204593364,0.6695338761726454);
\draw [line width=1.2pt,color=ttffqq] (8.05149245177472,4.3100077761152775)-- (8.91903245177472,4.3100077761152775);
\draw [line width=1.2pt,color=ffwwqq] (7.9137562695479495,5.704511835772999)-- (9.503878358284538,5.970552485951136);
\draw [line width=1.2pt,color=ttffcc] (8.27388413624714,6.802248811249402)-- (9.428647188019202,7.349241835772999);
\draw [line width=1.2pt,color=xfqqff] (9.265267464957248,8.797992711192034)-- (9.428647188019202,7.349241835772999);
\draw [line width=1.2pt,color=xfqqff] (9.428647188019202,7.349241835772999)-- (8.431295780484827,7.950481676057909);
\draw [line width=1.2pt,color=ffwwqq] (6.61466537622066,11.256682061013898)-- (8.204787464957247,11.522722711192033);
\draw [line width=1.2pt,color=ffwwqq] (7.199511282730477,12.917226770849755)-- (6.61466537622066,11.256682061013898);
\draw [line width=1.2pt,color=ttffcc] (6.044748230958415,12.370233746326157)-- (7.199511282730477,12.917226770849755);
\draw [line width=1.2pt,color=ttffcc] (7.199511282730477,12.917226770849755)-- (5.969517060693078,13.74892309614802);
\draw [line width=1.2pt,color=ttffcc] (5.969517060693078,13.74892309614802)-- (6.044748230958415,12.370233746326157);
\draw [line width=1.2pt,color=ttffqq] (2.975506142221828,13.74892309614802)-- (3.3941461422218278,14.82892309614802);
\draw [line width=1.2pt,color=ffwwqq] (2.5266061422218278,14.82892309614802)-- (0.9364840534852403,14.562882445969885);
\draw [line width=1.2pt,color=ffwwqq] (0.9364840534852403,14.562882445969885)-- (1.9417602357120112,13.168378386312163);
\draw [line width=1.2pt,color=ffwwqq] (1.9417602357120112,13.168378386312163)-- (2.5266061422218278,14.82892309614802);
\draw [line width=1.2pt,color=xfqqff] (-0.21827899828682096,14.015889421446287)-- (-1.0522506827592393,13.168378386312163);
\draw [line width=1.2pt,color=xfqqff] (-1.0522506827592393,13.168378386312163)-- (-0.05489927522486582,12.567138546027252);
\draw [line width=1.2pt,color=xfqqff] (-0.05489927522486582,12.567138546027252)-- (-0.21827899828682096,14.015889421446287);
\draw [line width=1.2pt,color=xfqqff] (5.1355453762206595,12.901412061013897)-- (4.972165653158704,14.35016293643293);
\draw [line width=1.2pt,color=xfqqff] (4.972165653158704,14.35016293643293)-- (5.969517060693078,13.74892309614802);
\draw [line width=1.2pt,color=ttffcc] (0.9364840534852403,14.562882445969885)-- (-0.21827899828682096,14.015889421446287);
\draw [line width=1.2pt,color=ttffcc] (-0.21827899828682096,14.015889421446287)-- (1.0117152237505778,13.184193096148022);
\draw [line width=1.2pt,color=ttffcc] (1.0117152237505778,13.184193096148022)-- (0.9364840534852403,14.562882445969885);
\draw [line width=1.2pt,color=qqzzff] (-1.0522506827592393,13.168378386312163)-- (-2.0186901936961164,12.002408546027253);
\draw [line width=1.2pt,color=ttffqq] (-1.5697901936961163,10.922408546027253)-- (-2.4373301936961163,10.922408546027253);
\draw [line width=1.2pt,color=ttffqq] (-2.4373301936961163,10.922408546027253)-- (-2.0186901936961164,12.002408546027253);
\draw [line width=1.2pt,color=ffwwqq] (-1.4320540114693454,9.52790448636953)-- (-3.022176100205934,9.261863836191395);
\draw [line width=1.2pt,color=ffwwqq] (-3.022176100205934,9.261863836191395)-- (-2.4373301936961163,10.922408546027253);
\draw [line width=1.2pt,color=ffwwqq] (-2.4373301936961163,10.922408546027253)-- (-1.4320540114693454,9.52790448636953);
\draw [line width=1.2pt,color=ttffcc] (-1.7921818781685364,8.430167510893128)-- (-2.9469449299405976,7.883174486369531);
\draw [line width=1.2pt,color=xfqqff] (-2.7835652068786434,6.434423610950496)-- (-2.9469449299405976,7.883174486369531);
\draw [line width=1.2pt,color=xfqqff] (-2.9469449299405976,7.883174486369531)-- (-1.9495935224062233,7.281934646084621);
\draw [line width=1.2pt,color=xfqqff] (-1.9495935224062233,7.281934646084621)-- (-2.7835652068786434,6.434423610950496);
\draw [line width=1.2pt,color=qqzzff] (-1.2055456959417672,5.955663451235406)-- (-2.7835652068786434,6.434423610950496);
\draw [line width=1.2pt,color=qqzzff] (-2.7835652068786434,6.434423610950496)-- (-2.1719852068786434,4.789693610950497);
\draw [line width=1.2pt,color=qqzzff] (-2.1719852068786434,4.789693610950497)-- (-1.2055456959417672,5.955663451235406);
\draw [line width=1.2pt,color=ttffqq] (-2.1719852068786434,4.789693610950497)-- (-1.7230852068786433,3.709693610950497);
\draw [line width=1.2pt,color=ttffqq] (-1.7230852068786433,3.709693610950497)-- (-1.3044452068786434,4.789693610950497);

\draw [line width=1.2pt,color=ffwwqq] (-0.1329631181420563,3.9757342611286326)-- (-1.7230852068786433,3.709693610950497);
\draw [line width=1.2pt,color=ffwwqq] (-1.7230852068786433,3.709693610950497)-- (-0.7178090246518725,2.315189551292775);
\draw [line width=1.2pt,color=ffwwqq] (-0.7178090246518725,2.315189551292775)-- (-0.1329631181420563,3.9757342611286326);
\draw [line width=1.2pt,color=ttffcc] (0.4369540271201888,2.862182575816373)-- (-0.7178090246518725,2.315189551292775);
\draw [line width=1.2pt,color=ttffcc] (-0.7178090246518725,2.315189551292775)-- (0.5121851973855263,1.4834932259945095);
\draw [line width=1.2pt,color=ttffcc] (0.5121851973855263,1.4834932259945095)-- (0.4369540271201888,2.862182575816373);
\draw [->,>=stealth,thick,line width=2pt] (3.0875561158567764,0.4034932259945089) -- (3.9550961158567763,0.4034932259945089);
\draw [->,>=stealth,thick,line width=2pt] (2.4759761158567763,2.048223225994509) -- (4.539942022366593,2.0640379358303678);
\draw [->,>=stealth,thick,line width=2pt] (0.5121851973855259,1.4834932259945088) -- (3.5061961158567763,1.4834932259945088);
\draw [->,>=stealth,thick,line width=2pt] (3.9550961158567763,0.4034932259945089) -- (5.545218204593364,0.6695338761726454);
\draw [->,>=stealth,thick,line width=2pt] (5.545218204593364,0.6695338761726454) -- (1.5095366049198997,0.8822533857095992);
\draw [->,>=stealth,thick,line width=2pt] (1.5095366049198995,0.8822533857095992) -- (6.699981256365425,1.2165269006962423);
\draw [->,>=stealth,thick,line width=2pt] (6.699981256365425,1.2165269006962423) -- (0.5121851973855263,1.4834932259945088);
\draw [->,>=stealth,thick,line width=2pt] (0.5121851973855259,1.4834932259945088) -- (2.4759761158567763,2.048223225994509);
\begin{scriptsize}
\draw [fill=rvwvcq] (1.5095366049198995,0.8822533857095992) circle (1pt) node [color=rvwvcq,below] {$V_1$};
\draw [fill=rvwvcq] (3.0875561158567764,0.4034932259945089) circle (1pt)node [color=rvwvcq,below] {$V_2$};
\draw [fill=rvwvcq] (2.4759761158567763,2.048223225994509) circle (1pt)node [color=rvwvcq,above] {$W_1$};
\draw [fill=rvwvcq] (3.9550961158567763,0.4034932259945089) circle (1pt)node [color=rvwvcq,below] {$V_3$};
\draw [fill=rvwvcq] (3.5061961158567763,1.4834932259945088) circle (1pt)node [color=rvwvcq,right] {$W_2$};
\draw [fill=rvwvcq] (5.545218204593364,0.6695338761726454) circle (1pt)node [color=rvwvcq,below] {$V_4$};
\draw [fill=rvwvcq] (4.539942022366593,2.0640379358303678) circle (1pt)node [color=rvwvcq,above] {$W_3$};
\draw [fill=rvwvcq] (6.699981256365425,1.2165269006962423) circle (1pt)node [color=rvwvcq,below] {$V_5$};
\draw [fill=rvwvcq] (0.5121851973855259,1.4834932259945088) circle (1pt)node [color=rvwvcq,below] {$V_0$};
\draw [fill=rvwvcq] (7.533952940837843,2.0640379358303678) circle (1pt);
\draw [fill=ududff] (6.53660153330347,2.6652777761152775) circle (1pt);
\draw [fill=ududff] (0.5121851973855263,1.4834932259945088) circle (1pt);
\draw [fill=ududff] (1.3461568818579446,2.3310042611286343) circle (1pt);
\draw [fill=ududff] (5.469987034328026,2.048223225994509) circle (1pt);
\draw [fill=ududff] (6.922372940837843,3.708767935830368) circle (1pt);
\draw [fill=ududff] (7.533952940837843,2.0640379358303678) circle (1pt);
\draw [fill=ududff] (8.50039245177472,3.2300077761152775) circle (1pt);
\draw [fill=ududff] (8.50039245177472,3.2300077761152775) circle (1pt);
\draw [fill=ududff] (8.05149245177472,4.3100077761152775) circle (1pt);
\draw [fill=ududff] (8.91903245177472,4.3100077761152775) circle (1pt);
\draw [fill=ududff] (7.9137562695479495,5.704511835772999) circle (1pt);
\draw [fill=ududff] (9.503878358284538,5.970552485951136) circle (1pt);
\draw [fill=ududff] (8.91903245177472,4.3100077761152775) circle (1pt);
\draw [fill=ududff] (8.27388413624714,6.802248811249402) circle (1pt);
\draw [fill=ududff] (9.428647188019202,7.349241835772999) circle (1pt);
\draw [fill=ududff] (9.503878358284538,5.970552485951136) circle (1pt);
\draw [fill=ududff] (9.265267464957248,8.797992711192034) circle (1pt);
\draw [fill=ududff] (9.428647188019202,7.349241835772999) circle (1pt);
\draw [fill=ududff] (8.431295780484827,7.950481676057909) circle (1pt);
\draw [fill=ududff] (7.687247954020371,9.276752870907124) circle (1pt);
\draw [fill=ududff] (9.265267464957248,8.797992711192034) circle (1pt);
\draw [fill=ududff] (8.653687464957247,10.442722711192033) circle (1pt);
\draw [fill=ududff] (7.7861474649572475,10.442722711192033) circle (1pt);
\draw [fill=ududff] (8.653687464957247,10.442722711192033) circle (1pt);
\draw [fill=ududff] (8.204787464957247,11.522722711192033) circle (1pt);
\draw [fill=ududff] (6.61466537622066,11.256682061013898) circle (1pt);
\draw [fill=ududff] (8.204787464957247,11.522722711192033) circle (1pt);
\draw [fill=ududff] (7.199511282730477,12.917226770849755) circle (1pt);
\draw [fill=ududff] (6.044748230958415,12.370233746326157) circle (1pt);
\draw [fill=ududff] (7.199511282730477,12.917226770849755) circle (1pt);
\draw [fill=ududff] (5.969517060693078,13.74892309614802) circle (1pt);
\draw [fill=uuuuuu] (3.240851129039302,7.616208161071265) circle (2.0pt);
\draw [fill=ududff] (4.972165653158704,14.35016293643293) circle (1pt);
\draw [fill=ududff] (3.3941461422218278,14.82892309614802) circle (1pt);
\draw [fill=ududff] (4.005726142221828,13.184193096148022) circle (1pt);
\draw [fill=ududff] (3.3941461422218278,14.82892309614802) circle (1pt);
\draw [fill=ududff] (2.5266061422218278,14.82892309614802) circle (1pt);
\draw [fill=ududff] (2.975506142221828,13.74892309614802) circle (1pt);
\draw [fill=ududff] (2.5266061422218278,14.82892309614802) circle (1pt);
\draw [fill=ududff] (0.9364840534852403,14.562882445969885) circle (1pt);
\draw [fill=ududff] (1.9417602357120112,13.168378386312163) circle (1pt);
\draw [fill=ududff] (-0.21827899828682096,14.015889421446287) circle (1pt);
\draw [fill=ududff] (-1.0522506827592393,13.168378386312163) circle (1pt);
\draw [fill=ududff] (-0.05489927522486582,12.567138546027252) circle (1pt);
\draw [fill=ududff] (5.1355453762206595,12.901412061013897) circle (1pt);
\draw [fill=ududff] (4.972165653158704,14.35016293643293) circle (1pt);
\draw [fill=ududff] (5.969517060693078,13.74892309614802) circle (1pt);
\draw [fill=ududff] (0.9364840534852403,14.562882445969885) circle (1pt);
\draw [fill=ududff] (-0.21827899828682096,14.015889421446287) circle (1pt);
\draw [fill=ududff] (1.0117152237505778,13.184193096148022) circle (1pt);
\draw [fill=ududff] (-0.44067068275923926,11.523648386312162) circle (1pt);
\draw [fill=ududff] (-1.0522506827592393,13.168378386312163) circle (1pt);
\draw [fill=ududff] (-2.0186901936961164,12.002408546027253) circle (1pt);
\draw [fill=ududff] (-2.0186901936961164,12.002408546027253) circle (1pt);
\draw [fill=ududff] (-1.5697901936961163,10.922408546027253) circle (1pt);
\draw [fill=ududff] (-2.4373301936961163,10.922408546027253) circle (1pt);
\draw [fill=ududff] (-1.4320540114693454,9.52790448636953) circle (1pt);
\draw [fill=ududff] (-3.022176100205934,9.261863836191395) circle (1pt);
\draw [fill=ududff] (-2.4373301936961163,10.922408546027253) circle (1pt);
\draw [fill=ududff] (-1.7921818781685364,8.430167510893128) circle (1pt);
\draw [fill=ududff] (-2.9469449299405976,7.883174486369531) circle (1pt);
\draw [fill=ududff] (-3.022176100205934,9.261863836191395) circle (1pt);
\draw [fill=ududff] (-2.7835652068786434,6.434423610950496) circle (1pt);
\draw [fill=ududff] (-2.9469449299405976,7.883174486369531) circle (1pt);
\draw [fill=ududff] (-1.9495935224062233,7.281934646084621) circle (1pt);
\draw [fill=ududff] (-1.2055456959417672,5.955663451235406) circle (1pt);
\draw [fill=ududff] (-2.7835652068786434,6.434423610950496) circle (1pt);
\draw [fill=ududff] (-2.1719852068786434,4.789693610950497) circle (1pt);
\draw [fill=ududff] (-1.3044452068786434,4.789693610950497) circle (1pt);
\draw [fill=ududff] (-2.1719852068786434,4.789693610950497) circle (1pt);
\draw [fill=ududff] (-1.7230852068786433,3.709693610950497) circle (1pt);
\draw [fill=ududff] (-0.1329631181420563,3.9757342611286326) circle (1pt);
\draw [fill=ududff] (-1.7230852068786433,3.709693610950497) circle (1pt);
\draw [fill=ududff] (-0.7178090246518725,2.315189551292775) circle (1pt);
\draw [fill=ududff] (0.4369540271201888,2.862182575816373) circle (1pt);
\draw [fill=ududff] (-0.7178090246518725,2.315189551292775) circle (1pt);
\draw [fill=ududff] (0.5121851973855263,1.4834932259945095) circle (1pt);
\draw [fill=ududff] (5.969517060693078,13.74892309614802) circle (1pt);
\draw [fill=ududff] (5.1355453762206595,12.901412061013897) circle (1pt);
\draw[color=black] (3.5123261791394853,0.23) node {$s_8$};
\draw[color=black] (3.585064121379423,2.2128886377808525) node {$s_1$};
\draw[color=black] (3,1.6) node {$s_3$};
\draw[color=black] (4.758569589517084,0.3) node {$s_7$};
\draw[color=black] (3.473532609944852,0.6) node {$s_6$};
\draw[color=black] (4.5,.9) node {$s_5$};
\draw[color=black] (4.5,1.45) node {$s_4$};
\draw[color=black] (1.6,1.96) node {$s_2$};

\draw [->,>=stealth,thick,line width=.8pt,bend right] (4.2,0.45) to (3.6,.425);
\draw [->,>=stealth,thick,line width=.8pt,bend right] (4.5,0.7) to (4.5,0.5);
\draw [->,>=stealth,thick,line width=.8pt,bend right] (3.4,0.8) to (3.4,1);
\draw [->,>=stealth,thick,line width=.8pt,bend right] (4.5,1.3) to (4.5,1.1);
\draw [->,>=stealth,thick,line width=.8pt] (3.2,1.36) to (3.2,1.49);
\draw [->,>=stealth,thick,line width=.8pt,bend right] (2.2,1.42) to (2.1,1.92);
\draw [->,>=stealth,thick,line width=.8pt,bend right] (2.2,1.945) to (2.7,2.03);
\end{scriptsize}
\draw[white, fill=white](7.5,0.015242192519770936) rectangle (6.880092904848598,3.019682265207781);
\end{tikzpicture}}

\begin{tikzpicture}[scale=.5,xscale=1,ar/.style={->,thick,>=stealth}]
\draw(0,0)node(v1){$s_1$}(2,0)node(v2){$s_2$}(4,0)node(v4){$s_4$}(6,0)node(v5){$s_5$}(8,0)node(v6){$s_6$}
(10,0)node(v7){$s_7$}(12,0)node(v8){$s_8$}
    (4,2)node(v3){$s_3$};
\draw[ar](v1)to(v2);
\draw[ar](v4)to(v2);
\draw[ar](v4)to(v3);
\draw[ar](v4)to(v5);
\draw[ar](v6)to(v5);
\draw[ar](v6)to(v7);
\draw[ar](v7)to(v8);
\end{tikzpicture}
\caption{The intersection quiver of a $30$-gon of type $E_8$}
\label{fig:int-quiver-E8}
\end{figure}

\begin{figure}[ht]\centering
\definecolor{ttffcc}{rgb}{0.2,1.,0.8}
\definecolor{uuuuuu}{rgb}{0.26666666666666666,0.26666666666666666,0.26666666666666666}
\definecolor{aqaqaq}{rgb}{0.6274509803921569,0.6274509803921569,0.6274509803921569}
\definecolor{xfqqff}{rgb}{0.4980392156862745,0.,1.}
\definecolor{ududff}{rgb}{0.30,0.30,1}
\definecolor{ffwwqq}{rgb}{1.,0.4,0.}
\definecolor{ttffqq}{rgb}{0.2,1.,0.}
\definecolor{qqzzff}{rgb}{0.,0.6,1.}
\definecolor{rvwvcq}{rgb}{0.08235294117647059,0.396078431372549,0.7529411764705882}
\begin{tikzpicture}[scale=2.01, scale=.9,arrow/.style={->,>=stealth,,thick}]
\clip(.1,-0.3) rectangle (5.858641816625223,2.2);
\draw[help lines, line width=0.2pt,step=0.14cm] (0,-.1) grid (5.858641816625223,2.2);
\fill[fill=qqzzff!50,fill opacity=0.8] (1.4112195947114898,0.9032651191587705) -- (2.1666191951719984,0.5048481019457757) -- (2.0757977170866315,1.6671745559119926) -- cycle;
\fill[line width=2.pt,color=ttffqq,fill=ttffqq!30,fill opacity=.81] (2.1666191951719984,0.5048481019457757) -- (3.051017717086632,0.3324745559119927) -- (2.956017717086632,1.502444555911992) -- cycle;
\fill[line width=2.pt,color=ffwwqq,fill=ffwwqq!30,fill opacity=0.81] (3.051017717086632,0.3324745559119927) -- (4.006017717086632,0.20244455591199217) -- (3.822424194303723,1.6671783158166966) -- cycle;
\fill[line width=2.pt,color=xfqqff,fill=xfqqff!30,fill opacity=.81] (4.949787717086631,0.20244455591199217) -- (6.272424194303723,1.6671783158166966) -- (5.367222316678865,2.266357752569918) -- cycle;
\fill[line width=2.pt,color=xfqqff,fill=xfqqff!30,fill opacity=.81] (1.8286541943037244,2.9671783158166964) -- (1.4112195947114898,0.9032651191587705) -- (0.5060177170866318,1.502444555911992) -- cycle;
\fill[line width=2.pt,color=ttffcc,fill=ttffcc!60,fill opacity=.81] (4.006017717086632,0.20244455591199217) -- (4.949787717086631,0.20244455591199217) -- (4.525797717086632,1.6671745559119926) -- cycle;
\fill[line width=2.pt,color=ttffqq,fill=ttffqq!30,fill opacity=.81] (6.937002316678866,2.4310877525699186) -- (6.842002316678865,3.601057752569918) -- (7.726400838593499,3.428684206536135) -- cycle;
\fill[line width=2.pt,color=ffwwqq,fill=ffwwqq!30,fill opacity=0.81] (7.54280731581059,4.893417966440839) -- (8.49780731581059,4.763387966440838) -- (7.726400838593499,3.428684206536135) -- cycle;
\fill[line width=2.pt,color=ttffcc,fill=ttffcc!60,fill opacity=.81] (8.07381731581059,6.228117966440838) -- (9.01758731581059,6.228117966440838) -- (8.49780731581059,4.763387966440838) -- cycle;
\fill[line width=2.pt,color=xfqqff,fill=xfqqff!30,fill opacity=.81] (9.435021915402825,8.292031163098764) -- (9.01758731581059,6.228117966440838) -- (8.112385438185733,6.827297403194059) -- cycle;
\fill[line width=2.pt,color=qqzzff,fill=qqzzff!50,fill opacity=0.6600000262260437] (8.679622314942316,8.69044818031176) -- (9.435021915402825,8.292031163098764) -- (9.344200437317458,9.454357617064982) -- cycle;
\fill[line width=2.pt,color=ttffqq,fill=ttffqq!30,fill opacity=.81] (8.459801915402824,9.626731163098764) -- (9.344200437317458,9.454357617064982) -- (9.249200437317459,10.624327617064981) -- cycle;
\fill[line width=2.pt,color=ffwwqq,fill=ffwwqq!30,fill opacity=0.81] (8.294200437317459,10.754357617064983) -- (9.249200437317459,10.624327617064981) -- (9.06560691453455,12.089061376969687) -- cycle;
\fill[line width=2.pt,color=ttffcc,fill=ttffcc!60,fill opacity=.81] (8.121836914534551,12.089061376969687) -- (9.06560691453455,12.089061376969687) -- (8.641616914534552,13.553791376969688) -- cycle;
\fill[line width=2.pt,color=qqzzff,fill=qqzzff!50,fill opacity=0.6600000262260437] (7.736415036909694,14.15297081372291) -- (6.981015436449185,14.551387830935905) -- (7.071836914534552,13.389061376969687) -- cycle;
\fill[line width=2.pt,color=ttffqq,fill=ttffqq!30,fill opacity=.81] (6.981015436449185,14.551387830935905) -- (6.096616914534552,14.723761376969687) -- (6.191616914534552,13.553791376969688) -- cycle;
\fill[line width=2.pt,color=ffwwqq,fill=ffwwqq!30,fill opacity=0.81] (6.096616914534552,14.723761376969687) -- (5.141616914534552,14.853791376969689) -- (5.325210437317461,13.389057617064983) -- cycle;
\fill[line width=2.pt,color=xfqqff,fill=xfqqff!30,fill opacity=.81] (4.197846914534553,14.853791376969689) -- (2.8752104373174605,13.389057617064983) -- (3.7804123149423186,12.789878180311762) -- cycle;
\fill[line width=2.pt,color=xfqqff,fill=xfqqff!30,fill opacity=.81] (7.318980437317459,12.089057617064984) -- (7.736415036909694,14.15297081372291) -- (8.641616914534552,13.553791376969688) -- cycle;
\fill[line width=2.pt,color=ttffcc,fill=ttffcc!60,fill opacity=.81] (5.141616914534552,14.853791376969689) -- (4.197846914534553,14.853791376969689) -- (4.621836914534552,13.389061376969687) -- cycle;
\fill[line width=2.pt,color=qqzzff,fill=qqzzff!50,fill opacity=0.6600000262260437] (2.9660319154028274,12.226731163098767) -- (2.8752104373174605,13.389057617064983) -- (2.210632314942319,12.625148180311761) -- cycle;
\fill[line width=2.pt,color=ttffqq,fill=ttffqq!30,fill opacity=.81] (2.210632314942318,12.625148180311761) -- (2.3056323149423186,11.455178180311762) -- (1.4212337930276853,11.627551726345546) -- cycle;
\fill[line width=2.pt,color=ffwwqq,fill=ffwwqq!30,fill opacity=0.81] (1.604827315810594,10.16281796644084) -- (0.6498273158105938,10.292847966440842) -- (1.4212337930276853,11.627551726345546) -- cycle;
\fill[line width=2.pt,color=ttffcc,fill=ttffcc!60,fill opacity=.81] (1.0738173158105937,8.828117966440843) -- (0.13004731581059303,8.828117966440843) -- (0.6498273158105938,10.292847966440842) -- cycle;
\fill[line width=2.pt,color=xfqqff,fill=xfqqff!30,fill opacity=.81] (-0.2873872837816407,6.764204769782916) -- (0.13004731581059303,8.828117966440843) -- (1.035249193435451,8.228938529687621) -- cycle;
\fill[line width=2.pt,color=qqzzff,fill=qqzzff!50,fill opacity=0.6600000262260437] (0.4680123166788679,6.36578775256992) -- (-0.2873872837816407,6.764204769782916) -- (-0.1965658056962738,5.601878315816698) -- cycle;
\fill[line width=2.pt,color=ttffqq,fill=ttffqq!30,fill opacity=.81] (0.6878327162183595,5.429504769782916) -- (-0.1965658056962738,5.601878315816698) -- (-0.10156580569627494,4.431908315816699) -- cycle;
\fill[line width=2.pt,color=ffwwqq,fill=ffwwqq!30,fill opacity=0.81] (0.8534341943037251,4.3018783158166976) -- (-0.10156580569627494,4.431908315816699) -- (0.0820277170866337,2.9671745559119937) -- cycle;
\fill[line width=2.pt,color=ttffcc,fill=ttffcc!60,fill opacity=.81] (1.0257977170866326,2.9671745559119937) -- (0.0820277170866337,2.9671745559119937) -- (0.5060177170866318,1.5024445559119926) -- cycle;
\draw [->,>=stealth,line width=1.5pt] (1.4112195947114898,0.9032651191587705) -- node[below] {$s_4$}(0.5060177170866318,1.502444555911992);
\draw [->,>=stealth,line width=1.5pt] (0.5060177170866318,1.502444555911992) -- node[below] {$s_3$}(2.956017717086632,1.502444555911992);
\draw [->,>=stealth,line width=1.5pt] (2.0757977170866315,1.6671745559119926) -- node[above] {$s_1$}(3.822424194303723,1.6671783158166966);
\draw [->,>=stealth,line width=1.5pt] (0.5060177170866318,1.502444555911992) -- node[above] {$s_2$} (2.0757977170866315,1.6671745559119926);
\draw [->,>=stealth,line width=1.5pt] (4.006017717086632,0.20244455591199217) -- node[below] {$s_8$}(4.949787717086631,0.20244455591199217);
\draw [->,>=stealth,line width=1.5pt] (2.1666191951719984,0.5048481019457757) -- node[below] {$s_5$}(1.4112195947114898,0.9032651191587705);
\draw [->,>=stealth,line width=1.5pt] (4.949787717086631,0.20244455591199217) -- node[above] {$s_7$}(3.051017717086632,0.3324745559119927);
\draw [->,>=stealth,line width=1.5pt] (3.051017717086632,0.3324745559119927) --  node[below] {$s_6$}(2.1666191951719984,0.5048481019457757);
\begin{scriptsize}
\draw [fill=rvwvcq] (1.4112195947114898,0.9032651191587705) circle (1pt)node[color=rvwvcq,below] {$V_1$};
\draw [fill=rvwvcq] (2.1666191951719984,0.5048481019457757) circle (1pt)node[color=rvwvcq,below] {$V_2$};
\draw [fill=rvwvcq] (2.0757977170866315,1.6671745559119926) circle (1pt)node[color=rvwvcq,above] {$W_1$};
\draw [fill=rvwvcq] (3.051017717086632,0.3324745559119927) circle (1pt)node[color=rvwvcq,below] {$V_3$};
\draw [fill=rvwvcq] (2.956017717086632,1.502444555911992) circle (1pt)node[color=rvwvcq,right] {$W_2$};
\draw [fill=rvwvcq] (4.006017717086632,0.20244455591199217) circle (1pt)node[color=rvwvcq,below] {$V_4$};
\draw [fill=rvwvcq] (3.822424194303723,1.6671783158166966) circle (1pt)node[color=rvwvcq,above] {$W_3$};
\draw [fill=rvwvcq] (4.949787717086631,0.20244455591199217) circle (1pt)node[color=rvwvcq,below] {$V_5$};
\draw [fill=rvwvcq] (0.5060177170866318,1.502444555911992) circle (1pt) node[color=rvwvcq,left] {$V_0$};
\draw [fill=ududff] (6.272424194303723,1.6671783158166966) circle (.8pt);
\draw [fill=ududff] (5.367222316678865,2.266357752569918) circle (.8pt);
\draw [fill=ududff] (1.8286541943037244,2.9671783158166964) circle (.8pt);
\draw [fill=ududff] (4.525797717086632,1.6671745559119926) circle (.8pt);
\draw [fill=ududff] (6.181602716218356,2.8295047697829134) circle (.8pt);
\draw [fill=ududff] (6.272424194303723,1.6671783158166966) circle (.8pt);
\draw [fill=ududff] (6.937002316678865,2.4310877525699186) circle (.8pt);
\draw [fill=ududff] (6.937002316678866,2.4310877525699186) circle (.8pt);
\draw [fill=ududff] (6.842002316678865,3.601057752569918) circle (.8pt);
\draw [fill=ududff] (7.726400838593499,3.428684206536135) circle (.8pt);
\draw [fill=ududff] (7.54280731581059,4.893417966440839) circle (.8pt);
\draw [fill=ududff] (8.49780731581059,4.763387966440838) circle (.8pt);
\draw [fill=ududff] (7.726400838593499,3.428684206536135) circle (.8pt);
\draw [fill=ududff] (8.07381731581059,6.228117966440838) circle (.8pt);
\draw [fill=ududff] (9.01758731581059,6.228117966440838) circle (.8pt);
\draw [fill=ududff] (8.49780731581059,4.763387966440838) circle (.8pt);
\draw [fill=ududff] (9.435021915402825,8.292031163098764) circle (.8pt);
\draw [fill=ududff] (9.01758731581059,6.228117966440838) circle (.8pt);
\draw [fill=ududff] (8.112385438185733,6.827297403194059) circle (.8pt);
\draw [fill=ududff] (8.679622314942316,8.69044818031176) circle (.8pt);
\draw [fill=ududff] (9.435021915402825,8.292031163098764) circle (.8pt);
\draw [fill=ududff] (9.344200437317458,9.454357617064982) circle (.8pt);
\draw [fill=ududff] (8.459801915402824,9.626731163098764) circle (.8pt);
\draw [fill=ududff] (9.344200437317458,9.454357617064982) circle (.8pt);
\draw [fill=ududff] (9.249200437317459,10.624327617064981) circle (.8pt);
\draw [fill=ududff] (8.294200437317459,10.754357617064983) circle (.8pt);
\draw [fill=ududff] (9.249200437317459,10.624327617064981) circle (.8pt);
\draw [fill=ududff] (9.06560691453455,12.089061376969687) circle (.8pt);
\draw [fill=ududff] (8.121836914534551,12.089061376969687) circle (.8pt);
\draw [fill=ududff] (9.06560691453455,12.089061376969687) circle (.8pt);
\draw [fill=ududff] (8.641616914534552,13.553791376969688) circle (.8pt);
\draw [fill=uuuuuu] (4.573817315810592,7.52811796644084) circle (2.0pt);
\draw[color=uuuuuu] (-0.2799094183976468,2.7304248550831547) node {$N$};
\draw [fill=ududff] (7.736415036909694,14.15297081372291) circle (.8pt);
\draw [fill=ududff] (6.981015436449185,14.551387830935905) circle (.8pt);
\draw [fill=ududff] (7.071836914534552,13.389061376969687) circle (.8pt);
\draw [fill=ududff] (6.981015436449185,14.551387830935905) circle (.8pt);
\draw [fill=ududff] (6.096616914534552,14.723761376969687) circle (.8pt);
\draw [fill=ududff] (6.191616914534552,13.553791376969688) circle (.8pt);
\draw [fill=ududff] (6.096616914534552,14.723761376969687) circle (.8pt);
\draw [fill=ududff] (5.141616914534552,14.853791376969689) circle (.8pt);
\draw [fill=ududff] (5.325210437317461,13.389057617064983) circle (.8pt);
\draw [fill=ududff] (4.197846914534553,14.853791376969689) circle (.8pt);
\draw [fill=ududff] (2.8752104373174605,13.389057617064983) circle (.8pt);
\draw [fill=ududff] (3.7804123149423186,12.789878180311762) circle (.8pt);
\draw [fill=ududff] (7.318980437317459,12.089057617064984) circle (.8pt);
\draw [fill=ududff] (7.736415036909694,14.15297081372291) circle (.8pt);
\draw [fill=ududff] (8.641616914534552,13.553791376969688) circle (.8pt);
\draw [fill=ududff] (5.141616914534552,14.853791376969689) circle (.8pt);
\draw [fill=ududff] (4.197846914534553,14.853791376969689) circle (.8pt);
\draw [fill=ududff] (4.621836914534552,13.389061376969687) circle (.8pt);
\draw [fill=ududff] (2.9660319154028274,12.226731163098767) circle (.8pt);
\draw [fill=ududff] (2.8752104373174605,13.389057617064983) circle (.8pt);
\draw [fill=ududff] (2.210632314942319,12.625148180311761) circle (.8pt);
\draw [fill=ududff] (2.210632314942318,12.625148180311761) circle (.8pt);
\draw [fill=ududff] (2.3056323149423186,11.455178180311762) circle (.8pt);
\draw [fill=ududff] (1.4212337930276853,11.627551726345546) circle (.8pt);
\draw [fill=ududff] (1.604827315810594,10.16281796644084) circle (.8pt);
\draw [fill=ududff] (0.6498273158105938,10.292847966440842) circle (.8pt);
\draw [fill=ududff] (1.4212337930276853,11.627551726345546) circle (.8pt);
\draw [fill=ududff] (1.0738173158105937,8.828117966440843) circle (.8pt);
\draw [fill=ududff] (0.13004731581059303,8.828117966440843) circle (.8pt);
\draw [fill=ududff] (0.6498273158105938,10.292847966440842) circle (.8pt);
\draw [fill=ududff] (-0.2873872837816407,6.764204769782916) circle (.8pt);
\draw [fill=ududff] (0.13004731581059303,8.828117966440843) circle (.8pt);
\draw [fill=ududff] (1.035249193435451,8.228938529687621) circle (.8pt);
\draw [fill=ududff] (0.4680123166788679,6.36578775256992) circle (.8pt);
\draw [fill=ududff] (-0.2873872837816407,6.764204769782916) circle (.8pt);
\draw [fill=ududff] (-0.1965658056962738,5.601878315816698) circle (.8pt);
\draw [fill=ududff] (0.6878327162183595,5.429504769782916) circle (.8pt);
\draw [fill=ududff] (-0.1965658056962738,5.601878315816698) circle (.8pt);
\draw [fill=ududff] (-0.10156580569627494,4.431908315816699) circle (.8pt);
\draw [fill=ududff] (0.8534341943037251,4.3018783158166976) circle (.8pt);
\draw [fill=ududff] (-0.10156580569627494,4.431908315816699) circle (.8pt);
\draw [fill=ududff] (0.0820277170866337,2.9671745559119937) circle (.8pt);
\draw [fill=ududff] (1.0257977170866326,2.9671745559119937) circle (.8pt);
\draw [fill=ududff] (0.0820277170866337,2.9671745559119937) circle (.8pt);
\draw [fill=ududff] (8.641616914534552,13.553791376969688) circle (.8pt);
\draw [fill=ududff] (7.318980437317459,12.089057617064984) circle (.8pt);
\end{scriptsize}
\end{tikzpicture}
\definecolor{ttffcc}{rgb}{0.2,1.,0.8}
\definecolor{uuuuuu}{rgb}{0.26666666666666666,0.26666666666666666,0.26666666666666666}
\definecolor{aqaqaq}{rgb}{0.6274509803921569,0.6274509803921569,0.6274509803921569}
\definecolor{xfqqff}{rgb}{0.4980392156862745,0.,1.}
\definecolor{ududff}{rgb}{0.30196078431372547,0.30196078431372547,1.}
\definecolor{ffwwqq}{rgb}{1.,0.4,0.}
\definecolor{ttffqq}{rgb}{0.2,1.,0.}
\definecolor{qqzzff}{rgb}{0.,0.6,1.}
\definecolor{ffqqqq}{rgb}{1.,0.,0.}
\begin{tikzpicture}[scale=1.37, scale=.9,arrow/.style={->,>=stealth,,thick}]
\clip(0,-0.3) rectangle (8.502058495478874,3.5);
\draw[help lines, line width=0.2pt,step=0.2cm] (0,0) grid (8.502058495478874,3.5);
\fill[line width=2.pt,color=qqzzff,fill=qqzzff!50,fill opacity=0.6600000262260437] (2.0052560334676004,0.9937842816950129) -- (3.496126166402167,0.5957201695999972) -- (3.1523254452095255,2.5688590043548976) -- cycle;
\fill[line width=2.pt,color=ttffqq,fill=ttffqq!30,fill opacity=.81] (3.496126166402167,0.5957201695999972) -- (4.806256784755807,0.39805909389926697) -- (4.1026460334676,1.9937842816950129) -- cycle;
\fill[line width=2.pt,color=ffwwqq,fill=ffwwqq!30,fill opacity=0.81] (4.806256784755807,0.39805909389926697) -- (6.5084809600814175,0.39805909389926697) -- (5.447535782183381,2.5688590043548896) -- cycle;
\fill[line width=2.pt,color=xfqqff,fill=xfqqff!30,fill opacity=.81] (7.9052560334676,1.4937842816950129) -- (9.04492578218338,2.5688590043548896) -- (7.54492578218338,3.5688590043548896) -- cycle;
\fill[line width=2.pt,color=xfqqff,fill=xfqqff!30,fill opacity=.81] (1.644925782183381,3.0688590043548896) -- (2.0052560334676004,0.9937842816950129) -- (0.5052560334676004,1.9937842816950129) -- cycle;
\fill[line width=2.pt,color=ttffcc,fill=ttffcc!60,fill opacity=.81] (6.5084809600814175,0.39805909389926697) -- (7.9052560334676,1.4937842816950129) -- (6.749715445209525,2.5688590043548976) -- cycle;
\fill[line width=2.pt,color=qqzzff,fill=qqzzff!50,fill opacity=0.6600000262260437] (8.701125060990739,4.54199783910979) -- (9.04492578218338,2.5688590043548896) -- (10.191995193925305,4.143933727014774) -- cycle;
\fill[line width=2.pt,color=ttffqq,fill=ttffqq!30,fill opacity=.81] (10.191995193925305,4.143933727014774) -- (9.488384442637098,5.73965891481052) -- (10.798515060990738,5.54199783910979) -- cycle;
\fill[line width=2.pt,color=ffwwqq,fill=ffwwqq!30,fill opacity=0.81] (9.737569883092702,7.7127977495654125) -- (11.439794058418311,7.7127977495654125) -- (10.798515060990738,5.54199783910979) -- cycle;
\fill[line width=2.pt,color=ttffcc,fill=ttffcc!60,fill opacity=.81] (10.284253470160237,8.787872472225297) -- (11.681028543546418,9.883597660021042) -- (11.439794058418311,7.7127977495654125) -- cycle;
\fill[line width=2.pt,color=xfqqff,fill=xfqqff!30,fill opacity=.81] (11.3206982922622,11.958672382680918) -- (11.681028543546418,9.883597660021042) -- (10.181028543546418,10.883597660021042) -- cycle;
\fill[line width=2.pt,color=qqzzff,fill=qqzzff!50,fill opacity=0.6600000262260437] (9.829828159327633,12.356736494775934) -- (11.3206982922622,11.958672382680916) -- (10.976897571069559,13.931811217435818) -- cycle;
\fill[line width=2.pt,color=ttffqq,fill=ttffqq!30,fill opacity=.81] (9.666766952715918,14.129472293136548) -- (10.976897571069559,13.931811217435818) -- (10.273286819781351,15.527536405231565) -- cycle;
\fill[line width=2.pt,color=ffwwqq,fill=ffwwqq!30,fill opacity=0.81] (8.57106264445574,15.527536405231565) -- (10.273286819781351,15.527536405231565) -- (9.212341641883315,17.698336315687186) -- cycle;
\fill[line width=2.pt,color=ttffcc,fill=ttffcc!60,fill opacity=.81] (7.8155665684971325,16.60261112789144) -- (9.212341641883315,17.698336315687186) -- (8.05680105362524,18.77341103834707) -- cycle;
\fill[line width=2.pt,color=qqzzff,fill=qqzzff!50,fill opacity=0.6600000262260437] (6.556801053625239,19.77341103834707) -- (5.065930920690672,20.171475150442088) -- (5.4097316418833135,18.198336315687186) -- cycle;
\fill[line width=2.pt,color=ttffqq,fill=ttffqq!30,fill opacity=.81] (5.065930920690672,20.171475150442088) -- (3.7558003023370317,20.369136226142817) -- (4.459411053625239,18.77341103834707) -- cycle;
\fill[line width=2.pt,color=ffwwqq,fill=ffwwqq!30,fill opacity=0.81] (3.7558003023370317,20.369136226142817) -- (2.0535761270114214,20.369136226142817) -- (3.1145213049094576,18.198336315687193) -- cycle;
\fill[line width=2.pt,color=xfqqff,fill=xfqqff!30,fill opacity=.81] (0.6568010536252391,19.27341103834707) -- (-0.4828686950905414,18.198336315687193) -- (1.0171313049094586,17.198336315687193) -- cycle;
\fill[line width=2.pt,color=xfqqff,fill=xfqqff!30,fill opacity=.81] (6.917131304909458,17.698336315687193) -- (6.556801053625239,19.77341103834707) -- (8.05680105362524,18.77341103834707) -- cycle;
\fill[line width=2.pt,color=ttffcc,fill=ttffcc!60,fill opacity=.81] (2.0535761270114214,20.369136226142817) -- (0.6568010536252391,19.27341103834707) -- (1.8123416418833136,18.198336315687186) -- cycle;
\fill[line width=2.pt,color=qqzzff,fill=qqzzff!50,fill opacity=0.6600000262260437] (-0.13906797389789993,16.225197480932295) -- (-0.4828686950905414,18.198336315687193) -- (-1.6299381068324656,16.62326159302731) -- cycle;
\fill[line width=2.pt,color=ttffqq,fill=ttffqq!30,fill opacity=.81] (-1.6299381068324656,16.62326159302731) -- (-0.9263273555442595,15.027536405231565) -- (-2.236457973897899,15.225197480932295) -- cycle;
\fill[line width=2.pt,color=ffwwqq,fill=ffwwqq!30,fill opacity=0.81] (-1.1755127959998628,13.054397570476672) -- (-2.877736971325472,13.054397570476672) -- (-2.236457973897899,15.225197480932295) -- cycle;
\fill[line width=2.pt,color=ttffcc,fill=ttffcc!60,fill opacity=.81] (-1.7221963830673985,11.979322847816787) -- (-3.118971456453579,10.883597660021042) -- (-2.877736971325472,13.054397570476672) -- cycle;
\fill[line width=2.pt,color=xfqqff,fill=xfqqff!30,fill opacity=.81] (-2.7586412051693614,8.808522937361166) -- (-3.118971456453579,10.883597660021042) -- (-1.618971456453579,9.883597660021042) -- cycle;
\fill[line width=2.pt,color=qqzzff,fill=qqzzff!50,fill opacity=0.6600000262260437] (-1.267771072234794,8.41045882526615) -- (-2.7586412051693614,8.808522937361168) -- (-2.41484048397672,6.835384102606266) -- cycle;
\fill[line width=2.pt,color=ttffqq,fill=ttffqq!30,fill opacity=.81] (-1.1047098656230787,6.637723026905537) -- (-2.41484048397672,6.835384102606266) -- (-1.711229732688512,5.239658914810519) -- cycle;
\fill[line width=2.pt,color=ffwwqq,fill=ffwwqq!30,fill opacity=0.81] (-0.009005557362900873,5.239658914810519) -- (-1.711229732688512,5.239658914810519) -- (-0.6502845547904759,3.0688590043548984) -- cycle;
\fill[line width=2.pt,color=ttffcc,fill=ttffcc!60,fill opacity=.81] (0.7464905185957065,4.164584192150645) -- (-0.6502845547904759,3.0688590043548984) -- (0.5052560334675995,1.9937842816950138) -- cycle;
\draw [->,>=stealth,line width=1.5pt] (3.1523254452095255,2.5688590043548976) -- node[above] {$s_1$}(5.447535782183381,2.5688590043548896);
\draw [->,>=stealth,line width=1.5pt] (0.5052560334676004,1.9937842816950129) -- (4.1026460334676,1.9937842816950129);
\draw ($ (0.5052560334676004,1.9937842816950129)!.7!(4.1026460334676,1.9937842816950129)$) node[above] {$s_3$};
\draw [->,>=stealth,line width=1.5pt] (7.9052560334676,1.4937842816950129) -- node[below] {$s_4$}(0.5052560334676004,1.9937842816950129);
\draw [->,>=stealth,line width=1.5pt] (0.5052560334676004,1.9937842816950129) -- node[above] {$s_2$}(3.1523254452095255,2.5688590043548976);
\draw [->,>=stealth,line width=1.5pt] (4.806256784755807,0.39805909389926697) -- node[below] {$s_8$}(6.5084809600814175,0.39805909389926697);
\draw [->,>=stealth,line width=1.5pt] (6.5084809600814175,0.39805909389926697) -- node[above] {$s_7$}(3.496126166402167,0.5957201695999972);
\draw [->,>=stealth,line width=1.5pt] (3.496126166402167,0.5957201695999972) -- node[below] {$s_6$}(2.0052560334676004,0.9937842816950129);
\draw [->,>=stealth,line width=1.5pt] (2.0052560334676004,0.9937842816950129) -- node[below] {$s_5$}(7.9052560334676,1.4937842816950129);
\begin{scriptsize}
\draw [fill=rvwvcq] (2.0052560334676004,0.9937842816950129) circle (1pt)node[color=rvwvcq,below] {$V_1$};
\draw [fill=rvwvcq] (3.496126166402167,0.5957201695999972) circle (1pt)node[color=rvwvcq,below] {$V_{2}$};
\draw [fill=rvwvcq] (3.1523254452095255,2.5688590043548976) circle (1pt)node[color=rvwvcq,above] {$W_1$};
\draw [fill=rvwvcq] (4.806256784755807,0.39805909389926697) circle (1pt)node[color=rvwvcq,below] {$V_3$};
\draw [fill=rvwvcq] (4.1026460334676,1.9937842816950129) circle (1pt)node[color=rvwvcq,right] {$W_2$};
\draw [fill=rvwvcq] (6.5084809600814175,0.39805909389926697) circle (1pt)node[color=rvwvcq,below] {$V_4$};
\draw [fill=rvwvcq] (5.447535782183381,2.5688590043548896) circle (1pt)node[color=rvwvcq,above] {$W_3$};
\draw [fill=rvwvcq] (7.9052560334676,1.4937842816950129) circle (1pt)node[color=rvwvcq,below] {$V_5$};
\draw [fill=rvwvcq] (0.5052560334676004,1.9937842816950129) circle (1pt)node[color=rvwvcq,left] {$V_0$};
\draw [fill=ududff] (9.04492578218338,2.5688590043548896) circle (1pt);
\draw [fill=ududff] (7.54492578218338,3.5688590043548896) circle (1pt);
\draw [fill=ududff] (1.644925782183381,3.0688590043548896) circle (1pt);
\draw [fill=ududff] (6.749715445209525,2.5688590043548976) circle (1pt);
\draw [fill=ududff] (8.701125060990739,4.54199783910979) circle (1pt);
\draw [fill=ududff] (9.04492578218338,2.5688590043548896) circle (1pt);
\draw [fill=ududff] (10.191995193925305,4.143933727014774) circle (1pt);
\draw [fill=ududff] (10.191995193925305,4.143933727014774) circle (1pt);
\draw [fill=ududff] (9.488384442637098,5.73965891481052) circle (1pt);
\draw [fill=ududff] (10.798515060990738,5.54199783910979) circle (1pt);
\draw [fill=ududff] (9.737569883092702,7.7127977495654125) circle (1pt);
\draw [fill=ududff] (11.439794058418311,7.7127977495654125) circle (1pt);
\draw [fill=ududff] (10.798515060990738,5.54199783910979) circle (1pt);
\draw [fill=ududff] (10.284253470160237,8.787872472225297) circle (1pt);
\draw [fill=ududff] (11.681028543546418,9.883597660021042) circle (1pt);
\draw [fill=ududff] (11.439794058418311,7.7127977495654125) circle (1pt);
\draw [fill=ududff] (11.3206982922622,11.958672382680918) circle (1pt);
\draw [fill=ududff] (11.681028543546418,9.883597660021042) circle (1pt);
\draw [fill=ududff] (10.181028543546418,10.883597660021042) circle (1pt);
\draw [fill=ududff] (9.829828159327633,12.356736494775934) circle (1pt);
\draw [fill=ududff] (11.3206982922622,11.958672382680916) circle (1pt);
\draw [fill=ududff] (10.976897571069559,13.931811217435818) circle (1pt);
\draw [fill=ududff] (9.666766952715918,14.129472293136548) circle (1pt);
\draw [fill=ududff] (10.976897571069559,13.931811217435818) circle (1pt);
\draw [fill=ududff] (10.273286819781351,15.527536405231565) circle (1pt);
\draw [fill=ududff] (8.57106264445574,15.527536405231565) circle (1pt);
\draw [fill=ududff] (10.273286819781351,15.527536405231565) circle (1pt);
\draw [fill=ududff] (9.212341641883315,17.698336315687186) circle (1pt);
\draw [fill=ududff] (7.8155665684971325,16.60261112789144) circle (1pt);
\draw [fill=ududff] (9.212341641883315,17.698336315687186) circle (1pt);
\draw [fill=ududff] (8.05680105362524,18.77341103834707) circle (1pt);
\draw [fill=uuuuuu] (4.2810285435464195,10.383597660021042) circle (2.0pt);
\draw[color=uuuuuu] (-0.0020254019125345267,4.127453417742386) node {$N$};
\draw [fill=ududff] (6.556801053625239,19.77341103834707) circle (1pt);
\draw [fill=ududff] (5.065930920690672,20.171475150442088) circle (1pt);
\draw [fill=ududff] (5.4097316418833135,18.198336315687186) circle (1pt);
\draw [fill=ududff] (5.065930920690672,20.171475150442088) circle (1pt);
\draw [fill=ududff] (3.7558003023370317,20.369136226142817) circle (1pt);
\draw [fill=ududff] (4.459411053625239,18.77341103834707) circle (1pt);
\draw [fill=ududff] (3.7558003023370317,20.369136226142817) circle (1pt);
\draw [fill=ududff] (2.0535761270114214,20.369136226142817) circle (1pt);
\draw [fill=ududff] (3.1145213049094576,18.198336315687193) circle (1pt);
\draw [fill=ududff] (0.6568010536252391,19.27341103834707) circle (1pt);
\draw [fill=ududff] (-0.4828686950905414,18.198336315687193) circle (1pt);
\draw [fill=ududff] (1.0171313049094586,17.198336315687193) circle (1pt);
\draw [fill=ududff] (6.917131304909458,17.698336315687193) circle (1pt);
\draw [fill=ududff] (6.556801053625239,19.77341103834707) circle (1pt);
\draw [fill=ududff] (8.05680105362524,18.77341103834707) circle (1pt);
\draw[color=ududff] (0.029032009399319887,4.141570422884136) node {$I'_{1}$};
\draw [fill=ududff] (2.0535761270114214,20.369136226142817) circle (1pt);
\draw [fill=ududff] (0.6568010536252391,19.27341103834707) circle (1pt);
\draw [fill=ududff] (1.8123416418833136,18.198336315687186) circle (1pt);
\draw [fill=ududff] (-0.13906797389789993,16.225197480932295) circle (1pt);
\draw [fill=ududff] (-0.4828686950905414,18.198336315687193) circle (1pt);
\draw [fill=ududff] (-1.6299381068324656,16.62326159302731) circle (1pt);
\draw [fill=ududff] (-1.6299381068324656,16.62326159302731) circle (1pt);
\draw [fill=ududff] (-0.9263273555442595,15.027536405231565) circle (1pt);
\draw [fill=ududff] (-2.236457973897899,15.225197480932295) circle (1pt);
\draw [fill=ududff] (-1.1755127959998628,13.054397570476672) circle (1pt);
\draw [fill=ududff] (-2.877736971325472,13.054397570476672) circle (1pt);
\draw [fill=ududff] (-2.236457973897899,15.225197480932295) circle (1pt);
\draw [fill=ududff] (-1.7221963830673985,11.979322847816787) circle (1pt);
\draw [fill=ududff] (-3.118971456453579,10.883597660021042) circle (1pt);
\draw [fill=ududff] (-2.877736971325472,13.054397570476672) circle (1pt);
\draw [fill=ududff] (-2.7586412051693614,8.808522937361166) circle (1pt);
\draw [fill=ududff] (-3.118971456453579,10.883597660021042) circle (1pt);
\draw [fill=ududff] (-1.618971456453579,9.883597660021042) circle (1pt);
\draw [fill=ududff] (-1.267771072234794,8.41045882526615) circle (1pt);
\draw [fill=ududff] (-2.7586412051693614,8.808522937361168) circle (1pt);
\draw [fill=ududff] (-2.41484048397672,6.835384102606266) circle (1pt);
\draw [fill=ududff] (-1.1047098656230787,6.637723026905537) circle (1pt);
\draw [fill=ududff] (-2.41484048397672,6.835384102606266) circle (1pt);
\draw [fill=ududff] (-1.711229732688512,5.239658914810519) circle (1pt);
\draw [fill=ududff] (-0.009005557362900873,5.239658914810519) circle (1pt);
\draw [fill=ududff] (-1.711229732688512,5.239658914810519) circle (1pt);
\draw [fill=ududff] (-0.6502845547904759,3.0688590043548984) circle (1pt);
\draw [fill=ududff] (0.7464905185957065,4.164584192150645) circle (1pt);
\draw [fill=ududff] (-0.6502845547904759,3.0688590043548984) circle (1pt);
\draw [fill=ududff] (8.05680105362524,18.77341103834707) circle (1pt);
\draw [fill=ududff] (6.917131304909458,17.698336315687193) circle (1pt);
\end{scriptsize}
\end{tikzpicture}
\definecolor{ttffcc}{rgb}{0.2,1.,0.8}
\definecolor{uuuuuu}{rgb}{0.26666666666666666,0.26666666666666666,0.26666666666666666}
\definecolor{aqaqaq}{rgb}{0.6274509803921569,0.6274509803921569,0.6274509803921569}
\definecolor{xfqqff}{rgb}{0.4980392156862745,0.,1.}
\definecolor{ududff}{rgb}{0.30196078431372547,0.30196078431372547,1.}
\definecolor{ffwwqq}{rgb}{1.,0.4,0.}
\definecolor{ttffqq}{rgb}{0.2,1.,0.}
\definecolor{qqzzff}{rgb}{0.,0.6,1.}
\definecolor{rvwvcq}{rgb}{0.08235294117647059,0.396078431372549,0.7529411764705882}
\begin{tikzpicture}[scale=1.8, scale=.9,arrow/.style={->,>=stealth,,thick}]
\clip(0.11,-0.35) rectangle (6.65,2.753006741045315);
\draw[help lines, line width=0.2pt,step=0.15cm] (0.13,-0.1) grid (6.65,2.753006741045315);
\fill[line width=2.pt,color=qqzzff,fill=qqzzff!50,fill opacity=0.6600000262260437] (1.4944807351834926,0.9070957970844746) -- (2.589278857558635,0.5206552338376964) -- (2.759058857558635,2.0710052338376954) -- cycle;
\fill[line width=2.pt,color=ttffqq,fill=ttffqq!30,fill opacity=.81] (2.589278857558635,0.5206552338376964) -- (3.538178857558635,0.3262752338376964) -- (3.489278857558635,1.5062752338376955) -- cycle;
\fill[line width=2.pt,color=ffwwqq,fill=ffwwqq!30,fill opacity=0.81] (3.538178857558635,0.3262752338376964) -- (5.187768857558636,0.3262752338376964) -- (4.905685334775726,2.0710089937423986) -- cycle;
\fill[line width=2.pt,color=xfqqff,fill=xfqqff!30,fill opacity=.81] (6.341094041114868,0.7486023217431645) -- (7.905685334775726,2.0710089937423986) -- (6.900483457150869,2.670188430495619) -- cycle;
\fill[line width=2.pt,color=xfqqff,fill=xfqqff!30,fill opacity=.81] (2.053870151219493,2.8286819058369295) -- (1.4944807351834926,0.9070957970844746) -- (0.48927885755863487,1.5062752338376955) -- cycle;
\fill[line width=2.pt,color=ttffcc,fill=ttffcc!60,fill opacity=.81] (5.187768857558636,0.3262752338376964) -- (6.341094041114868,0.7486023217431645) -- (5.759058857558635,2.0710052338376954) -- cycle;
\fill[line width=2.pt,color=qqzzff,fill=qqzzff!50,fill opacity=0.6600000262260437] (8.075465334775727,3.6213589937423976) -- (7.905685334775726,2.0710089937423986) -- (9.170263457150869,3.234918430495619) -- cycle;
\fill[line width=2.pt,color=ttffqq,fill=ttffqq!30,fill opacity=.81] (9.170263457150869,3.234918430495619) -- (9.121363457150867,4.414918430495618) -- (10.070263457150869,4.220538430495618) -- cycle;
\fill[line width=2.pt,color=ffwwqq,fill=ffwwqq!30,fill opacity=0.81] (9.78817993436796,5.96527219040032) -- (11.43776993436796,5.96527219040032) -- (10.070263457150869,4.220538430495617) -- cycle;
\fill[line width=2.pt,color=ttffcc,fill=ttffcc!60,fill opacity=.81] (10.855734750811727,7.287675102494851) -- (12.00905993436796,7.710002190400319) -- (11.43776993436796,5.96527219040032) -- cycle;
\fill[line width=2.pt,color=xfqqff,fill=xfqqff!30,fill opacity=.81] (12.56844935040396,9.631588299152774) -- (12.00905993436796,7.710002190400319) -- (11.003858056743102,8.309181627153539) -- cycle;
\fill[line width=2.pt,color=qqzzff,fill=qqzzff!50,fill opacity=0.6600000262260437] (11.473651228028816,10.018028862399552) -- (12.56844935040396,9.631588299152774) -- (12.738229350403959,11.181938299152772) -- cycle;
\fill[line width=2.pt,color=ttffqq,fill=ttffqq!30,fill opacity=.81] (11.78932935040396,11.376318299152771) -- (12.738229350403959,11.181938299152772) -- (12.689329350403959,12.361938299152772) -- cycle;
\fill[line width=2.pt,color=ffwwqq,fill=ffwwqq!30,fill opacity=0.81] (11.03973935040396,12.361938299152772) -- (12.689329350403959,12.361938299152772) -- (12.40724582762105,14.106672059057473) -- cycle;
\fill[line width=2.pt,color=ttffcc,fill=ttffcc!60,fill opacity=.81] (11.253920644064818,13.684344971152006) -- (12.40724582762105,14.106672059057473) -- (11.825210644064818,15.429074971152005) -- cycle;
\fill[line width=2.pt,color=qqzzff,fill=qqzzff!50,fill opacity=0.6600000262260437] (10.82000876643996,16.028254407905226) -- (9.725210644064816,16.414694971152002) -- (9.555430644064817,14.864344971152004) -- cycle;
\fill[line width=2.pt,color=ttffqq,fill=ttffqq!30,fill opacity=.81] (9.725210644064816,16.414694971152002) -- (8.776310644064818,16.609074971152005) -- (8.825210644064818,15.429074971152005) -- cycle;
\fill[line width=2.pt,color=ffwwqq,fill=ffwwqq!30,fill opacity=0.81] (8.776310644064818,16.609074971152005) -- (7.126720644064816,16.609074971152005) -- (7.408804166847726,14.8643412112473) -- cycle;
\fill[line width=2.pt,color=xfqqff,fill=xfqqff!30,fill opacity=.81] (5.973395460508584,16.186747883246536) -- (4.408804166847726,14.8643412112473) -- (5.4140060444725835,14.26516177449408) -- cycle;
\fill[line width=2.pt,color=xfqqff,fill=xfqqff!30,fill opacity=.81] (10.26061935040396,14.10666829915277) -- (10.82000876643996,16.028254407905226) -- (11.825210644064818,15.429074971152005) -- cycle;
\fill[line width=2.pt,color=ttffcc,fill=ttffcc!60,fill opacity=.81] (7.126720644064816,16.609074971152005) -- (5.973395460508584,16.186747883246536) -- (6.555430644064817,14.864344971152004) -- cycle;
\fill[line width=2.pt,color=qqzzff,fill=qqzzff!50,fill opacity=0.6600000262260437] (4.2390241668477255,13.313991211247302) -- (4.408804166847726,14.8643412112473) -- (3.1442260444725836,13.700431774494081) -- cycle;
\fill[line width=2.pt,color=ttffqq,fill=ttffqq!30,fill opacity=.81] (3.1442260444725836,13.700431774494081) -- (3.193126044472585,12.520431774494082) -- (2.2442260444725832,12.71481177449408) -- cycle;
\fill[line width=2.pt,color=ffwwqq,fill=ffwwqq!30,fill opacity=0.81] (2.526309567255492,10.97007801458938) -- (0.8767195672554919,10.97007801458938) -- (2.2442260444725832,12.714811774494082) -- cycle;
\fill[line width=2.pt,color=ttffcc,fill=ttffcc!60,fill opacity=.81] (1.4587547508117247,9.647675102494848) -- (0.30542956725549253,9.22534801458938) -- (0.8767195672554919,10.97007801458938) -- cycle;
\fill[line width=2.pt,color=xfqqff,fill=xfqqff!30,fill opacity=.81] (-0.25395984878050726,7.303761905836925) -- (0.30542956725549253,9.22534801458938) -- (1.3106314448803502,8.626168577836161) -- cycle;
\fill[line width=2.pt,color=qqzzff,fill=qqzzff!50,fill opacity=0.6600000262260437] (0.8408382735946365,6.917321342590148) -- (-0.25395984878050726,7.303761905836925) -- (-0.42373984878050663,5.753411905836927) -- cycle;
\fill[line width=2.pt,color=ttffqq,fill=ttffqq!30,fill opacity=.81] (0.5251601512194917,5.559031905836928) -- (-0.42373984878050663,5.753411905836927) -- (-0.3748398487805069,4.573411905836927) -- cycle;
\fill[line width=2.pt,color=ffwwqq,fill=ffwwqq!30,fill opacity=0.81] (1.274750151219493,4.573411905836927) -- (-0.3748398487805069,4.573411905836927) -- (-0.09275632599759831,2.8286781459322263) -- cycle;
\fill[line width=2.pt,color=ttffcc,fill=ttffcc!60,fill opacity=.81] (1.0605688575586338,3.2510052338376934) -- (-0.09275632599759831,2.8286781459322263) -- (0.48927885755863443,1.5062752338376946) -- cycle;
\draw [->,>=stealth,line width=1.5pt] (1.4944807351834926,0.9070957970844746) -- node[below] {$s_4$}(0.48927885755863487,1.5062752338376955);
\draw [->,>=stealth,line width=1.5pt] (0.48927885755863487,1.5062752338376955) -- node[above right] {$s_3$}(3.489278857558635,1.5062752338376955);
\draw [->,>=stealth,line width=1.5pt] (2.759058857558635,2.0710052338376954) -- node[above] {$s_1$}(4.905685334775726,2.0710089937423986);
\draw [->,>=stealth,line width=1.5pt] (0.48927885755863487,1.5062752338376955) -- node[above] {$s_2$}(2.759058857558635,2.0710052338376954);
\draw [->,>=stealth,line width=1.5pt] (3.538178857558635,0.3262752338376964) -- node[below] {$s_8$}(5.187768857558636,0.3262752338376964);
\draw [->,>=stealth,line width=1.5pt] (6.341094041114868,0.7486023217431645) -- node[above] {$s_5$}(1.4944807351834921,0.9070957970844746);
\draw [->,>=stealth,line width=1.5pt] (5.187768857558636,0.3262752338376964) -- (2.589278857558635,0.5206552338376964);
\draw ($(5.187768857558636,0.3262752338376964)!.2! (2.589278857558635,0.5206552338376964)$) node[above] {$s_7$};
\draw [->,>=stealth,line width=1.5pt] (2.589278857558635,0.5206552338376964) -- (6.341094041114868,0.7486023217431645);
\draw ($(2.589278857558635,0.5206552338376964)!.2! (6.341094041114868,0.7486023217431645)$) node[above] {$s_6$};
\begin{scriptsize}
\draw [fill=rvwvcq] (1.4944807351834926,0.9070957970844746) circle (1pt)node[color=rvwvcq,below] {$V_1$};
\draw [fill=rvwvcq] (2.589278857558635,0.5206552338376964) circle (1pt)node[color=rvwvcq,below] {$V_2$};
\draw [fill=rvwvcq] (2.759058857558635,2.0710052338376954) circle (1pt)node[color=rvwvcq,above] {$W_1$};
\draw [fill=rvwvcq] (3.538178857558635,0.3262752338376964) circle (1pt)node[color=rvwvcq,below] {$V_3$};
\draw [fill=rvwvcq] (3.489278857558635,1.5062752338376955) circle (1pt)node[color=rvwvcq,right] {$W_2$};
\draw [fill=rvwvcq] (5.187768857558636,0.3262752338376964) circle (1pt)node[color=rvwvcq,below] {$V_4$};
\draw [fill=rvwvcq] (4.905685334775726,2.0710089937423986) circle (1pt)node[color=rvwvcq,above] {$W_3$};
\draw [fill=rvwvcq] (6.341094041114868,0.7486023217431645) circle (1pt)node[color=rvwvcq,below] {$V_5$};
\draw [fill=rvwvcq] (0.48927885755863487,1.5062752338376955) circle (1pt)node[color=rvwvcq,left] {$V_0$};
\draw [fill=ududff] (7.905685334775726,2.0710089937423986) circle (1pt);
\draw [fill=ududff] (6.900483457150869,2.670188430495619) circle (1pt);
\draw [fill=ududff] (2.053870151219493,2.8286819058369295) circle (1pt);
\draw [fill=ududff] (5.759058857558635,2.0710052338376954) circle (1pt);
\draw [fill=ududff] (8.075465334775727,3.6213589937423976) circle (1pt);
\draw [fill=ududff] (7.905685334775726,2.0710089937423986) circle (1pt);
\draw [fill=ududff] (9.170263457150869,3.234918430495619) circle (1pt);
\draw [fill=ududff] (9.170263457150869,3.234918430495619) circle (1pt);
\draw [fill=ududff] (9.121363457150867,4.414918430495618) circle (1pt);
\draw [fill=ududff] (10.070263457150869,4.220538430495618) circle (1pt);
\draw [fill=ududff] (9.78817993436796,5.96527219040032) circle (1pt);
\draw [fill=ududff] (11.43776993436796,5.96527219040032) circle (1pt);
\draw [fill=ududff] (10.070263457150869,4.220538430495617) circle (1pt);
\draw [fill=ududff] (10.855734750811727,7.287675102494851) circle (1pt);
\draw [fill=ududff] (12.00905993436796,7.710002190400319) circle (1pt);
\draw [fill=ududff] (11.43776993436796,5.96527219040032) circle (1pt);
\draw [fill=ududff] (12.56844935040396,9.631588299152774) circle (1pt);
\draw [fill=ududff] (12.00905993436796,7.710002190400319) circle (1pt);
\draw [fill=ududff] (11.003858056743102,8.309181627153539) circle (1pt);
\draw [fill=ududff] (11.473651228028816,10.018028862399552) circle (1pt);
\draw [fill=ududff] (12.56844935040396,9.631588299152774) circle (1pt);
\draw [fill=ududff] (12.738229350403959,11.181938299152772) circle (1pt);
\draw [fill=ududff] (11.78932935040396,11.376318299152771) circle (1pt);
\draw [fill=ududff] (12.738229350403959,11.181938299152772) circle (1pt);
\draw [fill=ududff] (12.689329350403959,12.361938299152772) circle (1pt);
\draw [fill=ududff] (11.03973935040396,12.361938299152772) circle (1pt);
\draw [fill=ududff] (12.689329350403959,12.361938299152772) circle (1pt);
\draw [fill=ududff] (12.40724582762105,14.106672059057473) circle (1pt);
\draw [fill=ududff] (11.253920644064818,13.684344971152006) circle (1pt);
\draw [fill=ududff] (12.40724582762105,14.106672059057473) circle (1pt);
\draw [fill=ududff] (11.825210644064818,15.429074971152005) circle (1pt);
\draw [fill=uuuuuu] (6.157244750811726,8.46767510249485) circle (2.0pt);
\draw [fill=ududff] (10.82000876643996,16.028254407905226) circle (1pt);
\draw [fill=ududff] (9.725210644064816,16.414694971152002) circle (1pt);
\draw [fill=ududff] (9.555430644064817,14.864344971152004) circle (1pt);
\draw [fill=ududff] (9.725210644064816,16.414694971152002) circle (1pt);
\draw [fill=ududff] (8.776310644064818,16.609074971152005) circle (1pt);
\draw [fill=ududff] (8.825210644064818,15.429074971152005) circle (1pt);
\draw [fill=ududff] (8.776310644064818,16.609074971152005) circle (1pt);
\draw [fill=ududff] (7.126720644064816,16.609074971152005) circle (1pt);
\draw [fill=ududff] (7.408804166847726,14.8643412112473) circle (1pt);
\draw [fill=ududff] (5.973395460508584,16.186747883246536) circle (1pt);
\draw [fill=ududff] (4.408804166847726,14.8643412112473) circle (1pt);
\draw [fill=ududff] (5.4140060444725835,14.26516177449408) circle (1pt);
\draw [fill=ududff] (10.26061935040396,14.10666829915277) circle (1pt);
\draw [fill=ududff] (10.82000876643996,16.028254407905226) circle (1pt);
\draw [fill=ududff] (11.825210644064818,15.429074971152005) circle (1pt);
\draw [fill=ududff] (7.126720644064816,16.609074971152005) circle (1pt);
\draw [fill=ududff] (5.973395460508584,16.186747883246536) circle (1pt);
\draw [fill=ududff] (6.555430644064817,14.864344971152004) circle (1pt);
\draw [fill=ududff] (4.2390241668477255,13.313991211247302) circle (1pt);
\draw [fill=ududff] (4.408804166847726,14.8643412112473) circle (1pt);
\draw [fill=ududff] (3.1442260444725836,13.700431774494081) circle (1pt);
\draw [fill=ududff] (3.1442260444725836,13.700431774494081) circle (1pt);
\draw [fill=ududff] (3.193126044472585,12.520431774494082) circle (1pt);
\draw [fill=ududff] (2.2442260444725832,12.71481177449408) circle (1pt);
\draw [fill=ududff] (2.526309567255492,10.97007801458938) circle (1pt);
\draw [fill=ududff] (0.8767195672554919,10.97007801458938) circle (1pt);
\draw [fill=ududff] (2.2442260444725832,12.714811774494082) circle (1pt);
\draw [fill=ududff] (1.4587547508117247,9.647675102494848) circle (1pt);
\draw [fill=ududff] (0.30542956725549253,9.22534801458938) circle (1pt);
\draw [fill=ududff] (0.8767195672554919,10.97007801458938) circle (1pt);
\draw [fill=ududff] (-0.25395984878050726,7.303761905836925) circle (1pt);
\draw [fill=ududff] (0.30542956725549253,9.22534801458938) circle (1pt);
\draw [fill=ududff] (1.3106314448803502,8.626168577836161) circle (1pt);
\draw [fill=ududff] (0.8408382735946365,6.917321342590148) circle (1pt);
\draw [fill=ududff] (-0.25395984878050726,7.303761905836925) circle (1pt);
\draw [fill=ududff] (-0.42373984878050663,5.753411905836927) circle (1pt);
\draw [fill=ududff] (0.5251601512194917,5.559031905836928) circle (1pt);
\draw [fill=ududff] (-0.42373984878050663,5.753411905836927) circle (1pt);
\draw [fill=ududff] (-0.3748398487805069,4.573411905836927) circle (1pt);
\draw [fill=ududff] (1.274750151219493,4.573411905836927) circle (1pt);
\draw [fill=ududff] (-0.3748398487805069,4.573411905836927) circle (1pt);
\draw [fill=ududff] (-0.09275632599759831,2.8286781459322263) circle (1pt);
\draw [fill=ududff] (1.0605688575586338,3.2510052338376934) circle (1pt);
\draw [fill=ududff] (-0.09275632599759831,2.8286781459322263) circle (1pt);
\draw [fill=ududff] (11.825210644064818,15.429074971152005) circle (1pt);
\draw [fill=ududff] (10.26061935040396,14.10666829915277) circle (1pt);
\end{scriptsize}
\end{tikzpicture}
\definecolor{ttffcc}{rgb}{0.2,1.,0.8}
\definecolor{uuuuuu}{rgb}{0.26666666666666666,0.26666666666666666,0.26666666666666666}
\definecolor{aqaqaq}{rgb}{0.6274509803921569,0.6274509803921569,0.6274509803921569}
\definecolor{xfqqff}{rgb}{0.4980392156862745,0.,1.}
\definecolor{ududff}{rgb}{0.30196078431372547,0.30196078431372547,1.}
\definecolor{ffwwqq}{rgb}{1.,0.4,0.}
\definecolor{ttffqq}{rgb}{0.2,1.,0.}
\definecolor{qqzzff}{rgb}{0.,0.6,1.}
\definecolor{rvwvcq}{rgb}{0.08235294117647059,0.396078431372549,0.7529411764705882}
\begin{tikzpicture}[scale=2.05, scale=.9,arrow/.style={->,>=stealth,,thick}]
\clip(0.16,0) rectangle (5.899460258052324,2.910347861601015);
\draw[help lines, line width=0.2pt,step=0.135cm] (0.20164062625996548,0.1) grid (5.899460258052324,2.910347861601015);
\fill[line width=2.pt,color=qqzzff,fill=qqzzff!50,fill opacity=0.6600000262260437] (1.3965242753986016,0.996181247831571) -- (2.351923875859111,0.697764230618576) -- (2.261102397773744,2.060090684584793) -- cycle;
\fill[line width=2.pt,color=ttffqq,fill=ttffqq!30,fill opacity=.81] (2.351923875859111,0.697764230618576) -- (3.240222397773744,0.4153606845847938) -- (2.991322397773744,1.495360684584793) -- cycle;
\fill[line width=2.pt,color=ffwwqq,fill=ffwwqq!30,fill opacity=0.81] (3.240222397773744,0.4153606845847938) -- (4.312692435425764,0.4195296942616159) -- (3.907728874990835,2.060094444489496) -- cycle;
\fill[line width=2.pt,color=xfqqff,fill=xfqqff!30,fill opacity=.81] (5.191322397773745,0.5953606845847924) -- (6.407728874990836,2.060094444489496) -- (5.502526997365978,2.559273881242718) -- cycle;
\fill[line width=2.pt,color=xfqqff,fill=xfqqff!30,fill opacity=.81] (1.707728874990834,2.960094444489497) -- (1.3965242753986016,0.996181247831571) -- (0.4913223977737433,1.495360684584793) -- cycle;
\fill[line width=2.pt,color=ttffcc,fill=ttffcc!60,fill opacity=.81] (4.312692435425764,0.4195296942616159) -- (5.191322397773745,0.5953606845847924) -- (4.761102397773745,2.060090684584793) -- cycle;
\fill[line width=2.pt,color=qqzzff,fill=qqzzff!50,fill opacity=0.6600000262260437] (6.316907396905469,3.422420898455713) -- (6.407728874990836,2.060094444489496) -- (7.272306997365979,3.124003881242718) -- cycle;
\fill[line width=2.pt,color=ttffqq,fill=ttffqq!30,fill opacity=.81] (7.272306997365979,3.124003881242718) -- (7.023406997365979,4.204003881242717) -- (7.911705519280612,3.921600335208935) -- cycle;
\fill[line width=2.pt,color=ffwwqq,fill=ffwwqq!30,fill opacity=0.81] (7.506741958845683,5.562165085436815) -- (8.579211996497703,5.566334095113637) -- (7.911705519280612,3.921600335208935) -- cycle;
\fill[line width=2.pt,color=ttffcc,fill=ttffcc!60,fill opacity=.81] (8.148991996497703,7.0310640951136385) -- (9.027621958845684,7.206895085436815) -- (8.579211996497703,5.566334095113637) -- cycle;
\fill[line width=2.pt,color=xfqqff,fill=xfqqff!30,fill opacity=.81] (9.338826558437916,9.170808282094741) -- (9.027621958845684,7.206895085436815) -- (8.122420081220826,7.706074522190037) -- cycle;
\fill[line width=2.pt,color=qqzzff,fill=qqzzff!50,fill opacity=0.6600000262260437] (8.383426957977406,9.469225299307736) -- (9.338826558437916,9.170808282094741) -- (9.248005080352549,10.533134736060958) -- cycle;
\fill[line width=2.pt,color=ttffqq,fill=ttffqq!30,fill opacity=.81] (8.359706558437916,10.81553828209474) -- (9.248005080352549,10.533134736060958) -- (8.99910508035255,11.613134736060957) -- cycle;
\fill[line width=2.pt,color=ffwwqq,fill=ffwwqq!30,fill opacity=0.81] (7.926635042700529,11.608965726384135) -- (8.99910508035255,11.613134736060957) -- (8.59414151991762,13.253699486288838) -- cycle;
\fill[line width=2.pt,color=ttffcc,fill=ttffcc!60,fill opacity=.81] (7.71551155756964,13.077868495965662) -- (8.59414151991762,13.253699486288838) -- (8.16392151991762,14.718429486288839) -- cycle;
\fill[line width=2.pt,color=qqzzff,fill=qqzzff!50,fill opacity=0.6600000262260437] (7.258719642292762,15.217608923042063) -- (6.303320041832253,15.516025940255057) -- (6.3941415199176195,14.15369948628884) -- cycle;
\fill[line width=2.pt,color=ttffqq,fill=ttffqq!30,fill opacity=.81] (6.303320041832253,15.516025940255057) -- (5.4150215199176195,15.798429486288839) -- (5.663921519917619,14.71842948628884) -- cycle;
\fill[line width=2.pt,color=ffwwqq,fill=ffwwqq!30,fill opacity=0.81] (5.4150215199176195,15.798429486288839) -- (4.342551482265599,15.794260476612017) -- (4.747515042700528,14.153695726384136) -- cycle;
\fill[line width=2.pt,color=xfqqff,fill=xfqqff!30,fill opacity=.81] (3.4639215199176183,15.618429486288841) -- (2.2475150427005275,14.153695726384136) -- (3.1527169203253855,13.654516289630916) -- cycle;
\fill[line width=2.pt,color=xfqqff,fill=xfqqff!30,fill opacity=.81] (6.947515042700529,13.253695726384136) -- (7.258719642292762,15.217608923042063) -- (8.16392151991762,14.71842948628884) -- cycle;
\fill[line width=2.pt,color=ttffcc,fill=ttffcc!60,fill opacity=.81] (4.342551482265599,15.794260476612017) -- (3.4639215199176183,15.618429486288841) -- (3.8941415199176186,14.15369948628884) -- cycle;
\fill[line width=2.pt,color=qqzzff,fill=qqzzff!50,fill opacity=0.6600000262260437] (2.3383365207858944,12.79136927241792) -- (2.2475150427005275,14.153695726384136) -- (1.3829369203253847,13.089786289630915) -- cycle;
\fill[line width=2.pt,color=ttffqq,fill=ttffqq!30,fill opacity=.81] (1.3829369203253847,13.089786289630915) -- (1.6318369203253846,12.009786289630917) -- (0.7435383984107515,12.292189835664697) -- cycle;
\fill[line width=2.pt,color=ffwwqq,fill=ffwwqq!30,fill opacity=0.81] (1.1485019588456806,10.651625085436818) -- (0.07603192119366042,10.647456075759996) -- (0.7435383984107515,12.292189835664697) -- cycle;
\fill[line width=2.pt,color=ttffcc,fill=ttffcc!60,fill opacity=.81] (0.5062519211936607,9.182726075759994) -- (-0.37237804115432027,9.006895085436817) -- (0.07603192119366042,10.647456075759996) -- cycle;
\fill[line width=2.pt,color=xfqqff,fill=xfqqff!30,fill opacity=.81] (-0.6835826407465522,7.0429818887788915) -- (-0.37237804115432027,9.006895085436817) -- (0.5328238364705378,8.507715648683597) -- cycle;
\fill[line width=2.pt,color=qqzzff,fill=qqzzff!50,fill opacity=0.6600000262260437] (0.2718169597139575,6.744564871565897) -- (-0.6835826407465522,7.0429818887788915) -- (-0.5927611626611853,5.6806554348126745) -- cycle;
\fill[line width=2.pt,color=ttffqq,fill=ttffqq!30,fill opacity=.81] (0.2955373592534478,5.398251888778892) -- (-0.5927611626611853,5.6806554348126745) -- (-0.34386116266118627,4.600655434812676) -- cycle;
\fill[line width=2.pt,color=ffwwqq,fill=ffwwqq!30,fill opacity=0.81] (0.7286088749908339,4.604824444489498) -- (-0.34386116266118627,4.600655434812676) -- (0.06110239777374282,2.960090684584795) -- cycle;
\fill[line width=2.pt,color=ttffcc,fill=ttffcc!60,fill opacity=.81] (0.9397323601217238,3.135921674907971) -- (0.06110239777374282,2.960090684584795) -- (0.4913223977737431,1.495360684584794) -- cycle;
\draw [->,>=stealth,line width=1.5pt] (1.3965242753986016,0.996181247831571) -- node[below] {$s_4$}(0.4913223977737433,1.495360684584793);
\draw [->,>=stealth,line width=1.5pt] (0.4913223977737433,1.495360684584793) -- node[above] {$s_3$}(2.991322397773744,1.495360684584793);
\draw [->,>=stealth,line width=1.5pt] (2.261102397773744,2.060090684584793) -- node[above] {$s_1$}(3.907728874990835,2.060094444489496);
\draw [->,>=stealth,line width=1.5pt] (0.4913223977737433,1.495360684584793) -- node[above] {$s_2$}(2.261102397773744,2.060090684584793);
\draw [->,>=stealth,line width=1.5pt] (4.312692435425764,0.4195296942616159) -- node[below] {$s_7$}(5.191322397773745,0.5953606845847924);
\draw [->,>=stealth,line width=1.5pt] (3.240222397773744,0.4153606845847938) -- node[below] {$s_8$}(4.312692435425764,0.4195296942616159);
\draw [->,>=stealth,line width=1.5pt] (5.191322397773745,0.5953606845847924) -- node[above] {$s_6$}(2.351923875859111,0.697764230618576);
\draw [->,>=stealth,line width=1.5pt] (2.351923875859111,0.697764230618576) -- node[below] {$s_5$}(1.3965242753986016,0.996181247831571);
\begin{scriptsize}
\draw [fill=rvwvcq] (1.3965242753986016,0.996181247831571) circle (1pt)node[color=rvwvcq,below] {$V_1$};
\draw [fill=rvwvcq] (2.351923875859111,0.697764230618576) circle (1pt)node[color=rvwvcq,below] {$V_2$};
\draw [fill=rvwvcq] (2.261102397773744,2.060090684584793) circle (1pt)node[color=rvwvcq,above] {$W_1$};
\draw [fill=rvwvcq] (3.240222397773744,0.4153606845847938) circle (1pt)node[color=rvwvcq,below] {$V_3$};
\draw [fill=rvwvcq] (2.991322397773744,1.495360684584793) circle (1pt)node[color=rvwvcq,right] {$W_2$};
\draw [fill=rvwvcq] (4.312692435425764,0.4195296942616159) circle (1pt)node[color=rvwvcq,below] {$V_4$};
\draw [fill=rvwvcq] (3.907728874990835,2.060094444489496) circle (1pt)node[color=rvwvcq,above] {$W_3$};
\draw [fill=rvwvcq] (5.191322397773745,0.5953606845847924) circle (1pt)node[color=rvwvcq,below] {$V_5$};
\draw [fill=rvwvcq] (0.4913223977737433,1.495360684584793) circle (1pt)node[color=rvwvcq,left] {$V_0$};
\draw [fill=ududff] (6.407728874990836,2.060094444489496) circle (1pt);
\draw [fill=ududff] (5.502526997365978,2.559273881242718) circle (1pt);
\draw [fill=ududff] (1.707728874990834,2.960094444489497) circle (1pt);
\draw [fill=ududff] (4.761102397773745,2.060090684584793) circle (1pt);
\draw [fill=ududff] (6.316907396905469,3.422420898455713) circle (1pt);
\draw [fill=ududff] (6.407728874990836,2.060094444489496) circle (1pt);
\draw [fill=ududff] (7.272306997365979,3.124003881242718) circle (1pt);
\draw [fill=ududff] (7.272306997365979,3.124003881242718) circle (1pt);
\draw [fill=ududff] (7.023406997365979,4.204003881242717) circle (1pt);
\draw [fill=ududff] (7.911705519280612,3.921600335208935) circle (1pt);
\draw [fill=ududff] (7.506741958845683,5.562165085436815) circle (1pt);
\draw [fill=ududff] (8.579211996497703,5.566334095113637) circle (1pt);
\draw [fill=ududff] (7.911705519280612,3.921600335208935) circle (1pt);
\draw [fill=ududff] (8.148991996497703,7.0310640951136385) circle (1pt);
\draw [fill=ududff] (9.027621958845684,7.206895085436815) circle (1pt);
\draw [fill=ududff] (8.579211996497703,5.566334095113637) circle (1pt);
\draw [fill=ududff] (9.338826558437916,9.170808282094741) circle (1pt);
\draw [fill=ududff] (9.027621958845684,7.206895085436815) circle (1pt);
\draw [fill=ududff] (8.122420081220826,7.706074522190037) circle (1pt);
\draw [fill=ududff] (8.383426957977406,9.469225299307736) circle (1pt);
\draw [fill=ududff] (9.338826558437916,9.170808282094741) circle (1pt);
\draw [fill=ududff] (9.248005080352549,10.533134736060958) circle (1pt);
\draw [fill=ududff] (8.359706558437916,10.81553828209474) circle (1pt);
\draw [fill=ududff] (9.248005080352549,10.533134736060958) circle (1pt);
\draw [fill=ududff] (8.99910508035255,11.613134736060957) circle (1pt);
\draw [fill=ududff] (7.926635042700529,11.608965726384135) circle (1pt);
\draw [fill=ududff] (8.99910508035255,11.613134736060957) circle (1pt);
\draw [fill=ududff] (8.59414151991762,13.253699486288838) circle (1pt);
\draw [fill=ududff] (7.71551155756964,13.077868495965662) circle (1pt);
\draw [fill=ududff] (8.59414151991762,13.253699486288838) circle (1pt);
\draw [fill=ududff] (8.16392151991762,14.718429486288839) circle (1pt);
\draw [fill=uuuuuu] (4.327621958845682,8.106895085436816) circle (2.0pt);
\draw [fill=ududff] (7.258719642292762,15.217608923042063) circle (1pt);
\draw [fill=ududff] (6.303320041832253,15.516025940255057) circle (1pt);
\draw [fill=ududff] (6.3941415199176195,14.15369948628884) circle (1pt);
\draw [fill=ududff] (6.303320041832253,15.516025940255057) circle (1pt);
\draw [fill=ududff] (5.4150215199176195,15.798429486288839) circle (1pt);
\draw [fill=ududff] (5.663921519917619,14.71842948628884) circle (1pt);
\draw [fill=ududff] (5.4150215199176195,15.798429486288839) circle (1pt);
\draw [fill=ududff] (4.342551482265599,15.794260476612017) circle (1pt);
\draw [fill=ududff] (4.747515042700528,14.153695726384136) circle (1pt);
\draw [fill=ududff] (3.4639215199176183,15.618429486288841) circle (1pt);
\draw [fill=ududff] (2.2475150427005275,14.153695726384136) circle (1pt);
\draw [fill=ududff] (3.1527169203253855,13.654516289630916) circle (1pt);
\draw [fill=ududff] (6.947515042700529,13.253695726384136) circle (1pt);
\draw [fill=ududff] (7.258719642292762,15.217608923042063) circle (1pt);
\draw [fill=ududff] (8.16392151991762,14.71842948628884) circle (1pt);

\draw [fill=ududff] (4.342551482265599,15.794260476612017) circle (1pt);
\draw [fill=ududff] (3.4639215199176183,15.618429486288841) circle (1pt);
\draw [fill=ududff] (3.8941415199176186,14.15369948628884) circle (1pt);
\draw [fill=ududff] (2.3383365207858944,12.79136927241792) circle (1pt);
\draw [fill=ududff] (2.2475150427005275,14.153695726384136) circle (1pt);
\draw [fill=ududff] (1.3829369203253847,13.089786289630915) circle (1pt);
\draw [fill=ududff] (1.3829369203253847,13.089786289630915) circle (1pt);
\draw [fill=ududff] (1.6318369203253846,12.009786289630917) circle (1pt);
\draw [fill=ududff] (0.7435383984107515,12.292189835664697) circle (1pt);
\draw [fill=ududff] (1.1485019588456806,10.651625085436818) circle (1pt);
\draw [fill=ududff] (0.07603192119366042,10.647456075759996) circle (1pt);
\draw [fill=ududff] (0.7435383984107515,12.292189835664697) circle (1pt);
\draw [fill=ududff] (0.5062519211936607,9.182726075759994) circle (1pt);
\draw [fill=ududff] (-0.37237804115432027,9.006895085436817) circle (1pt);
\draw [fill=ududff] (0.07603192119366042,10.647456075759996) circle (1pt);
\draw [fill=ududff] (-0.6835826407465522,7.0429818887788915) circle (1pt);
\draw [fill=ududff] (-0.37237804115432027,9.006895085436817) circle (1pt);
\draw [fill=ududff] (0.5328238364705378,8.507715648683597) circle (1pt);
\draw [fill=ududff] (0.2718169597139575,6.744564871565897) circle (1pt);
\draw [fill=ududff] (-0.6835826407465522,7.0429818887788915) circle (1pt);
\draw [fill=ududff] (-0.5927611626611853,5.6806554348126745) circle (1pt);
\draw [fill=ududff] (0.2955373592534478,5.398251888778892) circle (1pt);
\draw [fill=ududff] (-0.5927611626611853,5.6806554348126745) circle (1pt);
\draw [fill=ududff] (-0.34386116266118627,4.600655434812676) circle (1pt);
\draw [fill=ududff] (0.7286088749908339,4.604824444489498) circle (1pt);
\draw [fill=ududff] (-0.34386116266118627,4.600655434812676) circle (1pt);
\draw [fill=ududff] (0.06110239777374282,2.960090684584795) circle (1pt);
\draw [fill=ududff] (0.9397323601217238,3.135921674907971) circle (1pt);
\draw [fill=ududff] (0.06110239777374282,2.960090684584795) circle (1pt);
\draw [fill=ududff] (8.16392151991762,14.71842948628884) circle (1pt);
\draw [fill=ududff] (6.947515042700529,13.253695726384136) circle (1pt);
\end{scriptsize}
\end{tikzpicture}
\caption{Four stable 30-gons of type $E_8$}
\label{fig:E8-1}
\end{figure}

\begin{figure}[ht]\centering
\definecolor{ttffcc}{rgb}{0.2,1.,0.8}
\definecolor{uuuuuu}{rgb}{0.26666666666666666,0.26666666666666666,0.26666666666666666}
\definecolor{aqaqaq}{rgb}{0.6274509803921569,0.6274509803921569,0.6274509803921569}
\definecolor{xfqqff}{rgb}{0.4980392156862745,0.,1.}
\definecolor{ududff}{rgb}{0.30196078431372547,0.30196078431372547,1.}
\definecolor{ffwwqq}{rgb}{1.,0.4,0.}
\definecolor{ttffqq}{rgb}{0.2,1.,0.}
\definecolor{qqzzff}{rgb}{0.,0.6,1.}
\definecolor{rvwvcq}{rgb}{0.08235294117647059,0.396078431372549,0.7529411764705882}
\begin{tikzpicture}[scale=1.67, scale=.9,arrow/.style={->,>=stealth,,thick}]
\clip(0.06907443200495603,-0.16257290577795844) rectangle (7.075512428024496,3.057959562207037);
\draw[help lines, line width=0.2pt,step=0.18cm] (0.06907443200495603,0) grid (7.075512428024496,3.057959562207037);
\fill[line width=2.pt,color=qqzzff,fill=qqzzff!50,fill opacity=0.6600000262260437] (1.687344961817413,0.8072909699098726) -- (3.072874961817413,0.4272909699098728) -- (2.461294961817413,2.0720209699098726) -- cycle;
\fill[line width=2.pt,color=ttffqq,fill=ttffqq!30,fill opacity=.81] (3.072874961817413,0.4272909699098728) -- (3.940414961817413,0.4272909699098728) -- (3.491514961817413,1.5072909699098729) -- cycle;
\fill[line width=2.pt,color=ffwwqq,fill=ffwwqq!30,fill opacity=0.81] (3.940414961817413,0.4272909699098728) -- (5.405984961817413,0.8072909699098726) -- (4.507921439034504,2.0720247298145757) -- cycle;
\fill[line width=2.pt,color=xfqqff,fill=xfqqff!30,fill opacity=.81] (6.68530010232606,1.2403246446116063) -- (7.507921439034504,2.0720247298145757) -- (6.312091439034504,2.772024729814576) -- cycle;
\fill[line width=2.pt,color=xfqqff,fill=xfqqff!30,fill opacity=.81] (1.314136298525857,2.3389910551128423) -- (1.687344961817413,0.8072909699098726) -- (0.49151496181741283,1.5072909699098729) -- cycle;
\fill[line width=2.pt,color=ttffcc,fill=ttffcc!60,fill opacity=.81] (5.405984961817413,0.8072909699098726) -- (6.68530010232606,1.2403246446116063) -- (5.461294961817413,2.0720209699098726) -- cycle;
\fill[line width=2.pt,color=qqzzff,fill=qqzzff!50,fill opacity=0.6600000262260437] (6.896341439034504,3.7167547298145753) -- (7.507921439034504,2.0720247298145757) -- (8.281871439034504,3.336754729814576) -- cycle;
\fill[line width=2.pt,color=ttffqq,fill=ttffqq!30,fill opacity=.81] (8.281871439034504,3.336754729814576) -- (7.8329714390345035,4.416754729814576) -- (8.700511439034504,4.416754729814576) -- cycle;
\fill[line width=2.pt,color=ffwwqq,fill=ffwwqq!30,fill opacity=0.81] (7.8024479162515945,5.681488489719279) -- (9.268017916251594,6.061488489719279) -- (8.700511439034504,4.416754729814576) -- cycle;
\fill[line width=2.pt,color=ttffcc,fill=ttffcc!60,fill opacity=.81] (8.044012775742946,6.893184815017546) -- (9.323327916251593,7.326218489719279) -- (9.268017916251594,6.061488489719279) -- cycle;
\fill[line width=2.pt,color=xfqqff,fill=xfqqff!30,fill opacity=.81] (8.950119252960036,8.857918574922248) -- (9.323327916251593,7.326218489719279) -- (8.127497916251592,8.02621848971928) -- cycle;
\fill[line width=2.pt,color=qqzzff,fill=qqzzff!50,fill opacity=0.6600000262260437] (7.564589252960036,9.237918574922249) -- (8.950119252960036,8.857918574922248) -- (8.338539252960036,10.502648574922247) -- cycle;
\fill[line width=2.pt,color=ttffqq,fill=ttffqq!30,fill opacity=.81] (7.470999252960036,10.502648574922247) -- (8.338539252960036,10.502648574922247) -- (7.889639252960036,11.582648574922247) -- cycle;
\fill[line width=2.pt,color=ffwwqq,fill=ffwwqq!30,fill opacity=0.81] (6.424069252960035,11.202648574922247) -- (7.889639252960036,11.582648574922247) -- (6.991575730177127,12.84738233482695) -- cycle;
\fill[line width=2.pt,color=ttffcc,fill=ttffcc!60,fill opacity=.81] (5.71226058966848,12.414348660125217) -- (6.991575730177127,12.84738233482695) -- (5.7675705896684795,13.679078660125217) -- cycle;
\fill[line width=2.pt,color=qqzzff,fill=qqzzff!50,fill opacity=0.6600000262260437] (4.5717405896684795,14.379078660125217) -- (3.1862105896684794,14.759078660125217) -- (3.7977905896684794,13.114348660125216) -- cycle;
\fill[line width=2.pt,color=ttffqq,fill=ttffqq!30,fill opacity=.81] (3.1862105896684794,14.759078660125217) -- (2.3186705896684794,14.759078660125217) -- (2.7675705896684795,13.679078660125217) -- cycle;
\fill[line width=2.pt,color=ffwwqq,fill=ffwwqq!30,fill opacity=0.81] (2.3186705896684794,14.759078660125217) -- (0.8531005896684789,14.379078660125217) -- (1.751164112451388,13.114344900220514) -- cycle;
\fill[line width=2.pt,color=xfqqff,fill=xfqqff!30,fill opacity=.81] (-0.426214550840168,13.946044985423484) -- (-1.248835887548612,13.114344900220514) -- (-0.05300588754861213,12.414344900220513) -- cycle;
\fill[line width=2.pt,color=xfqqff,fill=xfqqff!30,fill opacity=.81] (4.944949252960035,12.847378574922248) -- (4.5717405896684795,14.379078660125217) -- (5.7675705896684795,13.679078660125217) -- cycle;
\fill[line width=2.pt,color=ttffcc,fill=ttffcc!60,fill opacity=.81] (0.8531005896684789,14.379078660125217) -- (-0.426214550840168,13.946044985423484) -- (0.7977905896684794,13.114348660125216) -- cycle;
\fill[line width=2.pt,color=qqzzff,fill=qqzzff!50,fill opacity=0.6600000262260437] (-0.6372558875486121,11.469614900220515) -- (-1.248835887548612,13.114344900220514) -- (-2.0227858875486113,11.849614900220514) -- cycle;
\fill[line width=2.pt,color=ttffqq,fill=ttffqq!30,fill opacity=.81] (-2.0227858875486113,11.849614900220514) -- (-1.5738858875486113,10.769614900220514) -- (-2.4414258875486112,10.769614900220514) -- cycle;
\fill[line width=2.pt,color=ffwwqq,fill=ffwwqq!30,fill opacity=0.81] (-1.5433623647657022,9.50488114031581) -- (-3.008932364765702,9.12488114031581) -- (-2.4414258875486112,10.769614900220514) -- cycle;
\fill[line width=2.pt,color=ttffcc,fill=ttffcc!60,fill opacity=.81] (-1.7849272242570535,8.293184815017543) -- (-3.0642423647657004,7.860151140315811) -- (-3.008932364765702,9.12488114031581) -- cycle;
\fill[line width=2.pt,color=xfqqff,fill=xfqqff!30,fill opacity=.81] (-2.6910337014741437,6.328451055112842) -- (-3.0642423647657004,7.860151140315811) -- (-1.8684123647656996,7.16015114031581) -- cycle;
\fill[line width=2.pt,color=qqzzff,fill=qqzzff!50,fill opacity=0.6600000262260437] (-1.3055037014741435,5.948451055112841) -- (-2.6910337014741437,6.328451055112842) -- (-2.0794537014741437,4.6837210551128425) -- cycle;
\fill[line width=2.pt,color=ttffqq,fill=ttffqq!30,fill opacity=.81] (-1.2119137014741437,4.6837210551128425) -- (-2.0794537014741437,4.6837210551128425) -- (-1.6305537014741436,3.6037210551128425) -- cycle;
\fill[line width=2.pt,color=ffwwqq,fill=ffwwqq!30,fill opacity=0.81] (-0.1649837014741431,3.9837210551128432) -- (-1.6305537014741436,3.6037210551128425) -- (-0.7324901786912346,2.33898729520814) -- cycle;
\fill[line width=2.pt,color=ttffcc,fill=ttffcc!60,fill opacity=.81] (0.5468249618174124,2.772020969909873) -- (-0.7324901786912346,2.33898729520814) -- (0.49151496181741283,1.5072909699098727) -- cycle;
\draw [->,>=stealth,line width=1.5pt] (3.072874961817413,0.4272909699098728) -- node[below] {$s_8$}(3.940414961817413,0.4272909699098728);
\draw [->,>=stealth,line width=1.5pt] (2.461294961817413,2.0720209699098726) -- node[above] {$s_1$}(4.507921439034504,2.0720247298145757);
\draw [->,>=stealth,line width=1.5pt] (0.49151496181741283,1.5072909699098729) -- (3.491514961817413,1.5072909699098729);
\draw ($(0.49151496181741283,1.5072909699098729)!.7!(3.491514961817413,1.5072909699098729)$) node[above] {$s_3$};
\draw [->,>=stealth,line width=1.5pt] (6.68530010232606,1.2403246446116063) -- node[below] {$s_4$}(0.49151496181741283,1.5072909699098729);
\draw [->,>=stealth,line width=1.5pt] (0.49151496181741283,1.5072909699098729) -- node[above] {$s_2$}(2.461294961817413,2.0720209699098726);
\draw [->,>=stealth,line width=1.5pt] (3.940414961817413,0.4272909699098728) -- node[below] {$s_7$}(1.6873449618174128,0.8072909699098726);
\draw [->,>=stealth,line width=1.5pt] (1.687344961817413,0.8072909699098726) -- node[above] {$s_6$}(5.405984961817413,0.8072909699098726);
\draw [->,>=stealth,line width=1.5pt] (5.405984961817413,0.8072909699098726) -- node[below] {$s_5$}(6.68530010232606,1.2403246446116063);
\begin{scriptsize}
\draw [fill=rvwvcq] (1.687344961817413,0.8072909699098726) circle (1pt)node[color=rvwvcq,below] {$V_1$};
\draw [fill=rvwvcq] (3.072874961817413,0.4272909699098728) circle (1pt)node[color=rvwvcq,below] {$V_2$};
\draw [fill=rvwvcq] (2.461294961817413,2.0720209699098726) circle (1pt)node[color=rvwvcq,above] {$W_1$};
\draw [fill=rvwvcq] (3.940414961817413,0.4272909699098728) circle (1pt)node[color=rvwvcq,below] {$V_3$};
\draw [fill=rvwvcq] (3.491514961817413,1.5072909699098729) circle (1pt)node[color=rvwvcq,right] {$W_2$};
\draw [fill=rvwvcq] (5.405984961817413,0.8072909699098726) circle (1pt)node[color=rvwvcq,below] {$V_4$};
\draw [fill=rvwvcq] (4.507921439034504,2.0720247298145757) circle (1pt)node[color=rvwvcq,above] {$W_3$};
\draw [fill=rvwvcq] (6.68530010232606,1.2403246446116063) circle (1pt)node[color=rvwvcq,below] {$V_5$};
\draw [fill=rvwvcq] (0.49151496181741283,1.5072909699098729) circle (1pt)node[color=rvwvcq,left] {$V_0$};
\draw [fill=ududff] (7.507921439034504,2.0720247298145757) circle (1pt);
\draw [fill=ududff] (6.312091439034504,2.772024729814576) circle (1pt);
\draw [fill=ududff] (1.314136298525857,2.3389910551128423) circle (1pt);
\draw [fill=ududff] (5.461294961817413,2.0720209699098726) circle (1pt);
\draw [fill=ududff] (6.896341439034504,3.7167547298145753) circle (1pt);
\draw [fill=ududff] (7.507921439034504,2.0720247298145757) circle (1pt);
\draw [fill=ududff] (8.281871439034504,3.336754729814576) circle (1pt);
\draw [fill=ududff] (8.281871439034504,3.336754729814576) circle (1pt);
\draw [fill=ududff] (7.8329714390345035,4.416754729814576) circle (1pt);
\draw [fill=ududff] (8.700511439034504,4.416754729814576) circle (1pt);
\draw [fill=ududff] (7.8024479162515945,5.681488489719279) circle (1pt);
\draw [fill=ududff] (9.268017916251594,6.061488489719279) circle (1pt);
\draw [fill=ududff] (8.700511439034504,4.416754729814576) circle (1pt);
\draw [fill=ududff] (8.044012775742946,6.893184815017546) circle (1pt);
\draw [fill=ududff] (9.323327916251593,7.326218489719279) circle (1pt);
\draw [fill=ududff] (9.268017916251594,6.061488489719279) circle (1pt);
\draw [fill=ududff] (8.950119252960036,8.857918574922248) circle (1pt);
\draw [fill=ududff] (9.323327916251593,7.326218489719279) circle (1pt);
\draw [fill=ududff] (8.127497916251592,8.02621848971928) circle (1pt);
\draw [fill=ududff] (7.564589252960036,9.237918574922249) circle (1pt);
\draw [fill=ududff] (8.950119252960036,8.857918574922248) circle (1pt);
\draw [fill=ududff] (8.338539252960036,10.502648574922247) circle (1pt);
\draw [fill=ududff] (7.470999252960036,10.502648574922247) circle (1pt);
\draw [fill=ududff] (8.338539252960036,10.502648574922247) circle (1pt);
\draw [fill=ududff] (7.889639252960036,11.582648574922247) circle (1pt);
\draw [fill=ududff] (6.424069252960035,11.202648574922247) circle (1pt);
\draw [fill=ududff] (7.889639252960036,11.582648574922247) circle (1pt);
\draw [fill=ududff] (6.991575730177127,12.84738233482695) circle (1pt);
\draw [fill=ududff] (5.71226058966848,12.414348660125217) circle (1pt);
\draw [fill=ududff] (6.991575730177127,12.84738233482695) circle (1pt);
\draw [fill=ududff] (5.7675705896684795,13.679078660125217) circle (1pt);
\draw [fill=uuuuuu] (3.129542775742946,7.593184815017545) circle (2.0pt);
\draw [fill=ududff] (4.5717405896684795,14.379078660125217) circle (1pt);
\draw [fill=ududff] (3.1862105896684794,14.759078660125217) circle (1pt);
\draw [fill=ududff] (3.7977905896684794,13.114348660125216) circle (1pt);
\draw [fill=ududff] (3.1862105896684794,14.759078660125217) circle (1pt);
\draw [fill=ududff] (2.3186705896684794,14.759078660125217) circle (1pt);
\draw [fill=ududff] (2.7675705896684795,13.679078660125217) circle (1pt);
\draw [fill=ududff] (2.3186705896684794,14.759078660125217) circle (1pt);
\draw [fill=ududff] (0.8531005896684789,14.379078660125217) circle (1pt);
\draw [fill=ududff] (1.751164112451388,13.114344900220514) circle (1pt);
\draw [fill=ududff] (-0.426214550840168,13.946044985423484) circle (1pt);
\draw [fill=ududff] (-1.248835887548612,13.114344900220514) circle (1pt);
\draw [fill=ududff] (-0.05300588754861213,12.414344900220513) circle (1pt);
\draw [fill=ududff] (4.944949252960035,12.847378574922248) circle (1pt);
\draw [fill=ududff] (4.5717405896684795,14.379078660125217) circle (1pt);
\draw [fill=ududff] (5.7675705896684795,13.679078660125217) circle (1pt);
\draw [fill=ududff] (0.8531005896684789,14.379078660125217) circle (1pt);
\draw [fill=ududff] (-0.426214550840168,13.946044985423484) circle (1pt);
\draw [fill=ududff] (0.7977905896684794,13.114348660125216) circle (1pt);
\draw [fill=ududff] (-0.6372558875486121,11.469614900220515) circle (1pt);
\draw [fill=ududff] (-1.248835887548612,13.114344900220514) circle (1pt);
\draw [fill=ududff] (-2.0227858875486113,11.849614900220514) circle (1pt);
\draw [fill=ududff] (-2.0227858875486113,11.849614900220514) circle (1pt);
\draw [fill=ududff] (-1.5738858875486113,10.769614900220514) circle (1pt);
\draw [fill=ududff] (-2.4414258875486112,10.769614900220514) circle (1pt);
\draw [fill=ududff] (-1.5433623647657022,9.50488114031581) circle (1pt);
\draw [fill=ududff] (-3.008932364765702,9.12488114031581) circle (1pt);
\draw [fill=ududff] (-2.4414258875486112,10.769614900220514) circle (1pt);
\draw [fill=ududff] (-1.7849272242570535,8.293184815017543) circle (1pt);
\draw [fill=ududff] (-3.0642423647657004,7.860151140315811) circle (1pt);
\draw [fill=ududff] (-3.008932364765702,9.12488114031581) circle (1pt);
\draw [fill=ududff] (-2.6910337014741437,6.328451055112842) circle (1pt);
\draw [fill=ududff] (-3.0642423647657004,7.860151140315811) circle (1pt);
\draw [fill=ududff] (-1.8684123647656996,7.16015114031581) circle (1pt);
\draw [fill=ududff] (-1.3055037014741435,5.948451055112841) circle (1pt);
\draw [fill=ududff] (-2.6910337014741437,6.328451055112842) circle (1pt);
\draw [fill=ududff] (-2.0794537014741437,4.6837210551128425) circle (1pt);
\draw [fill=ududff] (-1.2119137014741437,4.6837210551128425) circle (1pt);
\draw [fill=ududff] (-2.0794537014741437,4.6837210551128425) circle (1pt);
\draw [fill=ududff] (-1.6305537014741436,3.6037210551128425) circle (1pt);
\draw [fill=ududff] (-0.1649837014741431,3.9837210551128432) circle (1pt);
\draw [fill=ududff] (-1.6305537014741436,3.6037210551128425) circle (1pt);
\draw [fill=ududff] (-0.7324901786912346,2.33898729520814) circle (1pt);
\draw [fill=ududff] (0.5468249618174124,2.772020969909873) circle (1pt);
\draw [fill=ududff] (-0.7324901786912346,2.33898729520814) circle (1pt);
\draw [fill=ududff] (5.7675705896684795,13.679078660125217) circle (1pt);
\draw [fill=ududff] (4.944949252960035,12.847378574922248) circle (1pt);
\end{scriptsize}
\end{tikzpicture}
\definecolor{ttffcc}{rgb}{0.2,1.,0.8}
\definecolor{uuuuuu}{rgb}{0.26666666666666666,0.26666666666666666,0.26666666666666666}
\definecolor{aqaqaq}{rgb}{0.6274509803921569,0.6274509803921569,0.6274509803921569}
\definecolor{xfqqff}{rgb}{0.4980392156862745,0.,1.}
\definecolor{ududff}{rgb}{0.30196078431372547,0.30196078431372547,1.}
\definecolor{ffwwqq}{rgb}{1.,0.4,0.}
\definecolor{ttffqq}{rgb}{0.2,1.,0.}
\definecolor{qqzzff}{rgb}{0.,0.6,1.}
\definecolor{rvwvcq}{rgb}{0.08235294117647059,0.396078431372549,0.7529411764705882}
\begin{tikzpicture}[scale=1.7, scale=.9,arrow/.style={->,>=stealth,,thick}]
\clip(0.1925760469266172,-0.15847710284996636) rectangle (7.107441726166619,3.1022252967431663);
\draw[help lines, line width=0.2pt,step=0.17cm] (0.1925760469266172,0) grid (7.107441726166619,3.1022252967431663);
\fill[line width=2.pt,color=qqzzff,fill=qqzzff!50,fill opacity=0.6600000262260437] (1.5035475233911475,0.8822533857095992) -- (3.0875561158567764,0.4034932259945089) -- (2.4759761158567763,2.048223225994509) -- cycle;
\fill[line width=2.pt,color=ttffqq,fill=ttffqq!30,fill opacity=.81] (3.0875561158567764,0.4034932259945089) -- (3.9550961158567763,0.4034932259945089) -- (3.5061961158567763,1.4834932259945088) -- cycle;
\fill[line width=2.pt,color=ffwwqq,fill=ffwwqq!30,fill opacity=0.81] (3.9550961158567763,0.4034932259945089) -- (5.545218204593364,0.6695338761726454) -- (4.522602593073869,2.048226985899212) -- cycle;
\fill[line width=2.pt,color=xfqqff,fill=xfqqff!30,fill opacity=.81] (6.699981256365425,1.2165269006962423) -- (7.522602593073872,2.048226985899212) -- (6.525251185539498,2.649466826184122) -- cycle;
\fill[line width=2.pt,color=xfqqff,fill=xfqqff!30,fill opacity=.81] (1.3288174525652208,2.3151933111974787) -- (1.5035475233911475,0.8822533857095992) -- (0.5061961158567739,1.4834932259945088) -- cycle;
\fill[line width=2.pt,color=ttffcc,fill=ttffcc!60,fill opacity=.81] (5.545218204593364,0.6695338761726454) -- (6.699981256365425,1.2165269006962423) -- (5.475976115856779,2.048223225994509) -- cycle;
\fill[line width=2.pt,color=qqzzff,fill=qqzzff!50,fill opacity=0.6600000262260437] (6.911022593073872,3.692956985899212) -- (7.522602593073872,2.048226985899212) -- (8.495031185539501,3.214196826184122) -- cycle;
\fill[line width=2.pt,color=ttffqq,fill=ttffqq!30,fill opacity=.81] (8.495031185539501,3.214196826184122) -- (8.046131185539501,4.2941968261841215) -- (8.913671185539501,4.2941968261841215) -- cycle;
\fill[line width=2.pt,color=ffwwqq,fill=ffwwqq!30,fill opacity=0.81] (7.891055574020006,5.672889935910688) -- (9.481177662756593,5.938930586088825) -- (8.913671185539501,4.2941968261841215) -- cycle;
\fill[line width=2.pt,color=ttffcc,fill=ttffcc!60,fill opacity=.81] (8.257172522247947,6.770626911387091) -- (9.411935574020008,7.317619935910688) -- (9.481177662756593,5.938930586088825) -- cycle;
\fill[line width=2.pt,color=xfqqff,fill=xfqqff!30,fill opacity=.81] (9.237205503194081,8.750559861398568) -- (9.411935574020008,7.317619935910688) -- (8.414584166485634,7.918859776195598) -- cycle;
\fill[line width=2.pt,color=qqzzff,fill=qqzzff!50,fill opacity=0.6600000262260437] (7.653196910728453,9.229320021113658) -- (9.237205503194081,8.750559861398568) -- (8.625625503194081,10.395289861398567) -- cycle;
\fill[line width=2.pt,color=ttffqq,fill=ttffqq!30,fill opacity=.81] (7.758085503194081,10.395289861398567) -- (8.625625503194081,10.395289861398567) -- (8.176725503194081,11.475289861398567) -- cycle;
\fill[line width=2.pt,color=ffwwqq,fill=ffwwqq!30,fill opacity=0.81] (6.586603414457494,11.209249211220431) -- (8.176725503194081,11.475289861398567) -- (7.154109891674587,12.853982971125134) -- cycle;
\fill[line width=2.pt,color=ttffcc,fill=ttffcc!60,fill opacity=.81] (5.999346839902525,12.306989946601536) -- (7.154109891674587,12.853982971125134) -- (5.93010475116594,13.6856792964234) -- cycle;
\fill[line width=2.pt,color=qqzzff,fill=qqzzff!50,fill opacity=0.6600000262260437] (4.932753343631567,14.28691913670831) -- (3.348744751165938,14.7656792964234) -- (3.960324751165938,13.120949296423401) -- cycle;
\fill[line width=2.pt,color=ttffqq,fill=ttffqq!30,fill opacity=.81] (3.348744751165938,14.7656792964234) -- (2.481204751165938,14.7656792964234) -- (2.930104751165938,13.6856792964234) -- cycle;
\fill[line width=2.pt,color=ffwwqq,fill=ffwwqq!30,fill opacity=0.81] (2.481204751165938,14.7656792964234) -- (0.8910826624293504,14.499638646245264) -- (1.913698273948845,13.120945536518697) -- cycle;
\fill[line width=2.pt,color=xfqqff,fill=xfqqff!30,fill opacity=.81] (-0.2636803893427109,13.952645621721667) -- (-1.0863017260511576,13.120945536518697) -- (-0.08895031851678414,12.519705696233787) -- cycle;
\fill[line width=2.pt,color=xfqqff,fill=xfqqff!30,fill opacity=.81] (5.107483414457493,12.85397921122043) -- (4.932753343631567,14.28691913670831) -- (5.93010475116594,13.6856792964234) -- cycle;
\fill[line width=2.pt,color=ttffcc,fill=ttffcc!60,fill opacity=.81] (0.8910826624293504,14.499638646245264) -- (-0.2636803893427109,13.952645621721667) -- (0.9603247511659356,13.120949296423401) -- cycle;
\fill[line width=2.pt,color=qqzzff,fill=qqzzff!50,fill opacity=0.6600000262260437] (-0.4747217260511576,11.476215536518698) -- (-1.0863017260511576,13.120945536518697) -- (-2.058730318516787,11.954975696233788) -- cycle;
\fill[line width=2.pt,color=ttffqq,fill=ttffqq!30,fill opacity=.81] (-2.058730318516787,11.954975696233788) -- (-1.6098303185167868,10.874975696233788) -- (-2.477370318516787,10.874975696233788) -- cycle;
\fill[line width=2.pt,color=ffwwqq,fill=ffwwqq!30,fill opacity=0.81] (-1.4547547069972921,9.49628258650722) -- (-3.044876795733879,9.230241936329085) -- (-2.477370318516787,10.874975696233788) -- cycle;
\fill[line width=2.pt,color=ttffcc,fill=ttffcc!60,fill opacity=.81] (-1.8208716552252326,8.39854561103082) -- (-2.975634706997294,7.851552586507221) -- (-3.044876795733879,9.230241936329085) -- cycle;
\fill[line width=2.pt,color=xfqqff,fill=xfqqff!30,fill opacity=.81] (-2.800904636171367,6.418612661019342) -- (-2.975634706997294,7.851552586507221) -- (-1.9782832994629196,7.250312746222312) -- cycle;
\fill[line width=2.pt,color=qqzzff,fill=qqzzff!50,fill opacity=0.6600000262260437] (-1.2168960437057388,5.939852501304252) -- (-2.800904636171367,6.418612661019342) -- (-2.189324636171367,4.773882661019343) -- cycle;
\fill[line width=2.pt,color=ttffqq,fill=ttffqq!30,fill opacity=.81] (-1.3217846361713672,4.773882661019343) -- (-2.189324636171367,4.773882661019343) -- (-1.7404246361713671,3.6938826610193427) -- cycle;
\fill[line width=2.pt,color=ffwwqq,fill=ffwwqq!30,fill opacity=0.81] (-0.1503025474347801,3.9599233111974783) -- (-1.7404246361713671,3.6938826610193427) -- (-0.7178090246518725,2.315189551292775) -- cycle;
\fill[line width=2.pt,color=ttffcc,fill=ttffcc!60,fill opacity=.81] (0.4369540271201888,2.862182575816373) -- (-0.7178090246518725,2.315189551292775) -- (0.506196115856774,1.4834932259945095) -- cycle;
\draw [line width=2.pt,color=ffwwqq] (3.9550961158567763,0.4034932259945089)-- (5.545218204593364,0.6695338761726454);
\draw [line width=2.pt,color=ttffqq] (8.046131185539501,4.2941968261841215)-- (8.913671185539501,4.2941968261841215);
\draw [line width=2.pt,color=ffwwqq] (7.891055574020006,5.672889935910688)-- (9.481177662756593,5.938930586088825);
\draw [line width=2.pt,color=ttffcc] (8.257172522247947,6.770626911387091)-- (9.411935574020008,7.317619935910688);
\draw [line width=2.pt,color=xfqqff] (9.237205503194081,8.750559861398568)-- (9.411935574020008,7.317619935910688);
\draw [line width=2.pt,color=xfqqff] (9.411935574020008,7.317619935910688)-- (8.414584166485634,7.918859776195598);
\draw [line width=2.pt,color=ffwwqq] (6.586603414457494,11.209249211220431)-- (8.176725503194081,11.475289861398567);
\draw [line width=2.pt,color=ffwwqq] (7.154109891674587,12.853982971125134)-- (6.586603414457494,11.209249211220431);
\draw [line width=2.pt,color=ttffcc] (5.999346839902525,12.306989946601536)-- (7.154109891674587,12.853982971125134);
\draw [line width=2.pt,color=ttffcc] (7.154109891674587,12.853982971125134)-- (5.93010475116594,13.6856792964234);
\draw [line width=2.pt,color=ttffcc] (5.93010475116594,13.6856792964234)-- (5.999346839902525,12.306989946601536);
\draw [line width=2.pt,color=ttffqq] (2.930104751165938,13.6856792964234)-- (3.348744751165938,14.7656792964234);
\draw [line width=2.pt,color=ffwwqq] (2.481204751165938,14.7656792964234)-- (0.8910826624293504,14.499638646245264);
\draw [line width=2.pt,color=ffwwqq] (0.8910826624293504,14.499638646245264)-- (1.913698273948845,13.120945536518697);
\draw [line width=2.pt,color=ffwwqq] (1.913698273948845,13.120945536518697)-- (2.481204751165938,14.7656792964234);
\draw [line width=2.pt,color=xfqqff] (-0.2636803893427109,13.952645621721667)-- (-1.0863017260511576,13.120945536518697);
\draw [line width=2.pt,color=xfqqff] (-1.0863017260511576,13.120945536518697)-- (-0.08895031851678414,12.519705696233787);
\draw [line width=2.pt,color=xfqqff] (-0.08895031851678414,12.519705696233787)-- (-0.2636803893427109,13.952645621721667);
\draw [line width=2.pt,color=xfqqff] (5.107483414457493,12.85397921122043)-- (4.932753343631567,14.28691913670831);
\draw [line width=2.pt,color=xfqqff] (4.932753343631567,14.28691913670831)-- (5.93010475116594,13.6856792964234);
\draw [line width=2.pt,color=ttffcc] (0.8910826624293504,14.499638646245264)-- (-0.2636803893427109,13.952645621721667);
\draw [line width=2.pt,color=ttffcc] (-0.2636803893427109,13.952645621721667)-- (0.9603247511659356,13.120949296423401);
\draw [line width=2.pt,color=ttffcc] (0.9603247511659356,13.120949296423401)-- (0.8910826624293504,14.499638646245264);
\draw [line width=2.pt,color=qqzzff] (-1.0863017260511576,13.120945536518697)-- (-2.058730318516787,11.954975696233788);
\draw [line width=2.pt,color=ttffqq] (-1.6098303185167868,10.874975696233788)-- (-2.477370318516787,10.874975696233788);
\draw [line width=2.pt,color=ttffqq] (-2.477370318516787,10.874975696233788)-- (-2.058730318516787,11.954975696233788);
\draw [line width=2.pt,color=ffwwqq] (-1.4547547069972921,9.49628258650722)-- (-3.044876795733879,9.230241936329085);
\draw [line width=2.pt,color=ffwwqq] (-3.044876795733879,9.230241936329085)-- (-2.477370318516787,10.874975696233788);
\draw [line width=2.pt,color=ffwwqq] (-2.477370318516787,10.874975696233788)-- (-1.4547547069972921,9.49628258650722);
\draw [line width=2.pt,color=ttffcc] (-1.8208716552252326,8.39854561103082)-- (-2.975634706997294,7.851552586507221);
\draw [line width=2.pt,color=xfqqff] (-2.800904636171367,6.418612661019342)-- (-2.975634706997294,7.851552586507221);
\draw [line width=2.pt,color=xfqqff] (-2.975634706997294,7.851552586507221)-- (-1.9782832994629196,7.250312746222312);
\draw [line width=2.pt,color=xfqqff] (-1.9782832994629196,7.250312746222312)-- (-2.800904636171367,6.418612661019342);
\draw [line width=2.pt,color=qqzzff] (-1.2168960437057388,5.939852501304252)-- (-2.800904636171367,6.418612661019342);
\draw [line width=2.pt,color=qqzzff] (-2.800904636171367,6.418612661019342)-- (-2.189324636171367,4.773882661019343);
\draw [line width=2.pt,color=qqzzff] (-2.189324636171367,4.773882661019343)-- (-1.2168960437057388,5.939852501304252);
\draw [line width=2.pt,color=ttffqq] (-2.189324636171367,4.773882661019343)-- (-1.7404246361713671,3.6938826610193427);
\draw [line width=2.pt,color=ttffqq] (-1.7404246361713671,3.6938826610193427)-- (-1.3217846361713672,4.773882661019343);
\draw [line width=2.pt,color=ffwwqq] (-0.1503025474347801,3.9599233111974783)-- (-1.7404246361713671,3.6938826610193427);
\draw [line width=2.pt,color=ffwwqq] (-1.7404246361713671,3.6938826610193427)-- (-0.7178090246518725,2.315189551292775);
\draw [line width=2.pt,color=ffwwqq] (-0.7178090246518725,2.315189551292775)-- (-0.1503025474347801,3.9599233111974783);
\draw [->,>=stealth,line width=1.5pt] (3.0875561158567764,0.4034932259945089) -- node[below] {$s_8$}(3.9550961158567763,0.4034932259945089);
\draw [->,>=stealth,line width=1.5pt] (2.4759761158567763,2.048223225994509) -- node[above] {$s_1$}(4.522602593073869,2.048226985899212);
\draw [->,>=stealth,line width=1.5pt] (0.5061961158567739,1.4834932259945088) -- (3.5061961158567763,1.4834932259945088);
\draw ($(0.5061961158567739,1.4834932259945088)!.7!(3.5061961158567763,1.4834932259945088)$) node[above] {$s_3$};
\draw [->,>=stealth,line width=1.5pt] (3.9550961158567763,0.4034932259945089) -- node[below] {$s_7$}(5.545218204593364,0.6695338761726454);
\draw [->,>=stealth,line width=1.5pt] (5.545218204593364,0.6695338761726454) -- node[below] {$s_6$}(1.5035475233911475,0.8822533857095992);
\draw [->,>=stealth,line width=1.5pt] (1.5035475233911475,0.8822533857095992) -- (6.699981256365425,1.2165269006962423);
\draw ($(1.5035475233911475,0.8822533857095992)!.3!(6.699981256365425,1.2165269006962423)$) node[above] {$s_5$};
\draw [->,>=stealth,line width=1.5pt] (6.699981256365425,1.2165269006962423) -- (0.506196115856774,1.4834932259945088);
\draw ($(6.699981256365425,1.2165269006962423)!.33!(0.506196115856774,1.4834932259945088)$) node[above] {$s_4$};
\draw [->,>=stealth,line width=1.5pt] (0.5061961158567739,1.4834932259945088) -- node[above] {$s_2$}(2.4759761158567763,2.048223225994509);
\begin{scriptsize}
\draw [fill=rvwvcq] (1.5035475233911475,0.8822533857095992) circle (1pt)node[color=rvwvcq,below] {$V_1$};
\draw [fill=rvwvcq] (3.0875561158567764,0.4034932259945089) circle (1pt)node[color=rvwvcq,below] {$V_2$};
\draw [fill=rvwvcq] (2.4759761158567763,2.048223225994509) circle (1pt)node[color=rvwvcq,above] {$W_1$};
\draw [fill=rvwvcq] (3.9550961158567763,0.4034932259945089) circle (1pt)node[color=rvwvcq,below] {$V_3$};
\draw [fill=rvwvcq] (3.5061961158567763,1.4834932259945088) circle (1pt)node[color=rvwvcq,right] {$W_2$};
\draw [fill=rvwvcq] (5.545218204593364,0.6695338761726454) circle (1pt)node[color=rvwvcq,below] {$V_4$};
\draw [fill=rvwvcq] (4.522602593073869,2.048226985899212) circle (1pt)node[color=rvwvcq,above] {$W_3$};
\draw [fill=rvwvcq] (6.699981256365425,1.2165269006962423) circle (1pt)node[color=rvwvcq,below] {$V_5$};
\draw [fill=rvwvcq] (0.5061961158567739,1.4834932259945088) circle (1pt)node[color=rvwvcq,left] {$V_0$};
\draw [fill=ududff] (7.522602593073872,2.048226985899212) circle (1pt);
\draw [fill=ududff] (6.525251185539498,2.649466826184122) circle (1pt);
\draw [fill=ududff] (1.3288174525652208,2.3151933111974787) circle (1pt);
\draw [fill=ududff] (5.475976115856779,2.048223225994509) circle (1pt);
\draw [fill=ududff] (6.911022593073872,3.692956985899212) circle (1pt);
\draw [fill=ududff] (7.522602593073872,2.048226985899212) circle (1pt);
\draw [fill=ududff] (8.495031185539501,3.214196826184122) circle (1pt);
\draw [fill=ududff] (8.495031185539501,3.214196826184122) circle (1pt);
\draw [fill=ududff] (8.046131185539501,4.2941968261841215) circle (1pt);
\draw [fill=ududff] (8.913671185539501,4.2941968261841215) circle (1pt);
\draw [fill=ududff] (7.891055574020006,5.672889935910688) circle (1pt);
\draw [fill=ududff] (9.481177662756593,5.938930586088825) circle (1pt);
\draw [fill=ududff] (8.913671185539501,4.2941968261841215) circle (1pt);
\draw [fill=ududff] (8.257172522247947,6.770626911387091) circle (1pt);
\draw [fill=ududff] (9.411935574020008,7.317619935910688) circle (1pt);
\draw [fill=ududff] (9.481177662756593,5.938930586088825) circle (1pt);
\draw [fill=ududff] (9.237205503194081,8.750559861398568) circle (1pt);
\draw [fill=ududff] (9.411935574020008,7.317619935910688) circle (1pt);
\draw [fill=ududff] (8.414584166485634,7.918859776195598) circle (1pt);
\draw [fill=ududff] (7.653196910728453,9.229320021113658) circle (1pt);
\draw [fill=ududff] (9.237205503194081,8.750559861398568) circle (1pt);
\draw [fill=ududff] (8.625625503194081,10.395289861398567) circle (1pt);
\draw [fill=ududff] (7.758085503194081,10.395289861398567) circle (1pt);
\draw [fill=ududff] (8.625625503194081,10.395289861398567) circle (1pt);
\draw [fill=ududff] (8.176725503194081,11.475289861398567) circle (1pt);
\draw [fill=ududff] (6.586603414457494,11.209249211220431) circle (1pt);
\draw [fill=ududff] (8.176725503194081,11.475289861398567) circle (1pt);
\draw [fill=ududff] (7.154109891674587,12.853982971125134) circle (1pt);
\draw [fill=ududff] (5.999346839902525,12.306989946601536) circle (1pt);
\draw [fill=ududff] (7.154109891674587,12.853982971125134) circle (1pt);
\draw [fill=ududff] (5.93010475116594,13.6856792964234) circle (1pt);
\draw [fill=uuuuuu] (3.218150433511357,7.584586261208955) circle (2.0pt);
\draw [fill=ududff] (4.932753343631567,14.28691913670831) circle (1pt);
\draw [fill=ududff] (3.348744751165938,14.7656792964234) circle (1pt);
\draw [fill=ududff] (3.960324751165938,13.120949296423401) circle (1pt);
\draw [fill=ududff] (3.348744751165938,14.7656792964234) circle (1pt);
\draw [fill=ududff] (2.481204751165938,14.7656792964234) circle (1pt);
\draw [fill=ududff] (2.930104751165938,13.6856792964234) circle (1pt);
\draw [fill=ududff] (2.481204751165938,14.7656792964234) circle (1pt);
\draw [fill=ududff] (0.8910826624293504,14.499638646245264) circle (1pt);
\draw [fill=ududff] (1.913698273948845,13.120945536518697) circle (1pt);
\draw [fill=ududff] (-0.2636803893427109,13.952645621721667) circle (1pt);
\draw [fill=ududff] (-1.0863017260511576,13.120945536518697) circle (1pt);
\draw [fill=ududff] (-0.08895031851678414,12.519705696233787) circle (1pt);
\draw [fill=ududff] (5.107483414457493,12.85397921122043) circle (1pt);
\draw [fill=ududff] (4.932753343631567,14.28691913670831) circle (1pt);
\draw [fill=ududff] (5.93010475116594,13.6856792964234) circle (1pt);
\draw [fill=ududff] (0.8910826624293504,14.499638646245264) circle (1pt);
\draw [fill=ududff] (-0.2636803893427109,13.952645621721667) circle (1pt);
\draw [fill=ududff] (0.9603247511659356,13.120949296423401) circle (1pt);
\draw [fill=ududff] (-0.4747217260511576,11.476215536518698) circle (1pt);
\draw [fill=ududff] (-1.0863017260511576,13.120945536518697) circle (1pt);
\draw [fill=ududff] (-2.058730318516787,11.954975696233788) circle (1pt);
\draw [fill=ududff] (-2.058730318516787,11.954975696233788) circle (1pt);
\draw [fill=ududff] (-1.6098303185167868,10.874975696233788) circle (1pt);
\draw [fill=ududff] (-2.477370318516787,10.874975696233788) circle (1pt);
\draw [fill=ududff] (-1.4547547069972921,9.49628258650722) circle (1pt);
\draw [fill=ududff] (-3.044876795733879,9.230241936329085) circle (1pt);
\draw [fill=ududff] (-2.477370318516787,10.874975696233788) circle (1pt);
\draw [fill=ududff] (-1.8208716552252326,8.39854561103082) circle (1pt);
\draw [fill=ududff] (-2.975634706997294,7.851552586507221) circle (1pt);
\draw [fill=ududff] (-3.044876795733879,9.230241936329085) circle (1pt);
\draw [fill=ududff] (-2.800904636171367,6.418612661019342) circle (1pt);
\draw [fill=ududff] (-2.975634706997294,7.851552586507221) circle (1pt);
\draw [fill=ududff] (-1.9782832994629196,7.250312746222312) circle (1pt);
\draw [fill=ududff] (-1.2168960437057388,5.939852501304252) circle (1pt);
\draw [fill=ududff] (-2.800904636171367,6.418612661019342) circle (1pt);
\draw [fill=ududff] (-2.189324636171367,4.773882661019343) circle (1pt);
\draw [fill=ududff] (-1.3217846361713672,4.773882661019343) circle (1pt);
\draw [fill=ududff] (-2.189324636171367,4.773882661019343) circle (1pt);
\draw [fill=ududff] (-1.7404246361713671,3.6938826610193427) circle (1pt);
\draw [fill=ududff] (-0.1503025474347801,3.9599233111974783) circle (1pt);
\draw [fill=ududff] (-1.7404246361713671,3.6938826610193427) circle (1pt);
\draw [fill=ududff] (-0.7178090246518725,2.315189551292775) circle (1pt);
\draw [fill=ududff] (0.4369540271201888,2.862182575816373) circle (1pt);
\draw [fill=ududff] (-0.7178090246518725,2.315189551292775) circle (1pt);
\draw [fill=ududff] (5.93010475116594,13.6856792964234) circle (1pt);
\draw [fill=ududff] (5.107483414457493,12.85397921122043) circle (1pt);
\end{scriptsize}
\end{tikzpicture}
\definecolor{ttffcc}{rgb}{0.2,1.,0.8}
\definecolor{uuuuuu}{rgb}{0.26666666666666666,0.26666666666666666,0.26666666666666666}
\definecolor{aqaqaq}{rgb}{0.6274509803921569,0.6274509803921569,0.6274509803921569}
\definecolor{xfqqff}{rgb}{0.4980392156862745,0.,1.}
\definecolor{ududff}{rgb}{0.30196078431372547,0.30196078431372547,1.}
\definecolor{ffwwqq}{rgb}{1.,0.4,0.}
\definecolor{ttffqq}{rgb}{0.2,1.,0.}
\definecolor{qqzzff}{rgb}{0.,0.6,1.}
\definecolor{rvwvcq}{rgb}{0.08235294117647059,0.396078431372549,0.7529411764705882}
\begin{tikzpicture}[scale=1.68, scale=.9, xscale=1,arrow/.style={->,>=stealth,,thick}]
\clip(0,-0.1233017230992943) rectangle (7.012561070485398,3);
\draw[help lines, line width=0.2pt,step=0.18cm] (0,0) grid (7.012561070485398,3);
\fill[line width=2.pt,color=qqzzff,fill=qqzzff!50,fill opacity=0.6600000262260437] (1.9081716794437384,0.8894085563132137) -- (3.084329801818881,0.4085879930664357) -- (2.472749801818881,2.053317993066436) -- cycle;
\fill[line width=2.pt,color=ttffqq,fill=ttffqq!30,fill opacity=.81] (3.084329801818881,0.4085879930664357) -- (3.951869801818881,0.4085879930664357) -- (3.502969801818881,1.4885879930664354) -- cycle;
\fill[line width=2.pt,color=ffwwqq,fill=ffwwqq!30,fill opacity=0.81] (3.951869801818881,0.4085879930664357) -- (5.417439801818881,0.48858799306643574) -- (4.519376279035972,2.053321752971139) -- cycle;
\fill[line width=2.pt,color=xfqqff,fill=xfqqff!30,fill opacity=.81] (6.654784985375114,0.7309150809719034) -- (7.519376279035972,2.053321752971139) -- (6.114174401411114,2.6525011897243607) -- cycle;
\fill[line width=2.pt,color=xfqqff,fill=xfqqff!30,fill opacity=.81] (1.3675610954797381,2.810994665065671) -- (1.9081716794437384,0.8894085563132137) -- (0.5029698018188804,1.4885879930664354) -- cycle;
\fill[line width=2.pt,color=ttffcc,fill=ttffcc!60,fill opacity=.81] (5.417439801818881,0.48858799306643574) -- (6.654784985375114,0.7309150809719034) -- (5.472749801818882,2.053317993066436) -- cycle;
\fill[line width=2.pt,color=qqzzff,fill=qqzzff!50,fill opacity=0.6600000262260437] (6.907796279035972,3.6980517529711396) -- (7.519376279035972,2.053321752971139) -- (8.083954401411116,3.2172311897243615) -- cycle;
\fill[line width=2.pt,color=ttffqq,fill=ttffqq!30,fill opacity=.81] (8.083954401411116,3.2172311897243615) -- (7.635054401411116,4.297231189724362) -- (8.502594401411116,4.297231189724362) -- cycle;
\fill[line width=2.pt,color=ffwwqq,fill=ffwwqq!30,fill opacity=0.81] (7.604530878628207,5.861964949629066) -- (9.070100878628207,5.941964949629066) -- (8.502594401411116,4.297231189724362) -- cycle;
\fill[line width=2.pt,color=ttffcc,fill=ttffcc!60,fill opacity=.81] (7.888065695071974,7.264367861723598) -- (9.125410878628207,7.506694949629066) -- (9.070100878628207,5.941964949629066) -- cycle;
\fill[line width=2.pt,color=xfqqff,fill=xfqqff!30,fill opacity=.81] (8.584800294664207,9.428281058381522) -- (9.125410878628207,7.506694949629066) -- (7.720209001003349,8.105874386382288) -- cycle;
\fill[line width=2.pt,color=qqzzff,fill=qqzzff!50,fill opacity=0.6600000262260437] (7.408642172289065,9.9091016216283) -- (8.584800294664207,9.428281058381522) -- (7.973220294664207,11.073011058381523) -- cycle;
\fill[line width=2.pt,color=ttffqq,fill=ttffqq!30,fill opacity=.81] (7.105680294664207,11.073011058381523) -- (7.973220294664207,11.073011058381523) -- (7.524320294664207,12.153011058381523) -- cycle;
\fill[line width=2.pt,color=ffwwqq,fill=ffwwqq!30,fill opacity=0.81] (6.058750294664207,12.073011058381523) -- (7.524320294664207,12.153011058381523) -- (6.626256771881298,13.717744818286228) -- cycle;
\fill[line width=2.pt,color=ttffcc,fill=ttffcc!60,fill opacity=.81] (5.388911588325065,13.475417730380759) -- (6.626256771881298,13.717744818286228) -- (5.444221588325066,15.04014773038076) -- cycle;
\fill[line width=2.pt,color=qqzzff,fill=qqzzff!50,fill opacity=0.6600000262260437] (4.039019710700208,15.639327167133981) -- (2.8628615883250648,16.120147730380758) -- (3.474441588325065,14.475417730380759) -- cycle;
\fill[line width=2.pt,color=ttffqq,fill=ttffqq!30,fill opacity=.81] (2.8628615883250648,16.120147730380758) -- (1.9953215883250648,16.120147730380758) -- (2.444221588325065,15.04014773038076) -- cycle;
\fill[line width=2.pt,color=ffwwqq,fill=ffwwqq!30,fill opacity=0.81] (1.9953215883250648,16.120147730380758) -- (0.5297515883250643,16.04014773038076) -- (1.4278151111079733,14.475413970476055) -- cycle;
\fill[line width=2.pt,color=xfqqff,fill=xfqqff!30,fill opacity=.81] (-0.7075935952311685,15.79782064247529) -- (-1.5721848888920267,14.475413970476055) -- (-0.16698301126716864,13.876234533722833) -- cycle;
\fill[line width=2.pt,color=xfqqff,fill=xfqqff!30,fill opacity=.81] (4.579630294664208,13.717741058381524) -- (4.039019710700208,15.639327167133981) -- (5.444221588325066,15.04014773038076) -- cycle;
\fill[line width=2.pt,color=ttffcc,fill=ttffcc!60,fill opacity=.81] (0.5297515883250643,16.04014773038076) -- (-0.7075935952311685,15.79782064247529) -- (0.4744415883250639,14.475417730380759) -- cycle;
\fill[line width=2.pt,color=qqzzff,fill=qqzzff!50,fill opacity=0.6600000262260437] (-0.9606048888920267,12.830683970476056) -- (-1.5721848888920267,14.475413970476055) -- (-2.1367630112671705,13.311504533722832) -- cycle;
\fill[line width=2.pt,color=ttffqq,fill=ttffqq!30,fill opacity=.81] (-2.1367630112671705,13.311504533722832) -- (-1.6878630112671704,12.231504533722834) -- (-2.5554030112671704,12.231504533722834) -- cycle;
\fill[line width=2.pt,color=ffwwqq,fill=ffwwqq!30,fill opacity=0.81] (-1.6573394884842614,10.66677077381813) -- (-3.122909488484261,10.58677077381813) -- (-2.5554030112671704,12.231504533722834) -- cycle;
\fill[line width=2.pt,color=ttffcc,fill=ttffcc!60,fill opacity=.81] (-1.9408743049280286,9.264367861723596) -- (-3.1782194884842614,9.02204077381813) -- (-3.122909488484261,10.58677077381813) -- cycle;
\fill[line width=2.pt,color=xfqqff,fill=xfqqff!30,fill opacity=.81] (-2.6376089045202615,7.100454665065673) -- (-3.1782194884842614,9.02204077381813) -- (-1.7730176108594033,8.422861337064907) -- cycle;
\fill[line width=2.pt,color=qqzzff,fill=qqzzff!50,fill opacity=0.6600000262260437] (-1.4614507821451195,6.619634101818894) -- (-2.6376089045202615,7.100454665065673) -- (-2.0260289045202615,5.455724665065672) -- cycle;
\fill[line width=2.pt,color=ttffqq,fill=ttffqq!30,fill opacity=.81] (-1.1584889045202615,5.455724665065672) -- (-2.0260289045202615,5.455724665065672) -- (-1.5771289045202614,4.375724665065672) -- cycle;
\fill[line width=2.pt,color=ffwwqq,fill=ffwwqq!30,fill opacity=0.81] (-0.11155890452026096,4.455724665065672) -- (-1.5771289045202614,4.375724665065672) -- (-0.6790653817373524,2.810990905160967) -- cycle;
\fill[line width=2.pt,color=ttffcc,fill=ttffcc!60,fill opacity=.81] (0.5582798018188804,3.053317993066436) -- (-0.6790653817373524,2.810990905160967) -- (0.50296980181888,1.4885879930664352) -- cycle;
\draw [->,>=stealth,line width=1.5pt] (1.9081716794437384,0.8894085563132137) -- node[below] {$s_4$}(0.5029698018188804,1.4885879930664354);
\draw [->,>=stealth,line width=1.5pt] (0.5029698018188804,1.4885879930664354) -- (3.502969801818881,1.4885879930664354);
\draw ($(0.5029698018188804,1.4885879930664354)!.7!(3.502969801818881,1.4885879930664354)$) node[above] {$s_3$};
\draw [->,>=stealth,line width=1.5pt] (3.084329801818881,0.4085879930664357) -- node[below] {$s_8$}(3.951869801818881,0.4085879930664357);
\draw [->,>=stealth,line width=1.5pt] (2.472749801818881,2.053317993066436) -- node[above] {$s_1$}(4.519376279035972,2.053321752971139);
\draw [->,>=stealth,line width=1.5pt] (0.5029698018188804,1.4885879930664354) -- node[above] {$s_2$}(2.472749801818881,2.053317993066436);
\draw [->,>=stealth,line width=1.5pt] (5.417439801818881,0.48858799306643574) -- node[below] {$s_6$}(6.654784985375114,0.7309150809719034);
\draw [->,>=stealth,line width=1.5pt] (3.951869801818881,0.4085879930664357) -- node[below] {$s_7$}(5.417439801818881,0.48858799306643574);
\draw [->,>=stealth,line width=1.5pt] (6.654784985375114,0.7309150809719034) -- node[above] {$s_5$}(1.908171679443738,0.8894085563132137);
\begin{scriptsize}
\draw [fill=rvwvcq] (1.9081716794437384,0.8894085563132137) circle (1pt)node[color=rvwvcq,below] {$V_1$};
\draw [fill=rvwvcq] (3.084329801818881,0.4085879930664357) circle (1pt)node[color=rvwvcq,below] {$V_2$};
\draw [fill=rvwvcq] (2.472749801818881,2.053317993066436) circle (1pt)node[color=rvwvcq,above] {$W_1$};
\draw [fill=rvwvcq] (3.951869801818881,0.4085879930664357) circle (1pt)node[color=rvwvcq,below] {$V_3$};
\draw [fill=rvwvcq] (3.502969801818881,1.4885879930664354) circle (1pt)node[color=rvwvcq,right] {$W_2$};
\draw [fill=rvwvcq] (5.417439801818881,0.48858799306643574) circle (1pt)node[color=rvwvcq,below] {$V_4$};
\draw [fill=rvwvcq] (4.519376279035972,2.053321752971139) circle (1pt)node[color=rvwvcq,above] {$W_3$};
\draw [fill=rvwvcq] (6.654784985375114,0.7309150809719034) circle (1pt)node[color=rvwvcq,below] {$V_5$};
\draw [fill=rvwvcq] (0.5029698018188804,1.4885879930664354) circle (1pt)node[color=rvwvcq,left] {$V_0$};
\draw [fill=ududff] (7.519376279035972,2.053321752971139) circle (1pt);
\draw [fill=ududff] (6.114174401411114,2.6525011897243607) circle (1pt);
\draw [fill=ududff] (1.3675610954797381,2.810994665065671) circle (1pt);
\draw [fill=ududff] (5.472749801818882,2.053317993066436) circle (1pt);
\draw [fill=ududff] (6.907796279035972,3.6980517529711396) circle (1pt);
\draw [fill=ududff] (7.519376279035972,2.053321752971139) circle (1pt);
\draw [fill=ududff] (8.083954401411116,3.2172311897243615) circle (1pt);
\draw [fill=ududff] (8.083954401411116,3.2172311897243615) circle (1pt);
\draw [fill=ududff] (7.635054401411116,4.297231189724362) circle (1pt);
\draw [fill=ududff] (8.502594401411116,4.297231189724362) circle (1pt);
\draw [fill=ududff] (7.604530878628207,5.861964949629066) circle (1pt);
\draw [fill=ududff] (9.070100878628207,5.941964949629066) circle (1pt);
\draw [fill=ududff] (8.502594401411116,4.297231189724362) circle (1pt);
\draw [fill=ududff] (7.888065695071974,7.264367861723598) circle (1pt);
\draw [fill=ududff] (9.125410878628207,7.506694949629066) circle (1pt);
\draw [fill=ududff] (9.070100878628207,5.941964949629066) circle (1pt);
\draw [fill=ududff] (8.584800294664207,9.428281058381522) circle (1pt);
\draw [fill=ududff] (9.125410878628207,7.506694949629066) circle (1pt);
\draw [fill=ududff] (7.720209001003349,8.105874386382288) circle (1pt);
\draw [fill=ududff] (7.408642172289065,9.9091016216283) circle (1pt);
\draw [fill=ududff] (8.584800294664207,9.428281058381522) circle (1pt);
\draw [fill=ududff] (7.973220294664207,11.073011058381523) circle (1pt);
\draw [fill=ududff] (7.105680294664207,11.073011058381523) circle (1pt);
\draw [fill=ududff] (7.973220294664207,11.073011058381523) circle (1pt);
\draw [fill=ududff] (7.524320294664207,12.153011058381523) circle (1pt);
\draw [fill=ududff] (6.058750294664207,12.073011058381523) circle (1pt);
\draw [fill=ududff] (7.524320294664207,12.153011058381523) circle (1pt);
\draw [fill=ududff] (6.626256771881298,13.717744818286228) circle (1pt);
\draw [fill=ududff] (5.388911588325065,13.475417730380759) circle (1pt);
\draw [fill=ududff] (6.626256771881298,13.717744818286228) circle (1pt);
\draw [fill=ududff] (5.444221588325066,15.04014773038076) circle (1pt);
\draw [fill=uuuuuu] (2.973595695071973,8.264367861723597) circle (2.0pt);
\draw[color=uuuuuu] (0.3476704471452668,3.1784014731494254) node {$N$};
\draw [fill=ududff] (4.039019710700208,15.639327167133981) circle (1pt);
\draw [fill=ududff] (2.8628615883250648,16.120147730380758) circle (1pt);
\draw [fill=ududff] (3.474441588325065,14.475417730380759) circle (1pt);
\draw [fill=ududff] (2.8628615883250648,16.120147730380758) circle (1pt);
\draw [fill=ududff] (1.9953215883250648,16.120147730380758) circle (1pt);
\draw [fill=ududff] (2.444221588325065,15.04014773038076) circle (1pt);
\draw [fill=ududff] (1.9953215883250648,16.120147730380758) circle (1pt);
\draw [fill=ududff] (0.5297515883250643,16.04014773038076) circle (1pt);
\draw [fill=ududff] (1.4278151111079733,14.475413970476055) circle (1pt);
\draw [fill=ududff] (-0.7075935952311685,15.79782064247529) circle (1pt);
\draw [fill=ududff] (-1.5721848888920267,14.475413970476055) circle (1pt);
\draw [fill=ududff] (-0.16698301126716864,13.876234533722833) circle (1pt);
\draw [fill=ududff] (4.579630294664208,13.717741058381524) circle (1pt);
\draw [fill=ududff] (4.039019710700208,15.639327167133981) circle (1pt);
\draw [fill=ududff] (5.444221588325066,15.04014773038076) circle (1pt);
\draw[color=ududff] (0.3702146651208945,3.188648844956529) node {$I'_{1}$};
\draw [fill=ududff] (0.5297515883250643,16.04014773038076) circle (1pt);
\draw [fill=ududff] (-0.7075935952311685,15.79782064247529) circle (1pt);
\draw [fill=ududff] (0.4744415883250639,14.475417730380759) circle (1pt);
\draw [fill=ududff] (-0.9606048888920267,12.830683970476056) circle (1pt);
\draw [fill=ududff] (-1.5721848888920267,14.475413970476055) circle (1pt);
\draw [fill=ududff] (-2.1367630112671705,13.311504533722832) circle (1pt);
\draw [fill=ududff] (-2.1367630112671705,13.311504533722832) circle (1pt);
\draw [fill=ududff] (-1.6878630112671704,12.231504533722834) circle (1pt);
\draw [fill=ududff] (-2.5554030112671704,12.231504533722834) circle (1pt);
\draw [fill=ududff] (-1.6573394884842614,10.66677077381813) circle (1pt);
\draw [fill=ududff] (-3.122909488484261,10.58677077381813) circle (1pt);
\draw [fill=ududff] (-2.5554030112671704,12.231504533722834) circle (1pt);
\draw [fill=ududff] (-1.9408743049280286,9.264367861723596) circle (1pt);
\draw [fill=ududff] (-3.1782194884842614,9.02204077381813) circle (1pt);
\draw [fill=ududff] (-3.122909488484261,10.58677077381813) circle (1pt);
\draw [fill=ududff] (-2.6376089045202615,7.100454665065673) circle (1pt);
\draw [fill=ududff] (-3.1782194884842614,9.02204077381813) circle (1pt);
\draw [fill=ududff] (-1.7730176108594033,8.422861337064907) circle (1pt);
\draw [fill=ududff] (-1.4614507821451195,6.619634101818894) circle (1pt);
\draw [fill=ududff] (-2.6376089045202615,7.100454665065673) circle (1pt);
\draw [fill=ududff] (-2.0260289045202615,5.455724665065672) circle (1pt);
\draw [fill=ududff] (-1.1584889045202615,5.455724665065672) circle (1pt);
\draw [fill=ududff] (-2.0260289045202615,5.455724665065672) circle (1pt);
\draw [fill=ududff] (-1.5771289045202614,4.375724665065672) circle (1pt);
\draw [fill=ududff] (-0.11155890452026096,4.455724665065672) circle (1pt);
\draw [fill=ududff] (-1.5771289045202614,4.375724665065672) circle (1pt);
\draw [fill=ududff] (-0.6790653817373524,2.810990905160967) circle (1pt);
\draw [fill=ududff] (0.5582798018188804,3.053317993066436) circle (1pt);
\draw [fill=ududff] (-0.6790653817373524,2.810990905160967) circle (1pt);
\draw [fill=ududff] (5.444221588325066,15.04014773038076) circle (1pt);
\draw [fill=ududff] (4.579630294664208,13.717741058381524) circle (1pt);
\end{scriptsize}
\end{tikzpicture}
\definecolor{ttffcc}{rgb}{0.2,1.,0.8}
\definecolor{uuuuuu}{rgb}{0.26666666666666666,0.26666666666666666,0.26666666666666666}
\definecolor{aqaqaq}{rgb}{0.6274509803921569,0.6274509803921569,0.6274509803921569}
\definecolor{xfqqff}{rgb}{0.4980392156862745,0.,1.}
\definecolor{ududff}{rgb}{0.30196078431372547,0.30196078431372547,1.}
\definecolor{ffwwqq}{rgb}{1.,0.4,0.}
\definecolor{ttffqq}{rgb}{0.2,1.,0.}
\definecolor{qqzzff}{rgb}{0.,0.6,1.}
\definecolor{rvwvcq}{rgb}{0.08235294117647059,0.396078431372549,0.7529411764705882}
\begin{tikzpicture}[scale=1.46, scale=.9, xscale=1,arrow/.style={->,>=stealth,,thick}]
\clip(0,-0.18140970733848455) rectangle (8.055281353641172,3.695894376175167);
\draw[help lines, line width=0.2pt,step=0.2cm] (-0.03714134341956283,0) grid (8.055281353641172,3.695894376175167);
\fill[line width=2.pt,color=qqzzff,fill=qqzzff!50,fill opacity=0.6600000262260437] (2.03671749770347,0.391053935364154) -- (3.6764531369532563,0.391055911064663) -- (3.150166909445395,2.5748286580240394) -- cycle;
\fill[line width=2.pt,color=ttffqq,fill=ttffqq!30,fill opacity=.81] (3.6764531369532563,0.391055911064663) -- (5.003097497703469,0.4997539353641538) -- (4.100487497703469,1.9997539353641534) -- cycle;
\fill[line width=2.pt,color=ffwwqq,fill=ffwwqq!30,fill opacity=0.81] (5.003097497703469,0.4997539353641538) -- (6.5652629286567175,0.7929645476979998) -- (5.44537724641925,2.5748286580240305) -- cycle;
\fill[line width=2.pt,color=xfqqff,fill=xfqqff!30,fill opacity=.81] (7.603097497703469,1.4997539353641538) -- (9.042767246419249,2.5748286580240305) -- (7.509147246419248,4.18352865802403) -- cycle;
\fill[line width=2.pt,color=xfqqff,fill=xfqqff!30,fill opacity=.81] (1.9427672464192494,3.07482865802403) -- (2.03671749770347,0.391053935364154) -- (0.5030974977034695,1.9997539353641534) -- cycle;
\fill[line width=2.pt,color=ttffcc,fill=ttffcc!60,fill opacity=.81] (6.5652629286567175,0.7929645476979998) -- (7.603097497703469,1.4997539353641538) -- (6.747556909445395,2.5748286580240394) -- cycle;
\fill[line width=2.pt,color=qqzzff,fill=qqzzff!50,fill opacity=0.6600000262260437] (8.516481018911389,4.758601404983407) -- (9.042767246419249,2.5748286580240305) -- (10.156216658161174,4.758603380683915) -- cycle;
\fill[line width=2.pt,color=ttffqq,fill=ttffqq!30,fill opacity=.81] (10.156216658161174,4.758603380683915) -- (9.253606658161175,6.258603380683915) -- (10.580251018911387,6.367301404983406) -- cycle;
\fill[line width=2.pt,color=ffwwqq,fill=ffwwqq!30,fill opacity=0.81] (9.46036533667392,8.149165515309438) -- (11.022530767627167,8.442376127643284) -- (10.580251018911387,6.367301404983406) -- cycle;
\fill[line width=2.pt,color=ttffcc,fill=ttffcc!60,fill opacity=.81] (10.166990179369094,9.51745085030317) -- (11.204824748415845,10.224240237969324) -- (11.022530767627167,8.442376127643284) -- cycle;
\fill[line width=2.pt,color=xfqqff,fill=xfqqff!30,fill opacity=.81] (11.110874497131624,12.9080149606292) -- (11.204824748415845,10.224240237969324) -- (9.671204748415844,11.832940237969325) -- cycle;
\fill[line width=2.pt,color=qqzzff,fill=qqzzff!50,fill opacity=0.6600000262260437] (9.471138857881837,12.90801298492869) -- (11.110874497131624,12.9080149606292) -- (10.584588269623762,15.091787707588576) -- cycle;
\fill[line width=2.pt,color=ttffqq,fill=ttffqq!30,fill opacity=.81] (9.25794390887355,14.983089683289085) -- (10.584588269623762,15.091787707588576) -- (9.681978269623762,16.591787707588576) -- cycle;
\fill[line width=2.pt,color=ffwwqq,fill=ffwwqq!30,fill opacity=0.81] (8.119812838670514,16.29857709525473) -- (9.681978269623762,16.591787707588576) -- (8.562092587386296,18.373651817914606) -- cycle;
\fill[line width=2.pt,color=ttffcc,fill=ttffcc!60,fill opacity=.81] (7.524258018339545,17.66686243024845) -- (8.562092587386296,18.373651817914606) -- (7.706551999128222,19.44872654057449) -- cycle;
\fill[line width=2.pt,color=qqzzff,fill=qqzzff!50,fill opacity=0.6600000262260437] (6.172931999128221,21.05742654057449) -- (4.533196359878435,21.057424564873983) -- (5.059482587386296,18.873651817914606) -- cycle;
\fill[line width=2.pt,color=ttffqq,fill=ttffqq!30,fill opacity=.81] (4.533196359878435,21.057424564873983) -- (3.206551999128222,20.94872654057449) -- (4.109161999128222,19.44872654057449) -- cycle;
\fill[line width=2.pt,color=ffwwqq,fill=ffwwqq!30,fill opacity=0.81] (3.206551999128222,20.94872654057449) -- (1.6443865681749736,20.655515928240646) -- (2.764272250412441,18.873651817914613) -- cycle;
\fill[line width=2.pt,color=xfqqff,fill=xfqqff!30,fill opacity=.81] (0.6065519991282224,19.94872654057449) -- (-0.8331177495875579,18.873651817914613) -- (0.700502250412443,17.264951817914614) -- cycle;
\fill[line width=2.pt,color=xfqqff,fill=xfqqff!30,fill opacity=.81] (6.266882250412442,18.373651817914613) -- (6.172931999128221,21.05742654057449) -- (7.706551999128221,19.44872654057449) -- cycle;
\fill[line width=2.pt,color=ttffcc,fill=ttffcc!60,fill opacity=.81] (1.6443865681749736,20.655515928240646) -- (0.6065519991282224,19.94872654057449) -- (1.4620925873862962,18.873651817914606) -- cycle;
\fill[line width=2.pt,color=qqzzff,fill=qqzzff!50,fill opacity=0.6600000262260437] (-0.30683152207969755,16.689879070955236) -- (-0.8331177495875579,18.873651817914613) -- (-1.946567161329483,16.68987709525473) -- cycle;
\fill[line width=2.pt,color=ttffqq,fill=ttffqq!30,fill opacity=.81] (-1.946567161329483,16.68987709525473) -- (-1.0439571613294838,15.18987709525473) -- (-2.3706015220796957,15.081179070955239) -- cycle;
\fill[line width=2.pt,color=ffwwqq,fill=ffwwqq!30,fill opacity=0.81] (-1.2507158398422291,13.299314960629207) -- (-2.812881270795476,13.006104348295361) -- (-2.3706015220796957,15.081179070955239) -- cycle;
\fill[line width=2.pt,color=ttffcc,fill=ttffcc!60,fill opacity=.81] (-1.957340682537403,11.931029625635475) -- (-2.995175251584154,11.22424023796932) -- (-2.812881270795476,13.006104348295361) -- cycle;
\fill[line width=2.pt,color=xfqqff,fill=xfqqff!30,fill opacity=.81] (-2.9012250002999327,8.540465515309444) -- (-2.995175251584154,11.22424023796932) -- (-1.4615552515841532,9.61554023796932) -- cycle;
\fill[line width=2.pt,color=qqzzff,fill=qqzzff!50,fill opacity=0.6600000262260437] (-1.2614893610501454,8.540467491009954) -- (-2.9012250002999327,8.540465515309444) -- (-2.3749387727920706,6.356692768350069) -- cycle;
\fill[line width=2.pt,color=ttffqq,fill=ttffqq!30,fill opacity=.81] (-1.0482944120418587,6.4653907926495595) -- (-2.3749387727920706,6.356692768350069) -- (-1.4723287727920713,4.856692768350069) -- cycle;
\fill[line width=2.pt,color=ffwwqq,fill=ffwwqq!30,fill opacity=0.81] (0.08983665816117714,5.1499033806839165) -- (-1.4723287727920713,4.856692768350069) -- (-0.35244309055460477,3.074828658024039) -- cycle;
\fill[line width=2.pt,color=ttffcc,fill=ttffcc!60,fill opacity=.81] (0.6853914784921464,3.7816180456901947) -- (-0.35244309055460477,3.074828658024039) -- (0.503097497703469,1.9997539353641542) -- cycle;
\draw [->,>=stealth,line width=1.5pt] (3.150166909445395,2.5748286580240394) -- node[above] {$s_1$}(5.44537724641925,2.5748286580240305);
\draw [->,>=stealth,line width=1.5pt] (0.5030974977034695,1.9997539353641534) -- (4.100487497703469,1.9997539353641534);
\draw ($(0.5030974977034695,1.9997539353641534)!.7!(4.100487497703469,1.9997539353641534)$) node[above] {$s_3$};
\draw [->,>=stealth,line width=1.5pt] (7.603097497703469,1.4997539353641538) -- node[below] {$s_4$}(0.503097497703469,1.9997539353641534);
\draw [->,>=stealth,line width=1.5pt] (0.5030974977034695,1.9997539353641534) -- node[above] {$s_2$}(3.150166909445395,2.5748286580240394);
\draw [->,>=stealth,line width=1.5pt] (6.5652629286567175,0.7929645476979998) -- node[below] {$s_5$}(7.603097497703469,1.4997539353641538);
\draw [->,>=stealth,line width=1.5pt] (2.03671749770347,0.391053935364154) -- node[below] {$s_8$}(3.6764531369532563,0.391055911064663);
\draw [->,>=stealth,line width=1.5pt] (3.6764531369532563,0.391055911064663) -- node[below] {$s_7$}(5.003097497703469,0.4997539353641538);
\draw [->,>=stealth,line width=1.5pt] (5.003097497703469,0.4997539353641538) -- node[below] {$s_6$}(6.5652629286567175,0.7929645476979998);
\begin{scriptsize}
\draw [fill=rvwvcq] (2.03671749770347,0.391053935364154) circle (1pt)node[color=rvwvcq,below] {$V_1$};
\draw [fill=rvwvcq] (3.6764531369532563,0.391055911064663) circle (1pt)node[color=rvwvcq,below] {$V_2$};
\draw [fill=rvwvcq] (3.150166909445395,2.5748286580240394) circle (1pt)node[color=rvwvcq,above] {$W_1$};
\draw [fill=rvwvcq] (5.003097497703469,0.4997539353641538) circle (1pt)node[color=rvwvcq,below] {$V_3$};
\draw [fill=rvwvcq] (4.100487497703469,1.9997539353641534) circle (1pt)node[color=rvwvcq,right] {$W_2$};
\draw [fill=rvwvcq] (6.5652629286567175,0.7929645476979998) circle (1pt)node[color=rvwvcq,below] {$V_4$};
\draw [fill=rvwvcq] (5.44537724641925,2.5748286580240305) circle (1pt)node[color=rvwvcq,above] {$W_3$};
\draw [fill=rvwvcq] (7.603097497703469,1.4997539353641538) circle (1pt)node[color=rvwvcq,right] {$V_5$};
\draw [fill=rvwvcq] (0.5030974977034695,1.9997539353641534) circle (1pt)node[color=rvwvcq,left] {$V_0$};
\draw [fill=ududff] (9.042767246419249,2.5748286580240305) circle (1pt);
\draw [fill=ududff] (7.509147246419248,4.18352865802403) circle (1pt);
\draw [fill=ududff] (1.9427672464192494,3.07482865802403) circle (1pt);
\draw [fill=ududff] (6.747556909445395,2.5748286580240394) circle (1pt);
\draw [fill=ududff] (8.516481018911389,4.758601404983407) circle (1pt);
\draw [fill=ududff] (9.042767246419249,2.5748286580240305) circle (1pt);
\draw [fill=ududff] (10.156216658161174,4.758603380683915) circle (1pt);
\draw [fill=ududff] (10.156216658161174,4.758603380683915) circle (1pt);
\draw [fill=ududff] (9.253606658161175,6.258603380683915) circle (1pt);
\draw [fill=ududff] (10.580251018911387,6.367301404983406) circle (1pt);
\draw [fill=ududff] (9.46036533667392,8.149165515309438) circle (1pt);
\draw [fill=ududff] (11.022530767627167,8.442376127643284) circle (1pt);
\draw [fill=ududff] (10.580251018911387,6.367301404983406) circle (1pt);
\draw [fill=ududff] (10.166990179369094,9.51745085030317) circle (1pt);
\draw [fill=ududff] (11.204824748415845,10.224240237969324) circle (1pt);
\draw [fill=ududff] (11.022530767627167,8.442376127643284) circle (1pt);
\draw [fill=ududff] (11.110874497131624,12.9080149606292) circle (1pt);
\draw [fill=ududff] (11.204824748415845,10.224240237969324) circle (1pt);
\draw [fill=ududff] (9.671204748415844,11.832940237969325) circle (1pt);
\draw [fill=ududff] (9.471138857881837,12.90801298492869) circle (1pt);
\draw [fill=ududff] (11.110874497131624,12.9080149606292) circle (1pt);
\draw [fill=ududff] (10.584588269623762,15.091787707588576) circle (1pt);
\draw [fill=ududff] (9.25794390887355,14.983089683289085) circle (1pt);
\draw [fill=ududff] (10.584588269623762,15.091787707588576) circle (1pt);
\draw [fill=ududff] (9.681978269623762,16.591787707588576) circle (1pt);
\draw [fill=ududff] (8.119812838670514,16.29857709525473) circle (1pt);
\draw [fill=ududff] (9.681978269623762,16.591787707588576) circle (1pt);
\draw [fill=ududff] (8.562092587386296,18.373651817914606) circle (1pt);
\draw [fill=ududff] (7.524258018339545,17.66686243024845) circle (1pt);
\draw [fill=ududff] (8.562092587386296,18.373651817914606) circle (1pt);
\draw [fill=ududff] (7.706551999128222,19.44872654057449) circle (1pt);
\draw [fill=uuuuuu] (4.1048247484158455,10.724240237969322) circle (2.0pt);
\draw [fill=ududff] (6.172931999128221,21.05742654057449) circle (1pt);
\draw [fill=ududff] (4.533196359878435,21.057424564873983) circle (1pt);
\draw [fill=ududff] (5.059482587386296,18.873651817914606) circle (1pt);
\draw [fill=ududff] (4.533196359878435,21.057424564873983) circle (1pt);
\draw [fill=ududff] (3.206551999128222,20.94872654057449) circle (1pt);
\draw [fill=ududff] (4.109161999128222,19.44872654057449) circle (1pt);
\draw [fill=ududff] (3.206551999128222,20.94872654057449) circle (1pt);
\draw [fill=ududff] (1.6443865681749736,20.655515928240646) circle (1pt);
\draw [fill=ududff] (2.764272250412441,18.873651817914613) circle (1pt);
\draw [fill=ududff] (0.6065519991282224,19.94872654057449) circle (1pt);
\draw [fill=ududff] (-0.8331177495875579,18.873651817914613) circle (1pt);
\draw [fill=ududff] (0.700502250412443,17.264951817914614) circle (1pt);
\draw [fill=ududff] (6.266882250412442,18.373651817914613) circle (1pt);
\draw [fill=ududff] (6.172931999128221,21.05742654057449) circle (1pt);
\draw [fill=ududff] (7.706551999128221,19.44872654057449) circle (1pt);
\draw [fill=ududff] (1.6443865681749736,20.655515928240646) circle (1pt);
\draw [fill=ududff] (0.6065519991282224,19.94872654057449) circle (1pt);
\draw [fill=ududff] (1.4620925873862962,18.873651817914606) circle (1pt);
\draw [fill=ududff] (-0.30683152207969755,16.689879070955236) circle (1pt);
\draw [fill=ududff] (-0.8331177495875579,18.873651817914613) circle (1pt);
\draw [fill=ududff] (-1.946567161329483,16.68987709525473) circle (1pt);
\draw [fill=ududff] (-1.946567161329483,16.68987709525473) circle (1pt);
\draw [fill=ududff] (-1.0439571613294838,15.18987709525473) circle (1pt);
\draw [fill=ududff] (-2.3706015220796957,15.081179070955239) circle (1pt);
\draw [fill=ududff] (-1.2507158398422291,13.299314960629207) circle (1pt);
\draw [fill=ududff] (-2.812881270795476,13.006104348295361) circle (1pt);
\draw [fill=ududff] (-2.3706015220796957,15.081179070955239) circle (1pt);
\draw [fill=ududff] (-1.957340682537403,11.931029625635475) circle (1pt);
\draw [fill=ududff] (-2.995175251584154,11.22424023796932) circle (1pt);
\draw [fill=ududff] (-2.812881270795476,13.006104348295361) circle (1pt);
\draw [fill=ududff] (-2.9012250002999327,8.540465515309444) circle (1pt);
\draw [fill=ududff] (-2.995175251584154,11.22424023796932) circle (1pt);
\draw [fill=ududff] (-1.4615552515841532,9.61554023796932) circle (1pt);
\draw [fill=ududff] (-1.2614893610501454,8.540467491009954) circle (1pt);
\draw [fill=ududff] (-2.9012250002999327,8.540465515309444) circle (1pt);
\draw [fill=ududff] (-2.3749387727920706,6.356692768350069) circle (1pt);
\draw [fill=ududff] (-1.0482944120418587,6.4653907926495595) circle (1pt);
\draw [fill=ududff] (-2.3749387727920706,6.356692768350069) circle (1pt);
\draw [fill=ududff] (-1.4723287727920713,4.856692768350069) circle (1pt);
\draw [fill=ududff] (0.08983665816117714,5.1499033806839165) circle (1pt);
\draw [fill=ududff] (-1.4723287727920713,4.856692768350069) circle (1pt);
\draw [fill=ududff] (-0.35244309055460477,3.074828658024039) circle (1pt);
\draw [fill=ududff] (0.6853914784921464,3.7816180456901947) circle (1pt);
\draw [fill=ududff] (-0.35244309055460477,3.074828658024039) circle (1pt);
\draw [fill=ududff] (7.706551999128222,19.44872654057449) circle (1pt);
\draw [fill=ududff] (6.266882250412442,18.373651817914613) circle (1pt);
\end{scriptsize}
\end{tikzpicture}
\caption{Another four stable 30-gons of type $E_8$}
\label{fig:E8-2}
\end{figure}
\end{example}
%





\end{document}